\documentclass[openany,twoside,12pt]{book}

\usepackage{letltxmacro}
\usepackage{amsmath}
\usepackage{amssymb}
\usepackage{graphicx}
\usepackage{amsthm}
\usepackage{pifont}
\usepackage{eucal}
\usepackage{mathrsfs}
\usepackage{color}
\usepackage{enumitem}
\usepackage{hyperref}
\usepackage{siunitx}
\usepackage{cancel}
\usepackage[normalem]{ulem}
\usepackage[all]{xy}
\usepackage{tikz-cd}
\usepackage{nameref}
\usepackage[a4paper, twoside, inner=30mm, outer=30mm, top=30mm, bottom=35mm]{geometry}
\usepackage{mathalfa}
\usepackage[cbgreek]{textgreek}
\DeclareFontFamily{OT1}{pzc}{}
\DeclareFontShape{OT1}{pzc}{m}{it}{<-> s * [1.20] pzcmi7t}{}
\DeclareMathAlphabet{\mathpzc}{OT1}{pzc}{m}{it}

\newtheorem{theorem}{Theorem}[chapter]
\newtheorem*{theorem*}{Theorem}
\newtheorem*{``theorem"*}{``Very Inaccurate Theorem"}
\newtheorem{lemma}[theorem]{Lemma}
\newtheorem{proposition}[theorem]{Proposition}
\newtheorem{corollary}[theorem]{Corollary}
\newtheorem{definition}[theorem]{Definition}
\newcommand\at[2]{\left.#1\right|_{#2}}

\newenvironment{PROOF}[1]{{\flushleft {\bfseries Proof of #1}:}}{\hfill\ensuremath{\Box}}

\theoremstyle{definition}

\theoremstyle{remark}

\numberwithin{section}{chapter}
\numberwithin{equation}{chapter}
\makeindex

\newcommand{\brk}[1]{\left(#1\right)}
\newcommand{\Brk}[1]{\left[#1\right]}
\newcommand{\BRK}[1]{\left\{#1\right\}}
\newcommand{\mymat}[1]{\begin{matrix} #1 \end{matrix}}
\newcommand{\Cases}[1]{\begin{cases} #1 \end{cases}}
\newcommand{\bra}{\langle}
\newcommand{\ket}{\rangle}

\newcommand{\secref}[1]{Section~\ref{#1}}

\newcommand{\thmref}[1]{Theorem~\ref{#1}}
\newcommand{\defref}[1]{Definition~\ref{#1}}
\newcommand{\chapref}[1]{Chapter~\ref{#1}}
\newcommand{\propref}[1]{Proposition~\ref{#1}}
\newcommand{\lemref}[1]{Lemma~\ref{#1}}
\newcommand{\corrref}[1]{Corollary~\ref{#1}}

\newcommand{\beq}{\begin{equation}}
	\newcommand{\eeq}{\end{equation}}

\newcommand{\frakA}{\mathfrak{A}}

\newcommand{\frakG}{\mathfrak{G}}
\newcommand{\frakS}{\mathfrak{S}}
\newcommand{\frakT}{\mathfrak{T}}
\newcommand{\frakX}{\mathfrak{X}}

\newcommand{\frakn}{\mathfrak{n}}
\newcommand{\frakt}{\mathfrak{t}}

\newcommand{\frakC}{\mathfrak{C}}
\newcommand{\frakI}{\mathfrak{I}}
\newcommand{\frakP}{\mathfrak{P}}
\newcommand{\frakB}{\mathfrak{B}}

\newcommand{\bbC}{{\mathbb C}}
\newcommand{\bbE}{{\mathbb E}}
\newcommand{\bbF}{{\mathbb F}}
\newcommand{\bbG}{{\mathbb G}}
\newcommand{\bbJ}{{\mathbb J}}
\newcommand{\bbM}{{\mathbb M}}
\newcommand{\bbN}{{\mathbb N}}
\newcommand{\bbP}{{\mathbb P}}
\newcommand{\bbR}{{\mathbb R}}

\newcommand{\bbU}{{\mathbb U}}
\newcommand{\bbZ}{{\mathbb Z}}
\newcommand{\bbL}{\mathbb{L}}

\newcommand{\scrC}{\mathscr{C}}

\newcommand{\scrH}{\mathscr{H}}
\newcommand{\scrN}{\mathscr{N}}
\newcommand{\scrR}{\mathscr{R}}
\newcommand{\scrW}{\mathscr{W}}
\newcommand{\scrL}{\mathscr{L}}
\newcommand{\scrS}{\mathscr{S}}
\newcommand{\scrM}{\mathscr{M}}
\newcommand{\scrU}{\mathscr{U}}

\newcommand{\scrV}{\mathscr{V}}

\newcommand{\calG}{{\mathcal{G}}}
\newcommand{\calL}{{\mathcal{L}}}
\newcommand{\calP}{{\mathcal{P}}}

\newcommand{\calS}{{\mathcal{S}}}

\newcommand{\module}{\scrH}
\newcommand{\bQ}{\mathfrak{Q}}

\newcommand{\bA}{A}
\newcommand{\fbA}{\mathpzc{A}}
\newcommand{\fbC}{\mathpzc{C}}

\newcommand{\bD}{\mathfrak{A}}
\newcommand{\tbD}{\hat{\mathfrak{A}}}
\newcommand{\fbD}{\mathfrak{D}}

\newcommand{\bB}{\frakB}
\newcommand{\bC}{\mathfrak{C}}
\newcommand{\bP}{\mathfrak{G}}
\newcommand{\tbP}{\frakP}
\newcommand{\ttbP}{\tilde{\frakP}}
\newcommand{\cttbP}{\tilde{\mathpzc{P}}}
\newcommand{\trace}{\mathrm{tr}}
\renewcommand{\ker}{\operatorname{Ker}}
\newcommand{\image}{\operatorname{Im}}

\newcommand{\Rm}{\mathrm{Rm}}

\newcommand{\ord}{\operatorname{ord}}
\newcommand{\supp}{\operatorname{supp}}
\newcommand{\End}{{\operatorname{End}}}
\newcommand{\id}{{\operatorname{Id}}}
\newcommand{\Ah}{{\operatorname{A}}}
\newcommand{\Gh}{{\operatorname{G}}}
\newcommand{\MCh}{{\operatorname{MC}}}

\newcommand{\textand}{\quad\text{ and }\quad}
\newcommand{\Textand}{\qquad\text{ and }\qquad}

\newcommand{\overbar}[1]{\mkern 1.5mu\overline{\mkern-1.5mu#1\mkern-1.5mu}\mkern 1.5mu}
\newcommand{\tM}{\tilde{M}}
\newcommand{\dM}{{\partial M}}
\newcommand{\Nzero}{{\bbN_0}}

\newcommand{\tE}{\tilde{\E}}
\newcommand{\tF}{\tilde{\bbF}}
\newcommand{\tg}{\tilde{g}}

\newcommand{\E}{\bbE}

\newcommand{\g}{g}

\newcommand{\PtD}{\bbP^{\frakt}}
\newcommand{\PnD}{\bbP^{\frakn}}
\newcommand{\PttD}{\bbP^{\frakt\frakt}}
\newcommand{\PtnD}{\bbP^{\frakt\frakn}}
\newcommand{\PntD}{\bbP^{\frakn\frakt}}
\newcommand{\PnnD}{\bbP^{\frakn\frakn}}
\newcommand{\nabg}{\nabla^g}
\newcommand{\nabU}{\nabla}
\newcommand{\gD}{g_{\partial}}
\newcommand{\OP}{\mathrm{OP}}

\newcommand{\fbN}{\fbD}

\newcommand{\dU}{d_{\nabU}}
\newcommand{\delU}{\delta_{\nabU}}

\newcommand{\bHg}{\mathpzc{H}_g}

\newcommand{\Hg}{H_{g}}

\newcommand{\dg}{d_{\nabg}}

\newcommand{\deltag}{\delta_{\nabg}}

\newcommand{\dgV}{d_{\nabg}^V}

\newcommand{\deltagV}{\delta_{\nabg}^V}
\newcommand{\dertZero}{\at{\tfrac{d}{dt}}{t=0}}
\newcommand{\dersZero}{\at{\tfrac{d}{ds}}{s=0}}
\newcommand{\Def}{\delta_{g}^*}
\newcommand{\Defd}{\mathrm{Def}_{\gD}}
\newcommand{\Defh}{\mathrm{Def}_{\Ah_{g}}}

\newcommand{\Volume}{\mathrm{Vol}}

\newcommand{\pzcd}{\mathpzc{d}}
\newcommand{\pzcdel}{\text{\textit{\textbf{\textdelta}}}}

\newcommand{\starG}{\star_g}
\newcommand{\starGV}{\starG^V}
\newcommand{\G}{\bP}
\newcommand{\GV}{\bP_V}
\newcommand{\dr}{\partial_r}
\newcommand{\idr}{i_{\partial_r}}
\newcommand{\PttG}{\bbP^{\frakt\frakt}_g}
\newcommand{\PtnG}{\bbP^{\frakt\frakn}_g}
\newcommand{\PntG}{\bbP^{\frakn\frakt}_g}
\newcommand{\PnnG}{\bbP^{\frakn\frakn}_g}

\newcommand{\NN}{\mathrm{N}}

\newcommand{\DD}{\mathrm{D}}

\newcommand{\R}{\mathrm{R}}

\newcommand{\N}{\mathrm{N}}
\newcommand{\D}{\mathrm{D}}

\newcommand{\LkmAll}{\Lambda}
\newcommand{\GkmAll}{\calG}
\newcommand{\WkmAll}{\Omega}
\newcommand{\CkmAll}{\scrC}

\newcommand{\plLkmAll}{\Lambda}

\newcommand{\plWkmAll}{\Omega}
\newcommand{\plCkmAll}{\scrC}

\newcommand{\Lkm}[2]{\LkmAll^{#1,#2}(M)}
\newcommand{\Gkm}[2]{\GkmAll^{#1,#2}(M)}
\newcommand{\Wkm}[2]{\WkmAll^{#1,#2}(M)}
\newcommand{\Ckm}[2]{\CkmAll^{#1,#2}(M)}

\newcommand{\ixi}{i_{\xi^{\sharp}}}
\newcommand{\ixiV}{i^V_{\xi^{\sharp}}}

\newcommand{\plLkm}[2]{\plLkmAll^{#1,#2}(\dM)}

\newcommand{\plWkm}[2]{\plWkmAll^{#1,#2}(\dM)}
\newcommand{\plCkm}[2]{\plCkmAll^{#1,#2}(\dM)}

\newcommand{\dBianchi}{d_{g}}
\newcommand{\delBianchi}{\delta_g}
\newcommand{\dBianchiV}{d_g^V}
\newcommand{\delBianchiV}{\delta_g^V}

\newcommand{\DBianchi}{\pzcd_g}
\newcommand{\DelBianchi}{\pzcdel_g}

\newcommand{\calPG}{\calP_\calG}

\makeatletter
\newcommand*\owedge{\mathpalette\@owedge\relax}
\newcommand*\@owedge[1]{%
	\mathbin{%
		\ooalign{%
			$#1\m@th\bigcirc$\cr
			\hidewidth$#1\m@th\wedge$\hidewidth\cr
		}%
	}%
}
\makeatother

\numberwithin{equation}{section}

\begin{document}
\frontmatter
\title{Hodge Theory for Linearized Boundary-Value Problems on General Geometric Structures}
	
\author{Roee Leder\footnote{roee.leder@mail.huji.ac.il}  \\ Institute of Mathematics \\ The Hebrew University \\ Jerusalem 9190401 Israel} 

\maketitle
\chapter*{\centering Abstract}
{
\makeatletter
\def\blfootnote{\gdef\@thefnmark{}\@footnotetext}
\makeatother
\blfootnote{\textit{2020 Mathematics Subject Classification.}
Primary 58J10, 35N25, 53C21; Secondary 58J32, 35S15, 58A14.}
}
We develop a framework that systematically casts solvability and uniqueness conditions for linearized overdetermined boundary-value problems into cohomological terms. The theory is designed to be applicable without assumptions on the underlying geometric structure and provides tools to study the resulting cohomology explicitly. To achieve this generality, we develop the notion of an elliptic pre-complex, which generalizes the machineries of Hodge theory to sequences of Douglis--Nirenberg systems in the pseudodifferential calculus of boundary-value problems that interact through Green's formulae, the notion of overdetermined ellipticity, and a condition we call the order-reduction property, which relaxes the rigid requirement that the sequence form a cochain complex. This property typically arises from linearized symmetries and constraints, as we demonstrate through several geometric examples that have long resisted analysis, including exterior covariant derivatives, the Killing and Hessian equations, and the Riemann curvature equations.

\tableofcontents
\mainmatter

\chapter{Introduction}
\label{chp:intro_exm}

\section{Background and motivation} 
\label{sec:opening} 
This work addresses overdetermined geometric boundary-value problems. Roughly speaking, these are systems of partial differential equations on a manifold with boundary that have more equations than unknowns. More accurately, they are covariant equations for an unknown section of a vector bundle, whose local representation exhibits this mismatch. Problems arising in differential geometry are often of this kind, since they are characterized by constraints and gauge symmetries that add extra equations to the main system \cite{Aub98}.  

Overdetermined systems are, in particular, not elliptic in the standard sense. The lack of ellipticity usually makes it difficult to identify conditions for their global solvability and uniqueness. This work aims to systemically resolve this difficulty for a large class of such linear systems, by drawing an analogy with the prototypical example provided by \emph{Hodge theory}. To illustrate the picture we aim to generalize, we begin by recalling the components of this example (e.g., \cite{Sch95b,Tay11b}):

Let $M$ be a compact smooth manifold with boundary, with pullback operators $\jmath^*:\Omega^{k}(M)\rightarrow\Omega^{k}(\dM)$ and exterior derivatives $d:\Omega^{k}(M)\rightarrow \Omega^{k+1}(M)$ acting on differential forms. By setting $\Omega^{k}_{\DD}(M):=\Omega^{k}(M)\cap\ker\jmath^*$, these operators fit into a \emph{cochain complex}:
\beq 
\begin{tikzcd}
\dots \arrow[r] &[-1.2em]
\Omega^{k-1}_{\DD}(M) \arrow[r, "d"] &
\Omega^{k}_{\DD}(M) \arrow[r, "d"] & 
\Omega^{k+1}_{\DD}(M) \arrow[r, "d"] &
\Omega^{k+2}_{\DD}(M) \arrow[r] &[-1.2em]
\dots
\end{tikzcd}
\label{eq:Hodge_complex} 
\eeq
We denote its cohomology by $\module_{\DD}^{\bullet}$. What Hodge theory brings to this purely differential-topological picture is the following: upon introducing a Riemannian metric $g$ on $M$, one can canonically identify $\module_{\DD}^{\bullet}$ with the $L^{2}$-cokernels of the overdetermined boundary-value problems 
\beq
\begin{aligned}
& d\omega = \eta, && \text{in } M,
\\
& \jmath^*\omega=0, 
\qquad && \text{on } \dM,
\end{aligned}
\label{eq:deRham}
\eeq
where $\omega\in\Omega^{k}(M)$ is the unknown and $\eta\in\Omega_{\DD}^{k+1}(M)$ is a prescribed form. This identification is made possible by exploiting the inherent gauge freedom of the equations \eqref{eq:deRham}, encoded by the complex \eqref{eq:Hodge_complex}: namely, if $\omega\in\Omega_{\DD}^{k}(M)$ is a solution, then $\omega+d\rho$ is also a solution for any $\rho\in\Omega^{k-1}_{\DD}(M)$. Upon introducing $\delta_{g}:\Omega^{k}(M)\rightarrow\Omega^{k-1}(M)$, the formal $L^2$-adjoint of the exterior derivative, the condition $\delta_{g}\omega=0$ provides a canonical mechanism for fixing this gauge, allowing us to write (up to the canonical isomorphism):  
\beq 
\module^{k}_{\D}=\ker(d,\delta_{g})\cap\Omega_{\DD}^{k}(M).
\label{eq:Hodge_iso}
\eeq 
This, in turn, yields the following statement concerning \eqref{eq:deRham}, which is the central subject of our generalization:
\begin{theorem*}[See, e.g., \cite{Sch95b}]
\label{thm:Hodge_cohomology}
Given a prescribed $\eta\in\Omega_{\DD}^{k+1}(M)$, the boundary-value problem \eqref{eq:deRham} admits a solution $\omega\in\Omega_{\DD}^{k}(M)$ satisfying $\delBianchi\omega=0$ if and only if
\[
d\eta = 0, \qquad \eta \perp_{L^2} \module_{\DD}^{k+1}.
\]
Any solution satisfying $\delBianchi\omega=0$ is unique modulo $\module_{\DD}^{k}$.
\end{theorem*} 
Since the obstructions to solvability and uniqueness of the problem \eqref{eq:deRham} are now encoded by the kernel of a partial differential equation parameterized by an arbitrary metric, Hodge theory thus paves the way for its analysis through a variety of tools from different disciplines. These include homological algebra, index theory, Bochner techniques, and methods from differential topology.

Motivated by the success of this cohomological picture, our objective is to extend it to other overdetermined geometric problems of the form suggested by \eqref{eq:deRham}; i.e., roughly outlined at this stage, those that fall within the general template:
\beq
\begin{aligned}
& A\omega = \eta \qquad && \text{in } M,
\\
& B\omega=0, 
\qquad && \text{on } \dM,
\end{aligned}
\label{eq:linear_problem}
\eeq
where $A$ is an overdetermined linear operator between sections of vector bundles over a manifold $M$, arising from some geometric problem, and $B$ is a boundary operator ``naturally'' associated with $A$. We shall provide a list of concrete examples momentarily in \secref{sec:main_results_examples_1}.

In analogy with \eqref{eq:deRham}, the first step toward such a generalization is to establish a mechanism for embedding the linear map $A$ into a cochain complex analogous to \eqref{eq:Hodge_complex}; namely, a sequence of linear maps $(A_{\bullet})$ satisfying
\beq
\image A_{\alpha} \subseteq \ker A_{\alpha+1}, \qquad \alpha = 0,1,2,\dots
\label{eq:A0A1_intro} 
\eeq
with $A_{\alpha}=A$ for some $\alpha$, and whose associated {cohomology modules},
\[
\module^{\alpha+1} \simeq \ker A_{\alpha+1} / \image A_{\alpha},
\]
encode certain structural features of the underlying space or, in some cases, vanish entirely. Indeed, with such a cochain complex in hand, one can at once identify the corresponding gauge freedom in \eqref{eq:linear_problem} as
\beq 
\rho\mapsto\omega+A_{\alpha-1}\rho, 
\label{eq:gauge_freedom}
\eeq 
as well as the solvability conditions, which become: 
\[
A_{\alpha+1}\eta=0 \textand \eta\,\bot\,\module^{\alpha+1},
\]
where the orthogonality is understood in an appropriate sense. Upon identifying another operator $G$ that acts as a gauge-fixing mechanism for \eqref{eq:gauge_freedom}, playing the same role as the adjoint $\delta_{g}$ in \eqref{eq:deRham}, the goal is to show that the cohomology modules of the complex naturally admit a further characterization analogous to that provided by Hodge theory in \eqref{eq:Hodge_iso}.

Driven by this picture, many studies have been devoted to constructing and analyzing such cochain complexes in a variety of geometric contexts. These include, but are certainly not limited to, direct generalizations of Hodge theory, such as the theories of elliptic complexes and Dirac operators \cite{AS68,Sch95b,AKM06,Tay11b,BB12,SS19}; the theory of compatibility operators for overdetermined systems \cite{Cal61,Gol67,Spe69,BE69,GG88A,DS96}; Bernstein--Gelfand--Gelfand (BGG) sequences and related studies \cite{Eas00,CSS01,CS09,AH21,CELM21,Hu26}; exterior differential systems \cite{BCGGG91}; and potential theory in the context of constrained problems \cite{DPR16,Bog19,GR22}.

The machinery developed in previous studies, however, is usually local in nature and available only under additional rigid assumptions on the underlying geometric structure, most often constant curvature, homogeneity, or K\"ahler geometry. Without such assumptions, it is often difficult to arrange for a sequence of operators arising from a geometric problem to form a cochain complex supporting a Hodge-like framework. We shall demonstrate this shortcoming extensively below.
\section{Main examples and results}
\label{sec:main_results_examples_1}
\subsection{Elliptic pre-complexes and the order reduction property} 

The circumstances alluded to above are ubiquitous in geometry: overdetermined geometric problems give rise to sequences of operators that form cochain complexes under certain geometric assumptions, but cease to do so once those assumptions are relaxed. To overcome this shortcoming, we develop here a generalized Hodge theory that canonically \emph{lifts} such sequences so as to recover a cochain complex, and hence the cohomological picture. 

The key idea behind the generalized Hodge theory is to weaken the classical notion of an \emph{elliptic complex} \cite{AS68,Tay11b} so that it applies more broadly, leading to the notion of an \emph{elliptic pre-complex}. In a nutshell, what makes this generalization possible is the observation that, in general geometries, while the compositions of consecutive operators in a given sequence do not vanish, they do result in appropriate lower-order terms involving curvature, thus manifesting what we call the \emph{order-reduction property}. Essentially, it amounts to weakening the requirement in \eqref{eq:A0A1_intro} to
\[
\ord(A_{k+1}A_{k})\leq \ord(A_{k}),
\]
where $\ord(\cdot)$ denotes the order of the differential operator $A_{k}$, or, more generally, of a pseudodifferential one. We emphasize that this property must be verified at the operator level and cannot be deduced merely from the vanishing of the symbol, highlighting the fact that the lower-order terms arise from curvature. This idea was already explored in prototypical form in our previous work \cite{KL23} and is fully developed here, accommodating a much larger class of examples that could not previously be treated within this framework. 

The theory is based on the pseudodifferential calculus of boundary-value problems \cite{Bou71,RS82,Gru96}, thereby avoiding local techniques that require assumptions on the underlying structure. This calculus is needed because the standard pseudodifferential calculus does not apply directly on manifolds with nonempty boundary. Since this subject is technical in nature, we defer its presentation and treatment to a separate expository section (\secref{sec:Hodge_intro}) and to the technical body of the work itself (\secref{sec:per_pseudo}).

Indeed, we shall lay out the generalized Hodge theory in its abstract form, along with a precise formulation of elliptic pre-complexes, the order-reduction property and the other required analytical aspects, in a separate expository chapter (\secref{sec:Hodge_intro}). Here, instead, we take a more case-study-oriented approach, presenting the machinery through several geometric examples that have long resisted analysis. Before doing so, a few remarks are in order:

\begin{enumerate}
\item We state the examples in streamlined form, deliberately avoiding several technical aspects. A detailed exposition of each, including further extensions and subtleties, is deferred to the more extensive overview \secref{sec:examples_intro} and the subsequent technical chapters.

\item The list is not exhaustive, either in scope or in sophistication, and is intended to illustrate the reach of the theory and the class of geometric problems it is designed to encompass.  

\item The examples studied are presented here as ends in themselves. In the next \secref{sec:outlook}, we shall elaborate on the common nonlinear geometric origin of these problems---namely, how they arise from the linearization of nonlinear problems---and discuss the utility of our theory in this context.

\item Throughout, we use standard typography (e.g., $A$) for the operators in the original sequence and script typography (e.g., $\mathpzc{A}$) for their pseudodifferential \emph{lifts}. We defer a precise discussion and characterization of the lifting procedure to our exposition of the abstract theory in \secref{sec:Hodge_intro}, while emphasizing here that it is specifically designed to correct the failure of the original sequence to form a cochain complex without imposing geometric assumptions.

\end{enumerate}
\subsection{Exterior covariant derivatives} 

Let $\bbU\rightarrow M$ be a Riemannian vector bundle, and let $\nabU$ be a connection on $\bbU$. The \emph{exterior covariant derivatives} and their adjoints,  
\[
\dU:\Omega^{k}(M;\bbU) \to \Omega^{k+1}(M;\bbU), \qquad \delU:\Omega^{k+1}(M;\bbU) \to \Omega^{k}(M;\bbU),
\]  
arise in various geometric and analytical contexts, including Bochner techniques \cite[Ch.~9]{Pet16}, gauge theory \cite[Ch.~1.4,6]{RS17}, \cite[App.~C.6]{Tay11b}, and harmonic maps \cite{EL83}. Like in the classical theory, these naturally give rise to boundary-value problems falling within the scope of \eqref{eq:linear_problem}:
\beq 
\begin{aligned}
& \dU\omega = \eta, && \text{in } M,\\
& \jmath^*\omega=0, 
\qquad && \text{on } \dM .
\end{aligned}
\label{eq:boundary_exterior1} 
\eeq 
When $\nabU$ is locally flat, namely, when its associated curvature endomorphism $R_{\nabU}\in\Omega^{2}(M;\End(\bbU))$ satisfies $R_{\nabU}=0$, the operators $\dU$ fit into a cochain complex analogous to the one in \eqref{eq:Hodge_complex}:
\beq 
\begin{tikzcd}
\dots \arrow[r] &[-1.2em]
\Omega^{k-1}_{\DD}(M;\bbU) \arrow[r, "\dU"] &
\Omega^{k}_{\DD}(M;\bbU) \arrow[r, "\dU"] & 
\Omega^{k+1}_{\DD}(M;\bbU) \arrow[r, "\dU"] &
\Omega^{k+2}_{\DD}(M;\bbU) \arrow[r] &[-1.2em]
\dots
\end{tikzcd}
\label{eq:exterior_complex} 
\eeq
It is then straightforward to observe that the problem admits a Hodge-theoretic interpretation analogous to that of \eqref{eq:deRham}. For a general connection $\nabU$, however, lower-order curvature terms prevent the sequence from being a cochain complex, and the Hodge-like theory breaks down. 

The machinery developed in this work shall remedy this failure for an arbitrary connection $\nabU$: 
\begin{theorem}
\label{thm:exterior_examples_1}
Let $\bbU\rightarrow M$ be a Riemannian vector bundle over a compact manifold with boundary, and let $\nabU$ be any connection on $\bbU$. Then the following holds:

\begin{enumerate}
\item \emph{(Lifted complex and cohomology)} There exists a cochain complex of continuous linear maps
\beq 
\begin{tikzcd}
\dots \arrow[r] &[-1.2em]
\Omega^{k-1}_{\DD}(M;\bbU) \arrow[r, "\pzcd_{\nabU}"] &
\Omega^{k}_{\DD}(M;\bbU) \arrow[r, "\pzcd_{\nabU}"] & 
\Omega^{k+1}_{\DD}(M;\bbU) \arrow[r, "\pzcd_{\nabU}"] &
\Omega^{k+2}_{\DD}(M;\bbU) \arrow[r] &[-1.2em]
\dots
\end{tikzcd}
\label{eq:exterior_complexlift} 
\eeq
coinciding with \eqref{eq:exterior_complex} when $R^{\nabU}=0$, and whose cohomology modules $\module^{\bullet}_{\DD}(\bbU)$ are given by:
\beq 
\begin{split} 
&\module^{0}_{\DD}(\nabU)=\ker(\nabU,(\cdot)|_{\dM}), \\
&\module^{1}_{\DD}(\nabU)=\ker(\dU,\delU)\cap\Omega^{1}_{\DD}(M;\bbU), \\
&\module^{k}_{\DD}(\nabU)=\ker(\dU,\pzcdel_{\nabU})\cap\Omega^{k}_{\DD}(M;\bbU), \qquad k\geq 2,
\end{split} 
\label{eq:exterior_cohomology}
\eeq 
where $\pzcdel_{\nabU}:\Omega^{k}(M;\bbU)\rightarrow\Omega^{k-1}(M;\bbU)$ is the formal $L^{2}$-adjoint of $\pzcd_{\nabU}$ in \eqref{eq:exterior_complexlift}. 

\item \emph{(Solvability and uniqueness)} Given $\eta\in\Omega^{k+1}_{\DD}(M;\bbU)$, the boundary-value problem \eqref{eq:boundary_exterior1} admits a solution $\omega\in\Omega^k_{\DD}(M;\bbU)$ satisfying $\pzcdel_{\nabU}\omega=0$ if and only if
\[
\pzcd_{\nabU}\eta=0, \qquad \eta\bot_{L^{2}} \module^{k+1}_{\DD}(\nabU).
\]
Any solution satisfying $\pzcdel_{\nabU}\omega=0$ is unique modulo $\module^{k}_{\DD}(\nabU)$.
\end{enumerate}
\end{theorem}
The proof of this theorem will be contained in the more general treatment of \secref{sec:exterior_intro} and \secref{sec:exterior_covariant_derivatives}. 

We remark that the expression in \eqref{eq:exterior_cohomology} is precise, in the sense that the cohomology modules are indeed expressible in terms of the operators in the original complex \eqref{eq:exterior_complex}. This is an important feature of the theory, one that is prominent in all the examples below and that we shall discuss in greater depth in the sequel, \secref{sec:outlook}, and, indeed, throughout this work.

The immediate implications of this expression are that $\scrH^{0}_{\DD}(\nabU)$ is nothing but the space of parallel sections vanishing on the boundary, which is identically zero when the boundary is nonempty. Likewise, $\scrH^{1}_{\DD}(\nabU)$ is the space of tangential harmonic sections, which can be shown to vanish using a Bochner technique under appropriate estimates on the interior curvature $R_{\nabU}$ and convexity assumptions on the boundary. For the higher-degree cohomology modules, we shall elaborate in the body of the work (\secref{sec:main_results_intro}) on how the kernel of the lifted adjoint $\pzcdel_{\nabU}$ can also be expressed in terms of $\delU$, although this is more involved.
\subsection{The Killing equation} 
Given a Riemannian metric $g$ on $M$, the Killing equation asks whether a prescribed symmetric tensor $\sigma\in S^{2}(M)$ can be written as
\beq 
\Def Y=\sigma , \qquad \text{in } M,
\label{eq:Killing_1} 
\eeq 
where $\Def:\frakX(M)\rightarrow S^{2}(M)$ is the Killing operator, namely, $\Def Y:=\tfrac12\calL_Yg$. In geometric terms, the question underlying this equation is whether a symmetric tensor field is an infinitesimal deformation of the metric generated by a vector field. 

One can supplement this problem with boundary conditions, most prominently the vanishing trace on the boundary, corresponding to the fact that the associated flow fixes the boundary:
\beq 
Y|_{\dM}=0, \qquad \text{on } \dM.
\label{eq:Killing_2} 
\eeq
Due to its significance, the Killing equation has been studied in many geometric and analytic contexts \cite{Cal61,Gol67,BE69,GG88A,DS96,CELM21,PSU23,KL23}. We shall provide more historical and geometric context in \secref{sec:KillingHessianIntro}. For the expository purposes here, it suffices to note that previous studies have found that, when the metric $g$ is K\"ahler or has constant curvature, or, more generally, when $(M,g)$ is a symmetric space, the Killing operator can be embedded into a cochain complex

\beq 
\begin{tikzcd}
0 \arrow[r] &[-1.2em]
\frakX_{\DD}(M) \arrow[r, "\Def"] &
S^{2}_{\DD}(M) \arrow[r, "\Hg"] & 
\scrC^{2,2}_{\DD}(M) \arrow[r, "\dBianchi"] &
\Omega_{\DD}^{3}(M;\Lambda^{2}T^*M) \arrow[r, "\dBianchi"] &
\cdots
\end{tikzcd}
\label{eq:Killing_complex} 
\eeq
where $\Ckm{2}{2}$ is the module of algebraic curvature tensors, $\Hg:S^{2}(M)\rightarrow\Ckm{2}{2}$ is a second-order differential operator, and the subscript $\DD$ indicates a specific choice of boundary data, which will be detailed in \secref{sec:KillingHessianIntro}. As above, such cochain complexes in turn yield Hodge-theoretic frameworks for the solvability and uniqueness of \eqref{eq:Killing_1}--\eqref{eq:Killing_2}. For a general metric, however, curvature terms again prevent the sequence \eqref{eq:Killing_complex} from being a cochain complex, and such analyses break down.

As in the case of exterior covariant derivatives (\thmref{thm:exterior_chomology_intro}), our theory provides the following theorem, which replaces \eqref{eq:Killing_complex} with a lifted complex acting on the same domains and valid for arbitrary background geometry.
\begin{theorem}
\label{thm:Dirichlet_Lie} 
Let $(M,g)$ be a compact Riemannian manifold with boundary. Then the following holds:

\begin{enumerate}
\item \emph{(Lifted complex and cohomology)} There exists a cochain complex of continuous linear maps
\beq 
\begin{tikzcd}
0 \arrow[r] &[-1.2em]
\frakX_{\DD}(M) \arrow[r, "\Def"] &
S^{2}_{\DD}(M) \arrow[r, "\mathpzc{H}_{g}"] & 
\scrC^{2,2}_{\DD}(M) \arrow[r, "\pzcd_{g}"] &
\Omega^{3}_{\DD}(M;\Lambda^{2}T^*M) \arrow[r, "\pzcd_{g}"] &
\cdots
\end{tikzcd}
\label{eq:Killing_complexlifted} 
\eeq
coinciding with \eqref{eq:Killing_complex} when $g$ is locally flat, so that its cohomology modules $\mathscr{K}^{\bullet}_{\DD}(g)$ are given by: 
\beq 
\begin{split} 
\mathscr{K}^{0}_{\DD}(g)&= \ker(\Def,(\cdot)|_{\dM}), \\
\mathscr{K}^{1}_{\DD}(g)&= \ker(\Hg,\delta_{g})\cap S^{2}_{\DD}(M), \\
\mathscr{K}^{2}_{\DD}(g)&= \ker(\dBianchi,\mathpzc{H}_{g}^*)\cap \scrC^{2,2}_{\DD}, \\
\mathscr{K}^{k}_{\DD}(g)&= \ker(\dBianchi,\pzcdel_{g})\cap \Omega_{\DD}^{k}(M;\Lambda^{2}T^*M), \qquad \text{for } k\geq 3,
\end{split}
\label{eq:Kcohomology}
\eeq 
where $\delta_{g}:S^{2}(M)\rightarrow\frakX(M)$ is tensor divergence and $\pzcdel_{g}$ is the adjoint of $\pzcd_{g}$. 

\item \emph{(Solvability and uniqueness)} Given $\sigma\in S^{2}_{\DD}(M)$, the boundary-value problem \eqref{eq:Killing_1}--\eqref{eq:Killing_2} admits a solution $Y\in\frakX_{\DD}(M)$ if and only if
\[
\mathpzc{H}_{g}\sigma=0, \qquad \sigma\,\bot_{L^{2}}\,\mathscr{K}^{1}_{\DD}(g).
\]
The solution is unique modulo $\mathscr{K}^{0}_{\DD}(g)$.
\end{enumerate}
\end{theorem}
The precise definition of the modules and the associated boundary conditions, as well as the proof, will be contained in the more general treatment of \secref{sec:KillingHessianIntro} and \secref{sec:KillingHessianBody}.

Again, we emphasize that the cohomology modules \eqref{eq:Kcohomology} are expressible in terms of the operators in the original sequence. To illustrate the implications of this in the present example, note that, when the boundary is nonempty, $\mathscr{K}_{\DD}^0(g)=\{0\}$ for any $g$, as there are no Killing fields that vanish on the boundary. Moreover, by means of a Bochner technique, it can be shown that $\mathscr{K}^1_{\DD}(g)=\{0\}$ under assumptions on the interior curvature and convexity of the boundary (see, e.g., the calculation for \cite[Thm.~5]{Led25B}). For the higher-degree modules, as in the case of the exterior covariant derivative, we shall elaborate further in \secref{sec:general_intro} on how the kernels of the lifted terms $\pzcdel_{g}$ and $\mathpzc{H}_{g}^*$ can be refined so as to be expressed in terms of the original adjoints $\delta_{g}$ and $\Hg^*$. This can also yield vanishing results in certain cases.

\subsection{The Hessian equation} 
Closely related to the Killing equation, the Hessian equation asks whether a prescribed symmetric tensor $\sigma$ can be written as
\beq 
\Hg f=\sigma ,
\label{eq:Hessian_1} 
\eeq 
where $\Hg:C^{\infty}(M)\rightarrow S^{2}(M)$ is the usual Riemannian Hessian acting on functions. The requirement that the gradient vector field vanishes identically on the boundary translates into the Cauchy data:
\beq 
f|_{\dM}=0, \qquad \dr f|_{\dM}=0 \qquad \text{on } \dM,
\label{eq:Hessian_2} 
\eeq 
where $\dr$ is the normal derivative on the boundary. Under restrictive curvature assumptions, such as a locally symmetric structure or suitable nondegeneracy of the curvature operator (cf.~\cite{BE69,Bry13a,Bry25} and references therein), the solvability conditions can be encoded by a cochain complex of the form
\beq 
\begin{tikzcd}
0 \arrow[r] &[-1.2em]
C^{\infty}_{\DD}(M) \arrow[r, "\Hg"] &
S^{2}_{\DD}(M) \arrow[r, "\dBianchi"] & 
\Omega^{2,1}_{\DD}(M;T^*M) \arrow[r, "\dBianchi"] & \cdots
\end{tikzcd}
\label{eq:Hessian_complex} 
\eeq
Once these assumptions are removed, however, curvature terms again obstruct the construction.

Analogously to the previous examples, we prove the following.
\begin{theorem}
\label{thm:Hessian_intro}
Let $(M,g)$ be a compact Riemannian manifold with boundary. Then the following holds:

\begin{enumerate}
\item \emph{(Lifted complex and cohomology)} There exists a cochain complex of continuous linear maps
\beq 
\begin{tikzcd}
0 \arrow[r] &[-1.2em]
C^{\infty}_{\DD}(M) \arrow[r, "\Hg"] &
S^{2}_{\DD}(M) \arrow[r, "\pzcd_{g}"] & 
\Omega^{2,1}_{\DD}(M;T^*M) \arrow[r, "\pzcd_{g}"] &
\cdots
\end{tikzcd}
\label{eq:Hessian_complexlifted} 
\eeq
coinciding with \eqref{eq:Hessian_complex} when $g$ is locally flat, so that its cohomology modules $\mathscr{G}^{\bullet}_{\DD}(g)$ satisfy
\[
\begin{split} 
&\mathscr{G}^{0}_{\DD}(g)= \ker(\Hg,(\cdot)|_{\dM},\partial_{r}(\cdot)|_{\dM}), \\&
\mathscr{G}^{1}_{\DD}(g)= \ker(\dBianchi,\Hg^*)\cap S^{2}_{\DD}(M),  \\
&\mathscr{G}^{k}_{\DD}(g)=\ker(\dBianchi,\pzcdel_{g})\cap \Omega_{\DD}^{k}(M;T^*M), \qquad \text{for } k\geq 2,
\end{split}
\]
where $\Hg^*:S^{2}(M)\rightarrow C^{\infty}(M)$ is the adjoint of $\Hg$ and $\pzcdel_{g}$ is the adjoint of $\pzcd_{g}$. 

\item \emph{(Solvability and uniqueness)} Given $\sigma\in S^{2}_{\DD}(M)$, the boundary-value problem \eqref{eq:Hessian_1}--\eqref{eq:Hessian_2} admits a solution $f\in C^{\infty}_{\DD}(M)$ if and only if
\[
\pzcd_{g}\sigma=0, \qquad \sigma\,\bot_{L^{2}}\,\mathscr{G}^{1}_{\DD}(g).
\]
The solution is unique modulo $\mathscr{G}^{0}_{\DD}(g)$.
\end{enumerate}
\end{theorem}
Again, the precise definition of the modules and the associated boundary conditions, as well as the proof, will be contained in the more general treatment of \secref{sec:KillingHessianIntro} and \secref{sec:KillingHessianBody}.

The same remark concerning the cohomology modules in the previous examples holds verbatim here. In particular, we observe that $\mathscr{G}^{0}_{\DD}(g)=\{0\}$ when the boundary is nonempty, by the same argument as for the Killing equation, and that $\mathscr{G}^{1}_{\DD}(g)=\{0\}$ under topological and convex boundary assumptions, since these tensors are, in particular, Codazzi tensors (i.e., satisfy $\dBianchi\sigma=0$), for which rigidity results have been widely studied (see, e.g., \cite{Bou81,Bes87}).

\subsection{Linearized Riemann curvature} 
For the next example, consider the correspondence between a Riemannian metric and its Riemann curvature tensor, which may be viewed as a map,
\[
g\mapsto \Rm_g:\scrM(M)\rightarrow\Ckm{2}{2}.
\] 
Here, $\scrM(M)$ is the manifold of all Riemannian metrics on $M$ (\cite{Ebi70,GM91,GMN92,CR13}). The \emph{Riemann curvature equations}~\cite{DY86,Bry13b,Bry15} are, as their name suggests, the nonlinear problem of prescribing the range of this map; i.e., given an algebraic curvature tensor $\mathrm{T}\in\Ckm{2}{2}$, the problem is to find a Riemannian metric $g \in \scrM(M)$ yielding 
\[
\Rm_g = \mathrm{T}, \qquad \text{in } M. 
\]
In the presence of a boundary, we supplement the interior equation with \emph{Cauchy data} \cite{And08,AH08,AH22}:
\[
\gD = \gamma, \qquad \Ah_g = \mathrm{K}, \qquad \text{in } \dM. 
\]
The prescribed boundary data consist of a Riemannian metric $\gamma \in \scrM(\dM)$ and a symmetric tensor field $\mathrm{K} \in S^{2}(\dM)$. Here, $\gD = \PttD g$ denotes the pullback of $g$ to the boundary, where
\[
\PttD: S^{2}(M) \to S^{2}(\dM)
\]
is the linear map representing the tangential projection of a symmetric tensor field onto the boundary. The nonlinear mapping
\[
g \mapsto \Ah_g: \scrM(M) \to S^{2}(\dM)
\]
assigns to each Riemannian metric $g$ the induced second fundamental form $\Ah_g \in S^2(\dM)$ of the boundary. Altogether, we obtain the nonlinear overdetermined boundary-value problem for an unknown $g\in\scrM(M)$:
\[
\begin{gathered}
\Rm_g =\mathrm{T} , \\ 
\gD =\gamma, \qquad 
\Ah_{g} =\mathrm{K} .
\end{gathered}
\]
Linearizing the interior equations at a fixed background metric $g\in\scrM(M)$ yields a linear second-order differential operator $\D\Rm_{g}:S^{2}(M)\rightarrow\Ckm{2}{2}$ given by \cite[p.~559]{Tay11b}:
\[
\D\Rm_g\sigma:=\dertZero\Rm_{g+t\sigma},
\qquad
\sigma\in S^2(M).
\]
Linearizing the boundary data as well yields a boundary-value problem for the linearized Riemann curvature equations that falls within the template in \eqref{eq:linear_problem}:
\beq 
\begin{aligned}
& \D\Rm_{g}\sigma = T \qquad && \text{in } M,
\\
& \PttD\sigma=0,  \qquad \D\Ah_{g}\sigma=0 
\qquad && \text{on } \dM,
\end{aligned}
\label{eq:Riem_1}
\eeq 
where $\sigma\mapsto\PttD\sigma\in S^{2}(\dM)$ and
$\D\Ah_g\sigma=\dertZero \Ah_{g+t\sigma}\in S^{2}(\dM)$ are the linearized boundary operators.

When the background metric has constant curvature, it is known that the operators governing this problem can be organized within a cochain complex:
\beq 
\begin{tikzcd}
0 \arrow[r] &[-1.2em]
\frakX_{\DD}(M) \arrow[r, "\Def"] &
S^{2}_{\DD}(M) \arrow[r, "\D\Rm_{g}"] & 
\scrC^{2,2}_{\DD}(M) \arrow[r, "\dBianchi"] &
\Omega_{\DD}^{3}(M;\Lambda^{2}T^*M) \arrow[r, "\dBianchi"] &\cdots
\end{tikzcd}
\label{eq:curvture_complex} 
\eeq
which yields, in turn, can be shown to yield a Hodge-theoretic framework \cite{Cal61,GG88A,Eas00,KL22,KL21a}. We shall elaborate on the precise definition of the modules and the boundary conditions in \secref{sec:prescribed_curvature_intro}. For a general metric with boundary, however, the candidate sequence is no longer a cochain complex, and this analysis fails.

Analogously to the other examples, we shall prove the following theorem. For simplicity and expository purposes, we assume in the statement here that $\Rm_{g}|_{\dM}=0$, while allowing the curvature tensor to be arbitrary in the interior. In \secref{sec:prescribed_curvature_intro}, we explain how this assumption may be relaxed.
\begin{theorem}
\label{thm:Riem_intro}
Let $(M,g)$ be a compact Riemannian manifold with boundary satisfying $\Rm_{g}|_{\dM}=0$. Then the following holds:

\begin{enumerate}
\item \emph{(Lifted complex and cohomology)} There exists a cochain complex of continuous linear maps
\beq 
\begin{tikzcd}
0 \arrow[r] &[-1.2em]
\frakX_{\DD}(M) \arrow[r, "\Def"] &
S^{2}_{\DD}(M) \arrow[r, "\mathpzc{D}\mathpzc{Rm}_{g}"] & 
\scrC^{2,2}_{\DD}(M) \arrow[r, "\pzcd_{g}"] &
\Omega_{\DD}^{3}(M;\Lambda^{2}T^*M) \arrow[r, "\pzcd_{g}"] &\cdots
\end{tikzcd}
\label{eq:curvture_complexlifted} 
\eeq
coinciding with \eqref{eq:curvture_complex} when $g$ is locally flat, so that its cohomology modules $\mathscr{B}^{\bullet}_{\DD}(g)$ satisfy
\[
\begin{split} 
&\mathscr{B}^{0}_{\DD}(g)= \ker(\Def,(\cdot)|_{\dM}), \\&
\mathscr{B}^{1}_{\DD}(g)= \ker(\D\Rm_{g},\delta_{g},\PttD,\D\Ah_{g}),  \\
&\mathscr{B}^{2}_{\DD}(g)=\ker(\dBianchi,\mathpzc{D}\mathpzc{Rm}_{g}^*)\cap \scrC^{2,2}_{\DD}(M).
\\ &\mathscr{B}^{k}_{\DD}(g)=\ker(\dBianchi,\pzcdel_{g})\cap \Omega^{k}_{\DD}(M;\Lambda^{2}T^*M), \qquad k\geq 3. 
\end{split}
\]

\item \emph{(Solvability and uniqueness)} Given $T\in\scrC^{2,2}_{\DD}(M)$, the boundary-value problem \eqref{eq:Riem_1} admits a solution $\sigma\in S^2_{\DD}(M)$ satisfying $\delta_{g}\sigma=0$ if and only if
\[
\pzcd_{g} T=0, \qquad T\,\bot_{L^{2}}\,\mathscr{B}^{2}_{\DD}(g).
\]
Any solution satisfying $\delta_{g}\sigma=0$ is unique modulo $\mathscr{B}^{1}_{\DD}(g)$.
\end{enumerate}
\end{theorem}
Again, the remarks concerning the cohomology modules from the previous examples hold verbatim here, although, distinctively, the expression for $\mathscr{B}^{2}_{\DD}(g)$ contains non-local terms. In this context, it is worthwhile to note that one can also formulate a similar picture for the linearized Ricci curvature equations
\[
\D\mathrm{Ric}_{g}\sigma=T, \qquad \text{in } M,
\]
which are an important subject of study in geometric analysis (e.g., \cite{DeT81,Ham84,And08,AH22,Hin24}) and are, naturally, closely related to the current example. The vanishing of the resulting cohomology modules in this picture becomes particularly rich from a geometric point of view, owing to the specific structure of the gauge and boundary conditions pertinent to Ricci curvature. In particular, in that setting one obtains vanishing results for the second cohomology module governing solvability. For this reason, and since the corresponding cohomological formulation requires additional structure and has further implications, we defer its treatment to a separate work \cite{Led25B}.

\section{Discussion and outlook} 
\label{sec:outlook} 
We make a few remarks regarding the scope and nature of the results laid out above.
\subsection{The explicit expressions for the cohomology} 
A key feature apparent in all examples above is that, as in classical Hodge theory, the cohomology modules of the lifted complexes admit explicit expressions in terms of the kernels of operators in the original sequence. These cohomological identities are not incidental, but rather constitute a structural feature of the theory. As in the classical Hodge isomorphisms \eqref{eq:Hodge_iso}, they pave the path to analyzing the resulting cohomology modules directly, with the aim of obtaining sharper solvability and uniqueness statements for the underlying problem by relating them to geometric properties. The resulting cohomology is therefore not merely abstract, but carries geometric significance, as already reflected by several vanishing statements referred to in the examples above.


\subsection{The analytical structure of the lifted complexes} 
Another technical remark concerns the nature of the operators in the lifted complexes. In the results above, they are described only as continuous maps between Fr\'echet spaces, since their more precise nature is not required for the precise statements of the results. However, as alluded to earlier, the operators in the lifted complexes are in fact constructed within a pseudodifferential calculus of boundary-value problems, namely, the Boutet de Monvel calculus, and differ from the original operators by zero-order terms in this calculus. In particular, the lifted operators admit well-behaved adjoints, which are used in the statements of all the theorems to identify and fix the inherent gauge of the problem, as is done in classical Hodge theory.

This nature is what stands behind the terminology ``lifting'', which is borrowed from other approaches to generalizing Hodge theory in a similar spirit (i.e., in which the composition of consecutive operators yields lower-order operations), although in an analytically different manner from ours \cite{KTT07,Wal15,SS19}. We provide a detailed comparison with these approaches in the body of the text (cf.~\secref{sec:proto_intro} and \secref{sec:comparison_theory}).

Since the theory is entirely based on this calculus, it naturally applies to a broader class of operators $A$ and $B$ in \eqref{eq:linear_problem} than is reflected in the examples and discussion above. In particular, it accommodates systems involving non-local terms that map boundary sections to interior sections, and vice versa, as well as terms of varying orders---known in the literature as Douglis--Nirenberg systems, which are encountered throughout geometric analysis \cite{DN55,Gru90,Kha23}. We shall develop the scope of the abstract theory to this extent in \secref{sec:general_intro}, and provide concrete examples of such systems in the more extensive treatment of \secref{sec:examples_intro}.

\subsection{Overdetermined boundary-value problems revisited} 
Apart from its purely analytic purpose, we expect the generalized Hodge theory to be useful for the analysis of nonlinear problems. Indeed, although the examples above are stated as linear boundary-value problems, they have a common origin: each arises from the linearization of a nonlinear overdetermined problem, whose linearized symmetries and constraints naturally give rise to a sequence of operators.


Let us make the common structural picture, underlying all examples in this work, more precise. Problems in geometry, geometric analysis, and the calculus of variations often reduce to finding a configuration $\gamma$ -- such as a metric tensor, a connection, or a gauge field -- that satisfies a nonlinear equation
\beq
\mathrm{F}_{\gamma} = 0, \qquad \gamma \in \scrU,
\label{eq:geoemtric_eq_intro}
\eeq
where $\scrU$ is an infinite-dimensional manifold modeled on section spaces over an underlying finite-dimensional manifold $M$. The nonlinear operator $\gamma\mapsto\mathrm{F}_{\gamma} : \scrU \to \scrV$ typically represents a geometric construct, such as the Riemann curvature tensor or an Euler--Lagrange operator, mapping into a target vector space or an infinite-dimensional vector bundle $\scrV$ \cite{Ham82}. 

As alluded to in the opening statement, such problems are often equipped with geometric constraints and gauge symmetries of the form
\beq 
\mathrm{F}_{\phi\cdot\gamma}=\phi\cdot\mathrm{F}_{\gamma}, \qquad \mathrm{G}_{\gamma}\mathrm{F}_{\gamma}=0.
\label{eq:F_constraints} 
\eeq 
where $\gamma\mapsto \mathrm{G}_{\gamma}$ is a nonlinear map, and $(\phi, \gamma) \mapsto \phi\cdot\gamma$ and $(\phi, V) \mapsto \phi\cdot V$ are Lie-group actions encoding the gauge symmetry. 

For many analytical approaches to such nonlinear problems, particularly those based on the implicit function theorem, a necessary first step is to study a linearized equation
\beq
A\omega
= \mathrm{D}_{\Gamma}\mathrm{F}_{\gamma}\omega
:= \mathrm{D}\mathrm{F}_{\gamma}\omega + \Gamma\omega,
\qquad \text{where} \qquad
\mathrm{D}\mathrm{F}_{\gamma}\omega
:= \dertZero \mathrm{F}_{\gamma_t},
\label{eq:derivative_intro}
\eeq
Here, $\gamma_t \in \scrU$ is a smooth path satisfying $\gamma_0 = \gamma$ and $\dertZero \gamma_t = \omega$, while the endomorphism $\omega \mapsto \Gamma\omega$ represents a linear error term, such as one arising from a connection on $\mathscr{V}$. Including this term broadens the scope of the implicit function theorem; see, for example, \cite[pp.~93--94, 213--216]{Ham82}. In the context of \eqref{eq:linear_problem}, the linear operator $A$ in \eqref{eq:derivative_intro} then corresponds to the interior operator in the template \eqref{eq:linear_problem}, while the equations $B\omega=0$ represents the boundary conditions dictated by the underlying structure of $\scrU$. 

By the chain rule, linearizing the constraints \eqref{eq:F_constraints} shows that, under the assumption $\mathrm{F}_{\gamma}=0$, one can embed $A_{1}=A$ into a cochain complex as in \eqref{eq:A0A1_intro}, where
\beq 
A_{0}:=\D(\phi\mapsto \phi\cdot\gamma), \qquad A_{2}:=\D\mathrm{G}_{\gamma}.
\label{eq:gauge_symmetries_sequence} 
\eeq 
However, when the background structure fails to satisfy $\mathrm{F}_{\gamma}=0$, the resulting sequence is no longer a cochain complex, and the analysis breaks down. As we demonstrate in \secref{sec:examples_intro}, all the sequences of operators presented in \secref{sec:main_results_examples_1} are essentially of this form, and their failure to be cochain complexes in general results from removing the background assumption $\mathrm{F}_{\gamma}=0$.

As surveyed in the examples above, existing techniques can sometimes address this failure by ``correcting'' the sequence into a cochain complex (of differential operators) while imposing additional restrictive assumptions on the background structure at the reference point $\gamma \in \scrU$, such as analyticity, homogeneity, or constant background curvature. In general geometry, however, to the best of our knowledge no technique is currently available for overcoming this failure. Returning to the nonlinear premise, these assumptions, in turn, limit the applicability of linear-to-nonlinear machinery, including perturbation theory, continuity methods, and implicit function theorems, to the original problem \eqref{eq:geoemtric_eq_intro}.

Our results and aims can thus be rephrased as follows: they provide a framework for studying linearized problems of the form \eqref{eq:geoemtric_eq_intro}--\eqref{eq:derivative_intro} without imposing restrictive background assumptions, allowing us to operate even when $\mathrm{F}_{\gamma} \neq 0$. It is this theory that we shall develop below.

\subsection{Structure of this work} 
Due to the length and technicality of this work, we found it prudent to assemble \chapref{chp:overview}, which remains expository and provides a detailed overview and summary of both the abstract framework (\secref{sec:Hodge_intro}) and the geometric examples to which it applies (\secref{sec:examples_intro}). Its purpose is to avoid redundancy and lack of clarity in the technical body of the text, and to make it easier to refer to the main results and ideas of this work.

The technical body of this work is then structured as follows.

In \chapref{chp:tech_setup}, we provide the necessary technical setup. In \secref{sec:Preliminaries}, we review the required analytical preliminaries, including the Boutet de Monvel calculus. In \secref{sec:Overdetermined}, we review overdetermined ellipticity in the Douglis--Nirenberg sense, introducing concepts and notation tailored to our needs.

In \chapref{chp:elliptic_pre_complexes}, we develop and prove all the results outlined in \secref{sec:Hodge_intro}.

In \chapref{chp:examples}, we provide the analysis underlying the examples outlined in \secref{sec:examples_intro}, including the proofs of the results in \secref{sec:main_results_examples_1} as particular cases.

\subsection*{Acknowledgments}
I am grateful to my advisor, Raz Kupferman, for his guidance and assistance in advancing this project. I also wish to thank Or Hershkovits for his kind support, and to Deane Yang for his constructive feedback and for suggesting to apply the theory to the Einstein equations, which led to a significant extension of its scope. Lastly, I extend my warm thanks to Rotem Assouline for many helpful discussions. 

This project was funded by the Israel Science Foundation Grant 560/22 and an Azrieli Graduate Fellowship from the Azrieli Foundation.  

\chapter{Overview and Summary}
\label{chp:overview}

As stated above, due to the length and technicality of this work, we found it prudent to assemble this chapter, which remains expository and provides a detailed overview and summary of both the abstract framework (\secref{sec:Hodge_intro}) and the geometric examples to which it applies (\secref{sec:examples_intro}). Its purpose is to avoid redundancy and lack of clarity in the technical body of the text, and to make it easier to refer to the main results and ideas of this work.

To that end, we focus here on demonstrating the applicability of the theory, stripped as much as possible of technical aspects, so others may recognize the framework as useful for their own applications. Accordingly, the emphasis in \secref{sec:Hodge_intro} is on presenting the theory in a sufficiently detailed yet black-box form, together with discussions of what makes the framework work. Similarly, in \secref{sec:examples_intro}, the emphasis is on the interaction between the theory and geometric aspects, rather than on technical details, which are deferred to later sections.

\section{Generalized Hodge theory: Overview}
\label{sec:Hodge_intro}

\subsection{Classical elliptic complexes}
As introduced in \secref{sec:main_results_examples_1}, the key idea behind this work's machinery is to weaken the classical notion of an {elliptic complex} \cite{AS68,Tay11b} so that it applies more broadly, leading to the notion of an {elliptic pre-complex}. For the introduction of our generalization, we first recall the notion of an elliptic complex, which unifies and abstracts the essential components of the classical Hodge theory for the de Rham complex. These components were recognized long ago \cite{AB67,AS68} and were later generalized to compact manifolds with boundary. Several expositions of elliptic complexes appear in the literature. We follow the exposition in \cite[Ch.~12.A]{Tay11b}, with certain elements adapted from \cite{Sch95b,KL23}, as this formulation is more suitable for our purposes; see also the discussion in \secref{sec:comparison_theory} for a comparison with other formulations.

The components can be summarized in the following diagram:

\beq
\begin{xy}
(-30,0)*+{0}="Em1";
(-30,-20)*+{0}="Gm1";
(0,0)*+{\Gamma(\E_0)}="E0";
(30,0)*+{\Gamma(\E_1)}="E1";
(60,0)*+{\Gamma(\E_2)}="E2";
(90,0)*+{\Gamma(\E_3)}="E3";
(100,0)*+{\cdots}="E4";
(0,-20)*+{\Gamma(\bbJ_0)}="G0";
(30,-20)*+{\Gamma(\bbJ_1)}="G1";
(60,-20)*+{\Gamma(\bbJ_2)}="G2";
(90,-20)*+{\Gamma(\bbJ_3)}="G3";
(100,-20)*+{\cdots}="G4";
{\ar@{->}@/^{1pc}/^{A_0}"E0";"E1"};
{\ar@{->}@/^{1pc}/^{A_0^*}"E1";"E0"};
{\ar@{->}@/^{1pc}/^{A_1}"E1";"E2"};
{\ar@{->}@/^{1pc}/^{A_1^*}"E2";"E1"};
{\ar@{->}@/^{1pc}/^{A_2}"E2";"E3"};
{\ar@{->}@/^{1pc}/^{A_2^*}"E3";"E2"};
{\ar@{->}@/^{1pc}/^0"Em1";"E0"};
{\ar@{->}@/^{1pc}/^0"E0";"Em1"};
{\ar@{->}@/_{0pc}/^{B_0}"E0";"G0"};
{\ar@{->}@/_{0pc}/^{B_1}"E1";"G1"};
{\ar@{->}@/_{0pc}/^{B_2}"E2";"G2"};
{\ar@{->}@/^{1pc}/^{B^*_0}"E1";"G0"};
{\ar@{->}@/^{1pc}/^{B^*_1}"E2";"G1"};
{\ar@{->}@/^{1pc}/^0"E0";"Gm1"};
{\ar@{->}@/^{1pc}/^{B^*_2}"E3";"G2"};
{\ar@{->}@/_{0pc}/^{B^*_3}"E3";"G3"};
\end{xy}
\label{eq:elliptic_complex_diagram}
\eeq
where $\E_\alpha \to M$ and $\bbJ_\alpha \to \dM$ are sequences of vector bundles over the interior and the boundary, respectively, and $\Gamma$ here denotes the global smooth sections functor.

The operators acting on the section spaces are, classically, differential operators, all of the same order, and are symbiotic with one another via \emph{Green's formulae}:  
\beq
\langle A_\alpha\psi, \eta \rangle_{L^2(M)} = \langle \psi, A_\alpha^*\eta \rangle_{L^2(M)} + \langle B_\alpha\psi, B^*_\alpha\eta \rangle_{L^2(\dM)},
\qquad
\psi \in \Gamma(\E_\alpha), \quad \eta \in \Gamma(\E_{\alpha+1}),
\label{eq:Green_elliptic_intro}
\eeq
where $\langle \cdot, \cdot \rangle_{L^2(M)}$ and $\langle \cdot, \cdot \rangle_{L^2(\dM)}$ denote $L^2$-pairings of interior and boundary sections, respectively, with respect to chosen Riemannian fiber metrics and volume forms.  

In addition, the following conditions are required:  
\begin{enumerate}[itemsep=0pt,label=(\alph*)]
\item The interior operators satisfy $A_{\alpha+1} A_{\alpha}=0$; that is, $(A_{\bullet})$ is a cochain complex.  
\item The boundary-value problem associated with $(D_\alpha^*D_\alpha, T_\alpha)$ is \emph{elliptic}, where $D_\alpha = A_{\alpha-1}^* \oplus A_\alpha$, so that $D_\alpha^*D_\alpha $ is the ``Laplacian": 
\[
D_\alpha^*D_\alpha = A_{\alpha-1}A_{\alpha-1}^* + A^*_\alpha A_\alpha,
\]
and where $T_\alpha$ is a boundary operator defined by one of the following:
\begin{itemize}
\item[$(\NN)$] $T_\alpha = B^*_{\alpha-1} \oplus B^*_\alpha A_\alpha$;
\item[$(\DD)$] $T_\alpha = B_{\alpha} \oplus B_{\alpha-1}A^*_{\alpha-1}$.
\end{itemize}
\label{eq:classical_ellipticity}
\end{enumerate}

Under these conditions, the elliptic complex is said to satisfy generalized \emph{Neumann conditions} (usually termed \emph{absolute} due to their classical correspondence with absolute cohomology) in the case of $(\NN)$, and generalized \emph{Dirichlet conditions} (similarly termed \emph{relative} due to the classical correspondence with relative cohomology) in the case of $(\DD)$; cf.\ \cite{Sch95b, Tay11a, Tay11b}.

The main result concerning elliptic complexes is an $L^{2}$-orthogonal, topologically direct decomposition of Fréchet spaces, the \emph{Hodge decomposition}. For Neumann conditions $(\NN)$, this is:
\beq
\Gamma(\bbE_{\alpha+1})= 
\lefteqn{\overbrace{\phantom{\image(A_{\alpha})\oplus \module_{\N}^{\alpha+1}}}^{\ker(A_{\alpha+1})}} \image(A_{\alpha}) \oplus \underbrace{\module_{\N}^{\alpha+1}\oplus \image(A^*_{\alpha+1}|_{\ker B_{\alpha+1}^*})}_{\ker(A_{\alpha}^*|_{\ker B_{\alpha}^*})}
\label{eq:Hodge_intro}
\eeq
where the finite-dimensional space $\module_{\N}^{\alpha+1} := \ker(A_{\alpha+1}, A_\alpha^*, B^*_\alpha)$ satisfies:
\beq
\module_{\N}^{\alpha+1} \simeq \ker(A_{\alpha+1})\big/\image(A_{\alpha}) \simeq \ker(A^*_{\alpha}|_{\ker B_{\alpha}^*})\big/\image(A^*_{\alpha+1}|_{\ker B_{\alpha+1}^*}).
\label{eq:cohomology_hodge_classic}
\eeq 
Distinctively, the cohomology of the complex $(A_{\bullet})$ is identified not only with that of the adjoint complex $(A^*_{\bullet})$, supplemented with its natural boundary conditions, but is shown to be isomorphic to the kernel of an elliptic boundary-value problem.

Analogously, for Dirichlet conditions $(\DD)$, the decomposition takes the form:
\beq
\Gamma(\bbE_{\alpha+1})= 
\lefteqn{\overbrace{\phantom{\image(A_{\alpha}|_{\ker B_{\alpha}})\oplus \module_{\DD}^{\alpha+1}}}^{\ker(A_{\alpha+1}|_{\ker B_{\alpha+1}})}} \image(A_{\alpha}|_{\ker B_{\alpha}}) \oplus \underbrace{\module_{\DD}^{\alpha+1}\oplus \image(A^*_{\alpha+1})}_{\ker(A_{\alpha}^*)}
\label{eq:Hodge_intro_dirichlet}
\eeq
where the finite-dimensional space $\module_{\DD}^{\alpha+1} := \ker(A_{\alpha+1}, A_\alpha^*, B_{\alpha+1})$ satisfies:
\beq 
\module_{\DD}^{\alpha+1} \simeq \ker(A_{\alpha+1}|_{\ker B_{\alpha+1}})\big/\image(A_{\alpha}|_{\ker B_{\alpha}}) \simeq \ker(A^*_{\alpha})\big/\image(A^*_{\alpha+1}).
\label{eq:cohomology_hodge_classic_dirichlet}
\eeq

In both cases, identifying the cohomology in terms of elliptic boundary-value problems is the most important feature of the theory, and paves the way for the application of machinery such as unique continuation, Weitzenböck formulae, and Bochner techniques to obtain estimates on the dimension of the cohomology modules and relate them to geometric features (cf. \cite{PRS08, Pet16}).

Conversely, from the perspective of the motivation in \secref{sec:opening}, these Hodge decompositions provide a \emph{cohomological formulation} of the solvability and uniqueness conditions for the boundary value problem $A_{\alpha}\omega = \eta$ \cite{Sch95b}, as described in \eqref{sec:opening}. 

However, despite the apparent generality of elliptic complexes, a cohomological formulation for problems beyond the de Rham complex is often out of reach. The reason is twofold. First, as alluded to in \secref{sec:opening}, constructing a sequence $(A_\bullet)$ that incorporates $A$ into a cochain complex is generally not possible. Second, geometric problems often fail to satisfy the ellipticity conditions, as they are typically overdetermined—again due to gauge invariance and geometric constraints. Specifically, in the diagram \eqref{eq:elliptic_complex_diagram}, $A_{\alpha}$ and $A_{\alpha-1}$ may not have the same order, implying that the associated Laplacian $D^*_{\alpha}D_{\alpha}$ fails to be elliptic even in the interior. Indeed, as aptly put in \cite[Ch.~12.A, p.~465]{Tay11b}:
\begin{quotation}
``With this sketch of elliptic complexes done, it is time to deliver the bad news. The regularity (i.e., ellipticity) hypothesis is rarely satisfied, other than for the de Rham complex.''
\end{quotation}
  
\subsection{Elliptic pre-complexes: The prototypical setting} 
\label{sec:elliptic_pre_complexes_intro}

To overcome these shortcomings, and extend the scope of Hodge theory, we examine sequences of mappings that interact in a manner that is weaker, yet more ubiquitous, than in elliptic complexes. 

For introductory purposes, and given its practical relevance in numerous examples---indeed, in all the examples presented in \secref{sec:opening}---we begin with a prototypical setting that most clearly demonstrates the main ideas behind our theory: a diagram of operators that shares the structure of \eqref{eq:elliptic_complex_diagram}, yet is weakened in several key respects. Our main result is that such a diagram can be ``lifted'' to a full-fledged complex, to which all the results of Hodge theory apply, with its cohomology admitting an explicit expression in terms of the original operators in the diagram.

Once the principal ideas and machinery are clear for this prototypical setting, we shall subsequently present a more general diagram, applicable to a much broader class of systems, including non-local terms and mappings not only between interior sections but also between boundary sections, and allowing varying orders and mixing between section spaces over the interior and the boundary---namely, Douglis--Nirenberg systems \cite{DN55,Gru90}.

\label{sec:proto_intro}
For the statement of our main definition of an \emph{elliptic pre-complex} (not to be confused with \emph{elliptic quasicomplexes}\footnote{We developed elliptic pre-complexes prior to becoming aware of the existing notion of \emph{elliptic quasicomplexes} \cite{KTT07, Wal15, SS19}. We are grateful to Sylvie Paycha, Elmar Schrohe, and Jörg Seiler for bringing these to our attention. Despite some conceptual similarities, the two notions differ in their analytical structure and applications. A detailed comparison of these frameworks is provided in \secref{sec:comparison_theory}, once the necessary technical tools and terminology have been introduced.}), let $A \mapsto \ord(A) \in \Nzero$ denote the correspondence between a differential operator and its order. We first state the definition as it stands, and then elaborate on the terminology and role of its various components:

\begin{definition}[Elliptic pre-complex --- Prototypical]
\label{def:elliptic_pseudo_complex_intro}
Consider a diagram of differential operators \eqref{eq:elliptic_complex_diagram}, each of order $m_{\alpha}:=\ord(A_{\alpha})$ (where we allow $m_{\alpha_1} \ne m_{\alpha_2}$ for $\alpha_1 \ne \alpha_2$), satisfying Green's formulae \eqref{eq:Green_elliptic_intro}. 

We say the diagram is an \emph{elliptic pre-complex} if the following conditions hold:
\begin{enumerate}[itemsep=0pt,label=(\alph*)]
\item $B_{\alpha}$ and $B_{\alpha}^*$ are \emph{normal systems of trace operators}.
\item The interior operators $(A_{\bullet})$ satisfy the \emph{order-reduction property}:
\[
\ord(A_{\alpha+1} A_{\alpha}) \leq \ord(A_{\alpha}). 
\]
\item One of the following systems is \emph{overdetermined elliptic}: 
\begin{itemize}[itemsep=0pt]
\item[$(\NN)$] \emph{(Neumann conditions)}: $A_{\alpha} \oplus A_{\alpha-1}^* \oplus B^*_{\alpha-1}$. 
\item[$(\DD)$] \emph{(Dirichlet conditions)}: $A_{\alpha} \oplus A_{\alpha-1}^* \oplus B_{\alpha}$.

In the $(\DD)$ case, the following \emph{boundary order reduction} must also be satisfied:
\[
\ker B_{\alpha-1} \subseteq \ker (B_{\alpha}A_{\alpha-1}).
\] 
\end{itemize}
In the $(\NN)$ case, we say the elliptic pre-complex is based on \emph{Neumann conditions}; in the $(\DD)$ case, we say it is based on \emph{Dirichlet conditions}. 
\end{enumerate}
\end{definition}

\subsubsection{Green's formulae and normality}
\label{sec:Generalized_green_intro}
The definition of a normal system of trace operators, in the context of Green's formula, is taken directly from \cite[Ch.~1.4]{Gru96}. Essentially, when we say that a boundary operator $B$ is a \emph{normal system of trace operators}, we mean that it can be decomposed as:
\beq 
B=\bigoplus_{j=0}^{r}B_{j}, \qquad r\in\Nzero,
\label{eq:system_of_trace_operators} 
\eeq 
where each $B_{j}$ is a differential boundary operator of order $j$ whose leading normal derivative coefficient on the boundary is a surjective morphism. This, in turn, implies by the trace theorem that the entire boundary operator $B$, along with all of its continuous Sobolev extensions, is surjective, and that its kernel is dense in the $L^{2}$ topology \cite[Ch.~1.4]{Gru96}.

This requirement extends the classical setting presented in \secref{sec:Hodge_intro}, where the ellipticity of the boundary-value problem \eqref{eq:classical_ellipticity} implicitly forces the boundary operators to be normal. Here, because we relax the ellipticity assumption, this normality condition must be imposed separately.

We remark that the role of surjectivity of boundary operators arising in Green's formulae has been studied extensively \cite{LM72, Gru96, Tay11a}, and their significance in those contexts motivated much of their role in our own theory.

\subsubsection{Overdetermined ellipticity} 
\label{sec:overdetermined_ellipticity_intro}
The limiting requirement of ellipticity in classical elliptic complexes is generalized here by the condition of \emph{overdetermined ellipticity}, which, in view of the discussion in \secref{sec:opening}, is better suited to accommodate problems arising in geometry.

To understand how overdetermined ellipticity differs from classical ellipticity, one may take two equivalent paths. The first proceeds via the principal symbol (a delicate object when operating with systems of varying order and on manifolds with boundary): for overdetermined ellipticity, the symbol is required to be injective, rather than bijective as in the elliptic case; this is classically known as the Lopatinskii--Shapiro condition \cite[pp.~239--240]{Hor03}. 

The second path avoids symbolic calculus and proceeds via the Fredholm properties of the continuous extensions of the system acting between Sobolev spaces. Specifically, at the functional-analytic level, overdetermined ellipticity of a system of the form $A_{1}\oplus A_{2}\oplus B$ (where $A_{i}$ are interior operators of order $m_{i}$ and $B$ is a system of trace operators as in \eqref{eq:system_of_trace_operators}) equates to the existence of a continuous Sobolev extension that is \emph{semi-Fredholm} (cf.~\cite[Ch.~4.5]{Kat80}); that is, for some $W^{s,p}$-Sobolev norm, an a priori estimate holds:
\beq  
\|\psi\|_{s,p} \leq C \bigl(\|A_1 \psi\|_{s-m_{1},p} + \|A_2 \psi\|_{s-m_{2},p} +\sum_{j=0}^{r} \|B_j \psi\|_{s - j -1+ 1/p,p} + \|\psi\|_0 \bigr),
\label{eq:Sobolev_estimate}
\eeq 
where $s \in \Nzero$ is a sufficiently large Sobolev exponent.

Once \eqref{eq:Sobolev_estimate} holds for some $s\in\Nzero$, it can in fact be shown to hold for every $s > r + 1/p-1$ and $1 < p < \infty$. 

Either way, overdetermined ellipticity offers considerably more flexibility than standard ellipticity; for instance, it allows $A_{1}$ and $A_{2}$ to be of different orders and to be considered separately rather than jointly as a ``Laplacian'' operator. This flexibility is not merely technical: as demonstrated in \secref{sec:opening}, in geometric problems, $A_{1}$ is often the actual linearized operator we wish to invert, while $A_{2}$ incorporates compatibility or gauge conditions. We shall see this explicitly in the examples studied below.

Within our framework, \defref{def:elliptic_pseudo_complex_intro} accommodates two distinct types of elliptic pre-complexes: those based on Neumann conditions (denoted $\NN$) and those based on Dirichlet conditions (denoted $\DD$). This nomenclature reflects the fact that these settings generalize the boundary conditions naturally associated with the elliptic boundary-value problems for the ``Laplacians'' in the classical theory.

\subsubsection{The order-reduction property}
\label{sec:order-reduction_intro}
Perhaps most importantly from the geometric point of view, the limiting requirement in the classical theory that the sequence of operators $(A_{\bullet})$ forms a cochain complex (i.e., $A_{\alpha+1}A_{\alpha}=0$) is replaced by a significantly weaker condition, which we call the \emph{order-reduction property}. Specifically, while the composition $A_{\alpha+1}A_{\alpha}$ is nominally an operator of order $m_{\alpha+1}+m_{\alpha}$ (where $m_{k} = \ord(A_{k})$), this property dictates that its actual order, perhaps unexpectedly, does not exceed $m_{\alpha}$. This naturally encompasses the case $A_{\alpha+1}A_{\alpha}=0$, but more generally allows for $A_{\alpha+1}A_{\alpha}\ne 0$, thereby considerably broadening the applicability of the theory.

In this context, we note that in the Dirichlet case in \defref{def:elliptic_pseudo_complex_intro}, an additional order-reduction property is also required on the boundary; specifically, $\ker B_{\alpha-1}\subseteq \ker (B_{\alpha}A_{\alpha-1})$. This must be separately imposed; in the classical setting, it is a consequence of the interior identity $A_{\alpha+1}A_{\alpha}=0$, Green's formula \eqref{eq:Green_elliptic_intro}, and ellipticity. 

We also remark that the order reduction properties in \defref{def:elliptic_pseudo_complex_intro} cannot, in general, be deduced solely by analyzing the principal symbols of the differential operators. A vanishing principal symbol for $A_{\alpha+1} A_{\alpha}$ merely indicates a cancellation of leading-order terms, which does not guarantee the strict bound on the order. Our analysis intrinsically requires these bounds to hold at the operator level.

From the perspective outlined in \secref{sec:opening}, the order-reduction property in our examples typically arises from geometric interactions between the systems in the sequence. In particular, it is usually the result of underlying curvature, manifesting naturally during the linearization of geometric symmetries and constraints, as demonstrated in detail in \secref{sec:examples_intro}.
\subsection{Main result: The lifted complex (prototypical)}
\label{sec:main_results_intro}
Our main theorem (again, stated here for the prototypical case of \defref{def:elliptic_pseudo_complex_intro}) states that every elliptic pre-complex $(\bA_{\bullet})$ can be \emph{lifted} to a fully fledged cochain complex, denoted in script by $(\fbA_{\bullet})$:

\begin{theorem}[Lifted complex --- Prototypical]
\label{thm:lifted_complexIntro}
Every elliptic pre-complex $(\bA_{\bullet})$ induces a sequence of continuous linear maps of Fréchet spaces
\[
\fbA_{\alpha+1} : \Gamma(\bbE_{\alpha+1}) \;\rightarrow\; \Gamma(\bbE_{\alpha+2}),
\]
\emph{uniquely} characterized by the following properties, depending on the conditions on which the elliptic pre-complex is based in \defref{def:elliptic_pseudo_complex_intro}:
\begin{enumerate}[label=(\roman*), itemsep=0.5em]
\item[$\mathrm{(\NN)}$] For Neumann conditions:
\beq 
\begin{aligned}
&(a) \qquad &&\fbA_{\alpha+1}\fbA_{\alpha} = 0 &&&& \text{identically,} \\
&(b) \qquad &&\fbA_{\alpha+1} = \bA_{\alpha+1} &&&& \text{on } \ker(\fbA_{\alpha}^*|_{\ker B^*_{\alpha}}).
\end{aligned}
\label{eq:lifted_A_intro1} 
\eeq 
\item[$\mathrm{(\DD)}$] For Dirichlet conditions:
\beq 
\begin{aligned}
&(a)\qquad &&\fbA_{\alpha+1}\fbA_{\alpha} = 0 &&&& \text{on } \ker B_{\alpha}, \\
&(b)\qquad &&\fbA_{\alpha+1} = \bA_{\alpha+1} &&&& \text{on } \ker\fbA^*_{\alpha}.
\end{aligned}
\label{eq:lifted_A_intro2} 
\eeq 
\end{enumerate}
\end{theorem}

The lifting of an elliptic pre-complex is carried out inductively: 

At the base of the induction, the characterizations in \eqref{eq:lifted_A_intro1} and \eqref{eq:lifted_A_intro2} dictate that $\fbA_{0}=\bA_{0}$ and $\fbA_{0}^*=\bA_{0}^*$, consistent with the convention $\fbA_{-1}=\bA_{-1}=0$ in the diagram \eqref{eq:elliptic_complex_diagram}.

From the base case onward, the notation $\fbA^*_{\alpha}$ denotes the formal adjoint of $\fbA_{\alpha}$. The adjoint is well-defined because, as a byproduct of the purely functional-analytic characterization in \eqref{eq:lifted_A_intro1} and \eqref{eq:lifted_A_intro2} (which uniquely determines the operators), we show that the lifted operators differ from the original ones by negligible terms within a pseudodifferential calculus for boundary-value problems \cite{Hor03, RS82, Gru90, Gru96}; namely, the \emph{Boutet de Monvel calculus} \cite{Bou71}. In a nutshell, this broader calculus is required because the standard pseudodifferential calculus does not apply to manifolds with boundary. The Boutet de Monvel calculus provides a complete extension that suits our purposes and ensures the lifted operators in \thmref{thm:lifted_complexIntro} admit adjoints. We provide a detailed exposition of this calculus, on its various aspects, in the body of the paper (\secref{sec:per_pseudo}).

The crux of the induction step can then be sketched as follows: when combined with overdetermined ellipticity, Green's formulae, and the normality of the corresponding boundary operators in \defref{def:elliptic_pseudo_complex_intro}, the induction hypothesis yields that $\fbA_{\alpha}$ induces what we call an \emph{auxiliary decomposition}. This takes the form of an $L^{2}$-orthogonal, topologically direct decomposition of Fréchet spaces (focusing, without loss of generality, on the Neumann case):
\[
\Gamma(\bbE_{\alpha+1}) = \image(\fbA_{\alpha})\oplus\ker(\fbA_{\alpha}^*|_{\ker B_{\alpha}^*}),
\]
where the direct summands are, in particular, closed subspaces in the Fréchet topology. The induction step then consists of showing that defining $\fbA_{\alpha+1}$ to be zero on $\image(\fbA_{\alpha})$ and equal to $\bA_{\alpha+1}$ on $\ker(\fbA_{\alpha}^*|_{\ker B_{\alpha}^*})$ indeed produces a system within the calculus, which in turn induces its own auxiliary decomposition. Once the induction is complete, the auxiliary decompositions refine into full Hodge decompositions, much as in the classical theory of \secref{sec:Hodge_intro}.

An important element in the proof is the following: systems within the Boutet de Monvel calculus are characterized by an \emph{order}, which quantifies the total number of derivatives, and a \emph{class}, which measures the number of normal derivatives at the boundary. In our framework, the order-reduction property provides that the ``lifting'' difference
\beq
\fbC_\alpha := \fbA_\alpha - \bA_\alpha
\label{eq:correction_intro}
\eeq
is a term of zero order and class. This is an indispensable part of the theory, from several reasons. First, as mentioned above, it guarantees that the adjoints $\fbA^*_\alpha$ exist and differ from $\bA_\alpha^*$ by a system of order and class zero:
\[
\fbC_{\alpha}^* = \fbA^*_{\alpha} - \bA_{\alpha}^*.
\] 
Consequently, since operators of order and class zero are $L^{2}$-continuous and thus integrate by parts without producing a boundary term, $\fbA_{\alpha}$ inherits the generalized Green's formula:
\beq
\bra \fbA_\alpha\psi,\eta\ket_{L^{2}(M)} = \bra\psi,\fbA_\alpha^*\eta\ket_{L^{2}(M)} + \bra B_\alpha\psi, B^*_\alpha\eta\ket_{L^{2}(\dM)},
\label{eq:generlized_Green_intro}
\eeq
leaving the boundary terms unchanged. 

Second, the fact that the lifting term is of order zero ensures that the perturbed systems corresponding to those in \defref{def:elliptic_pseudo_complex_intro} are also overdetermined elliptic, since semi-Fredholmness (i.e., estimates of the form \eqref{eq:Sobolev_estimate}) is stable under compact perturbations: 
\[
\begin{aligned}
\NN&:\qquad &&\fbA_{\alpha}\oplus\fbA_{\alpha-1}^*\oplus B^*_{\alpha-1} \\
\DD&:\qquad &&\fbA_{\alpha}\oplus\fbA_{\alpha-1}^*\oplus B_{\alpha} 
\end{aligned}
\]
Thus, as overdetermined elliptic systems, they yield approximate left-inverses that are also in the calculus. This will be relied upon heavily in the proof of the induction step, together with the property stating that any operator within the calculus possessing a continuous extension between spaces of $L^{2}$ sections must necessarily be of zero order and class.

Third, it is this specific form and these properties of the lifting terms that allow us to obtain an explicit expression for the cohomology of the lifted complex, as will become clear in the Hodge theories detailed below.

Before presenting the emergent Hodge theories from the lifting, we make a few remarks on the context and scope of this framework:
\begin{enumerate}[label=(\arabic*), itemsep=0.5em]
\item The terminology \emph{lifting} originates in other approaches to generalizing Hodge theory using the Boutet de Monvel calculus, albeit in a manner different from ours; we refer to \cite{KTT07,Wal15,SS19} and \secref{sec:comparison_theory} for a comparison.

\item The inductive structure suggests (as will become apparent in the body of the paper) that one can technically designate a distinct point $\alpha_0\in\Nzero$ in the diagram \eqref{eq:elliptic_pre_complex_diagram_intro} up to which the requirements of an elliptic pre-complex hold. The lifting and its associated constructions can then be carried out strictly up to that level, yielding the notion of an $\alpha_0$-\emph{elliptic pre-complex}. This is merely a technical distinction; for introductory purposes, we state our theorems here for diagrams in which all segments satisfy the full requirements.

\item The theory of elliptic pre-complexes holds verbatim for closed manifolds, where many of its technical aspects relax. This is essentially because, in the boundaryless setting, the Boutet de Monvel calculus reduces to the standard pseudodifferential setting, in which every operator admits an adjoint.  

\item The motivation for studying the theory on manifolds with boundary lies not only in its significantly richer structure, but also in its relevance to actual nonlinear boundary-value problems, as premised in \secref{sec:opening}. 

\item Since the ultimate goal is to address nonlinear problems, we also consider the case where the vector bundles and systems in the diagram \eqref{eq:elliptic_pre_complex_diagram_intro} are parameterized tamely and smoothly by a topological space. This study is technical in nature and builds upon the theory of tame families of linear maps presented by Hamilton in the context of the Nash--Moser inverse function theorem \cite{Ham82}.
\end{enumerate}

\subsubsection{Hodge theory for Neumann conditions}
\label{sec:Hodge_intro_NN}
\thmref{thm:lifted_complexIntro} applied in the $\NN$-case implies the existence of a cochain complex:
\beq
\begin{tikzcd}
\cdots \arrow[r, "\fbA_{\alpha-1}"] &  \Gamma(\bbE_\alpha) \arrow[rr, "\fbA_\alpha"] &  & \Gamma(\bbE_{\alpha+1}) \arrow[rr, "\fbA_{\alpha+1}"] &  & \Gamma(\bbE_{\alpha+2}) \arrow[r, "\fbA_{\alpha+2}"]  & \cdots
\end{tikzcd}
\label{eq:Neumann_complexIntroPro}
\eeq
In analogy with \eqref{eq:Hodge_intro}, this cochain complex gives rise to Hodge decompositions of Neumann type:
\begin{theorem}[Neumann Hodge decomposition]
\label{thm:hodge_like_lifted_complexIntro}
In the setting of \thmref{thm:lifted_complexIntro}, under Neumann conditions, every $\alpha\in\Nzero\cup\BRK{-1}$ yields an $L^2$-orthogonal, topologically direct compound decomposition
\beq
\Gamma(\bbE_{\alpha+1})=
\lefteqn{\overbrace{\phantom{\image(\fbA_\alpha)\oplus \module_{\N}^{\alpha+1}}}^{\ker(\fbA_{\alpha+1})}}
\image(\fbA_\alpha)
\oplus
\underbrace{\module_{\N}^{\alpha+1}\oplus \image(\fbA_{\alpha+1}^*|_{\ker B^*_{\alpha+1}})}_{\ker(\fbA_\alpha^*|_{\ker B^*_\alpha})},
\label{eq:HodgelikesmoothIntroPro}
\eeq
where the finite-dimensional space $\module_{\N}^{\alpha+1} := \ker(\fbA_{\alpha+1}, \fbA^*_\alpha, B^*_\alpha)$ satisfies:
\beq 
\module_{\N}^{\alpha+1} \simeq \ker(\fbA_{\alpha+1})\big/\image(\fbA_{\alpha})
\simeq
\ker(\fbA_{\alpha}^*|_{\ker B_{\alpha}^*})\big/\image(\fbA_{\alpha+1}^*|_{\ker B_{\alpha+1}^*}).
\label{eq:Neuamnn_cohomologyPro}
\eeq 
\end{theorem}
The explicit characterization of the lifting in \eqref{thm:lifted_complexIntro} then allows us to obtain reduced expressions for the adjoint and the cohomology, a distinctive feature of the theory alluded to in \secref{sec:opening}:

\begin{theorem}[Cohomological expressions]
\label{thm:coadjointN}
In the setting of \thmref{thm:hodge_like_lifted_complexIntro}, let 
\[
\cttbP_{\alpha-1} \colon \Gamma(\bbE_\alpha) \to \Gamma(\bbE_\alpha)
\] 
be the $L^2$-orthogonal projection onto $\ker(\fbA_{\alpha-1}^*|_{\ker{B_{\alpha-1}^*}})$, so $\id - \cttbP_{\alpha-1}$ is the projection onto $\image(\fbA_{\alpha-1})$. Then, for every $\alpha\leq \alpha_0$,   
\[
\fbA_\alpha = A_\alpha\cttbP_{\alpha-1} \Textand
\fbA_\alpha^* = \cttbP_{\alpha-1}A_\alpha^*, \qquad \text{on }\ker B_{\alpha}^* 
\]
and the cohomology can be rewritten in the following form:
\beq
\module_{\N}^{\alpha+1} = \ker(A_{\alpha+1}, \cttbP_{\alpha-1}A^*_\alpha, B_{\alpha}^*).
\label{eq:cohomology_groupsDIntroProN}
\eeq
\end{theorem}
We emphasize that the fact that the cohomology modules $\module_{\N}^{\alpha+1}$ can be expressed in terms of the original, non-lifted operators is an important feature of the theory. Without this property, the raw expression for the cohomology would be difficult to interpret. This characterization is a direct consequence of how the lifting is carried out in \thmref{thm:lifted_complexIntro}. 

We also remark that, for $\alpha=-1$ and $\alpha=0$, since $\cttbP_{-2}=0$ and $\cttbP_{-1}=\id$, the cohomologies there are expressible purely in terms of the original operators:
\beq 
\module^0_{\N}=\ker(A_0), \qquad \module^1_{\N}=\ker(A_1,A_0^*,B_0^*).
\label{eq:lower_cohomolgiesN}
\eeq 
In general, although \eqref{eq:cohomology_groupsDIntroProN} implies that it always holds that 
\[
\ker(A_{\alpha+1}, A_\alpha^*, B_{\alpha}^*) \subseteq \module^{\alpha+1}_{\N},
\]
there exist counterexamples to the reverse inclusion. 

A few remarks are now in order, which also apply to the Dirichlet case detailed below:
\begin{enumerate}[label=(\arabic*), itemsep=0.5em]
\item We show that the projections onto the various closed subspaces in \eqref{eq:HodgelikesmoothIntroPro} belong to the Boutet de Monvel calculus, and are of zero order and class as well. This fact allows, via a density and approximation argument, the derivation of $W^{s,p}$-Sobolev versions for every $1 < p < \infty$ and $s \in \Nzero$, in analogy with classical Hodge theory \cite{Sch95b, Tay11a}. For $s = 0$, these decompositions reduce to $L^{p}$-decompositions of section spaces. 

In this context, we remark that the study of $L^{p}$ Hodge decompositions is a subject with an extensive literature \cite{AKM06,HMP08,BB12}, typically treated in the highest functional-analytic generality or for the specific case of first-order operators. Here, we present these decompositions merely as byproducts of the fact that our operators lie in a pseudodifferential calculus that permits such results to follow.

\item As in the classical theory, the fact that the cohomology is expressed in terms of elliptic boundary-value problems can, under certain circumstances, be used to relate it to the underlying geometric structure, for example through unique continuation methods and Bochner techniques (cf.~\cite{PRS08} and \cite[Ch.~9]{Pet16}). Such relations can sometimes be established even in the presence of the non-local projection term in the general expression \eqref{eq:cohomology_groupsDIntroProN}.

We shall provide below, in \secref{sec:examples_intro}, some examples of this maneuver in connection with the examples laid out in \secref{sec:opening}. A much more profound realization of this undertaking is provided in our separate work \cite{Led25B}.

\item From the perspective of index theory, if a family of elliptic pre-complexes parameterized continuously by a topological space is \emph{finite}---in the sense that $\bA_{\alpha} = 0$ for $\alpha$ sufficiently large---it is shown that the \emph{Neumann Euler characteristic}  
\[
\mathscr{X}_{\N} = \sum_{\alpha} (-1)^{\alpha} \dim \module_{\N}^{\alpha}
\]
is constant as the parameter varies continuously.

\item We remark that, in this prototypical formulation, the Neumann version of the theory of elliptic pre-complexes was already developed to some extent in the preliminary study \cite{KL23}, although not under this terminology. The Dirichlet setting, however, was not treated there, nor was the Neumann theory developed as fully as it is here. The present paper therefore both completes and further generalizes this theory, treating the Dirichlet and Neumann settings together in order to demonstrate that they share essentially the same analytical heart, as in the theory of classical elliptic complexes.
\end{enumerate}
\subsubsection{Hodge theory for Dirichlet conditions}
\label{sec:Hodge_intro_DD}
Establishing the Hodge decompositions in the Dirichlet case follows lines similar to those in the Neumann case, with the necessary adaptations to account for the boundary kernels being absorbed into the range of the forward sequence rather than into that of its adjoint.

The starting point is the observation, arising as a byproduct of \thmref{thm:lifted_complexIntro} and the order-reduction property in the $\DD$ case \eqref{eq:D_boundary_reduction}, that one in fact has,
\[
\image(\fbA_{\alpha}|_{\ker B_{\alpha}}) \subseteq \ker B_{\alpha+1}.
\]
Hence, defining $\Gamma_{\D}(\bbE_\alpha) := \Gamma(\bbE_\alpha) \cap \ker B_\alpha$, we obtain the cochain complex (generalizing the sequence \eqref{eq:Hodge_complex} from Hodge theory):
\beq
\begin{tikzcd}
\cdots \arrow[r, "\fbA_{\alpha-1}|_{\ker B_{\alpha-1}}"] & \Gamma_{\D}(\bbE_\alpha) \arrow[rr, "\fbA_\alpha|_{\ker B_{\alpha}}"] &  & \Gamma_{\D}(\bbE_{\alpha+1}) \arrow[rr, "\fbA_{\alpha+1}|_{\ker B_{\alpha+1}}"] &  & \Gamma_{\D}(\bbE_{\alpha+2}) \arrow[r, "\fbA_{\alpha+2}|_{\ker B_{\alpha+2}}"] &  \cdots
\end{tikzcd}
\label{eq:Dirichlet_cochain_omplexIntroPro}
\eeq
As shall be explained in \secref{sec:examples_intro}, all the examples in \secref{sec:opening} fall within this pattern. In analogy with \eqref{eq:Hodge_intro_dirichlet}, we have:
\begin{theorem}[Dirichlet Hodge decomposition]
\label{thm:hodge_like_lifted_complexDIntro}
In the setting of \thmref{thm:lifted_complexIntro}, under Dirichlet conditions, every $\alpha\in\Nzero\cup\BRK{-1}$ yields an $L^2$-orthogonal, topologically direct compound decomposition
\beq
\Gamma(\bbE_{\alpha+1})=
\lefteqn{\overbrace{\phantom{\image(\fbA_\alpha|_{\ker B_\alpha})\oplus \module_{\D}^{\alpha+1}}}^{\ker(\fbA_{\alpha+1}|_{\ker B_{\alpha+1}})}}
\image(\fbA_\alpha|_{\ker B_\alpha})
\oplus
\underbrace{\module_{\D}^{\alpha+1}\oplus \image(\fbA_{\alpha+1}^*)}_{\ker(\fbA_\alpha^*)},
\label{eq:HodgelikesmoothDIntroPro}
\eeq
where the finite-dimensional space $\module_{\D}^{\alpha+1} := \ker(\fbA_{\alpha+1}, \fbA^*_\alpha, B_{\alpha+1})$ satisfies:
\[
\module_{\D}^{\alpha+1} \simeq \ker(\fbA_{\alpha+1}|_{\ker B_{\alpha+1}})\big/\image(\fbA_{\alpha}|_{\ker B_{\alpha}})
\simeq
\ker(\fbA^*_{\alpha})\big/\image(\fbA^*_{\alpha+1}).
\]
\end{theorem}

As in the Neumann case, the projections onto the various closed subspaces in \eqref{eq:HodgelikesmoothDIntroPro} belong to the pseudodifferential calculus. Hence, $W^{s,p}$-Sobolev versions hold for every $1<p<\infty$ and $s\in\Nzero$. 

Also, just as in the Neumann case, an important feature is that the cohomology modules reduce to an expression expressible in terms of the original operators: 

\begin{theorem}
\label{thm:coadjointD}
In the setting of \thmref{thm:hodge_like_lifted_complexDIntro}, let 
\[
\cttbP_{\alpha-1} \colon \Gamma(\bbE_\alpha) \to \Gamma(\bbE_\alpha)
\] 
be the $L^2$-orthogonal projection onto $\ker(\fbA_{\alpha-1}^*)$, so $\id - \cttbP_{\alpha-1}$ is the projection onto $\image(\fbA_{\alpha-1}|_{\ker B_{\alpha-1}})$. Then, for every $\alpha\leq \alpha_0$,  
\[
\fbA_\alpha = A_\alpha\cttbP_{\alpha-1} \Textand
\fbA_\alpha^* = \cttbP_{\alpha-1}A_\alpha^*,
\]
and the cohomology can be rewritten in the following form:
\beq
\module_{\D}^{\alpha+1} = \ker(A_{\alpha+1}, \cttbP_{\alpha-1}A^*_\alpha, B_{\alpha+1}).
\label{eq:cohomology_groupsDIntroPro}
\eeq
\end{theorem}
As in the Neumann case, for $\alpha=-1$ and $\alpha=0$, the cohomologies are expressible purely in terms of the original operators: 
\beq 
\module^0_{\D}=\ker(A_0,B_0), \qquad \module^1_{\D}=\ker(A_1,A_0^*,B_1).
\label{eq:lower_cohomolgies}
\eeq 
In general, although \eqref{eq:cohomology_groupsDIntroPro} implies that it always holds that 
\[
\ker(A_{\alpha+1}, A_\alpha^*, B_{\alpha+1}) \subseteq \module^{\alpha+1}_{\D},
\]
there exist counterexamples to the reverse inclusion. 

Similarly to the Neumann setting, if the elliptic pre-complex is finite and parameterized continuously by a topological space, then the \emph{Dirichlet Euler characteristic} 
\beq
\mathscr{X}_{\D}=\sum_{\alpha}(-1)^{\alpha}\dim{\module_{\D}^{\alpha}} 
\label{eq:D_Euler}
\eeq
is constant as the parameter varies continuously. 

\subsubsection{Cohomological formulations}
\label{sec:cohomological_formulation}
Using techniques similar to those in \cite{Sch95b,KL23}, one can use these Hodge decompositions to establish solvability and uniqueness for full non-homogeneous linear boundary-value problems associated with the systems in the lifted complex $(\fbA_{\bullet})$.  

For conciseness, and since our main aim is to obtain cohomological formulations for boundary-value problems involving the original operators in the elliptic pre-complex, we restrict attention here to homogeneous gauge and boundary conditions. Applications of these results to geometric problems, in the spirit of \secref{sec:opening}, are again deferred to \secref{sec:examples_intro}.

For Neumann conditions:
\begin{theorem}[Neumann cohomological formulation -- Prototypical]
\label{thm:NNintroPro}
Let $(\bA_{\bullet})$ be an elliptic pre-complex based on Neumann conditions. 
Given $\eta \in \Gamma(\bbE_{\alpha+1})$, the system
\[
\begin{aligned} 
&\bA_\alpha \psi = \eta,  && \qquad \text{in } M, \\
&B^*_{\alpha-1}\psi = 0 && \qquad \text{on } \dM
\end{aligned} 
\]
admits a solution $\psi \in \Gamma(\bbE_\alpha)$ satisfying $\fbA_{\alpha-1}^*\psi=0$ if and only if
\[
\fbA_{\alpha+1}\eta = 0, 
\qquad 
\eta \perp_{L^2} \module_{\N}^{\alpha+1}.
\]
Moreover, any solution satisfying $\fbA_{\alpha-1}^*\psi=0$ is unique modulo $\module_{\N}^{\alpha}$.
\end{theorem}

The proof follows directly from the decomposition \eqref{eq:HodgelikesmoothIntroPro}, using the relations $\fbA_{\alpha+1} \fbA_{\alpha} = 0$ and $\fbA_{\alpha} = \bA_{\alpha}$ on $\ker(\fbA_{\alpha-1}^*|_{\ker B^*_{\alpha-1}})$, together with the $L^{2}$-orthogonality of the decomposition. A corresponding Sobolev version is also available, obtained from the Sobolev versions of the Hodge decompositions, which remain $L^{2}$-orthogonal throughout the Sobolev grading.

The analog of the cohomological formulation for Dirichlet conditions is: 
\begin{theorem}[Dirichlet cohomological formulation -- Prototypical]
\label{thm:DDintroPro}
Let $(\bA_{\bullet})$ be an elliptic pre-complex based on Dirichlet conditions. 
Given $\eta \in \Gamma(\bbE_{\alpha+1})$, the system
\[
\begin{aligned} 
&\bA_\alpha \psi = \eta, && \qquad \text{in } M, \\
&B_\alpha\psi = 0 &&\qquad  \text{on } \dM
\end{aligned} 
\]
admits a solution $\psi \in \Gamma(\bbE_\alpha)$ satisfying $\fbA_{\alpha-1}^*\psi=0$ if and only if: 
\[
\fbA_{\alpha+1}\eta = 0, 
\qquad 
B_{\alpha+1} \eta = 0, 
\qquad 
\eta \perp_{L^2} \module_{\D}^{\alpha+1}.
\]
Moreover, any solution satisfying $\fbA_{\alpha-1}^*\psi=0$ is unique modulo $\module_{\D}^{\alpha}$.
\end{theorem}

The proof of \thmref{thm:DDintroPro} follows directly from the decompositions in \thmref{thm:hodge_like_lifted_complexDIntro}, invoking the relations $(\fbA_{\alpha+1}\oplus B_{\alpha+1})\fbA_{\alpha}=0$ on $\ker B_{\alpha}$ and $\fbA_{\alpha}=\bA_{\alpha}$ on $\ker(\fbA_{\alpha-1}^*)$. An implied Sobolev version is also valid.

\subsection{Elliptic pre-complexes: General setting} 
\label{sec:general_intro} 
Building on the prototypical setting, we now broaden our scope to encompass more general systems within the Boutet de Monvel calculus. Indeed, in geometric analysis, one encounters not just differential operators, but mappings from boundary sections to interior sections (e.g., in geometric inverse problems \cite{PSU23}), as well as non-local mappings between interior sections.

Because our main results in \secref{sec:main_results_intro} already rely on the Boutet de Monvel calculus, it is natural to develop this generalization within that framework. Accordingly, we recall the specifics of this calculus of pseudodifferential boundary-value problems more closely. 

Given vector bundles $\bbE, \bbF \to M$ and $\bbJ, \bbG \to \dM$ over the interior and boundary, respectively, the building blocks of the Boutet de Monvel calculus are \emph{Green operators}. These are linear systems taking the following matrix form:
\beq
\bD = \begin{pmatrix}
A_+ + G & K \\ 
T & Q 
\end{pmatrix}: 
\begin{matrix}\Gamma(\bbE) \\ \oplus \\ \Gamma(\bbJ)\end{matrix} \longrightarrow \begin{matrix}\Gamma(\bbF) \\ \oplus \\ \Gamma(\bbG)\end{matrix}.
\label{eq:Boutet_de_movel_intro} 
\eeq

The classes of non-local operators to which $A_+, G, K, T, Q$ belong are defined such that Green operators form an algebra closed under composition, adjunction, and inversion where applicable. We note that on a manifold with an empty boundary, $K, G, T$, and $Q$ are trivial, and $A_+=A$ reduces to a standard pseudodifferential operator.

Most prominently, the calculus extends to include systems in which $A_+, G, K, T$, and $Q$ have varying orders. These are referred to as \emph{Douglis--Nirenberg systems} \cite{DN55,RS82,Gru90,Kha23}, which are encountered throughout geometry and analysis, as demonstrated in the examples section below.  

To handle these systems more conveniently, we introduce the shorthand notation:
\[
\Gamma(\bbE;\bbJ) = \Gamma(\bbE) \oplus \Gamma(\bbJ)
\]
and denote this by $\Gamma(0;\bbJ)$ when $\bbE=\{0\}\times M$, or simply by $0$ if, in addition, $\bbJ=\{0\}\times\dM$.  

Consider now a diagram of Douglis--Nirenberg systems generalizing \eqref{eq:elliptic_complex_diagram}:
\beq
\begin{xy}
(-30,0)*+{0}="Em1";
(0,0)*+{\Gamma(\bbF_0;\bbG_0)}="E0";
(30,0)*+{\Gamma(\bbF_1;\bbG_{1})}="E1";
(60,0)*+{\Gamma(\bbF_2;\bbG_{2})}="E2";
(90,0)*+{\Gamma(\bbF_3;\bbG_{3})}="E3";
(101,0)*+{\cdots}="E4";
(-30,-25)*+{0}="Gm1";
(0,-25)*+{\Gamma(0;\bbL_0)}="G0";
(30,-25)*+{\Gamma(0;\bbL_1)}="G1";
(60,-25)*+{\Gamma(0;\bbL_2)}="G2";
(90,-25)*+{\Gamma(0;\bbL_3)}="G3";
(100,-25)*+{\cdots}="G4";
{\ar@{->}@/^{1pc}/^{\bD_{0}}"E0";"E1"};
{\ar@{->}@/^{1pc}/^{\bD_{0}^*}"E1";"E0"};
{\ar@{->}@/^{1pc}/^{\bD_{1}}"E1";"E2"};
{\ar@{->}@/^{1pc}/^{\bD_{1}^*}"E2";"E1"};
{\ar@{->}@/^{1pc}/^{\bD_{2}}"E2";"E3"};
{\ar@{->}@/^{1pc}/^{\bD_{2}^*}"E3";"E2"};
{\ar@{->}@/^{1pc}/^{0}"Em1";"E0"};
{\ar@{->}@/^{1pc}/^{0}"E0";"Em1"};
{\ar@{->}@/_{0pc}/^{\bB_0}"E0";"G0"};
{\ar@{->}@/_{0pc}/^{\bB_1}"E1";"G1"};
{\ar@{->}@/_{0pc}/^{\bB_2}"E2";"G2"};
{\ar@{->}@/_{0pc}/^{\bB_{3}}"E3";"G3"};
{\ar@{->}@/^{0pc}/^{0}"Em1";"Gm1"};
{\ar@{->}@/^{0.8pc}/^{\bB_{0}^*}"E1";"G0"};
{\ar@{->}@/^{0.8pc}/^{\bB_{1}^*}"E2";"G1"};
{\ar@{->}@/^{0.8pc}/^{\bB_{2}^*}"E3";"G2"};
{\ar@{->}@/^{0.8pc}/^{0}"E0";"Gm1"};
\end{xy}
\label{eq:elliptic_pre_complex_diagram_intro}
\eeq

where the notation $\Gamma(0;\bbL_{\alpha})$ on the range of $\bB_{\alpha}$ indicates that these boundary systems are permitted to take the form:
\beq
\begin{pmatrix} 0 & 0 \\  T & Q\end{pmatrix}.
\label{eq:system_boundary_intro} 
\eeq 
We remark that the diagram \eqref{eq:elliptic_complex_diagram} for elliptic complexes, as well as the prototypical setting introduced in \secref{sec:proto_intro}, can indeed be recovered from this more general diagram by substituting:  
\[
\bbF_{\alpha} = \bbE_{\alpha}, \qquad \bbG_{\alpha} = \{0\} \times \dM, \qquad \bbL_{\alpha} = \bbJ_{\alpha},
\]
and  
\beq
\bD_\alpha = \begin{pmatrix}
A_\alpha & 0 \\ 0 & 0 
\end{pmatrix}, \qquad 
\bD^*_\alpha = \begin{pmatrix}
A_\alpha^* & 0 \\ 0 & 0 
\end{pmatrix}, \qquad 
\bB_\alpha = \begin{pmatrix}
0 & 0 \\ B_\alpha & 0 
\end{pmatrix}, \qquad 
\bB^*_\alpha = \begin{pmatrix}
0 & 0 \\ B^*_\alpha & 0 
\end{pmatrix}.
\label{eq:particular_case_intro}
\eeq
We now outline how the various components of an elliptic pre-complex (\defref{def:elliptic_pseudo_complex_intro}) generalize to such diagrams. Unlike the prototypical \defref{def:elliptic_pseudo_complex_intro}, the precise definition in full generality requires further technical setup, which is provided in \chapref{chp:elliptic_pre_complexes}.

\subsubsection{Generalized Green's formulae and normality}
\label{sec:Generalized_green_intro2} 
Generalizing \eqref{eq:Green_elliptic_intro} and \defref{def:elliptic_pseudo_complex_intro}, the boundary system $\bB_\alpha$ is required to be \emph{normal}---that is, it is surjective and its kernel is dense in the $L^{2}$ topology \cite[Ch.~1.4]{Gru96}---and, together with $\bB_{\alpha}^*$, satisfies the following generalized Green's formula for all $\Psi \in \Gamma(\bbF_\alpha;\bbG_\alpha)$ and $\Theta \in \Gamma(\bbF_{\alpha+1};\bbG_{\alpha+1})$:  
\beq
\langle \bD_\alpha \Psi, \Theta \rangle_{L^{2}(\bbF_{\alpha+1};\bbG_{\alpha+1})} = \langle \Psi, \bD^*_\alpha \Theta \rangle_{L^{2}(\bbF_{\alpha};\bbG_{\alpha})} + \langle \bB_\alpha \Psi, \bB^*_\alpha \Theta \rangle_{L^{2}(0;\bbL_{\alpha})}.
\label{eq:Green_formula_intro}
\eeq
In this setting, $\bD_{\alpha}$ is referred to as an \emph{adapted Green system}, and $\bD^*_\alpha$ is its \emph{adapted adjoint}. 

We emphasize that, although we abuse notation, $\bD^*_\alpha$ and $\bB^*_{\alpha}$ are generally not the formal $L^2$-adjoints of $\bD_\alpha$ and $\bB_{\alpha}$, because true $L^2$-adjoints may not even exist. This is because the systems $\bD_\alpha$ and $\bB_{\alpha}$ are generally of nonzero class (e.g., when including differential operators in their bottom-left entry in \eqref{eq:Boutet_de_movel_intro}), which is an obstruction to the existence of such an adjoint within the calculus \cite[Ch.~1.2-1.3]{Gru96}. The existence of an \emph{adapted adjoint} $\bD^*_\alpha$, even for systems of nonzero class, is therefore made possible by allowing $\bB_\alpha$ and $\bB^*_\alpha$ to have a nonzero bottom-right entry as in \eqref{eq:system_boundary_intro}, unlike standard boundary operators that typically appear in Green's formulae.

When convenient, we refer to the diagram \eqref{eq:elliptic_pre_complex_diagram_intro} compactly as $(\bD_{\bullet})$, since the adapted adjoints and boundary systems are effectively determined by the primary sequence $(\bD_{\alpha})_{\alpha \in \Nzero}$.

\subsubsection{Overdetermined ellipticity} 
\label{sec:overdetermined_ellipticity_intro2} 
We again require \emph{overdetermined ellipticity}, now interpreted in the varying-order, or Douglis--Nirenberg, sense \cite{RS82, Gru90}.

As in the prototypical setting, our theory accommodates two types of elliptic pre-complexes, based on Neumann ($\NN$) or Dirichlet ($\DD$) conditions:
\beq
\begin{aligned}
&\NN:\qquad &&\bD_\alpha \oplus \bD^*_{\alpha-1} \oplus \bB^*_{\alpha-1} \qquad &&&&\text{is overdetermined elliptic}, \\
&\DD:\qquad &&\bD_\alpha \oplus \bD^*_{\alpha-1} \oplus \bB_\alpha \qquad &&&&\text{is overdetermined elliptic}.
\end{aligned}
\label{eq:overdetermined_intro}
\eeq

The overdetermined ellipticity of fully varying-order systems is most transparently understood via the semi-Fredholmness of their continuous Sobolev extensions, detailed in \secref{sec:overdetermined_douglas_nirenberg}. Validating this semi-Fredholmness is, however, more intricate than in the prototypical setting because Douglis--Nirenberg systems often lack a set of distinct, ``sharp'' uniform orders. Consequently, the system $\fbD_{\alpha}$ (and its adapted adjoint) may not be representable as a direct sum of operators with uniform orders, precluding the systematic identification of orders in \secref{sec:overdetermined_ellipticity_intro}.

The choice of specific orders for each component in a Douglis--Nirenberg system to allow a priori estimates of the form \eqref{eq:Sobolev_estimate} is classically known as the choice of \emph{weights}~\cite{DN55}. In \secref{sec:overdetermined_douglas_nirenberg}, we design a framework to systematically establish these a priori estimates by categorizing distinct classes of tuples of orders and classes for a given Douglis--Nirenberg system $\bD$. 

Addressing the definition, identification, and manipulation of these tuples is a central technical challenge of this work, ultimately allowing the resulting Hodge theories to assume the clean forms detailed below.

\subsubsection{Disrupted elliptic pre-complexes}
\label{sec:disrupted_intro} 
Technically speaking, there may be instances where overdetermined ellipticity fails on certain segments of the diagram \eqref{eq:elliptic_pre_complex_diagram_intro}, yet the machinery behind elliptic pre-complexes still applies, with most results remaining valid. We refer to such diagrams as \emph{disrupted elliptic pre-complexes}. They merit particular attention, as they arise concretely in some geometric examples (e.g., in the Ricci curvature equations; cf.~\cite{Led25B}). However, since their introduction is technical rather than fundamental, and in order to keep this exposition concise, we defer the detailed discussion to the examples section and the main body of the text (\secref{sec:variants}).
 
\subsubsection{The order-reduction property}
\label{sec:order-reduction_intro2} 
As introduced in \secref{sec:order-reduction_intro}, the rigid condition that $(\bD_{\bullet})$ forms a cochain complex is replaced by the \emph{order-reduction property}. Conceptually, as in \defref{def:elliptic_pseudo_complex_intro}, this amounts to the algebraic constraint that the composition is of lower order and class than nominally expected:
\beq
\begin{split}
&\operatorname{ord}(\bD_\alpha\bD_{\alpha-1}) \leq \operatorname{ord}(\bD_{\alpha-1}),\\
&\operatorname{class}(\bD_\alpha\bD_{\alpha-1}) \leq \operatorname{class}(\bD_{\alpha-1}).
\end{split} 
\label{eq:order_reduction_intro}
\eeq

More generally, however, the actual comparison of orders is delicate, since the operators composing the systems may have varying order, marking another technical challenge addressed in the main text.

For Dirichlet conditions, we also require the following order-reduction property on the boundary:
\beq
\DD:\qquad\bB_{\alpha}\bD_{\alpha-1}=0 \quad \text{ on } \ker\bB_{\alpha-1}.
\label{eq:D_boundary_reduction} 
\eeq 

Our previous remark holds verbatim: the order-reduction properties \eqref{eq:order_reduction_intro} and \eqref{eq:D_boundary_reduction} cannot generally be deduced via symbolic calculus alone. Indeed, this is further exacerbated by the fact that the varying orders of the systems make the symbol calculus a significantly more delicate apparatus, as alluded to above.  

\subsection{Main result: The lifted complex (General setting)}
Our main theorem cleanly generalizes \thmref{thm:lifted_complexIntro} to the abstract diagrams of \eqref{eq:elliptic_pre_complex_diagram_intro}. For the statement, we define the following spaces associated with any adapted Green system:
\[
\begin{aligned}
&\scrR(\bD) := \image \bD,                         &&\qquad \scrN(\bD) := \ker \bD, \\
&\scrR(\bD; \bB) := \image \left.\bD\right|_{\ker \bB},  &&\qquad \scrN(\bD, \bB) := \ker(\bD, \bB). 
\end{aligned}
\]

\begin{theorem}[Lifted complex -- General setting]
\label{thm:lifted_complexIntro2}
Every elliptic pre-complex $(\bD_{\bullet})$ induces a sequence of continuous linear maps of Fréchet spaces
\[
\fbD_{\alpha}:\Gamma(\bbF_{\alpha};\bbG_{\alpha}) \;\rightarrow\; \Gamma(\bbF_{\alpha+1};\bbG_{\alpha+1}),
\]
\emph{uniquely} characterized by the following properties:
\begin{enumerate}[label=(\roman*), itemsep=0.5em]
\item[$\mathrm{(\NN)}$] For Neumann conditions:
\begin{enumerate}[label=(\alph*), itemsep=0pt]
\item $\scrR(\fbD_{\alpha}) \subseteq \scrN(\fbD_{\alpha+1})$.
\item $\fbD_{\alpha+1} = \bD_{\alpha+1}$ on $\scrN(\fbD_{\alpha}^*, \bB^*_{\alpha})$.
\end{enumerate}
\item[$\mathrm{(\DD)}$] For Dirichlet conditions:
\begin{enumerate}[label=(\alph*), itemsep=0pt]
\item $\scrR(\fbD_{\alpha}; \bB_{\alpha}) \subseteq \scrN(\fbD_{\alpha+1})$.
\item $\fbD_{\alpha+1} = \bD_{\alpha+1}$ on $\scrN(\fbD^*_{\alpha})$.
\end{enumerate}
\end{enumerate}
As a byproduct of this characterization, the unique mappings $\fbD_{\alpha}$ are adapted Green systems, differing from the original systems by terms of zero order and zero class within the calculus. We refer to the resulting sequence collectively as $(\fbD_{\bullet})$, and call it \emph{the lifted complex} induced by $(\bD_{\bullet})$.
\end{theorem}

The discussion on the lifted complex following \thmref{thm:lifted_complexIntro}, alongside the results for the associated Hodge theories, holds essentially verbatim in this abstract setting, provided the relevant spaces are appropriately adapted. In fact, in the main body of the text, it is \thmref{thm:lifted_complexIntro2} that we prove directly, subsequently recovering \thmref{thm:lifted_complexIntro} and its corollaries via the substitutions \eqref{eq:particular_case_intro} (see \secref{sec:dirichelt_special_case} for the explicit translation).

In the case of \thmref{thm:lifted_complexIntro2}, executing the lifting procedure and establishing the general Hodge theory requires navigating through numerous subtleties arising from the arbitrary nature of the systems; addressing these technicalities is one of the main undertakings of this work.

\subsubsection{Hodge theories}
For the full statements of the Hodge theorems in this generality, see \thmref{thm:hodge_like_corrected_complex} and \thmref{thm:hodge_like_corrected_complexD} in the body of the text. At this stage, we provide a quick survey.
 
In the notation of \thmref{thm:lifted_complexIntro2}, under Neumann conditions, the resulting Hodge decomposition reads:
\beq
\Gamma(\bbF_{\alpha+1};\bbG_{\alpha+1})=
\lefteqn{\overbrace{\phantom{\scrR(\fbD_\alpha)\oplus \module_{\N}^{\alpha+1}}}^{\scrN(\fbD_{\alpha+1})}}
\scrR(\fbD_\alpha)
\oplus
\underbrace{\module_{\N}^{\alpha+1}\oplus \scrR(\fbD_{\alpha+1}^*;\bB^*_{\alpha+1})}_{\scrN(\fbD_\alpha^*,\bB^*_\alpha)},
\label{eq:HodgelikesmoothIntroGen}
\eeq
whereas under Dirichlet conditions, it reads:
\beq
\Gamma(\bbF_{\alpha+1};\bbG_{\alpha+1})=
\lefteqn{\overbrace{\phantom{\scrR(\fbD_\alpha;\bB_\alpha)\oplus \module_{\D}^{\alpha+1}}}^{\scrN(\fbD_{\alpha+1},\bB_{\alpha+1})}}
\scrR(\fbD_\alpha;\bB_\alpha)
\oplus
\underbrace{\module_{\D}^{\alpha+1}\oplus \scrR(\fbD_{\alpha+1}^*)}_{\scrN(\fbD_\alpha^*)}.
\label{eq:HodgelikesmoothDIntroGen}
\eeq
Moreover, generalizing the prototypical settings, the cohomology modules reduce to expressions involving the original systems in the diagram \eqref{eq:elliptic_pre_complex_diagram_intro}, as stated in \thmref{thm:coadjointN} and \thmref{thm:coadjointD}:
\beq
\module_{\N}^{\alpha+1}=\ker(\bD_{\alpha+1},\ttbP_{\alpha-1}\bD^*_\alpha,\bB^*_\alpha), \qquad \module_{\D}^{\alpha+1}=\ker(\bD_{\alpha+1},\ttbP_{\alpha-1}\bD^*_\alpha,\bB_{\alpha+1}),
\label{eq:cohomology_groupsIntro}
\eeq
where $\tbP_{\alpha-1}$ denotes the corresponding $L^{2}$ projection onto $\scrN(\fbD_{\alpha-1}^*)$ (resp.\ $\scrN(\fbD_{\alpha-1}^*;\bB_{\alpha-1}^*)$) in the decompositions \eqref{eq:HodgelikesmoothIntroGen} and \eqref{eq:HodgelikesmoothDIntroGen}. We reiterate that the specific Hodge theories detailed in \secref{sec:Hodge_intro_NN} and \secref{sec:Hodge_intro_DD} follow directly from this general framework through the substitutions \eqref{eq:particular_case_intro}.

Similarly, we state the resulting cohomological formulations for the associated boundary-value problems. In the statements, note that since the systems in the diagram \eqref{eq:elliptic_pre_complex_diagram_intro} mix boundary and interior sections, there is no longer a distinction between data prescribed on $M$ and data prescribed on $\dM$, unlike in the statements of the prototypical \thmref{thm:DDintroPro} and \thmref{thm:NNintroPro}. Instead, data are prescribed directly on the primary systems $\bD_{\alpha}$, while the homogeneous conditions on $\fbD_{\alpha-1}^*$ and the boundary systems serve as supplementary constraints enabling solvability and uniqueness:
\begin{theorem}[Neumann cohomological formulation -- General setting]
\label{thm:NNintro}
Let $(\bD_{\bullet})$ be an elliptic pre-complex based on Neumann conditions. 
Given $\Theta \in \Gamma(\bbF_{\alpha+1};\bbG_{\alpha+1})$, the boundary-value problem
\[
\begin{gathered} 
\bD_\alpha \Psi = \Theta, \\\fbD_{\alpha-1}^*\Psi= 0, \qquad \bB^*_{\alpha-1}\Psi = 0
\end{gathered}
\]
admits a solution $\Psi \in \Gamma(\bbF_\alpha;\bbG_\alpha)$ if and only if: 
\[
\fbD_{\alpha+1}\Theta = 0, 
\qquad 
\Theta \perp_{L^2} \module_{\N}^{\alpha+1}.
\]
The solution is unique modulo $\module_{\N}^{\alpha}$.
\end{theorem}

The analog for Dirichlet conditions is: 
\begin{theorem}[Dirichlet cohomological formulation -- General setting]
\label{thm:DDintro}
Let $(\bD_{\bullet})$ be an elliptic pre-complex based on Dirichlet conditions. 
Given $\Theta \in \Gamma(\bbF_{\alpha+1};\bbG_{\alpha+1})$, the boundary-value problem
\[
\begin{gathered} 
\bD_\alpha \Psi = \Theta, \\\fbD_{\alpha-1}^*\Psi = 0, \qquad \bB_\alpha\Psi = 0
\end{gathered}
\]
admits a solution $\Psi \in \Gamma(\bbF_\alpha;\bbG_\alpha)$ if and only if: 
\[
\fbD_{\alpha+1}\Theta = 0, 
\qquad 
\bB_{\alpha+1} \Theta = 0, 
\qquad 
\Theta \perp_{L^2} \module_{\D}^{\alpha+1}.
\]
The solution is unique modulo $\module_{\D}^{\alpha}$.
\end{theorem}

\section{Examples: Overview}
\label{sec:examples_intro}
\subsection{Outline}
Now that the abstract theory has been laid down, we are in a position to revisit the examples listed in \secref{sec:main_results_examples_1}. We will also present further variants based on the distinction between Neumann and Dirichlet conditions, as well as consider examples of sequences consisting of fully fledged Douglis--Nirenberg systems.

As hinted in \secref{sec:main_results_examples_1}, the most immediate examples of elliptic pre-complexes are those of {exterior covariant derivatives}, generalizing the classical de Rham complex. These were already touched upon in the rudimentary work \cite{KL23}. Here we study them in the highest generality, including operations between boundary sections (i.e., $Q \ne 0$ and $T \ne 0$ in \eqref{eq:Boutet_de_movel_intro}), thereby illustrating several types of elliptic pre-complexes in their arguably simplest nontrivial forms.

Another archetypal example comprises the much-studied Killing and Hessian equations, which were our primary motivation for the original onset of elliptic pre-complexes and generalized Hodge theory in \cite{KL23}. These equations were classically treated within the scope of the \emph{Calabi complex}, and motivated much of the classical theory of compatibility complexes for overdetermined systems. They are also relevant in many problems arising in the calculus of variations and material science.

Finally, we present the {Riemann curvature equations}, the last example from \secref{sec:main_results_examples_1}, which is perhaps the richest of all the examples considered in this work. As stated in \secref{sec:main_results_examples_1}, the treatment of the Ricci curvature equations is deferred to a separate work \cite{Led25B}. 

Following the premise established in the introduction in \secref{sec:opening}, our goal in the examples here is to demonstrate how nonlinear geometric problems falling within the setting of \eqref{eq:geoemtric_eq_intro}--\eqref{eq:gauge_symmetries_sequence} can be linearized and formulated within elliptic pre-complexes, and how the order-reduction property arises from the linearization of geometric constraints and the enforcement of gauge invariance. To this end, we examine the nonlinear components of these problems in detail, even though they are at times somewhat removed from the linear analysis.

\subsection{Exterior covariant derivatives}
\label{sec:exterior_intro}
Now that we have the machinery of an elliptic pre-complex in hand, we begin by placing the setup around \thmref{thm:exterior_chomology_intro} in a broader and more accurate context. Let $\bbU \to M$ be a Riemannian vector bundle equipped with a connection $\nabla$, and let
\[
\Omega^{\alpha}(M;\bbU) = \Gamma(\Lambda^{\alpha}T^*M \otimes \bbU).
\]
denote the space of $\bbU$-valued differential forms. The {exterior covariant derivatives} and their adjoints,  
\[
\dU:\Omega^{\alpha}(M;\bbU) \to \Omega^{\alpha+1}(M;\bbU), \qquad \delU:\Omega^{\alpha+1}(M;\bbU) \to \Omega^{\alpha}(M;\bbU),
\]  
arise in various geometric and analytical contexts, including Bochner techniques \cite[Ch.~9]{Pet16}, gauge theory \cite[Ch.~1.4,6]{RS17}, \cite[App.~C.6]{Tay11b}, and harmonic maps \cite{EL83}.  

As stated before \thmref{thm:exterior_chomology_intro}, although the resulting sequence of operators (sometimes referred to as the \emph{twisted de Rham complex} \cite[p.~458]{RS17}) provides a natural generalization of the de Rham complex (corresponding to the case $\bbU = M \times \bbR$), it does not form an elliptic complex. To see this, we substitute in the classical diagram \eqref{eq:elliptic_complex_diagram} : 
\[
\bbE_{\alpha} = \Lambda^{\alpha}T^*M \otimes \bbU \qquad \bbJ_{\alpha} = \Lambda^{\alpha}T^*\dM \otimes \jmath^*\bbU,
\]
where $\jmath:\dM \hookrightarrow M$ is the inclusion and $\jmath^*\bbU \to \dM$ is the pullback bundle, together with  
\[
A_{\alpha} = \dU \qquad A_{\alpha}^* = \delU \qquad B_{\alpha} = \PtD \qquad B^*_{\alpha} = \PnD,
\]
where $\PtD:\Omega^{\alpha}(M;\bbU)\rightarrow\Omega^{\alpha}(\dM;\jmath^*\bbU)$ and $\PnD:\Omega^{\alpha}(M;\bbU)\rightarrow\Omega^{\alpha-1}(\dM;\jmath^*\bbU)$ are the tangential and normal projections of differential forms. Note that $\PtD$ coincides with the pullback operation on the form part, while on the vector-bundle part it is given by pullback to $\jmath^*\bbU$. These yield Green formulae of the form \eqref{eq:Green_elliptic_intro}: 
\beq
\bra \dU\omega, \eta\ket_{L^{2}(M)} = \bra\omega, \delU\eta\ket_{L^{2}(M)} + \bra \PtD\omega, \PnD\eta\ket_{L^{2}(\dM)},
\label{eq:Green's_formula_exterior} 
\eeq
together with the ellipticity of the Neumann boundary-value problems for the ``Laplacian''  
\[
D^*_{\alpha}D_{\alpha}= \dU\delU + \delU\dU.
\]

However, unless $\nabla$ is locally flat, the sequence $(d_{\nabla})$ does not form a cochain complex, since in general we have  
\beq
A_{\alpha+1}A_{\alpha}=\dU \dU = \mathrm{R}_{\nabla},
\label{eq:emergent_curvature_exterior_intro}
\eeq  
where $\mathrm{R}_{\nabla} \in \Omega^{2}(M;\End(\bbU))$ is the curvature endomorphism of the connection $\nabla$, which does not necessarily vanish.

Our key observation is that, in the context of the order-reduction property introduced in \secref{sec:order-reduction_intro} and \secref{sec:order-reduction_intro2}, the following identity does hold for every connection $\nabla$, extending \eqref{eq:emergent_curvature_exterior_intro}:
\beq
\begin{pmatrix} \dU & 0 \\ \PtD & -d_{\jmath^*\nabla} \end{pmatrix}  
\begin{pmatrix} \dU & 0 \\ \PtD & -d_{\jmath^*\nabla} \end{pmatrix}  
= \begin{pmatrix} \mathrm{R}_{\nabla} & 0 \\ 0 & \mathrm{R}_{\jmath^*\nabla} \end{pmatrix}.
\label{eq:order_reduction_exterior}
\eeq
Here each Green operator on the left is of order and class $1$, while the operator on the right is of order zero and class zero. 

Coupled with the overdetermined ellipticity of the usual de Rham system, this identity allows exterior covariant derivatives to accommodate various elliptic pre-complexes based on either Neumann or Dirichlet conditions, independently of the curvature of the connection $\nabla$. We survey some of these below. 

\subsubsection{Dirichlet picture}
\label{para:DD_exterior_intro}
An elliptic pre-complex based on Dirichlet conditions, in the prototypical sense \defref{def:elliptic_pseudo_complex_intro}, is obtained by setting: 
\[
\bA_\alpha:=\dU:\Omega^{\alpha}(M;\bbU)\rightarrow \Omega^{\alpha+1}(M;\bbU).
\]
The required order-reduction property follows directly from \eqref{eq:order_reduction_exterior}, whereas the overdetermined ellipticity conditions in \eqref{eq:overdetermined_intro} reduce to the overdetermined ellipticity of the system
\beq
\dU \oplus \delU \oplus \PtD
\label{eq:Dirichelt_exterior_intro}
\eeq
which is established as in classical Hodge theory (cf.~\secref{sec:exterior_covariant_derivatives}). 

\thmref{thm:lifted_complexIntro} then asserts that the lifted complex consists of a sequence of operators 
\[
\mathpzc{d}_{\nabla}:\Omega^{\alpha}(M;\bbU) \rightarrow \Omega^{\alpha+1}(M;\bbU),
\]
differing from $\dU$ by ``correction'' terms of order and class zero, and satisfying
\[
\mathpzc{d}_{\nabla}\mathpzc{d}_{\nabla}\omega=0 \textand \PtD\mathpzc{d}_{\nabla}\omega=0 \qquad \text{for} \qquad \omega\in\Omega^{\alpha}(M;\bbU)\cap\ker\PtD, 
\]
with adjoints $\pzcdel_{\nabla}:\Omega^{\alpha+1}(M;\bbU)\rightarrow \Omega^{\alpha}(M;\bbU)$ satisfying $\pzcdel_{\nabla}\pzcdel_{\nabla}=0$ identically. The corresponding Dirichlet cochain complex provided by \eqref{eq:Dirichlet_cochain_omplexIntroPro} is then precisely the one presented in \eqref{eq:exterior_complexlift} (after relabeling $\alpha$ and $k$). 

The cohomology modules of the lifted complex can now be identified in view of \thmref{thm:hodge_like_lifted_complexDIntro}, yielding the ones in \eqref{eq:exterior_cohomology} (again, after relabeling $\alpha$ and $k$). In particular, for $\alpha = 0$, since $\dU = \nabla$ on zero forms and $\PtD = |_{\partial M}$ is the restriction to the boundary:
\[
\module_{\D}^{0}(\nabU) = \ker(\dU,\PtD) = \ker(\nabla,|_{\partial M}) = \{0\},
\]
which is the (trivial) space of all $\nabla$-parallel fields vanishing at the boundary. 

For $k > 0$,
\[
\module_{\D}^{\alpha}(\nabU)  = \ker(\dU,\pzcdel_{\nabU},\PtD)
\]
may be nontrivial (e.g., harmonic forms tangent to the boundary in the case $\bbU = M \times \bbR$; see \cite{Sch95b}). In any case, \thmref{thm:DDintroPro} provides us with the proof of \thmref{thm:exterior_examples_1}. Here we write it more explicitly, and split it for $k=0$ and $k>0$. The first case reads:

\begin{theorem}
Let $\omega \in \Omega^1_{M;\bbU}$.  
The boundary-value problem
\[
\begin{aligned}
& \nabla s = \omega \quad && \text{in } M, \\
& s|_{\dM} = 0 \quad && \text{on } \dM
\end{aligned}
\]
admits a solution $s \in \Omega^{0}_{M;\bbU}$ if and only if
\[
\mathpzc{d}_{\nabla}\omega = 0, 
\qquad 
\PtD\omega = 0, 
\qquad 
\omega \perp \module^1_{\D}(\nabU) .
\]
The solution is unique.
\end{theorem}

For the higher-rank segments:

\begin{theorem}
Let $k>0$ and $\omega \in \Omega^{\alpha+1}(M;\bbU)$. The boundary-value problem 
\[
\begin{aligned}
& \dU \psi = \omega && \text{in } M, \\
& \PtD \psi = 0  && \text{on } \dM
\end{aligned}
\]
admits a solution $\psi \in \Omega^{\alpha}(M;\bbU)$ satisfying $\pzcdel_{\nabla}\psi = 0$ if and only if 
\[
\mathpzc{d}_{\nabla} \omega = 0, 
\qquad 
\PtD \omega = 0,
\qquad 
\omega \perp \module_{\D}^{\alpha+1}(\nabU).
\]
Any solution satisfying $\pzcdel_{\nabla}\psi = 0$ is unique modulo $\module^{\alpha}_{\D}(\nabU)$. 
\end{theorem}

\subsubsection{Neumann picture} 
\label{para:NN_exterior_intro}
One can obtain a Neumann elliptic pre-complex by setting $\bA_\alpha := \dU$ and selecting the natural Neumann boundary conditions arising from the Green's formula \eqref{eq:Green's_formula_exterior}. This is relatively straightforward, follows lines similar to the Dirichlet case above, and is essentially contained in \cite[Sec.~1.3]{KL23}. For the sake of variety, we shall repeat the analysis explicitly in the body of the text (\secref{sec:NN_exterior_intro}), and instead demonstrate here the Neumann picture in higher generality.

To that end, consider the following sequence, falling into the more general scope of \secref{sec:general_intro}:
\beq 
\bD_\alpha:= 
\begin{pmatrix} \dU & 0 \\ \PtD & -d_{\jmath^*\nabla} \end{pmatrix}:\mymat{\Omega^{\alpha}(M;\bbU)\\\oplus\\\Omega^{\alpha-1}{(\partial M;\jmath^*\bbU)}}\longrightarrow\mymat{\Omega^{\alpha+1}(M;\bbU)\\\oplus\\\Omega^{\alpha}{(\partial M;\jmath^*\bbU)}}
\label{eq:exterior_disjoint_intro}
\eeq 

For this sequence, again, the order-reduction property \eqref{eq:order_reduction_intro} is satisfied by virtue of \eqref{eq:order_reduction_exterior}. The Neumann overdetermined ellipticity conditions required in \eqref{eq:overdetermined_intro} can be shown to decouple into those of systems  
\[
\begin{pmatrix} \dU \oplus \delU & 0 \\ \PnD & 0 \end{pmatrix}, \qquad \begin{pmatrix} 0 & 0 \\ 0 & d_{\jmath^*\nabla} \oplus \delta_{\jmath^*\nabla} \end{pmatrix},
\]
which, as in the Dirichlet case, are established as in classical Hodge theory.

Applying \thmref{thm:lifted_complexIntro2} under Neumann conditions yields a lifted complex $(\fbD_\bullet)$, satisfying $\fbD_{\alpha+1} \fbD_\alpha = 0$ identically, and taking the form: 
\[
\begin{aligned}
\fbD_0 &=\begin{pmatrix} 
\dU & 0 \\ 
\PtD & 0
\end{pmatrix}=\begin{pmatrix} 
\nabla & 0 \\ 
|_{\partial M} & 0
\end{pmatrix}
&&  \\[1em]
\fbD_\alpha &= 
\begin{pmatrix} 
\mathpzc{d}_{\nabla} & -\mathpzc{k}_{\nabla} \\ 
\PtD-\mathpzc{c}_{\nabla} & - \mathpzc{d}_{\jmath^*\nabla} 
\end{pmatrix}, \qquad \alpha>0.
\end{aligned}
\]

Note that, since the systems $\bD_\alpha$ in the original sequence include operators taking values in boundary sections, these are lifted as well, yielding operators $\mathpzc{d}_{\jmath^*\nabla}$ and $\PtD-\mathpzc{c}_{\nabla}$, which differ from $d_{\jmath^*\nabla}$ and $\PtD$ by terms of order and class zero. In particular, the upper-right corner in the higher segments of the lifted complex contains an operator $\mathpzc{k}_{\nabla}$ of order zero. 

It can be shown that $(\psi;\lambda) \in \scrN(\fbD_{\alpha}^*; \bB_{\alpha}^*)$ amounts to the conditions:
\[
\pzcdel_{\nabla} \psi = \mathpzc{c}_{\nabla}^* \lambda, \qquad \mathpzc{k}^*_{\nabla} \psi = \pzcdel_{\jmath^* \nabla} \lambda, \qquad \PnD \psi = -\lambda.
\]
Hence, the cohomology modules $\module^{\alpha}_{\NN}(\bbU)$ of the lifted complex depend consist of smooth vector-valued forms satisfying:
\beq
\dU \psi = 0, \qquad \pzcdel_{\nabU} \psi = \mathpzc{c}_{\nabla}^* \lambda, \qquad \mathpzc{k}^*_{\nabla} \psi =\pzcdel_{\jmath^* \nabla} \lambda, \qquad \PtD \psi = d_{\jmath^* \nabla} \lambda, \qquad \PnD \psi = -\lambda.
\label{eq:cohomology_neumann_exterior} 
\eeq
For $\alpha = 0$ and $\alpha=1$, the cohomology depends solely on the original systems. So for $\alpha=0$, these conditions reduce to:
\[
\nabla \psi = 0, \qquad \psi|_{\dM} = 0,
\]
and thus it always holds that $\module^{0}_{\NN} = \BRK{0}$. 

For $\alpha=1$, we have: 
\[
\dU\psi=0, \qquad \delta_{\nabU}\psi=0, \qquad \PtD\psi=d_{\jmath^*\nabU}\lambda, \qquad \PnD\psi=-\lambda. 
\]
Which can yield vanishing theorems under curvature assumptions, using e.g., Bochner technique. For $\alpha > 0$, in general $\module^{\alpha}_{\NN}(\bbU) \ne \BRK{0}$. 

Regardless of the vanishing of the cohomology modules, in the most general case, \thmref{thm:NNintro} takes the following form, once we decouple the boundary and interior equations:
\begin{theorem}
\label{thm:exterior_chomology_intro}
Given $\omega \in \Omega^{\alpha+1}(M;\bbU)$ and $\rho \in \Omega^{\alpha}(\dM;\jmath^*\bbU)$, the boundary-value problem
\[
\begin{aligned}
&\dU \psi = \omega, \qquad \pzcdel_{\nabla} \psi = \mathpzc{c}_{\nabla}^* \lambda \quad && \text{in } M, \\
& \PtD \psi - d_{\jmath^*\nabla}\lambda = \rho, \quad 
\mathpzc{k}_{\nabla} \psi = \pzcdel_{\jmath^*\nabla} \lambda, \quad 
\PnD \psi = -\lambda \quad && \text{on } \dM
\end{aligned}
\]
admits a solution $(\psi;\lambda) \in \Omega^{\alpha}(M;\bbU) \oplus \Omega^{\alpha-1}(\dM;\jmath^*\bbU)$ if and only if
\[
\mathpzc{d}_{\nabla} \omega = \mathpzc{k}_{\nabla} \rho, 
\qquad 
\PtD \omega - \mathpzc{c}_{\nabla}\omega = \mathpzc{d}_{\jmath^*\nabla} \rho, 
\qquad 
(\omega;\rho)\perp \module^{\alpha+1}_{\NN}.
\]
The solution is unique modulo $\module^{\alpha}_{\NN}$.
\end{theorem}
\subsection{Killing and Hessian equations} 
\label{sec:KillingHessianIntro}
In the same spirit, we revisit the Killing and Hessian equations presented in \secref{sec:main_results_examples_1}. The results presented there correspond to the Dirichlet setting, and here we shall also present the Neumann setting. Much of the discussion in this section is an adaptation and reframing of the discussion and analysis in \cite{KL23}. We include it here for the sake of completeness, to show how our sharper results apply, and to demonstrate how it fits into the outline of our premise in \secref{sec:opening}.

Historically, the study of these problems in the Riemannian setting originates from a problem posed by Calabi \cite{Cal61} more than sixty years ago, which was the primary subject of \cite{KL23}:

\begin{quote}
Let $(M,g)$ be a compact Riemannian manifold, possibly with nonempty boundary. 
What are necessary and sufficient conditions for a symmetric $(2,0)$-tensor on $M$ to be in the range of the Killing operator? 
\end{quote}

In the context of our premise in \secref{sec:opening}, this becomes: given a symmetric tensor field $\sigma \in S^{2}(M)$, identify the solvability conditions for the \emph{Killing equation}: 
\beq
\sigma = \tfrac{1}{2}\,\calL_{Y}g, \qquad Y\in\frakX(M)
\label{eq:Saint_Venant_Problem} 
\eeq
where $Y \mapsto \calL_{Y}g : \frakX(M) \to S^{2}(M)$ is the Killing operator.

One can refine the problem further; by imposing $Y = \nabla f$, where $f \in C^{\infty}_{M}$ is a smooth scalar function, this becomes the \emph{Hessian equation} \cite{BE69,Bry13a,Bry25}:
\beq
\sigma = H_{g}f,
\label{eq:Hessian_Problem} 
\eeq
where $\Hg: C^{\infty}_{M} \to S^{2}(M)$ is the Riemannian Hessian operator. 

Calabi provided an answer for \eqref{eq:Saint_Venant_Problem} when $(M,g)$ is closed, simply-connected, and has constant sectional curvature. In particular, he commented at the end of his paper:

\begin{quotation}
``In a subsequent article... (the) theorem will be supplemented by an analogue of Hodge's theorem... satisfying globally certain elliptic systems of equations."
\end{quotation}

Such an article has never been published.

From the perspective of the premise in \secref{sec:outlook}, the Killing equation \eqref{eq:Saint_Venant_Problem} can be viewed as the linearization of the nonlinear problem of finding isometric immersions between two different metrics $g,h$ on $M$:
\beq 
\phi:(M,h)\to(M,g) \qquad \text{such that }\qquad \phi^*g=h.
\label{eq:isometric_imm}  
\eeq 
Indeed, if one takes a smooth path $\phi_{t}:M\to M$ of immersions (not necessarily isometric) with $\phi_0=\id$ and $\dertZero\phi_{t}=Y$ (on a manifold with nonempty boundary, this might not be the flow of $Y$ but rather $\phi_{t}=\exp(tY)$), and assumes a perturbative structure $h=g+t\sigma+o(t)$, then the equation \eqref{eq:isometric_imm} can be written as: 
\[
\phi_{t}^*g=g+t\sigma+o(t);
\] 
Linearizing both sides thus yields \eqref{eq:Saint_Venant_Problem}. We remark that, since many problems in the calculus of variations and materials science can be formulated as problems of finding isometric immersions, the system \eqref{eq:Saint_Venant_Problem} has been central in these fields (see, for instance, \cite{BE69,Gur72,CCGK07,Yav13,GLM13,DPR16,Bog19,KL22,GR22}).  

In the case of $(M,g)=(\Omega,e)$---namely, simply-connected Euclidean domains---it is well known that $\sigma=\calL_{Y}e$ if and only if $\sigma$ satisfies the \emph{compatibility condition}, 
\beq
\nabla\times\nabla\times\sigma=0,
\label{eq:curlcurlIntro}
\eeq
where $\nabla\times\nabla\times$ is the curl-curl operator, mapping symmetric $(2,0)$-tensors into $(4,0)$-tensors satisfying the Bianchi symmetries of algebraic curvatures, given in Euclidean coordinates by:
\beq 
(\nabla\times\nabla\times\sigma)_{ijkl}=\partial^2_{ik}\sigma_{jl}-\partial^2_{jk}\sigma_{il}-\partial^2_{il}\sigma_{jk}+\partial^2_{jl}\sigma_{ik}. 
\label{eq:curlcurlIntro2}
\eeq 
Thus, for $\Omega$ Euclidean and simply-connected, the compatibility condition \eqref{eq:curlcurlIntro} becomes the solvability condition for the Killing equation \eqref{eq:Saint_Venant_Problem}. A similar statement can be obtained for the Hessian equation \eqref{eq:Hessian_Problem}, once one replaces $\nabla\times\nabla\times$ with a vectorial version of the curl operator \cite{BE69, Bry13a}.

Over the years, several authors have, to some extent, generalized this compatibility condition to the Riemannian setting \cite{Cal61,Eas00,GG88A,CSS01,CELM21}. As we noted in the discussion in \secref{sec:main_results_examples_1}, all previous works recognize that, at the heart of the analysis, lies the search for a second-order linear differential operator acting on symmetric tensor fields and generalizing \eqref{eq:curlcurlIntro2}, which annihilates the range of $Y \mapsto \calL_Y g$. In other words, from the perspective laid out in \secref{sec:opening}, all attempts to resolve the Killing equation have funneled toward the search for a cochain complex:
\beq
\begin{tikzcd}
\frakX(M) \arrow[rr, "Y\mapsto\calL_Yg"] &  & S^{2}(M)  \arrow[rr, "\text{?}"] &  & \Ckm22,
\end{tikzcd}
\label{eq:saint_venant_complex_mission}
\eeq
where $\Ckm22$ denotes the space of $(4,0)$-tensor fields satisfying algebraic Bianchi symmetries, notation that will be clarified below. Such a complex, in cases where it has been constructed, is known in the literature as a \emph{Calabi complex} \cite{Gol67,GG88A,Eas00}. The difficulty in producing such a complex in general Riemannian geometries is a prime example of the challenge outlined in \secref{sec:opening} and \secref{sec:main_results_examples_1}. However, despite extensive research, to the best of our knowledge, no work has managed to remove the assumption that the underlying Riemannian manifold has, at the very least, a parallel curvature tensor.

In the literature, the Killing equation \eqref{eq:Saint_Venant_Problem} is the archetypal example of an overdetermined system of partial differential equations (\cite{Gol67,Spe69,BE69}; see also the encyclopedic entry \cite{DS96} and references therein). The study of such systems has a long history, and from the premise outlined in \secref{sec:opening}, the state of the art regarding the existence of a compatibility complex for a given operator $A$ is applicable only under quite restrictive conditions on its coefficients. For example, a differential operator with constant coefficients always admits a compatibility complex. 

The Killing operator, however, does not satisfy the required conditions in a general Riemannian manifold—particularly in those having a boundary or non-trivial topology, areas of study with scarce literature.

Here we resolve these problems entirely by casting both \eqref{eq:Saint_Venant_Problem} and \eqref{eq:Hessian_Problem} within elliptic pre-complexes based on Neumann conditions. To set up these complexes, following Calabi's original work, we proceed within the general algebraic framework of \emph{double forms}, which are the section spaces: 
\[
\Wkm{\alpha}{m}=\Gamma(\Lkm{\alpha}{m}), \qquad \Lkm{\alpha}{m}=\Lambda^k T^*M \otimes \Lambda^m T^*M, \qquad k,m \in \Nzero. 
\]
See also \cite{Cal61,Gra70,Kul72,KL21a}; we provide an appendix with a full introduction (\secref{sec:Appendix}). As Calabi realized, many objects in differential geometry can be cast as double forms. Riemannian metrics, Hessians of scalar functions, and Lie derivatives of the metric can be viewed as symmetric elements of $\Wkm11$; for higher-order examples, the Riemannian curvature tensor $\Rm_\g$ can be viewed as a symmetric element of $\Wkm22$. 

Most prominently, however, curvature tensors satisfy additional symmetries---namely, \emph{algebraic Bianchi identities}. For this purpose, Calabi introduced an algebraic symmetry that can be imposed on all double forms, generalizing this identity. This symmetry can be expressed as the kernel of a smooth bundle map $\G : \Lkm{\alpha}{m} \to \Lkm{\alpha+1}{m-1}$; let the section spaces be:
\[
\Ckm{\alpha}{m}=\Gamma(\ker\G).
\]
Following Calabi, we call double forms satisfying these symmetries \emph{Bianchi forms}. Notable examples of Bianchi forms include standard differential forms $\Ckm{\alpha}{0}$, symmetric tensor fields $\Ckm{1}{1}$ (which we earlier denoted, as in the literature, by $S^2_M$; we will interchange between these as convenient), and $(4,0)$-covariant tensor fields satisfying the algebraic Bianchi identity, $\Ckm{2}{2}$.

From the perspective of the setting in \secref{sec:exterior_intro}, these double forms are essentially vector-valued forms taking values in tensor products of the exterior algebra. Therefore, exterior covariant derivatives, introduced in \secref{sec:exterior_intro}, apply to them as well, yielding the operations:
\[
\dg : \Wkm{\alpha}{m} \to \Wkm{\alpha+1}{m}
\qquad \text{and} \qquad
\deltag : \Wkm{\alpha+1}{m} \to \Wkm{\alpha}{m}
\]
where the background connection arises from the Riemannian metric. We use these to define differential operators that act on the vector part via involution:
\[
\dgV : \Wkm{\alpha}{m} \to \Wkm{\alpha}{m+1}
\qquad \text{and} \qquad
\deltagV : \Wkm{\alpha}{m+1} \to \Wkm{\alpha}{m}
\]
where $\dgV \psi = (\dg \psi^T)^T$ and $\deltagV \psi = (\deltag \psi^T)^T$.

Although $\deltag$ preserves the Bianchi symmetry, $\dg$ does not, unless $k>m$. To this end, we introduce the \emph{Bianchi derivative}:
\[
\dBianchi = \calP_{g}\dg,
\]
where $\calP_{g}:\Wkm{\alpha}{m}\to \Wkm{\alpha}{m}$ is the tensorial projection onto $\Ckm{\alpha}{m}$. Since $\calP_{g}$ coincides with symmetrization when $k=m=1$, we find that $\dBianchi$ coincides with the Killing operator up to the musical isomorphism: 
\beq 
\dBianchi\omega=\tfrac{1}{2}\calL_{\omega^\sharp}g.
\label{eq:BianchiKilling}
\eeq 

With this noted, Calabi used these operations to rewrite the curl-curl operator from \eqref{eq:curlcurlIntro} and its adjoint as second-order operators acting on Bianchi forms, $\Hg : \Ckm{\alpha}{m} \to \Ckm{\alpha+1}{m+1}$ and $\Hg^* : \Ckm{\alpha+1}{m+1} \to \Ckm{\alpha}{m}$, defined by
\beq 
\Hg = \tfrac12(\dg \dgV + \dgV \dg) 
\qquad \text{and} \qquad 
\Hg^* = \tfrac12(\deltag \deltagV + \deltagV \deltag).
\label{eq:RiemcurlcurlIntro}
\eeq 
It can be shown that these operators satisfy integration by parts formulas involving both tensorial and first-order normal boundary operators, denoted by $B_{H}$ and $B_{H}^*$:
\[
\bra \Hg \psi, \eta \ket = \bra \psi, \Hg^* \eta \ket + \bra B_H \psi, B_H^* \eta \ket. 
\]
See again \secref{sec:second_orderBianchi} for a full elaboration. An explicit calculation shows that, as with the exterior covariant derivative, the iterated operators 
\[
\Hg \dBianchi, \qquad \dBianchi \Hg
\] 
which might be expected to be of order three, are in fact differential operators of order $1$---hence, in the context of \secref{sec:order-reduction_intro}, yielding order-reduction properties. These properties arise from geometric considerations and the fact that
\[
\dg\dg \quad \text{and} \quad \dg\dgV-\dgV\dg
\]
are tensorial expressions depending explicitly on curvature, as in \secref{sec:exterior_intro}. Alternatively, we will show in \secref{sec:prescribed_curvature_intro} how directly this follows from linearizing geometric identities, where $\Hg$ is recognized as the leading term in the variational formula for the Riemann curvature tensor.

The required Neumann overdetermined ellipticity is verified through a computation involving a generalization of the \emph{Lopatinskii--Shapiro condition}, for which we refer the reader to the body of the text (cf.~\secref{sec:KillingHessianBody}). 

Once this condition is established for every $1 \le m \le d$, we find that the following sequence constitutes an elliptic pre-complex based on Neumann conditions (\defref{def:elliptic_pseudo_complex_intro}):
\beq 
\begin{split} 
&\bA_{0} := \dBianchi : \Ckm{0}{m} \rightarrow \Ckm{1}{m} \\
&\bA_{1} := \dBianchi : \Ckm{1}{m} \rightarrow \Ckm{2}{m} \\
&\vdots \\
&\bA_{m-1} := \dBianchi : \Ckm{m-1}{m} \rightarrow \Ckm{m}{m} \\
&\bA_{m} := \Hg : \Ckm{m}{m} \rightarrow \Ckm{m+1}{m+1} \\
&\bA_{m+1} := \dBianchi : \Ckm{m+1}{m+1} \rightarrow \Ckm{m+2}{m+1} \\
&\vdots
\end{split} 
\label{eq:calabi_complex} 
\eeq
Calabi considered these sequences in his work \cite{Cal61} on manifolds with empty boundary and constant curvature; here, we refer to them as \emph{Bianchi complexes}. For the full diagram falling into the scope of \eqref{eq:elliptic_complex_diagram}, we refer to \secref{sec:KillingHessianBody}.

At this point, it is a matter of verification that, for $m=1$ and $m=0$, respectively, the Dirichlet version of this complex yields the ones outlined in \thmref{thm:Dirichlet_Lie} and \thmref{thm:Hessian_intro}, in analogy with the Dirichlet picture for exterior covariant derivatives. We shall see this in a broader context in the next example, \secref{sec:prescribed_curvature_intro}, and focus here, for the sake of variety, on the Neumann picture.

The Neumann version of the lifted complex \thmref{thm:lifted_complexIntro} associated with this family of complexes gives rise to a uniquely determined sequence of operators
\[
\begin{aligned}
&\DBianchi: \Ckm{\alpha}{m} \to \Ckm{\alpha+1}{m}  
&\qquad\qquad
&k=0,\dots,m-1 \\
&\bHg: \Ckm{m}{m} \to \Ckm{m+1}{m+1} \\
&\DBianchi: \Ckm{\alpha}{m+1} \to \Ckm{\alpha+1}{m+1}  
&\qquad\qquad
&k=m+1,\dots,d-1,
\end{aligned}
\]
satisfying identically
\[
\DBianchi\DBianchi=0 \qquad \DBianchi \bHg=0 \qquad \bHg\DBianchi=0,
\]
such that $\DBianchi-\dBianchi$ and $\bHg-\Hg$ are Green operators of order zero, vanishing on $\Ckm{0}{m}$. We denote the corresponding adjoints by $\DelBianchi$ and $\bHg^*$, with $\DelBianchi-\delBianchi$ and $\bHg^*-\Hg^*$ of order zero as well. Thus, for every $0\leq k,m\leq d$, there corresponds a Hodge theory, together with cohomological formulations corresponding to the overdetermined boundary-value problems, as detailed in \secref{sec:Hodge_intro_NN}. 

We now specialize to the cases where $m=0$ and $m=1$, and show how the resulting cohomological formulations of \thmref{thm:NNintroPro} for these pre-complexes correspond to the solvability and uniqueness conditions for the Hessian equation \eqref{eq:Hessian_Problem} and the Killing equation \eqref{eq:Saint_Venant_Problem}, respectively.

We start with the Bianchi complex associated with $m=0$. The first operator in the sequence in this case is $\Hg:\Ckm{0}{0}\to \Ckm{1}{1}$, and we recognize that the relevant cohomology modules become:
\[
\mathscr{G}^{0}_{\NN}(g)=\ker(\Hg) \quad \text{and} \quad \mathscr{G}^{1}_{\NN}(g)=\ker(\dBianchi,\Hg^*,B_{H}^*).
\]
Thus, the cohomological formulation \thmref{thm:NNintroPro} for this complex at $\alpha=1$ yields precisely the solvability and uniqueness conditions for \eqref{eq:Hessian_Problem} given an arbitrary background metric (which is the Neumann version of \thmref{thm:Hessian_intro})

\begin{theorem}
\label{thm:hessian}
Given $\sigma\in\Ckm{1}{1}$, the problem \eqref{eq:Hessian_Problem} admits a solution $f\in C^{\infty}_{M}=\Ckm{0}{0}$ if and only if
\[
\DBianchi\sigma=0, \qquad \sigma\,\bot\,\mathscr{G}^{1}_{\NN}(g).
\]  
The solution is unique modulo $\mathscr{G}^{0}_{\NN}(g)$.
\end{theorem}

The uniqueness clause in the characterization of \thmref{thm:lifted_complexIntro} implies that the operator $\DBianchi$ coincides with the usual exterior covariant derivative when $g$ is Euclidean, as \eqref{eq:calabi_complex} in this case is a fully fledged cochain complex. Thus, in a Euclidean domain, the Poincaré Lemma implies $\ker(\dBianchi,\Hg^*,B_{H}^*)=\{0\}$, while $\ker(\Hg)$ is exactly the space of affine functions.

We proceed with the Bianchi complex associated with $m=1$. In this case, the first operator in the sequence coincides with the Killing operator once we recognize $\frakX(M)\simeq \Ckm{0}{1}$ through the musical isomorphism and $\Ckm{1}{1}=S^{2}(M)$. We therefore recognize that the relevant cohomology modules become:
\[
\mathscr{K}^{0}_{\NN}(g)=\ker(\dBianchi)=\ker(\Def) \quad \text{and} \quad \mathscr{K}_{\NN}(g)=\ker(\Hg,\delta_{g},\PnD_{g}),
\]
for which \thmref{thm:NNintroPro} becomes (the Neumann version of \thmref{thm:Dirichlet_Lie}): 

\begin{theorem}
\label{thm:Killing}
Given $\sigma\in \Ckm{1}{1}$, the problem \eqref{eq:Saint_Venant_Problem} admits a solution if and only if
\[
\bHg\sigma=0, \qquad \sigma\,\bot\,\mathscr{K}_{\NN}(g).
\]  
The solution is unique modulo $\mathscr{K}^{0}_{\NN}(g)$. 
\end{theorem}

As in the Hessian complex, the classical theorem for simply-connected Euclidean domains together with the uniqueness of the complex imply that $\bHg=\Hg$ and $\ker(\Hg,\delta_{g},\PnD_{g})=\{0\}$, consistent with the classical compatibility condition discussed in \eqref{eq:curlcurlIntro}. As for the zero cohomology, the kernel of the Killing operator $\ker(\dBianchi)=\ker(\Def)$ is known to be generically trivial on a closed manifold (cf.~\cite{Ebi70}).
\subsection{Riemann curvature equations}
\label{sec:prescribed_curvature_intro} 
Here we develop the setup underlying \thmref{thm:Riem_intro} more elaborately. The onset is the same overdetermined boundary-value problem derived there for an unknown $g\in\scrM(M)$, which falls at once into the scope outlined in \secref{sec:outlook} by writing: 
\beq
\begin{gathered}
\Rm_g -\mathrm{T}=0 , \\ 
\gD-\gamma =0 , \qquad 
\Ah_{g} -\mathrm{K}=0 .
\end{gathered}
\label{eq:prescribed_curvatureIntro}
\eeq

We shall now analyze the gauge symmetries and constraints of \eqref{eq:prescribed_curvatureIntro} within the framework of double forms. The boundary-value problem \eqref{eq:prescribed_curvatureIntro} possesses two natural symmetry groups: one associated with the interior equation and another associated with the boundary conditions.

The symmetry group for the boundary conditions in \eqref{eq:prescribed_curvatureIntro} consists of all boundary diffeomorphisms $\varphi:\dM\rightarrow\dM$, reflecting the invariance of the correspondences $g \mapsto \gD$ and $g \mapsto \Ah_g$ under pullback. Specifically, by naturality of the Levi-Civita connection, if $\tilde{\varphi}: M \to M$ is any diffeomorphism such that $\tilde{\varphi}|_{\dM} = \varphi$, then
\beq
\varphi^* \gD = \PttG \tilde{\varphi}^* g, \qquad \varphi^* \Ah_g = \Ah_{\tilde{\varphi}^* g}.
\label{eq:boundary_symmetries_intro} 
\eeq
The symmetry group associated with the interior equation consists of all diffeomorphisms $\phi: M \to M$ that fix the boundary, i.e., $\phi|_{\dM} = \id$ \cite{Kaz81,And08}, and the corresponding invariance is
\beq
\Rm_{\phi^* g} = \phi^* \Rm_g.
\label{eq:gauge_group_Rm}
\eeq

In addition, the components of \eqref{eq:prescribed_curvatureIntro} satisfy differential constraints: the \emph{differential Bianchi identity} and the \emph{Gauss--Mainardi--Codazzi equations}, which can be written as \cite[Ch.~3]{Pet16}: 
\beq
\begin{gathered} 
d_g \Rm_g = 0,\\
\PtD \Rm_g =(\PttD\Rm_{g},\PtnG\Rm_{g})= (\Rm_{\gD} + \tfrac{1}{2}\Ah_g \wedge \Ah_g, d_{\gD}\Ah_{g}).
\end{gathered} 
\label{eq:contraints_Rm_intro}
\eeq
The operator $\dBianchi: \Ckm{\alpha}{m} \to \Ckm{\alpha+1}{m}$ is the Bianchi derivative from \secref{sec:KillingHessianIntro}, which we remind the reader coincides with the exterior covariant derivative when $k \geq m$, and the boundary operator $\PtD$ is the usual tangential projection of a vector-valued form, as in \secref{sec:exterior_intro}. 

With these symmetries and constraints in place, we proceed to linearize \eqref{eq:prescribed_curvatureIntro} at a metric $g \in \scrM(M)$ as was done in \eqref{eq:Riem_1}, albeit this time incorporating an error term given by any tensorial smooth map $\Gamma : \Ckm{1}{1} \rightarrow \Ckm{2}{2}$ (which may arise from a connection on the space of Riemannian metrics, $\scrM(M)$), casting it into the form of the model linearized operator \eqref{eq:derivative_intro}. 

Casting the interior equation in the framework of double forms, the well-known variation formula for the Riemann curvature tensor \cite[p.~560]{Tay11b} yields a second-order linear operator
\[
\D\Rm_{g}: \Ckm{1}{1} \to \Ckm{2}{2}.
\]
whose leading order component (i.e., modulo zero order tensorial terms) can be seen to be $\Hg$ from \eqref{eq:RiemcurlcurlIntro}. For the boundary operators, the linearized Cauchy data induce mappings
\[
\PttG \oplus \D\Ah_{g}: \Ckm{1}{1} \to \plCkm{1}{1} \oplus \plCkm{1}{1}.
\]
Hence, setting as in \eqref{eq:derivative_intro} $\D_{\Gamma}\Rm_{g} = \D\Rm_{g} + \Gamma$ for a tensorial error term $\Gamma: \Ckm{1}{1} \to \Ckm{2}{2}$, which may arise from a connection on $\scrM(M)$ (\cite{Ham82,GMN92,CR13}) (and in this case, $g\mapsto\Rm_{g}$ is viewed as a section of the trivial bundle $\scrM(M)\times\Ckm{2}{2}$), we arrive at the perturbed overdetermined linear system generalizing \eqref{eq:Riem_1}: 
\beq
\begin{gathered}
\D_{\Gamma}\Rm_{g}\sigma = T, \\
\PttG\sigma = 0, \qquad \D\Ah_{g}\sigma = 0.
\end{gathered}
\label{eq:linearized_prescribed_curvatureIntroII}
\eeq

We cast \eqref{eq:linearized_prescribed_curvatureIntroII} within an elliptic pre-complex based on Dirichlet conditions. We do this in the prototypical setting of \defref{def:elliptic_pseudo_complex_intro}. An elliptic pre-complex based on Neumann conditions, associated with a non-homogeneous version of \eqref{eq:prescribed_curvatureIntro}, is also studied, though it is much more intricate and requires the full general Douglis-Nirenberg setting introduced in \secref{sec:general_intro}. We defer its treatment to the detailed example section in the body of the text (\secref{sec:prescribed_curvature}).

For the Dirichlet pre-complex, in the sense of \defref{def:elliptic_pseudo_complex_intro}, we introduce the following sequence:
\beq
\begin{split}
&\bA_{0} := \Def : \frakX(M) \rightarrow \Ckm{1}{1}, \\
&\bA_{1} := \D_{\Gamma}\Rm_{g} : \Ckm{1}{1} \rightarrow \Ckm{2}{2}, \\
&\bA_{2} := \dBianchi : \Ckm{2}{2} \rightarrow \Ckm{3}{2}, \\
&\bA_{3} := \dots
\end{split}
\label{eq:systems_curvature_intro}
\eeq
Here, recall that $\Def:\frakX(M)\to S^{2}(M)$ is the adjoint of the tensor-divergence $\delBianchi:\Ckm{1}{1} \to \frakX(M)$, i.e., the vector version of the Killing operator from \eqref{eq:BianchiKilling}, given explicitly by $\delta_{g}^*X=\tfrac{1}{2}\calL_{X}g$. Thus, by choosing $\Gamma$ appropriately, one retains from this pre-complex the Dirichlet picture also for the Killing equations, yielding \thmref{thm:Killing}. Notably, for $\dim M > 3$, the sequence extends beyond the $\Ckm{3}{2}$-level, resulting in a generalization of the Calabi pre-complex introduced earlier in \secref{sec:KillingHessianIntro}. To keep the discussion concise, we focus here only on the first two segments.

This sequence comes with the associated boundary systems:
\beq
\begin{split}
&B_{0} := |_{\partial M} : \frakX(M) \rightarrow \frakX(M)|_{\partial M}, \\
&B_{1} := \mathrm{S}^{-1}_{g}(\PttD \oplus \D\Ah_{g}) : \Ckm{1}{1} \rightarrow \plCkm{1}{1} \oplus \plCkm{1}{1}, \\
&B_{2} := \PtD : \Ckm{2}{2} \rightarrow \PtD(\Ckm{2}{2}), \\
&B_{3} := \dots
\end{split}
\label{eq:boundary_systems_curvature_intro}
\eeq
where $\mathrm{S}_{g}:\plCkm{1}{1} \oplus \plCkm{1}{1}\rightarrow \plCkm{1}{1} \oplus \plCkm{1}{1}$ is a zero-order isomorphism. 

In the context of the discussion around \eqref{eq:geoemtric_eq_intro}--\eqref{eq:gauge_symmetries_sequence}, an important remark here is that, once the boundary-value problem \eqref{eq:linearized_prescribed_curvatureIntroII} is identified, the systems $\bD_{0}$ and $\bD_{2}$, together with their boundary associates $\bB_{0}$ and $\bB_{2}$, are chosen specifically to reflect the linearization of the invariance condition \eqref{eq:gauge_group_Rm} and the geometric constraints \eqref{eq:contraints_Rm_intro}, respectively. As we have emphasized around \eqref{eq:gauge_symmetries_sequence}, this may be viewed as a general paradigm for constructing elliptic pre-complexes associated with linearized problems.

The key steps in verifying that \eqref{eq:systems_curvature_intro} is an elliptic pre-complex can be summarized as follows:  
\begin{enumerate}  
\item The required Green's formulae \eqref{eq:Green_formula_intro} for $\bD_{0}$ and $\bD_{2}$ follow from the well-known formulae for the Killing operator $\Def$ and the exterior covariant derivative $\dBianchi$, respectively. 

In the case of $\bD_{1}$, the corresponding formula for $\D_{\Gamma}\Rm_{g}$ follows from the Green formula for its leading-order term, which again coincides with the covariant curl--curl operator $\Hg$ from \eqref{eq:RiemcurlcurlIntro} (note that the Bianchi derivatives coincide with the exterior ones since $k \geq m$ here). We expand this formula here as: 
\[
\bra \Hg\sigma, \eta \ket = \bra \sigma, \Hg^*\eta \ket + \bra \PttD\sigma, \frakT^*_{g}\eta \ket - \bra \frakT_{g}\sigma, \PnnD\eta \ket,
\]
where $\frakT_{g}:\Ckm{1}{1}\rightarrow\plCkm{1}{1}$ is a first-order differential boundary operator satisfying
\[
S_{g}(\PttD \oplus \frakT_{g}) = \PttD \oplus \D\Ah_{g},
\]
and $\PttD \oplus \frakT_{g}$ is normal in the sense of \secref{sec:Generalized_green_intro}. 

This Green's formula persists after adding tensorial terms, since such terms integrate by parts without producing boundary contributions. Thus, for every tensorial error term $\Gamma$,
\[
\bra \D_{\Gamma}\Rm_{g}\sigma, \eta \ket = \bra \sigma, (\D_{\Gamma}\Rm_{g})^*\eta \ket + \bra \PttD\sigma, \frakT^*_{g}\eta \ket - \bra \frakT_{g}\sigma, \PnnD\eta \ket.
\]

\item Like before, the required Dirichlet overdetermined ellipticities in \eqref{eq:overdetermined_intro} will follow from a computation involving a generalization of the \emph{Lopatinskii--Shapiro condition}. See more details in the body of the text.

\item The order-reduction property \eqref{eq:order_reduction_intro} follows from the linearization of the geometric constraints in \eqref{eq:contraints_Rm_intro}, together with the invariances \eqref{eq:boundary_symmetries_intro}--\eqref{eq:gauge_group_Rm}. To illustrate how these nonlinear relations give rise to the identities in \eqref{eq:order_reduction_intro}, we outline the argument up to lower-order terms; thus, it suffices to assume, without loss of generality, that $\Gamma = 0$.

For $\alpha = 0$, applying the chain rule and linearizing both \eqref{eq:gauge_group_Rm} and \eqref{eq:boundary_symmetries_intro} yield, for every $X \in \mathfrak{X}_{M}$ with $\bB_{0}X = X|_{\dM} = 0$, the relations
\beq
\begin{aligned}
&\bD_{1}\bD_{0}X = 2 \D\Rm_{g}\Def{X} = \dertZero \underbrace{\Rm_{\varphi_{t}^*g}}_{=\varphi^*_{t}\Rm_{g}}=\calL_{X} \Rm_{g} \qquad \text{(a l.o.t. in } X\text{)}, \\[6pt]
&\bB_{1}\bD_{0}X = S^{-1}_{g}(\PttG \oplus \D\Ah_{g})\Def{X} = S^{-1}_{g}\dertZero\underbrace{(\PttD\varphi_{t}^*\gD\oplus \Ah_{\varphi^*_{t}g})}_{=\mathrm{const}}=0.
\end{aligned}
\label{eq:order_reduction_gauge}
\eeq
Here we have used the fact that the Killing operator is defined by
\[
\Def{X}=\tfrac{1}{2}\dertZero\varphi_{t}^*g.
\]  
Similarly, for $\alpha = 1$, we consider variations of metrics with fixed pullback metric $\gD$ and second fundamental form $\Ah_{g}$; namely, variations whose variational field $\sigma$ satisfies $\frakB_{1} \sigma = 0$, equivalently, $\PttD \sigma = 0$ and $\D\Ah_{g} \sigma = 0$. These boundary data can be smoothly imposed by means of the implicit function theorem, since the linearization of the boundary operators $\PttD \oplus \D\Ah_{g}$ is a submersion (cf.\ \cite[Prop.~21]{Led25B}). Applying the chain rule to the constraints \eqref{eq:contraints_Rm_intro} for such variations yields
\beq
\begin{aligned}
\bD_{2} \bD_{1} \sigma
&= \dBianchi \D_{\Gamma}\Rm_{g} \sigma
= \dertZero \underbrace{d_{g+t\sigma} \Rm_{g+t\sigma}}_{=\,0}
+ \Gamma\sigma + \text{l.o.t.\ in } \sigma, \\[6pt]
\bB_{2} \bD_{1} \sigma
&= \PtD \D_{\Gamma}\Rm_{g} \sigma
= \dertZero
\underbrace{
\bigl(\PttD_{g+t\sigma}\Rm_{g+t\sigma}+\PttD\Gamma\sigma,
\PtnD_{g+t\sigma}\Rm_{g+t\sigma}+\PtnG\Gamma\sigma\bigr)
}_{=\,(\PttD\Gamma\sigma,\PtnG\Gamma\sigma+\text{l.o.t.\ in } \sigma)}
\\&\qquad\qquad\qquad= (\mathrm{const},\text{l.o.t.\ in } \sigma).
\end{aligned}
\label{eq:order_reduction_curvature}
\eeq
In the second identity, we have used the fact that the tangential projection $\PttD$ is independent of the Riemannian metric, whereas the projection into the $\PtnD_{g}$ argument implicitly depends on the normal. Thus, for the identity $\bB_{2}\bD_{1} = 0$ to hold, we need to impose the extra conditions solely on the tensorial error term: 
\beq
\PttD\Gamma\sigma=0 \textand \PtnG\Gamma\sigma = -\PtnD_{g+t\sigma}\Rm_{g+t\sigma} \quad \text{for } \sigma\in\ker(\PttD, \D\Ah_{g}).
\label{eq:Gamma_curvature}
\eeq
We emphasize that this condition is tensorial in $\sigma$, since the resulting terms are of lower order by virtue of the linearized constraints above. Thus, they remain, in a sense, within the scope of the order-reduction property. For example, when $\Rm_{g}|_{\dM}=0$ and $\Gamma|_{\dM}=0$, as in the setting of \thmref{thm:Riem_intro}, or, more generally, when $\Gamma\sigma = \Gamma_{g}(\Rm_{g}, \sigma)$ arises from a specifically designed tame connection on $\scrM(M)\times\Ckm{2}{2}$ (cf.\ \cite[pp.~92--93]{Ham82}), this condition is satisfied trivially. This is how the relation $\Rm_{g}|_{\dM}=0$ in this theorem may be relaxed. We emphasize that the requirement \eqref{eq:Gamma_curvature} does not impose any other condition on $\Gamma$ or on the metric $g$. See also the discussion in \cite[Sec.~1]{Led25B} for a related situation in the linearized Ricci curvature equations.
\end{enumerate}

Now that \eqref{eq:systems_curvature_intro} is established as an elliptic pre-complex, we translate the cohomological formulation for \eqref{eq:linearized_prescribed_curvatureIntroII} obtained from \thmref{thm:DDintro} into a solvability and uniqueness statement for \eqref{eq:linearized_prescribed_curvatureIntroII}.

Let $\fbD_{2} = \mathpzc{d}_{g}$ denote the lifted version of $\bD_{2} = \dBianchi$ arising from the elliptic pre-complex. Observe that, directly from the expressions \eqref{eq:cohomology_groupsDIntroPro}, which in turn are computed from the original pre-complex \eqref{eq:systems_curvature_intro}, we have:
\beq 
\begin{split}
& \module_{\D}^0 = \ker(\Def,|_{\partial M}) = \left\{ 0 \right\}, \\
& \module_{\D}^1 =\mathscr{B}_{\DD}^{1}(M,g,\Gamma):= \ker(\delBianchi,\D_{\Gamma}\Rm_{g}, \PttD ,\D\Ah_{g}), \\
& \module_{\D}^2 =\mathscr{B}_{\DD}^{2}(M,g,\Gamma):= \ker((\D_{\Gamma}\Rm_{g})^*,\dBianchi, \PtD).
\end{split}
\label{eq:curvature_cohomologies_DD}
\eeq 

Here, the second identity follows from the well-known fact that there are no nontrivial isometries fixing the boundary \cite{And08,Hin24}. For the purposes of the present discussion, concerning the cohomological formulation of \eqref{eq:linearized_prescribed_curvatureIntroII}, we note that \thmref{thm:DDintro}, in the case $\alpha = 1$, takes the following form (which also recovers \eqref{thm:Riem_intro} when $\Rm_{g}|_{\dM}=0$ and $\Gamma=0$ by identifying that $\scrC^{2,2}_{\DD}(M):=\Ckm{2}{2}\cap\ker(\PttD,\D\Ah_{g})$ etc.):
\begin{theorem}
\label{thm:curvature_prescription_intro}
In the setting above, given $T \in \Ckm{2}{2}$, the boundary-value problem
\[
\begin{aligned}
& \D_{\Gamma}\Rm_{g}\sigma=T, \qquad &&\delBianchi \sigma = 0 \quad &&&& \text{in } M, \\
& \PttD\sigma= 0, \qquad &&\D\Ah_{g}\sigma=0 \quad &&&& \text{on } \dM
\end{aligned}
\]
admits a solution $\sigma \in \Ckm{1}{1}$ if and only if
\[
\mathpzc{d}_g T = 0, 
\qquad 
\PtD T = 0, 
\qquad 
T \perp_{L^2} \mathscr{B}_{\DD}^{2}(M,g,\Gamma).
\]
The solution is unique modulo $\mathscr{B}^{1}_{\DD}(M,g,\Gamma)$.
\end{theorem}

A corresponding statement in weaker Sobolev regularity is self-implied.

\include{Acknowweldgments}
\chapter{Technical Setup}
\label{chp:tech_setup}
\section{Preliminaries}
\label{sec:Preliminaries}
\subsection{Notation and basic notions}
$M$ shall stand for a compact, orientable smooth manifold with boundary. The boundary may be empty, in which case $M$ is a closed manifold. Given a fiber bundle $\bbE\rightarrow M$, denote by $\Gamma(\bbE)$ the space of sections over $M$. Although sometimes this abstract notation is used, in practice, the sole case of interest of fiber bundles in this work are tensor bundles (and subbundles of them), which are a particular case of vector bundles \cite[Ch.~10--12]{Lee12}.

Prime examples of section spaces considered in this paper are sections of symmetric $(2,0)$ tensors, denoted in the literature usually by $S^{2}(M)$ \cite{BE69,Ebi70}; double forms $\Wkm{k}{m}=\Gamma(\Lkm{k}{m})$  and in particular scalar differential forms $\Omega^{k}(M)=\Wkm{k}{0}$ (\cite{Cal61,Kul72} and \secref{sec:KillingHessianIntro}); and Bianchi forms $\Ckm{k}{m}=\Gamma(\Gkm{k}{m})$ \cite{Kul72}. Note that $\Ckm{1}{1}$ coincides with $S^{2}(M)$, and this is the notation we shall mostly use in this work. See \secref{sec:Appendix} for more details on Bianchi forms. 

Topological vector spaces are also considered: of type Hilbert, Banach, and Fréchet, and the intermediate category of tame Fréchet spaces \cite[Sec.~II, p.~133]{Ham82}. Refering to \cite{Ham82} for the technical details, these are Fréchet spaces whose topology results from a grading of Banach spaces satisfying desirable properties. Manifolds can be modeled on these topological vector spaces, resulting in corresponding categories of infinite dimensional manifolds (\cite[Ch.~2]{Lan01} and \cite[Sec.~2.3, p.~146]{Ham82}). To distinct from finite dimensional manifolds, infinite dimensional manifolds will be denoted by scripture, $\mathscr{X}$, $\mathscr{Y}$ etc. 
Infinite dimensional vector bundles (whether Hilbert, Banach \cite[Ch.~3]{Lan01} or (tame) Fréchet \cite[p.~150]{Ham82}) will also be denoted by scripture, e.g., $\mathscr{E}\rightarrow \mathscr{X}$. 

Section spaces are the most immediate example of tame Fréchet spaces \cite[p.~139]{Ham82}, with the required grading provided by the hierarchy of their Sobolev completions, which are Hilbert and Banach spaces (see the sequel \secref{sec:L2_sobolev} for definitions). The most important example of an open infinite dimensional manifold in this paper, other than section spaces, is the space of all Riemannian metrics over $M$. In the notations of \cite{BE69,Ebi70}, we denote this space by $\scrM_{M}$.

Consider also mappings between open subsets of infinite dimensional manifolds \cite[pp.~69-70]{Ham82}. Given infinite dimensional manifolds $\mathscr{X}$, $\mathscr{Y}$, whether they are Hilbert, Banach or tame Fréchet, and an open set $\mathscr{U}\subseteq\mathscr{X}$, let
\[
\bD:(\mathscr{U}\subseteq\mathscr{X})\rightarrow\mathscr{Y}
\] 
to emphasize the fact that the map $\bD$ is defined only on an open subset of $\mathscr{X}$. The prime example of such mappings are differential operators, generally nonlinear ones, parameterized by open sets of manifolds, as discussed in more detail e.g., in \cite[Sec.~II.2, p.~140]{Ham82}. Our primary case of interest will be families of linear maps \cite[pp.~70, 150]{Ham82}, which are mappings $\bD:(\mathscr{U}\times\mathscr{V}\subset\mathscr{X})\rightarrow\mathscr{Y}$ such that $\mathscr{V}$ and $\mathscr{Y}$ are topological vector spaces, and
\[
\Psi\mapsto\bD(\gamma)\Psi
\]
is a linear map for every $\gamma\in\scrU$, grouping the parentheses to reflect the linearity. When more convenient, and when there is no ambiguity, we shall instead use the notation employed \secref{sec:opening} for the nonlinear argument of $\bD$, i.e., 
\[
(\gamma,\Psi) \mapsto \bD_{\gamma} \Psi.
\]  
Given a smooth curve $x : (-\epsilon, \epsilon) \rightarrow \mathscr{X}$ with $x(0) = \gamma$ and $\dot{x}(0) = \Psi \in T_{\gamma}\mathscr{X}$, the \emph{linearization} (also called the \emph{differential} or \emph{Fréchet derivative}) of a map $\bD : (\scrU \subseteq \mathscr{X}) \rightarrow \mathscr{Y}$ at $\gamma\in \scrU$ in the direction of $\Psi \in T_{\gamma}\mathscr{X}$ is the element of $T_{\bD(\gamma)}\mathscr{Y}$ defined, if the limit exists, by
\beq
\D\bD_{\gamma}\Psi =\D\bD(\gamma)\Psi:=\dertZero \bD(x(t)).
\label{eq:D_dertivate}
\eeq
The expression in \eqref{eq:D_dertivate} is well defined, independent of the choice of curve $x(t)$ that satisfies $x(0) = \gamma$ and $\dot{x}(0) = \Psi$, and yields a continuous linear map $\D\bD(\gamma) : T_{\gamma}\mathscr{X} \rightarrow T_{\bD(\gamma)}\mathscr{Y}$ between topological vector spaces. In fact, the derivative defines a a smooth homomorphism of vector bundles over $\mathscr{X}$:
\[
\D\bD : T\mathscr{X}|_{\scrU} \rightarrow \bD^{*}T\mathscr{Y},
\] 
where $\bD^{*}T\mathscr{Y}$ is the pullback bundle of $T\mathscr{Y}$. A map is called \emph{smooth} if all its iterated derivatives exist and are continuous. 
In certain cases, the fibers $T_{\gamma}\mathscr{X}$ can be represented as product spaces. When this is possible, if $(\Psi_{1}, \Psi_{2}, \dots) \in T_{\gamma}\mathscr{X}$, we emphasize the linearity by writing
\[
T\mathscr{X}|_{\scrU} \ni (\gamma, \Psi_{1}, \Psi_{2}, \dots) \mapsto \D\bD(\gamma)\{\Psi_{1}, \Psi_{2}, \dots\}. 
\]
There is also a notion of partial derivatives for maps between manifolds \cite[p.~79]{Ham82}, where the construction involves differentiating with respect to only one of the variables. Specifically, for the case of interest here, if $\bD : (\scrU \times \scrW \subset \mathscr{X}) \rightarrow \mathscr{Y}$ is itself a family of linear maps operating as $(\gamma, \Psi) \mapsto \bD_{\gamma}\Psi$, then its partial derivative with respect to the $\gamma$-variable in the direction of $\sigma \in T_{\gamma}\scrU$ is denoted by
\beq
(\sigma, \Psi) \mapsto \D_{\sigma}\bD_{\gamma}\Psi = \dertZero \bD(x(t))\Psi.
\label{eq:partial_derivative}
\eeq
A map $\bD : (\scrU \subset \mathscr{X}) \rightarrow \mathscr{Y}$ is called \emph{tame} if it satisfies a tame estimate in a neighborhood of each point, as defined in \cite[p.~140]{Ham82}. A map is called a \emph{tame smooth map} if it is smooth and all its derivatives are tame \cite[p.~143]{Ham82}. Tameness is closed under composition. In the specific context of a family of linear maps $\bD : (\scrU \times \scrW \subset \mathscr{X}) \rightarrow \mathscr{Y}$ (cf. \cite[Lem.~2.17, p.~143]{Ham82}) the tame estimate amount to, for every $\gamma_0 \in \scrU$, the existence of a constant $C > 0$ independent of $\gamma$ and $\Psi$ such that in a neighborhood of $\gamma_0$:
\beq
\|\bD(\gamma)\Psi\|_{n} \leq C (1 + \|\gamma\|_{n+m})\|\Psi\|_{n+b},
\label{eq:tame_estimate}
\eeq
where $\|\cdot\|_{n}$, $\|\cdot\|_{n+m}$, and $\|\cdot\|_{n+b}$ are appropriate norms in the tame gradings of the domain and codomain, with $n \in \Nzero$ and $m,b\in\bbZ$. 


\subsection{Pseudodifferential boundary-value problems}
\label{sec:per_pseudo}
\subsubsection{Sobolev spaces} 
\label{sec:L2_sobolev}
For notation in this section, let  $(\tM,\tg)$ be a closed $d$-dimensional Riemannian manifold, endowed with a volume form $d\Volume\in\Omega_{\tM}^d$. Let $M \hookrightarrow \tM$ be a compact embedded submanifold of the same dimension having a smooth boundary. Since every compact Riemannian manifold with smooth boundary can be embedded in its closed double \cite[p.~226]{Lee12}, henceforth every compact Riemannian manifold with smooth boundary $M$ is viewed as smoothly embedded in a closed ambient Riemannian manifold $\tM$. Let $d\Volume_{\partial}\in \Omega_{\partial M}^{d-1}$ be a volume form over the boundary. 

Let $\tE,\tF\to\tM$ be Riemannian vector bundles over $\tM$, with fiber metrics $g=g_{\tE},g_{\tF}$; denote by $\E=\tE|_M$ and  $\bbF=\tF|_M$ the pullback bundles, which are vector bundles over $M$. Let $\bbJ,\bbG\to \dM$ be Riemannian vector bundles over $\dM$ with fiber metric $g_{\partial}=g_{\bbJ},g_{\bbG}$. We denote by $L^{p}\Gamma(\tE)$ the space of all $L^{p}$-sections. For $\psi\in L^p\Gamma(\tE)$ and $\eta\in L^{q}\Gamma(\tE)$, $1<p<\infty$, denote their $L^p$---$L^{q}$ coupling by
\[
\bra\psi,\eta\ket = \int_{\tM} (\psi,\eta)_{g}\,d\Volume.
\]
The same notation is used for the coupling associated with sections over $M$.
Likewise, for $\rho\in L^p\Gamma(\bbG)$ and $\tau\in L^{q}\Gamma(\bbG)$, denoted the induced $L^p$---$L^{q}$ coupling  on the boundary $\dM$ by
\[
\bra\rho,\tau\ket = \int_\dM (\rho,\tau)_{g_{\partial}}\, d\Volume_{\partial}. 
\]

The definition of Sobolev sections of vector bundles, $W^{s,p}\Gamma(\tE)$ defined for $s\in\bbR$ and $1<p<\infty$, goes through first defining scalar-valued Sobolev functions on $\bbR^d$, then on domains $\Omega\subset\bbR^d$, and then on closed manifolds by means of partitions of unity and coordinate charts. Finally, Sobolev sections of vector bundles over closed manifolds are defined \cite[Sec.~1.2.1.2]{RS82}.

There are several variants of Sobolev spaces. The spaces $H^{s,p}\Gamma(\tE)$ (also known as \emph{Bessel-potential} spaces) are defined for every $s\in\bbR$ and $1<p<\infty$ by means of the Fourier transform
\cite[pp.~42--46]{RS82}, \cite[pp.~291--293]{Gru90} . 
For $s\in\Nzero$, $H^{s,p}\Gamma(\tE)$ is the completion of $\Gamma(\tE)$ with respect to the Sobolev norm,
\[
\|\psi\|_{s,p}=\sum_{|\alpha|\leq s}\|\nabla^{\alpha}\psi\|_{L^{p}}
\]
where $\nabla$ is the Levi-civita connection on $\tE$, and the $L^{p}$ norm is the one induced by the fiber metric. 

Our goal is to pass to manifolds with boundary, where trace theorems are being invoked. For $s\in\bbR_{+}\setminus\Nzero$, the spaces $H^{s,p}\Gamma(\tE)$ are insufficient for these theorems to hold. This is where \emph{Besov spaces} $B^{s,p}\Gamma(\tE)$, $s\in\bbR$ and $1<p<\infty$, 
come in \cite[p.~293]{Gru90}, \cite[pp.~45--46]{RS82}. As in \cite{Gru90}, we set
\[
W^{s,p}\Gamma(\tE)=\begin{cases}
H^{s,p}\Gamma(\tE) & s\in\bbZ 
\\ B^{s,p}\Gamma(\tE) & s\in \bbR\setminus \bbZ. 
\end{cases}
\] 
For $s<0$, we note that the spaces $W^{s,p}\Gamma(\tE)$ consists of distributions. We define \cite[pp.~294--297]{Gru90},
\[
\begin{split}
&W^{s,p}\Gamma(\E) = W^{s,p}\Gamma(\tE)  / \{\omega \in W^{s,p}\Gamma(\tE) ~:~ \supp \omega \subseteq \overline{\tM\setminus M}\},
\end{split}
\]
and
\[
\begin{split}
W_0^{s,p}\Gamma(\E) = \{\omega\in W^{s,p}\Gamma(\tE) ~:~ \supp\omega \subseteq M\}.
\end{split}
\]
For $1/p+1/q=1$ \cite[p.~296]{Gru90}, 
\[
(W^{s,q}_0\Gamma(\bbE))^*\simeq W^{-s,p}\Gamma(\bbE),
\]
where the identification is given by the continuous extension of the pairing $\psi\mapsto \bra\psi,\cdot\ket$, which defines a dense inclusion of $\Gamma_{c}(\bbE)$ into $W^{s,q}_{0}(\bbE)$. 
\subsubsection{The calculus of pseudodifferential boundary-value problems} 
A differential operator  $\Gamma(\tE)\to\Gamma(\tF)$ is a linear map that can be represented as an $\bbR^{N_1}\to \bbR^{N_2}$ differential operator in any local trivializations of $\tE$ and $\tF$.
Since this definition is local, it extends to linear maps $\Gamma(\E)\to\Gamma(\bbF)$ and boundary differential operators $\Gamma(\E)\to\Gamma(\bbG)$. 

On closed manifolds $\tM$, differential operators are the prominent example of a larger class of continuous linear maps between Fr\'echet spaces
\[
A:\Gamma(\tE)\to  \Gamma(\tF),
\]
known as \emph{pseudodifferential operators}. Such an operator is characterized by its \emph{order} $m \in \mathbb{R}$, which generalizes the order of a differential operator. By definition, a pseudodifferential operator of order $m$ is also of any order greater than $m$. We define the \emph{sharp order} of $A$ as the minimal $m$ for which $A$ is of order $m$.  Determining the sharp order of a pseudodifferential operator (and, more generally, a larger class of operators) is a central theme in this paper. 

Pseudodifferential operators are closed under composition and always admit a \emph{formal adjoint}, which is a pseudodifferential operator of the same order, well defined by the formula:
\beq
\bra A\psi,\eta\ket=\bra \psi,A^*\eta\ket, \qquad \psi\in\Gamma(\tE),\eta\in\Gamma(\tF)
\label{eq:adjoint_neigh_def}
\eeq
Pseudodifferential operators are generally defined on a manifold without boundary. There is a subclass of pseudodifferential operators over $\tM$ that truncate ``nicely" to $M$. Such operators were introduced by H\"ormander \cite[p.~105]{Hor03}, and are known as pseudodifferential operators $A:\Gamma(\tE)\rightarrow\Gamma(\tF)$ that have the \emph{transmission property} with respect to $\dM$. The truncation to $M$ is denoted by $A_{+}:\Gamma(\bbE)\rightarrow\Gamma(\bbF)$. The space of operators having the transmission property is closed under adjoints, i.e., if $A$ has the transmission property then $A^*$ also have the transmission property, and it holds that: 
\[
\bra A_+\psi,\eta\ket =\bra \psi, (A^*)_+\eta\ket
\qquad
\text{for every $\psi\in\Gamma(\E)$  and $\eta\in\Gamma_c(\bbF)$}.
\]

For introduction of boundary operators, denote by $\rho_N:\Gamma(\E)\to (\Gamma(\jmath^*\E))^N$ the standard trace operator: 
\[
\rho_N\psi=(\D_{\frakn}^0\psi,\D_{\frakn}\psi,...,\D_{\frakn}^{N-1}\psi),
\]
where $\D_{\frakn}$ is the normal covariant derivative, (which is well-defined in a collar neighborhood of $\dM$, hence can be iterated) evaluated at the boundary, and $\D_{\frakn}^0$ is the trace on the boundary; the choice of connection on $\E$ is immaterial. 
A \emph{trace operator} $T$ of \emph{order} $m \in\bbR$ and \emph{class} $r\in\Nzero$ is a continuous linear map $T:\Gamma(\E)\to\Gamma(\bbG)$ of the form
\beq
T = \sum_{j=0}^{r-1} S_j\D_{\frakn}^j + \jmath^*Q_+,
\label{eq:TraceOp}
\eeq 
where $S_j:\Gamma(\jmath^*\bbE)\rightarrow\Gamma(\bbG)$ is a pseudodifferential operator on the boundary (which is a closed manifold) and $Q$ is a certain operator with the transmission property of order $m$ \cite[pp.~27--28, 33]{Gru96}. The operator $\rho_m$ above is an instance of a trace operator of order $m-1$ and class $m$. The class of trace operators can be extended to negative values \cite[pp.~309--311]{Gru90}. In simple terms, $T$ is of class $-r$ if it has zero class as defined above, and in addition the trace operator $T \D_{\frakn}^r$ has zero class.

Next we introduce the calculus of boundary-value problems, originating in work by Boutet de Monvel \cite{Bou71}. A \emph{Green operator} (\cite[pp.~169---173, Sec.~2.3.3]{RS82} and \cite[p.~315]{Gru90}) of order $m\in\bbR$ and class $r\in\bbZ$ is a system of operators $\bD$, which can be written in matrix form as
\beq
\bD=\begin{pmatrix}
A_++G & K \\
T & Q
\end{pmatrix}: 
\mymat{\Gamma(\E) \\\oplus\\ \Gamma(\bbJ)}
\longrightarrow  
\mymat{\Gamma(\bbF)  \\\oplus\\ \Gamma(\bbG)}.
\label{eq:full_Green_operator}
\eeq
Besides the elements $K$ and $G$, all of these operators belong to classes of operators that have already been introduced: $A$ is an operator with the transmission property of order $m$, $T$ is a trace operator of order $m-1$ and class $r$, and $Q$ is a pseudodifferential operator of order $m$. 

The operator $K: \Gamma(\bbJ) \to \Gamma(\bbF)$ belongs to the class of \emph{Poisson operators}, which arise as extension operators for boundary data (e.g., the solution operator to the Poisson problem \cite[Ch.~5.1]{Tay11b} or the right inverse of the trace operator \cite[Prop.~1.6.5, p.~80]{Gru90}). It possesses only an order---in this case, $m$---and maps boundary sections to interior sections. Poisson operators also arise as the ``adjoints" of trace operators of order $m-1$, specifically when these trace operators have zero class \cite[pp.~29--30]{Gru96}.

The operator $G$ is referred to as a \emph{singular Green operator}, a certain class of non-pseudodifferential operators \cite[pp.~30--32]{Gru96}. Like trace operators, they are characterized by both an order and a class. Singular Green operators were introduced to establish good composition rules \cite[p.~152]{RS82}. In \eqref{eq:full_Green_operator}, the singular Green operator $G$ is assumed to be of order $m-1$ and class $r \in \bbZ$. The adjoint of a Green operator with class $0$ is well-defined and is itself a singular Green operator of class $0$ with the same order. However, in general, operators with class $r > 0$ are not $L^{p}$-continuous and thus do not admit adjoints.

It is worth noting that there are differing conventions in the literature regarding the order of the trace operator $T$ and the singular Green operator $G$ in the matrix—specifically, whether they are defined as having order $m$ (e.g., in \cite{Gru90,Gru96}) or $m-1$ (e.g., \cite[Sec.~2.3.3.1, p.~169]{RS82}). Each approach has its advantages, ultimately leading to equivalent theoretical outcomes through the use of \emph{order-reducing operators}. For reasons discussed later in detail, we adopt the convention that the orders of $T$ and $G$ are $m-1$, and will interpret results cited from works assuming orders of $m$ in this light.

Omitting the notation of the vector bundles on which the Green operators are defined when there is no ambiguity, we let $\OP(m, r)$ denote the space of all Green operators of order $m \in \mathbb{Z}$ and class $r \in \mathbb{Z}$ (referred to as $\mathfrak{G}^{m, d}$ or $\OP(\mathfrak{S}^{m, d})$ in \cite[pp.~171--174]{RS82}). Note that, as with pseudodifferential operators, the terms “of order” and “of class” allow for elements in $\OP(m, r)$ to also belong to lower orders or classes. When these exist, the minimal such numbers are referred to as the \emph{sharp} order and \emph{sharp} class.

For reference's sake, we note that we will always assume that Green operators are \emph{classical} (as termed in \cite{RS82}) or \emph{polyhomogeneous} (as termed in \cite{Gru96}), meaning that they are associated with a certain asymptotic expansion in terms of homogeneous components. Let also $\OP(-\infty,r)=\bigcap_{m}\OP(m,r)$ and $\OP(-\infty,-\infty)=\bigcap_{m,r}\OP(m,r)$, the class of \emph{smoothing} operators. The union ${\bigcup}_{m, r} \OP(m, r)$ is designed to be closed under composition in the following manner:
\begin{theorem}[Composition rules]
\label{thm:composition_Green_operators}
Let $\bD_{i}\in\OP(m_i,r_i)$, $i\in\BRK{1,2}$. Then $\bD_{2}\bD_{1}\in \OP(m,r)$ where $m=m_1+m_2$ and $r=\max(r_2+m_1,r_1)$. 
\end{theorem}
Note that by setting various terms in $\bD$ to zero, we can focus on $\bD$ of any of the following forms:
\beq
\begin{gathered} 
 \begin{pmatrix}
A_+ & 0 \\ 0 & 0
\end{pmatrix}, \quad \text{or} \quad \begin{pmatrix}
0 & 0 \\ 0 & Q
\end{pmatrix}, 
\\ 
 \begin{pmatrix}
G & 0 \\ 0 & 0
\end{pmatrix}, \quad \text{or} \quad \begin{pmatrix}
0 & K \\ 0 & 0
\end{pmatrix}, \quad \text{or} \quad \begin{pmatrix}
0 & 0 \\ T & 0
\end{pmatrix}.
\end{gathered} 
\label{eq:isolated_Green_operator}
\eeq
Applying \thmref{thm:composition_Green_operators} on these specific Green operators, we retain composition rules for each of the different classes of operators. There are sixteen such rules in total. When both $T = 0$ and $G = 0$, the situation can be interpreted as $r = -\infty$, meaning there is no class.
\subsubsection{The symbol of a Green operator} 
\label{sec:principle_symbol_green} 
Green operators in $\OP(m, r)$ are associated with a well-defined \emph{principal symbol} \cite[p.~174]{RS82}:
\beq
\sigma(\bD) = \sigma_M(\bD) \oplus \sigma_\dM(\bD),
\label{eq:symbol_Green_operator} 
\eeq
where
\[
\sigma_M(\bD)(x,\xi) = \sigma_A(x,\xi): \E_x \to \bbF_x, \qquad x \in M,\, \xi \in T_{x}^*M
\]
is the interior symbol of $A$ in \eqref{eq:full_Green_operator}, as truncated from $\tM$. The second summand, $\sigma_\dM(\bD)(x,\xi)$, defined for each $x \in \dM$ and $\xi \in T_x^*\dM$, is the \emph{boundary symbol} of $\bD$.

The boundary symbol is an equivalence class of continuous linear map:
\beq
\begin{split}
&\sigma_\dM(\bD)(x,\xi): \mymat{\scrS(\overbar{\bbR}_+;\bbC \otimes \E_x) \\\oplus\\ \bbC \otimes \bbJ_x} \longrightarrow \mymat{\scrS(\overbar{\bbR}_+;\bbC \otimes \bbF_x) \\\oplus\\ \bbC \otimes \bbG_x},
\end{split}
\label{eq:boundary_symbol_Green} 
\eeq
where, for a vector bundle $\bbU \to M$, $\scrS(\overbar{\bbR}_+;\bbC \otimes \bbU_x)$ denotes the space of $\bbC \otimes \bbU_x$-valued Schwartz functions on the half-line $\overbar{\bbR}_+ = \{s \in \bbR~:~s \geq 0\}$. Note that by replacing $\bD$ with any isolated operator in the calculus, as in \eqref{eq:isolated_Green_operator}, we retain a separate notion of principal symbol for each class of operators in the calculus. 

In particular, for the operators $A$ and $Q$ in \eqref{eq:isolated_Green_operator}, the determination of $\sigma(\bD)$ is equivalent to determining the symbols $\sigma_A$ and $\sigma_Q$, respectively, as pseudodifferential operators over the manifolds without boundary $\tM$ and $\dM$. In view of this, we denote these symbols by $\sigma(A)(x,\xi)$ for $x \in M$ and $\xi \in T^*_{x}M$, and by $\sigma(Q)(x,\xi)$ for $x \in \partial M$ and $\xi \in T^*_{x}\partial M$, whenever there is no ambiguity.

A general definition of the boundary symbol of an arbitrary $\bD\in\OP(m,r)$ can be found in \cite[pp.~23--34]{Gru96} and \cite[p.~174]{RS82}. For our purposes, we outline the construction of the boundary symbol of a Green operator $\bD$ when it takes the form  
\beq
\bD = \begin{pmatrix} A_{+} & 0 \\ T & Q \end{pmatrix},  
\label{eq:Douglis_Nirenberg_Adapted}
\eeq  
where $A_{+}, T, Q$ are differential operators. In this specialized setting, let $x \in \dM$ and write $\xi \in T^*_xM$ in the form $\xi = \xi' + \xi_d \, dr$, where $\xi' \in T_x^* \dM$ and $dr$ is the unit covector normal to the boundary, so that $\xi_d \in \mathbb{R}$ is the normal component of $\xi$.
Consider the map
\[
\begin{pmatrix} \sigma(A)(x, \xi' + \xi_d \, dr) & 0 \\
\sigma(T)(x, \xi' + \xi_d \, dr) & \sigma(Q)(x, \xi)
\end{pmatrix} : \begin{matrix} \E_x \\\oplus\\ \bbJ_{x} \end{matrix} \longrightarrow \begin{matrix} \bbF_x \\\oplus\\ \bbG_x \end{matrix}
\]
where $\sigma(T)(x, \xi' + \xi_d \, dr)$ is obtained from \eqref{eq:TraceOp} by \cite[p.~27]{Gru96}
\[
\sigma(T)(x, \xi' + \xi_d \, dr) =
\sum_{0 \leq j < m} \xi_d^{m-j} \, \sigma_{S_j}(x, \xi').
\]
Note that we abuse notation here, and the endomorphism $\sigma(T)(x,\xi' + \xi_d\,dr)$ is not the symbol of the trace operator $T$ when considered in an isolated matrix as in \eqref{eq:isolated_Green_operator}, but it will be used to construct the latter momentarily.  

If one considers $\xi_d \in \mathbb{R}$ as an independent variable, then the restriction of this map to $\mathbb{E}_{x}$ can be extended to operate on complexified vector-valued functions, 
\[
F : \operatorname{Func}(\mathbb{R}; \mathbb{C} \otimes \mathbb{E}_x) \to  \begin{matrix} \operatorname{Func} \brk{\mathbb{R}; \mathbb{C} \otimes \mathbb{F}_x }\\\oplus\\ \mathbb{C} \otimes \mathbb{G}_x \end{matrix},
\]
given by
\[
F(\psi)(t) =
\begin{pmatrix} \sigma(A)(x, \xi' + t \, dr) \psi(t) \\ \sigma(T)(x, \xi' + t \, dr) \psi(t) \end{pmatrix}. 
\]
We then perform, formally, a one-dimensional Fourier transform, replacing $t \mapsto \iota \, \partial_s$. 
This yields a differential map, $\hat{F}$, given by
\[
\hat{F}(\psi)(s) = 
\begin{pmatrix} \sigma(E)(x, \xi' + \imath \, \partial_s \, dr) \psi(s) \\ \sigma(T)(x, \xi' + \imath \, \partial_s \, dr) \psi(s) \end{pmatrix}. 
\]
This map can be restricted to one-sided Schwartz functions, yielding a map
\[
\hat{F}  : \mathscr{S}(\overline{\mathbb{R}}_+; \mathbb{C} \otimes \mathbb{E}_x) \to\begin{matrix}  \mathscr{S} \brk{\overline{\mathbb{R}}_+; \mathbb{C} \otimes \mathbb{F}_x} \\\oplus\\ \mathbb{C} \otimes \mathbb{G}_x \end{matrix}.
\]
The boundary symbol of $\bD$ of the form \eqref{eq:Douglis_Nirenberg_Adapted2} is then the map as in \eqref{eq:boundary_symbol_Green}, 
whose operation is given by, for $\psi \in \mathscr{S}(\overline{\mathbb{R}}_+; \mathbb{C} \otimes \mathbb{E}_x)$ and $\lambda \in \mathbb{C} \otimes \mathbb{J}_{x}$: 
\beq
\sigma_\dM(\bD)(x, \xi')(\psi; \lambda) =
\begin{pmatrix} \{s \mapsto \sigma(E)(x, \xi' + \imath \, \partial_s \, dr) \psi(s)\} \\ \sigma(T)(x, \xi' + \imath \, \partial_s \, dr) \psi(0) + \sigma(Q)(x, \xi') \lambda \end{pmatrix}. 
\label{eq:boundary_symbol_Green_operation}
\eeq 

Back to the general case, the space of all principal symbols of operators in $\OP(m,r)$ is denoted here by $\mathrm{S}(m,r)$ (denoted by $\frakS^{(m),r}$ in \cite{RS82}). This space consists of equivalence classes of mappings of the form \eqref{eq:boundary_symbol_Green}, yielding a well-defined map:
\[
\sigma: \OP(m,r) \to \mathrm{S}(m,r).
\]
Since $\OP(m-1,r) \hookrightarrow \OP(m,r)$, it is important to note the distinction between $\sigma: \OP(m,r) \to \mathrm{S}(m,r)$ and $\sigma: \OP(m-1,r) \to \mathrm{S}(m-1,r)$. Under the equivalence relation, $\mathrm{S}(m-1,r)$ is identified as the zero space within $\mathrm{S}(m,r)$. Indeed, the range of the inclusion $\OP(m-1,r) \hookrightarrow \OP(m,r)$ is exactly the kernel of $\sigma: \OP(m,r) \to \mathrm{S}(m,r)$ \cite[Thm.~5, p.~174]{RS82}. Consequently, when performing calculations or comparing the symbols of Green operators, it is often clearer to specify the space of principal symbols in which these comparisons take place. This is illustrated in the following result \cite[p.~175]{RS82}:
\begin{theorem}
\label{thm:full_Green_compositon_rules}
Let $\bD$ and $\bQ$ be Green operators of orders $m_{A},m_{Q}\in \bbZ$ and classes $r_{A},r_{Q} \in \bbZ$. Then the following hold:
\begin{enumerate}
\item The symbol of the composition decomposes as: 
\beq
\sigma(\bQ \bD) = \sigma(\bQ) \circ \sigma(\bD) = (\sigma_M(\bQ) \circ \sigma_M(\bD)) \oplus (\sigma_\dM(\bQ) \circ \sigma_\dM(\bD)) \quad \text{in} \quad \mathrm{S}(m, r),
\label{eq:compostion_smr}
\eeq
where $m = m_A + m_Q$ and $r=\max{r_{A},r_{Q}+m_{A}}$. 

\item
If $m_{A}< m_{Q}$, then:
\beq
\sigma(\bD + \bQ) = \sigma(\bD) \quad \text{in} \quad \mathrm{S}(m, r).
\label{eq:lower_order_symbol}
\eeq

\item
If $r_{A}=0$, and $\bD^*$ is the adjoint of $\bD$, then:
\beq
\sigma(\bD^*) = \sigma(\bD)^*  \quad \text{in} \quad \mathrm{S}(m, 0).
\label{eq:adjoint_symbol}
\eeq
\end{enumerate}
\end{theorem}

A Green operator $\bD \in \OP(m,r)$ is called \emph{elliptic} if $\sigma(\bD)$ is invertible in $\mathrm{S}(m,r)$, in the sense of the component-wise composition of symbols in \eqref{eq:compostion_smr}.  

Note that if $\bD$ is of order both $m$ and $m'$, it may so happen that it is elliptic in $\OP(m, r)$ but not in $\OP(m', r')$. Thus, ellipticity must always be verified within a specified symbol space $\mathrm{S}(m, r)$—and it is immediate that the order $m$ in this space must coincide with the sharp order of $\bD$ (otherwise, the symbol vanishes due to \eqref{eq:lower_order_symbol}). However, when there is no ambiguity, we simply say that $\bD$ is elliptic without specifying the exact symbol space in which this holds.

Generalizing elliptic pseudodifferential operators on a closed manifold, an elliptic Green operator $\bD$ benefits from the existence of a parametrix in the calculus, i.e., a $\frakP\in\OP(-m,r-m)$ such that $\bD\frakP - \id$ and $\frakP\bD - \id$ are both elements in $\OP(-\infty,-\infty)$ (\cite[pp.~335--336]{Gru90} and \cite[pp.~194--195]{RS82}). In this setting it holds that:
\beq
\sigma(\bD)^{-1}=\sigma(\frakP), \quad \text{in} \quad \mathrm{S}(-m,r-m).
\label{eq:symbol_inverse} 
\eeq

A more flexible notion than the ellipticity of a Green operator, and more central to this work, is that of \emph{overdetermined ellipticity} \cite[p.~237]{RS82}, \cite[p.~315]{Gru90}:

\begin{definition}
\label{def:overdetermined_elliptic_standard}
A Green operator $\bD\in\OP(m,r)$ is called \emph{overdetermined (overdetermined) elliptic} if its symbol $\sigma(\bD)$ is injective in $\mathrm{S}(m,r)$.
\end{definition}
For a convenient criterion of the injectivity of a symbol, we go back to systems of the form \eqref{eq:Douglis_Nirenberg_Adapted2}. Consider the ordinary differential operator,
\beq
\sigma(A)(x, \xi' + \imath \, \partial_s \, dr) : C^{\infty}(\overline{\mathbb{R}}_+; \mathbb{C} \otimes \mathbb{E}_x) \to C^{\infty}(\overline{\mathbb{R}}_+; \mathbb{C} \otimes \mathbb{F}_x),
\label{eq:boundary_symbol_E1} 
\eeq
to Schwartz functions, supplemented by the \emph{initial condition map} 
\beq
\Xi_{x, \xi} = 
\begin{pmatrix} \sigma(T)(x, \xi' + \imath \, \partial_s \, dr)|_{s=0} & \sigma(Q)(x, \xi) \end{pmatrix} : \begin{matrix} C^{\infty}(\overline{\mathbb{R}}_+; \mathbb{C} \otimes \mathbb{E}_x) \\\oplus\\ \mathbb{C} \otimes \mathbb{J}_{x} \end{matrix} \to \mathbb{C} \otimes \mathbb{G}_x. 
\label{eq:Xi_map1}
\eeq
operating as:
\[
\Xi_{x, \xi}(\psi;\lambda)=\sigma(T)(x, \xi' + \imath \, \partial_s \, dr)\psi(0)+\sigma(Q)(x, \xi)\lambda. 
\]
These mappings coincide with those defined in the classical \emph{Lopatinski–Shapiro condition} when $Q = 0$ (i.e., when the initial condition is homogeneous). In this case, verifying that $\sigma_{\dM}(x,\xi')$ is invertible reduces to checking that a corresponding system of ODEs, equipped with the induced initial conditions, admits only the trivial bounded solution for positive time (cf. \cite[pp.~233--234]{Hor03}, \cite[Ch.~5.11]{Tay11a}. 

We then have:
\begin{proposition}
\label{prop:rud_lop} 
Given a Green operator $\bD$ as in \eqref{eq:Douglis_Nirenberg_Adapted2}, with an injective interior symbol, let $x \in \dM$ and $\xi' \in T_x^*\dM \setminus \{0\}$. Let $\bbM_{x,\xi'}^+ \subset C^{\infty}(\overbar{\bbR}_+; \bbC \otimes \E_x)$ denote the space of decaying solutions of the linear $\bbC \otimes \E_x$-valued ordinary differential equation:
\[
\sigma(A)(x, \xi' + \imath \, \partial_s \, dr)\psi(s) = 0.
\]
Then, the principal symbol $\sigma(\bD)$ is injective if and only if the restriction of the initial condition map:
\[
\Xi_{x,\xi'} : 
\begin{matrix}
\bbM_{x,\xi'}^+ \\
\oplus \\
\bbC \otimes \bbJ_{x}
\end{matrix}
\longrightarrow \bbC \otimes \bbG_{x},
\]
is injective for every $x \in \dM$ and $\xi' \in T^*_{x}\partial M \setminus \{0\}$.
\end{proposition}
\subsubsection{Mapping properties between Sobolev spaces} 
Green operators are continuous with respect to Sobolev norms, with the Sobolev exponents in the domain and codomain determined by their order and class: explicitly, if $\bD \in \OP(m, r)$, then for every $\Psi\in\Gamma(\bbE)\oplus\Gamma(\bbF)$:
\beq
\|\bD\Psi\|_{s-m,s-m+1-1/p,p} \lesssim \|\Psi\|_{s,s+1-1/p,p},
\label{eq:Green_operator_sobolev_norm}
\eeq
which is to say that $\bD$, as operating between spaces of smooth sections as in \eqref{eq:full_Green_operator}, continuously extends into a map between Sobolev spaces:
\beq
\begin{split}
&\bD:
\mymat{W^{s,p}\Gamma(\bbE) \\\oplus\\ W^{s+1-1/p,p}\Gamma(\bbJ)}
\longrightarrow
\mymat{W^{s-m,p}\Gamma(\bbF) \\\oplus\\ W^{s-m+1-1/p,p}\Gamma(\bbG)}
\end{split}
\label{eq:Green_mapping_proprety} 
\eeq

for all $\bbR\ni s > r + 1/p - 1$ (see \cite[Cor.~3.2, pp.~312--313]{Gru90}, \cite[p.~176]{RS82}, noting that results for $p \ne 2$ in the latter are inaccurate and thus the former reference is required). 

It is important to note that, for \eqref{eq:Green_mapping_proprety} to hold, the estimate \eqref{eq:Green_operator_sobolev_norm} must be valid for smooth sections up to the boundary, i.e., for all $\Psi \in \Gamma(\bbE)$. For a general Green operator $\bD \in \OP(m, r)$, it is not sufficient for the estimate to hold only on sections compactly supported in the interior. However, if $r \leq 0$ and $-1/p < s < 1/p$, the continuous extension is determined entirely by operations on compactly supported sections. Moreover, when $p = 2$, a stronger result holds: if $r = 0$, then for any $s \in \mathbb{R}$ such that $s - m < 1/2$, we have for every $\Psi \in \Gamma_{c}(\bbE) \oplus \Gamma(\bbJ)$ \cite[p.~160]{RS82}:

\beq
\|\bD\Psi\|_{s-m,s-m+1/2,2} \lesssim \|\Psi\|_{s,s+1/2,2},
\label{eq:Green_operator_sobolev_norm0}
\eeq
which then yields the continuous mapping property: 
\beq
\begin{split}
\bD:\mymat{W^{s,2}_0\Gamma(\bbE) \\\oplus\\ W^{s+1/2,2}_0\Gamma(\bbJ)}
\longrightarrow
\mymat{W^{s-m,2}_0\Gamma(\bbF) \\\oplus\\ W^{s-m+1/2,2}_0\Gamma(\bbG)}
\end{split}.
\label{eq:Green_mapping_proprety0} 
\eeq

By setting different elements in the matrix \eqref{eq:full_Green_operator} to zero as in \eqref{eq:isolated_Green_operator}, mapping properties for each operator component in its respective class are retained (see also \cite[Cor.~3.2, pp.~312--313]{Gru90}):
\beq
\begin{split}
&A_+ + G : W^{s,p}\Gamma(\bbE) \rightarrow W^{s-m,p}\Gamma(\bbE), \qquad s > r + 1/p - 1, \\
&T : W^{s,p}\Gamma(\bbE) \rightarrow W^{s-m+1-1/p,p}\Gamma(\bbG) \quad \text{(when of order $m-1$)}, \qquad s > r + 1/p - 1, \\
&K : W^{s+1-1/p,p}\Gamma(\bbJ) \rightarrow W^{s-m,p}\Gamma(\bbF), \qquad s \in \mathbb{R}, \\
&Q : W^{s+1-1/p,p}\Gamma(\bbJ) \rightarrow W^{s-m+1/p,p}(\bbG), \qquad s \in \mathbb{R}.
\end{split}
\label{eq:mapping_proprety_diffrent_components}
\eeq
As noted, if $T = 0$ and $G = 0$, this corresponds to $r = -\infty$, allowing for mapping properties in negative Sobolev spaces for the $A_+$, $Q$, and $K$ components.

Allegedly, an operator within the calculus can have continuous extensions that exceed those specified in \eqref{eq:mapping_proprety_diffrent_components}. As the next proposition shows, this phenomenon is sharp in the sense that it can only arise if the operator is of order and class are lower than indicated. In the statement, for reference's sake, we restrict attention to the case $p = 2$.

\begin{proposition}[Sharp order and class]
\label{prop:L2_continuity}
Let $A_{+}$, $G$, $T$, $K$, and $Q$ be operators within the calculus as given in \eqref{eq:full_Green_operator}, and let $m \in \mathbb{R}$, $r\in\bbZ$ and $s \in \mathbb{R}$ such that $s\leq r+1/2$. Then
\begin{enumerate}
\item $A_{+} + G$ is $W^{s,2}\Gamma(\bbE)\rightarrow W^{s-m,2}\Gamma(\bbF)$ continuous if and only if $A$ is of order $m$, and $G$ is of order $m-1$ and of class $r$.
\item $K$ is $W^{s+1/2,2}\Gamma(\bbJ)\rightarrow W^{s-m,2}\Gamma(\bbF)$ continuous if and only if it is of order $m$.
\item $T$ is $W^{s,2}(\bbE)\rightarrow W^{s-m+1/2,2}(\bbG)$ continuous if and only if it is of order $m-1$ and of class $r$.
\item $Q$ is $W^{s+1/2,2}\Gamma(\bbJ)\rightarrow W^{s-m+1/2,2}\Gamma(\bbG)$ continuous if and only if it is of order $m$.
\end{enumerate}
\end{proposition}

\begin{proof}
For the statement about the class of $T$ and $A_{+}+G$, see \cite[Thm.~3.10, p.~310]{Gru90} and the comment in \cite[p.~312]{Gru90}.

For the argument regarding the orders, the first direction simply follows from the mapping properties in \eqref{eq:mapping_proprety_diffrent_components}. In the other direction, without loss of generality, suppose that $T$ is of order strictly greater than $m-1$, say $m'-1 > m-1$, so $T$ is $W^{s,2} \rightarrow W^{s-m'+1/2,2}$ continuous. By composing the mapping property of $T$ with the compact continuous inclusion $W^{s-m,2} \hookrightarrow W^{s-m'+1/2,2}$, we conclude that $T$ is a compact operator from $W^{s,2}\rightarrow W^{s-m'+1/2,2}$. By applying \cite[Cor.~4, p.~193]{RS82}, considering $T$ as a full Green operator with the other components in \eqref{eq:full_Green_operator} being zero, the principal symbol $\sigma(T)$ must vanish as an element in $\mathrm{S}(m'-1,r)$—hence it is of order lower than $m'-1$. Since $m' -1> m-1$ is arbitrary, we conclude that $T$ is of order $m-1$ as required.
\end{proof}

If $\bD$ is an elliptic operator within the calculus, its continuous extensions \eqref{eq:Green_mapping_proprety}, with respect to $m \in \bbZ$ for which $\sigma(\bD)$ is invertible in $\mathrm{S}(m, r)$, are Fredholm mappings between Banach spaces. This condition is, in fact, also sufficient for a Green operator to be elliptic \cite[Thm.~7, p.~197]{RS82}. We will see later, in a more general context, that this fact generalizes to the statement that a Green operator is overdetermined elliptic if and only if its continuous extensions are \emph{semi-Fredholm} (e.g, \cite[Ch.~5]{Kat80} or \cite[Ch.~1.3]{EE18}). In the meantime, we present one direction:  
\begin{proposition}
\label{prop:overdetermined_tools_Green}
Let $\bD \in \OP(m,r)$ be overdetermined elliptic \defref{def:overdetermined_elliptic_standard}. Then, its continuous extensions are semi-Fredholm mappings; namely, there exists an a priori estimate,  
\beq
\begin{split} 
\|\psi\|_{s,p} + \|\lambda\|_{s+1-1/p,p} \lesssim &\|(A_{+} + G)\psi + K\lambda\|_{s-m,p} + \|T\psi + Q\lambda\|_{s-m+1-1/p,p} 
\\&+ \|\psi\|_{s_0,p} + \|\lambda\|_{s_0+1-1/p,p}
\end{split}
\label{eq:overdetermined_ellipticity_a_priori1}
\eeq
for every $s, s_0 \in \bbR$ such that $s > s_0 > r + 1 - 1/p$.  

In particular, $\ker \bD \subseteq W^{s,p} \Gamma(\E) \oplus W^{s-1/p,p} \Gamma(\bbJ)$ is finite-dimensional, independent of $s, p$, and consists of smooth sections. If, in addition, $\bD$ is injective, then it admits a left inverse which is an element of $\OP(-m, r-m)$.  
\end{proposition}
For an overdetermined elliptic system $\bD$, the kernel $\ker \bD$ is contained in $L^2\Gamma(\bbE) \oplus L^2\Gamma(\bbJ)$ and thus admits an $L^2$-orthogonal projection, which we denote by
\[
\frakI: L^2\Gamma(\bbE) \oplus L^2\Gamma(\bbJ) \to L^2\Gamma(\bbE) \oplus L^2\Gamma(\bbJ).
\]
Since $\ker \bD$ consists entirely of smooth functions, this projection restricts continuously to smooth sections,
\[
\frakI: \Gamma(\bbE) \oplus \Gamma(\bbJ) \to \Gamma(\bbE) \oplus \Gamma(\bbJ),
\]
and defines an integral operator with a smooth kernel. Consequently, $\frakI \in \OP(-\infty, -\infty)$. By continuity, the projection extends to a compact operator on Sobolev spaces,
\[
\frakI: W^{s,p}\Gamma(\bbE) \oplus W^{\gamma,p}(\bbJ) \to W^{s,p}\Gamma(\bbE) \oplus W^{\gamma,p}(\bbJ),
\]
with finite-dimensional range equal to $\ker \bD$ for every $s, \gamma\in\bbR$.

Using the finite-dimensionality of $\ker \bD$ and the Rellich embedding theorem (see \cite[p.~51]{Bre11} or \cite[p.~28]{EE18}), one obtains the following refinement of the estimate \eqref{eq:overdetermined_ellipticity_a_priori1}:
\beq
\begin{aligned}
\|\psi\|_{s,p} + \|\lambda\|_{s + 1 - 1/p,\,p}
\;\lesssim\; \|(A_{+} + G)\psi + K\lambda\|_{s - m,\,p}  &+ \|T\psi + Q\lambda\|_{s - m + 1 - 1/p,\,p} \\
& + \|\frakI(\psi, \lambda)\|_{0,0,p},
\end{aligned}
\label{eq:overdetermined_sharp}
\eeq
for every $s \in \mathbb{R}$ satisfying $s > r + 1/p - 1$, where $\|\cdot\|_{0,0,p}$ denotes the $L^p$-norm on $L^p\Gamma(\bbE) \oplus L^p\Gamma(\bbJ)$.

\subsubsection{Adjoints and Green's formulae} 
A \emph{system of trace operators associated with order}  $m\in\bbZ$ is a trace operator of the form $T=T_0\oplus T_1\oplus\cdots\oplus T_{m-1}$, where $T_i:\Gamma(\E)\to\Gamma(\bbJ_i)$ is of order $i$ and class $i+1$, with $\bbJ_i\to \dM$ a vector bundle over the boundary \cite[pp.~45--46]{Gru96}. As noted earlier, every component $T_i$ can be written as
\beq
T_i = \sum_{j=0}^{i-1} S_{ij}\D_{\frakn}^j + \jmath^*(Q_i)_+,
\label{eq:components_normal}
\eeq
where $S_{ij}$ is of order $i-j$ and $Q_{i}$ of order $i$. 

\begin{definition}
\label{def:normal_system_of_trace_operators1}
A system of trace operators $T_0\oplus T_1\oplus\cdots\oplus T_{m-1}$ associated with order $m\in\bbZ$ is said to be \emph{normal} if each $T_i$ of the form \eqref{eq:components_normal} satisfies that $S_{ii}:\Gamma(\jmath^*\E)\to \Gamma(\bbJ_i)$ is surjective.
\end{definition}

The normality of a system of trace operators implies surjectivity and $L^{p}$-density of the kernel \cite[pp.~80, 82]{Gru96}:

\begin{proposition}
\label{prop:normal_sytem}
Let $T$ be a normal system of trace operators associated with order $m\in\bbZ$. Then $T:\Gamma(\E)\to \Gamma(\bbG)$ is surjective, and $\ker T$ is dense in $L^{p}\Gamma(\bbE)$ for all $1<p<\infty$. 
\end{proposition}
In \cite[p.~37, prop.~1.3.2]{Gru96}, it is proven that every operator with the transmission property $A$ yields \emph{differential} systems of trace operators $B_A:\Gamma(\E)\to\Gamma(\bbG)$ and $B_{A}^*:\Gamma(\bbF)\to\Gamma(\bbG)$, such that: 
\beq
\bra A_+\psi,\eta\ket = \bra\psi,A_+^*\eta\ket + \bra B_A\psi, B_{A^*}\eta\ket
\label{eq:integration_by_parts_convention}
\eeq 
for every $\psi\in\Gamma(\bbE)$ and $\eta\in\Gamma(\bbF)$. In \cite[Cor.~1.6.2, p.~77]{Gru96}, it is shown that if, in addition, the operator is \emph{non-characteristic} with respect to the boundary, then $B_A$ and $B_{A^*}$ can be chosen to be normal systems of trace operators of integer order $m \in \mathbb{Z}$. For example, all elliptic pseudodifferential operators with the transmission property are, by definition, non-characteristic.\footnote{This is why, in the exposition of elliptic complexes in \secref{sec:Hodge_intro}, this requirement on the boundary operators is actually implicit.} The computation carried out in the proof of \cite[Thm.~1.4.6, pp.~53–54]{Gru96} then implies the following:

\begin{proposition}
\label{prop:direct_sum_trace_operator_elliptic_operator}
Let $A$ be an elliptic operator with the transmission property, let $B_{A}$ be its associated normal system of trace operators associated with order $m\in\bbZ$ from \eqref{eq:integration_by_parts_convention}, and let $T$ be a normal system of trace operators associated with order $m'\in\bbZ$. Then $TA_{+}\oplus B_{A}$ is a normal system of trace operators associated with order $m+m'$.  
\end{proposition}

For general Green operators we have \cite[p.~151, Cor.~11]{RS82}:
\begin{proposition}
\label{prop:formal_adjoint}
The adjoint of a trace operator of order $m$ and class $0$ is a Poisson operator of order $m+1$, and vice versa. If $\bD\in\OP(m,0)$, then there exists a uniquely determined $\bD^{*}\in\OP(m,0)$ defined by the relation:
\[
\bra \bD\Psi,\Theta\ket = \bra \Psi,\bD^{*}\Theta\ket, \qquad \Psi\in \Gamma_{c}(\bbE)\oplus\Gamma(\bbJ), \,\,\, \Theta\in \Gamma_{c}(\bbF)\oplus\Gamma(\bbG).
\]
\end{proposition}

The additional order for adjoints of Poisson operators is one reason why $T$ is defined to have order $m-1$, ensuring that if $\bD \in \OP(m,0)$, then $\bD^* \in \OP(m,0)$ as well. This makes $\OP(m,0)$ closed under adjunction, and consequently $\OP(0,0)$ is closed under both composition and adjunction, making it an algebra (this class is denoted by $\mathfrak{G}^{0,0}$ in \cite[p.~175]{RS82}).

The continuous extensions of a Green operator, as described in \eqref{eq:Green_mapping_proprety}, define continuous mappings between Banach spaces and therefore possess well-defined Banach duals. In the case $r = 0$, and following similar lines to \cite[pp.~288–289]{WRL95}, these Banach duals can be related to continuous extensions of the adjoint $\bD^* \in \OP(m, 0)$ as follows: using the duality $W^{s,2}_0 \simeq W^{-s,2}$ via the pairing $\bra \cdot, \cdot \ket$, the Banach dual of \eqref{eq:Green_mapping_proprety}, with $s$ replaced by $s - 1/2$, is given for any $s \in \mathbb{R}$ satisfying $s - m < 1$ by:

\beq
\begin{split}
\bD' : \mymat{W^{-s + m,2}\Gamma(\bbF) \\\oplus\\ W^{-s + m + 1/2,2}\Gamma(\bbG)} \longrightarrow \mymat{W^{-s,2}\Gamma(\bbE) \\\oplus\\ W^{-s + 1/2,2}\Gamma(\bbJ)}
\end{split}.
\label{eq:Green_mapping_proprety_dual} 
\eeq
This map is also retained as the continuous extension of the map $(\bD'\Theta)(\Psi) = \bra \Theta, \bD\Psi \ket$ when restricted to compactly supported $\Psi,\Theta$. Then by \propref{prop:formal_adjoint}, we have that $\bD' = \bD^*$ when restricted to compactly supported sections. In particular, due to the defining relation between a linear map between Banach spaces and its Banach dual, it follows that, as continuous linear maps in \eqref{eq:Green_mapping_proprety_dual} and \eqref{eq:Green_mapping_proprety0}:
\beq
\|\bD^*\|_{\mathrm{op}(-s + m, -s)}=\|\bD'\|_{\mathrm{op}(-s + m, -s)}=\|\bD\|_{\mathrm{op}(s-1/2, s-1/2 - m)},
\label{eq:norms_operator_duality_adjoints}
\eeq
where $\|\cdot\|_{\mathrm{op}(s, s')}$ denotes the appropriate operator norm for continuous linear maps $W^{s, s + 1/2}_2\rightarrow W^{s', s' + 1/2}_2$. This identity will be useful in later analysis.




\section{Douglis-Nirenberg systems}
\label{sec:Overdetermined} 
This section extends the notion of overdetermined ellipticity to general systems of \emph{varying orders}, also known as \emph{Douglis–Nirenberg systems} \cite[pp.~234–235]{RS82}, \cite[Cor.~5.5, p.~336]{Gru90}. In the context of the order-reduction property discussed in \secref{sec:order-reduction_intro}, one of our goals is to formalize a procedure for comparing the respective orders and classes of two such systems.

The machinery for comparing operator orders played an important role in the rudienmantry theory developed in \cite{KL23}. However, for systems with varying orders, the situation becomes more intricate, and a direct comparison is no longer straightforward. In this section, we develop a framework that renders these comparisons feasible by translating the problem into the analysis of the continuous mapping properties of the systems between Sobolev spaces.
\subsection{Basic definitions} 
Let $1 \leq i \leq i_0$, $1 \leq j \leq j_0$, $1 \leq k \leq k_0$, and $1 \leq l \leq l_0$ be sets of indices, with associated integers $(m_{i}^{j})$, $(\tau_{k}^{j})$, $(t_{i}^{l})$, $(\sigma_{k}^{l})$, and $(r^{j})$. Generalizing the matrix operation from \eqref{eq:full_Green_operator}, suppose there are direct sums of vector bundles,
\[
\bbE = \bigoplus_{j} \bbE_{j}, \qquad
\bbJ = \bigoplus_{l} \bbJ_{l}, \qquad
\bbF = \bigoplus_{i} \bbF_{i}, \qquad
\bbG = \bigoplus_{k} \bbG_{k},
\]
whose fibers decompose accordingly:
\beq
\bbE_{x} = \bigoplus_{j} \bbE_{j,x}, \qquad
\bbJ_{x} = \bigoplus_{l} \bbJ_{l,x}, \qquad
\bbF_{x} = \bigoplus_{i} \bbF_{i,x}, \qquad
\bbG_{x} = \bigoplus_{k} \bbG_{k,x},
\label{eq:vector_bundles_fibers} 
\eeq
thereby allowing the operator $\bD$ in \eqref{eq:full_Green_operator} to mix components across these decompositions.

We now generalize $\bD$ further by viewing it as a \emph{tensor} of operators in the calculus, expressed in component form as $\bD = (\bD_{ik}^{jl})$, where each block operator is given by
\beq
\bD_{ik}^{jl} =
\begin{pmatrix}
E_{j}^{i} & K_{i}^{l} \\
T_{k}^{j} & Q_{k}^{l}
\end{pmatrix}
: \begin{matrix} \Gamma(\bbE_{j}) \\ \oplus \\ \Gamma(\bbJ_{l}) \end{matrix}
\longrightarrow
\begin{matrix} \Gamma(\bbF_{i}) \\ \oplus \\ \Gamma(\bbG_{k}) \end{matrix}.
\label{eq:douglas_nirenberg_components}
\eeq
Here:
\begin{enumerate}
\item $E_{j}^{i} = (A_{i}^{j})_{+} + G_{i}^{j}$, where $A_{i}^{j}$ is of order $m_{i}^{j}$ and $G_{i}^{j}$ is of order $m_{i}^{j} - 1$ and class $r^{j}$;
\item $T_{k}^{j}$ is of order $\tau_{k}^{j}$ and class $r^{j}$;
\item $Q_{k}^{l}$ is of order $\sigma_{k}^{l}$;
\item $K_{i}^{l}$ is of order $q_{i}^{l}$.
\end{enumerate}

Overall, the operator $\bD$ acts as
\[
\bD : 
\begin{matrix} 
\Gamma(\bbE) \\ 
\oplus \\ 
\Gamma(\bbJ) 
\end{matrix}
\longrightarrow 
\begin{matrix} 
\Gamma(\bbF) \\ 
\oplus \\ 
\Gamma(\bbG) 
\end{matrix},
\]
where, due to the direct sum structure,
\[
\Gamma(\bbE) = \bigoplus_{j} \Gamma(\bbE_{j}), \quad 
\Gamma(\bbJ) = \bigoplus_{l} \Gamma(\bbJ_{l}), \quad 
\Gamma(\bbF) = \bigoplus_{i} \Gamma(\bbF_{i}), \quad 
\Gamma(\bbG) = \bigoplus_{k} \Gamma(\bbG_{k}).
\]

The action of $\bD$ is explicitly given by a “contraction” on $\Psi \in \Gamma(\bbE) \oplus \Gamma(\bbJ)$:
\beq
(\bD\Psi)_{ik} 
= \bD_{ik}^{jl} 
\begin{pmatrix} (\psi_{j}); (\lambda_{l}) \end{pmatrix}  
= 
\begin{pmatrix} 
E_{i}^{j} \psi_{j} + K_{i}^{l} \lambda_{l} \\ 
T_{k}^{j} \psi_{j} + Q_{k}^{l} \lambda_{l} 
\end{pmatrix} 
\in 
\begin{matrix} 
\Gamma(\bbF_{i}) \\ 
\oplus \\ 
\Gamma(\bbG_{k}) 
\end{matrix},
\label{eq:contraction_Green}
\eeq
where the Einstein summation convention is employed. Here, the input $\Psi \in \Gamma(\bbE) \oplus \Gamma(\bbJ)$ is expressed in coordinates as $\Psi = (\Psi_{jl}) = ((\psi_{j}); (\lambda_{l}))$, with $\psi_{j} \in \Gamma(\bbE_{j})$ and $\lambda_{l} \in \Gamma(\bbJ_{l})$, and the output $\bD \Psi \in \Gamma(\bbF) \oplus \Gamma(\bbG)$ is written in coordinates as $(\bD \Psi)_{ik}$, as above.

When dealing with Douglis-Nirenberg systems, it is convenient to record the classes and orders as tuples of integers, with the ranging over the indices only implied:
\begin{definition}
\label{def:corresponding}
Given a Douglis-Nirenberg system $\bD$, the tensor of integers
\beq
\begin{pmatrix}
m_{i}^{j} & q_{i}^{l} \\ \tau_{k}^{j} & \sigma_{k}^{l}
\end{pmatrix}
\label{eq:corresponding_orders}
\eeq
are called \emph{corresponding orders} for $\bD$. The integers $(r^{j})$ are called \emph{corresponding classes} for $\bD$.
\end{definition}

Occasionally, it may happen that the corresponding classes and corresponding orders of $\bD$ are constant, meaning there exist integers $r, m, q, \tau, \sigma \in \mathbb{Z}$ such that $r^{j} = r$ for all $j$, and:
\[
\begin{cases}
m_{i}^{j} = m, & \text{for all } i, j, \\
q_{i}^{l} = q, & \text{for all } i, l, \\
\tau_{k}^{j} = \tau, & \text{for all } k, j, \\
\sigma_{k}^{l} = \sigma, & \text{for all } k, l.
\end{cases}
\]
In such a situation, where no ambiguity arises, the corresponding classes are written simply as $(r)$, and the corresponding orders as:
\[
\begin{pmatrix}
m & q \\ \tau & \sigma
\end{pmatrix}.
\]
Note that if the corresponding classes are $(r)$ for some constant $r \in \bbZ$ and the corresponding orders are
\[
\begin{pmatrix}
m & m \\ m-1 & m
\end{pmatrix},
\]
then the system $\bD$ is simply an element of $\OP(m,r)$, i.e., a Green operator of order $m$ and class $r$ as defined in \eqref{eq:full_Green_operator}, regardless of whether the vector bundles decomposes further or not. For similar reasons that Green operators may not have unique orders or classes, the orders and classes of a Douglis-Nirenberg system are also not unique. In fact, every Douglis-Nirenberg system technically belongs to $\OP(m,r)$ for sufficiently large $m$ and $r$ that exceed the respective orders and classes of the components in \eqref{eq:douglas_nirenberg_components}. This nuance introduces complexity to the calculus of orders and classes, which are no longer represented by a single number.

Building down from Green operators, Douglis-Nirenberg systems essentially arise in one of the following ways:

\begin{definition}
\label{def:direct_sums_unions_sums}
Let $\bD_{2}$ and $\bD_{1}$ be two Douglis-Nirenberg systems:
\[
\bD_{2}:\Gamma(\bbE_{2};\bbJ_{2})\rightarrow \Gamma(\bbF_{2};\bbG_{2}), \qquad \bD_{1}:\Gamma(\bbE_{1};\bbJ_{1})\rightarrow \Gamma(\bbF_{1};\bbG_{1}).
\]
\begin{enumerate}
\item If $\bbE_{2}= \bbE_{1}$ and $\bbJ_{2}=\bbJ_{1}$, then the systems' \emph{direct sum} $\bD_{2}\oplus\bD_{1}$ is the system
\[
\bD_{2}\oplus\bD_{1}: \Gamma(\bbE_{1};\bbJ_{1})\rightarrow \Gamma(\bbF_{2}\oplus\bbF_{1};\bbG_{2}\oplus\bbG_{1}),
\]
operating as
\[
\Psi\mapsto (\bD_{2}\Psi, \bD_{1}\Psi).
\]

\item If $\bbF_{2}= \bbF_{1}$ and $\bbG_{2}=\bbG_{1}$, then the systems' \emph{co-direct sum} $\bD_{2}\oplus^*\bD_{1}$ is the system
\[
\bD_{2}\oplus^*\bD_{1}: \Gamma(\bbE_{2}\oplus\bbE_{1};\bbJ_{2}\oplus\bbJ_{1})\rightarrow \Gamma(\bbF_{1};\bbG_{1}),
\]
operating as
\[
(\Psi, \Upsilon) \mapsto \bD_{2}\Psi + \bD_{1}\Upsilon,
\]
where $\Psi\in \Gamma(\bbE_{2}; \bbF_{2})$ and $\Upsilon\in \Gamma(\bbE_{1}; \bbF_{1})$.

\item If both $\bbE_{2} = \bbE_{1}$, $\bbJ_{2} = \bbJ_{1}$ and $\bbF_{2} = \bbF_{1}$, $\bbG_{2} = \bbG_{1}$, then the systems' \emph{sum} $\bD_{2} + \bD_{1}$ is the system
\[
\bD_{2} + \bD_{1}: \Gamma(\bbE_{1}; \bbJ_{1})\rightarrow \Gamma(\bbF_{1}; \bbG_{1}),
\]
operating as
\[
\Psi \mapsto \bD_{2}\Psi + \bD_{1}\Psi.
\]

\item The systems' \emph{disjoint union} $\bD_{2} \sqcup \bD_{1}$ is the system
\[
\bD_{2}\sqcup\bD_{1}: \Gamma(\bbE_{2}\oplus\bbE_{1};\bbJ_{2}\oplus\bbJ_{1})\rightarrow \Gamma(\bbF_{2}\oplus\bbF_{1};\bbG_{2}\oplus\bbG_{1}),
\]
operating as
\[
(\Psi, \Upsilon) \mapsto (\bD_{2}\Psi, \bD_{1}\Upsilon),
\]
where $\Psi \in \Gamma(\bbE_{2}; \bbF_{2})$ and $\Upsilon \in \Gamma(\bbE_{1}; \bbF_{1})$.
\end{enumerate}
\end{definition}

By treating each component in the matrix \eqref{eq:douglas_nirenberg_components} as its own Douglis-Nirenberg system, one retains from \defref{def:direct_sums_unions_sums} the ability to take sums, direct sums, co-direct sums, and disjoint unions of any two operators from any of the different classes in the calculus.

In the specific case of a direct sum as defined above, it is clear that the corresponding class of $\bD_{2} \oplus \bD_{1}$ is given by $\tilde{\tilde{r}}^{j} = \max(r^{j}, \tilde{r}^{j})$, where $r^{j}$ (resp. $\tilde{r}^{j}$) are the corresponding classes of $\bD_{1}$ (resp. $\bD_{2}$). To handle the corresponding orders of the resulting system conveniently, we adopt the following convention:

\begin{definition}
\label{def:corresponding_orders_direct_sums}
Let $\bD_{2}$ and $\bD_{1}$ satisfy the conditions in item (1) of \defref{def:direct_sums_unions_sums}. Let their respective corresponding orders be:
\[
\begin{pmatrix}
\tilde{m}_{\ell}^{j} & \tilde{q}_{\ell}^{l} \\ \tilde{\tau}_{u}^{j} & \tilde{\sigma}_{u}^{l}
\end{pmatrix},
\qquad
\begin{pmatrix}
m_{i}^{j} & q_{i}^{l} \\ \tau_{k}^{j} & \sigma_{k}^{l}
\end{pmatrix}.
\]
The corresponding orders of $\bD_{2} \oplus \bD_{1}$ are then denoted by:
\[
\begin{pmatrix}
(\tilde{m}_{\ell}^{j}, m_{i}^{j}) & (\tilde{q}_{\ell}^{l}, q_{i}^{l}) \\
(\tilde{\tau}_{u}^{j}, \tau_{k}^{j}) & (\tilde{\sigma}_{u}^{l}, \sigma_{k}^{l})
\end{pmatrix},
\]
with the ranging over the indices implied.
\end{definition}

Next, by carefully interpreting the operation in \eqref{eq:contraction_Green}, we can generalize the composition rules in \thmref{thm:composition_Green_operators} to the Douglis-Nirenberg setting:
\begin{theorem}
\label{thm:composition_Douglis_Nirenberg}
Let $\bD_{1}$ and $\bD_{2}$ be two Douglis-Nirenberg systems,
\[
\bD_{2}: \begin{matrix}\oplus_{i}\Gamma(\bbF_{i}) \\\oplus\\ \oplus_{k}\Gamma(\bbG_{k})\end{matrix} \longrightarrow \begin{matrix}\oplus_{\ell}\Gamma(\bbU_{\ell}) \\ \oplus \\ \oplus_{u}\Gamma(\mathbb{L}_{u})\end{matrix}, \qquad
\bD_{1}: \begin{matrix}\oplus_{j}\Gamma(\bbE_{j}) \\\oplus \\\oplus_{l}\Gamma(\bbJ_{l})\end{matrix} \longrightarrow \begin{matrix}\oplus_{i}\Gamma(\bbF_{i}) \\ \oplus \\ \oplus_{k}\Gamma(\bbG_{k})\end{matrix}
\]
with corresponding classes $(\tilde{r}^{i})$ and $({r}^{j})$, and corresponding orders
\[
\begin{pmatrix}
\tilde{m}_{\ell}^{i} & \tilde{q}_{\ell}^{k} \\ \tilde{\tau}_{u}^{i} & \tilde{\sigma}_{u}^{k}
\end{pmatrix}, \qquad
\begin{pmatrix}
m_{i}^{j} & q_{i}^{l} \\ \tau_{k}^{j} & \sigma_{k}^{l}
\end{pmatrix}.
\]
Then, the composition $\bD_{2}\bD_{1}$ has corresponding classes
\beq
\tilde{\tilde{r}}^{j} = \underset{i,k}{\max}(\tilde{r}^{i} + m_{i}^{j}, \tau_{k}^{j}, r^{j}),
\label{eq:comp_douglas_nirenberg_1}
\eeq
and corresponding orders
\beq
\begin{pmatrix}
\tilde{\tilde{m}}^{j}_{\ell} & \tilde{\tilde{q}}^{l}_{\ell} \\ \tilde{\tilde{\tau}}^{j}_{u} & \tilde{\tilde{\sigma}}_{u}^{l}
\end{pmatrix} = \begin{pmatrix}
\underset{i,k}{\max}(\tilde{m}^{i}_{\ell} + m_{i}^{j}, \tilde{q}^{k}_{\ell} + \tau^{j}_{k}) & \underset{i,k}{\max}(\tilde{m}^{i}_{\ell} + q_{i}^{l}, \tilde{q}^{k}_{\ell} + \sigma_{k}^{l}) \\
\underset{i,k}{\max}(\tilde{\tau}^{i}_{u} + m^{j}_{i}, \tilde{\sigma}^{k}_{u} + \tau_{k}^{j}) & \underset{i,k}{\max}(\tilde{\tau}^{i}_{u} + q^{l}_{i}, \tilde{\sigma}^{k}_{u} + \sigma_{k}^{l})
\end{pmatrix}.
\label{eq:comp_douglas_nirenberg_2}
\eeq
\end{theorem}

\begin{proof}
Explicitly, using tensor notation and the ``contraction" operation, we have
\[
\begin{split}
(\bD_{2}\bD_{1})_{\ell u}^{jl} &= (\bD_{2})_{\ell u}^{ik} (\bD_{1})_{ik}^{jl} = 
\begin{pmatrix}
\tilde{E}_{\ell}^{i} & \tilde{K}_{\ell}^{k} \\ \tilde{T}_{u}^{i} & \tilde{Q}_{u}^{k}
\end{pmatrix}
\begin{pmatrix}
E^{j}_{i} & K_{i}^{l} \\ T_{k}^{j} & Q_{k}^{l}
\end{pmatrix}
\\&=
\begin{pmatrix}
\tilde{E}_{\ell}^{i} E^{j}_{i} + \tilde{K}_{\ell}^{k} T_{k}^{j} & \tilde{E}_{\ell}^{i} K_{i}^{l} + \tilde{K}_{\ell}^{k} Q_{k}^{l} \\
\tilde{T}_{u}^{i} E^{j}_{i} + \tilde{Q}_{u}^{k} T_{k}^{j} & \tilde{T}_{u}^{i} K_{i}^{l} + \tilde{Q}_{u}^{k} Q_{k}^{l}
\end{pmatrix}.
\end{split}
\]
While straightforward, the composition rules for each term require care to follow precisely. Verifying these rules confirms that $\bD_{2}\bD_{1}$ is indeed a Douglis-Nirenberg system, with the corresponding classes given by \eqref{eq:comp_douglas_nirenberg_1} and the corresponding orders by \eqref{eq:comp_douglas_nirenberg_2}.
\end{proof}
\subsection{Mapping properties} 
\label{sec:sharp_lenient}
Next, we consider mapping properties of Douglis-Nirenberg systems between Sobolev spaces of sections. Naively, by examining the mapping properties of each component of $\bD$ separately as in \eqref{eq:Green_mapping_proprety0}, along with the continuous inclusions $W^{s,p}\hookrightarrow W^{s',p}$ for every $s\geq s'$ and $1<p<\infty$, it follows that $\bD$ has the collective mapping properties:
\beq
\bD : \begin{matrix}\oplus_{j} W^{s+r^j, p} \Gamma(\bbE_{j}) \\\oplus\\ \oplus_{l} W^{s+t^{l}+1 - 1/p, p} \Gamma(\bbJ_{l})\end{matrix} 
\longrightarrow
\begin{matrix}\oplus_{i} W^{s+\underset{j,l}{\min}(r^{j} - m_{i}^{j}, t^{l} - q_{i}^{l}), p} \Gamma(\bbF_{i}) \\ \oplus \\ \oplus_{k} W^{s+\underset{j,l}{\min}(r^{j} - \tau_{k}^{j}, t^{l} - \sigma_{k}^{l})+1 - 1/p, p} \Gamma(\bbG_{k})\end{matrix} 
\label{eq:douglas_nirenberg_mapping_1}
\eeq
for $s\in \bbR$ with $s > 1/p - 1$ and $t^{l} \in \bbR$.


To abbreviate and extend these mapping properties, consider tuples of real numbers and operations on them. Given a tuple $S = (a_{\alpha})_{\alpha=1}^{\alpha_0}$ indexed by positive integers, with $a_{\alpha} \in \bbR$, denote:
\[
|S| = \alpha_0, \qquad \max S = \max_{\alpha} \{a_{\alpha}\}, \qquad \min S = \min_{\alpha} \{a_{\alpha}\}.
\]
For two such tuples $S = (a_{\alpha})$ and $T = (b_{\alpha})$ with $|S| = |T|$, define:
\[
S + T = (a_{\alpha} + b_{\alpha})_{\alpha=1}^{\alpha_0}, \qquad S - T = (a_{\alpha} - b_{\alpha})_{\alpha=1}^{\alpha_0}.
\]
Additionally, if $s \in \bbR$, let:
\[
S + s = s + S = (s + a_{\alpha})_{\alpha=1}^{\alpha_0}.
\]
For arbitrary tuples $S = (a_{\alpha})_{\alpha=1}^{\alpha_0}$ and $T = (b_{\beta})_{\beta=1}^{\beta_0}$, define their concatenation:
\[
(S, T) = (a_{1}, \dots, a_{\alpha_0}, b_1, \dots, b_{\beta_0}).
\]
With these notions established, we can now introduce the following:
\begin{definition}
Let $S = (a_{\alpha})_{\alpha=1}^{\alpha_0}$ and $T = (b_{\beta})_{\beta=1}^{\beta_0}$ be tuples of real numbers. Consider also direct sums of vector bundles indexed accordingly:
\[
\bbE = \bigoplus_{1 \leq \alpha \leq \alpha_0} \bbE_{\alpha}, \qquad 
\bbJ = \bigoplus_{1 \leq \beta \leq \beta_0} \bbJ_{\beta}.
\]
For any $1 < p < \infty$, define the direct sum of Sobolev spaces as:
\[
W^{S,T}_p(\bbE; \bbJ) = 
\begin{matrix}
\displaystyle\oplus_{\alpha} W^{a_{\alpha}, p} \Gamma(\bbE_{\alpha}) \\
\oplus \\
\displaystyle\oplus_{\beta} W^{b_{\beta}, p} \Gamma(\bbJ_{\beta})
\end{matrix}.
\]
\end{definition}

The space above is a Banach space equipped with the product norm, which, for 
\[
\Psi = (\psi_{\alpha}, \lambda_{\beta})_{\alpha,\beta} \in W_p^{S,T}(\bbE; \bbJ),
\]
takes the form:
\[
\|\Psi\|_{S,T,p} = \sum_{\alpha} \|\psi_{\alpha}\|_{a_{\alpha}, p} + \sum_{\beta} \|\lambda_{\beta}\|_{b_{\beta}, p}.
\]

When the tuples consist of only a single number, we shall write, for example, $S = (s)$ and $T = (s' + 1 - 1/p)$ for fixed $s, s' \in \mathbb{R}$, with the indexing understood from context. In such cases, the space $W^{S,T}_{p}(\bbE; \bbJ)$ coincides with the usual Sobolev space of sections of the bundles $\bbE$ and $\bbJ$, and we use one of the following equivalent notations, depending on convenience:
\[
W^{S,T}_{p}(\bbE; \bbJ) = W^{s, s' + 1 - 1/p}_{p}(\bbE; \bbJ) = 
\begin{matrix}
\bigoplus_{\alpha} W^{s,p}\Gamma(\bbE_{\alpha}) \\
\oplus \\
\bigoplus_{\beta} W^{s',p}\Gamma(\bbJ_{\beta})
\end{matrix}
= 
\begin{matrix}
\bigoplus W^{s,p}\Gamma(\bbE) \\
\oplus \\
W^{s',p}\Gamma(\bbJ)
\end{matrix}.
\]

A particular case of interest occurs when $s = 0$ and $s' = 1/p - 1$, in which case the space reduces to:
\[
W^{0,0}_{p}(\bbE; \bbJ) = L^{p}(\bbE; \bbJ) = 
\begin{matrix}
\bigoplus_{\alpha} L^p \Gamma(\bbE_{\alpha}) \\
\oplus \\
\bigoplus_{\beta} L^p \Gamma(\bbJ_{\beta})
\end{matrix},
\]
with the product norm:
\[
\|\Psi\|_{0,0,p} = \sum_{\alpha} \|\psi_{\alpha}\|_{0,p} + \sum_{\beta} \|\lambda_{\beta}\|_{0,p}.
\]

For $p=2$, this becomes a product Hilbert space. Additionally, the smooth version is defined as:
\[
\Gamma(\bbE; \bbJ) = \begin{matrix}\oplus_{\alpha} \Gamma(\bbE_{\alpha}) \\\oplus\\ \oplus_{\beta} \Gamma(\bbJ_{\beta})\end{matrix} .
\]
Finally, consider also direct sums of compactly supported sections:
\[
\Gamma_{c}(\bbE; \bbJ) = \begin{matrix}\oplus_{\alpha} \Gamma_{c}(\bbE_{\alpha}) \\ \oplus\\\oplus_{\beta} \Gamma(\bbJ_{\beta})\end{matrix}.
\]
and so taking completions with respect to the Sobolev norms yields the spaces:
\[
W^{S,T}_{p,0}(\bbE; \bbJ) = \begin{matrix}\oplus_{\alpha} W^{a_{\alpha}, p}_0 \Gamma(\bbE_{\alpha}) \\ \oplus \\ \oplus_{\beta} W^{b_{\beta}, p} \Gamma(\bbJ_{\beta})\end{matrix}
\]
and we have the duality relation $W^{S,T}_{p,0}(\bbE;\bbJ)\simeq (W^{-S,-T}_{q}(\bbE;\bbJ))^*$ where $1/p+1/q=1$ as usual. 

In these notations, given a Douglis–Nirenberg system $\bD$ with specified corresponding orders and classes as defined in \defref{def:corresponding}, we introduce the \emph{basic tuples} for $\bD$ as:
\beq
\begin{gathered}
J_0 = (r^j), \qquad 
I_0 = \big( \underset{j,l}{\min}\, (r^j - m_{i}^{j},\; t^l - q_{i}^{l}) \big), \\ 
L_0 = (t^l), \qquad 
K_0 = \big( \underset{j,l}{\min}\, (r^j - \tau_{k}^{j},\; t^l - \sigma_{k}^{l}) \big).
\end{gathered}
\label{eq:strict_tuples0}
\eeq

For $s \in \mathbb{R}$ with $s > 1/p - 1$, we define, more generally, the $p$-\emph{standard tuples} for $\bD$ as:
\beq
\begin{gathered}
J = s + J_0, \qquad 
I = s + I_0, \\
L = s + L_0 + 1 - 1/p, \qquad 
K = s + K_0 + 1 - 1/p.
\end{gathered}
\label{eq:strict_tuples}
\eeq

With these definitions, for every $s > 1/p - 1$, the mapping property in \eqref{eq:douglas_nirenberg_mapping_1} can be written more compactly as:
\beq
\bD: W^{J, L}_{p} (\bbE; \bbJ) \to W^{I, K}_p (\bbF; \bbG).
\label{eq:strict_mapping_property}
\eeq

For simplicity, when no ambiguity arises, we often omit the dependence on $p$ and refer to \eqref{eq:strict_tuples} simply as the \emph{standard tuples}. The mapping \eqref{eq:strict_mapping_property} is then referred to as the \emph{standard mapping property} of $\bD$.

By construction, the standard mapping properties \eqref{eq:strict_mapping_property} depend on the specified corresponding orders and classes for $\bD$. However, as noted in the discussion surrounding \propref{prop:L2_continuity}, these values are not uniquely determined: the sharp order and class of certain components in the matrix representation of $\bD$ may, in fact, be lower than initially assigned. Our goal here is to obtain a sharp characterization of these orders and classes directly from the mapping properties of $\bD$, in a manner analogous to the results established for Green operators in \propref{prop:L2_continuity}.

To illustrate why it is necessary to analyze the mapping properties directly in the varying order framework—rather than relying on a symbol calculus, as is possible in the case of Green operators—we turn to the following construction. Given a Douglis–Nirenberg system $\bD$ with corresponding orders as in \eqref{eq:corresponding_orders}, consider the systems
\[
\Lambda : \Gamma(\bbE; \bbJ) \to \Gamma(\bbE; \bbJ), \qquad 
\Pi : \Gamma(\bbF; \bbG) \to \Gamma(\bbF; \bbG),
\]
where $\Lambda$ and $\Pi$ consist of \emph{order-reducing operators} (see \cite[Cor.~5.5]{Gru90}).
\beq
\Lambda_{jl}^{j'l'}=\begin{pmatrix}\delta_{j}^{j'}\calL^{-r^{j'}}_{\bbE_{j'}} & 0 \\ 0 & \delta_{l}^{l'}\calL^{-t^{l'}+1}_{\bbJ_{l'}}\end{pmatrix}, \qquad \Pi_{i'k'}^{ik}=\begin{pmatrix}\delta_{i'}^{i}\calL^{\underset{j,l}{\min}(r^{j} - m_{i}^{j}, t^{l} - q_{i}^{l})}_{\bbF_{i}} & 0 \\ 0 & \delta_{k'}^{k}\calL_{\bbG_{k}}^{\underset{j,l}{\min}(r^{j} - \tau_{k}^{j}, t^{l} - \sigma_{k}^{l})+1}\end{pmatrix}.
\label{eq:order_reducing_opertors}
\eeq
By construction, the corresponding classes of $\Lambda$ are $(-r^{j'})$, and the corresponding orders are 
\[
\begin{pmatrix}
-\delta_{j}^{j'}r^{j'} & -\infty \\ -\infty & -\delta^{l'}_{l}t^{l'}
\end{pmatrix}
\]
whereas the corresponding classes of $\Pi$ are $(\underset{j,l}{\min}(r^{j}-m_i^{j},t^{l}-q^{l}_{i}))$ (indexed by $i$) and its corresponding orders are
\[
\begin{pmatrix}
\delta^{i}_{i'}\underset{j,l}{\min}(r^{j} - m_{i}^{j}, t^{l} - q_{i}^{l}) & -\infty \\ -\infty & \delta^{k}_{k'}\underset{j,l}{\min}(r^{j} - \tau_{k}^{j}, t^{l} - \sigma_{k}^{l})
\end{pmatrix}
\]
For every standard tuples $(J,L;I,K)$ for $\bD$ as in \eqref{eq:strict_tuples}, the mapping property \eqref{eq:douglas_nirenberg_mapping_1} then reads that
\[
\Lambda: \begin{matrix}W^{s,p}\Gamma(\bbE)
\\ \oplus \\ 
W^{s+1-1/p,p}\Gamma(\bbJ)\end{matrix}\longrightarrow \begin{matrix}\oplus_{j}W^{s+r^j,p}\Gamma(\bbE_{j})
\\ \oplus \\
\oplus_{l}W^{s+t^l+1-1/p,p}\Gamma(\bbJ_{l})\end{matrix}=W^{J,L}_{p}(\bbE;\bbJ)
\]
and
\[
\Pi:W^{I,K}_{p}(\bbF;\bbG)=\begin{matrix}\oplus_{i}
W^{s+\underset{j,l}{\min}{(r^j-m_{i}^{j},t^l+1-q_{i}^{l})},p}\Gamma(\bbF_{i})\\\oplus\\\oplus_{k}W^{s+\underset{j,l}{\min}{(r^j-\tau_{k}^{j},t^l-\sigma_{k}^{l})}+1-1/p,p}\Gamma(\bbG_{k})\end{matrix} \longrightarrow \begin{matrix}W^{s,p}\Gamma(\bbF)
\\ \oplus \\ 
W^{s+1-1/p,p}\Gamma(\bbG)\end{matrix}.
\]
The fundamental property of the order reducing operators $\Lambda$ and $\Pi$ is that the above continuous extensions are isomorphisms of Banach spaces, and as such yield inverses within the calculus going in the opposite directions, denoted by $\Lambda^{-1}$ and $\Pi^{-1}$. 

By carefully following the composition rules provided by \thmref{thm:composition_Douglis_Nirenberg}, we find that $\Pi\bD\Lambda$, with $\Pi,\bD,\Lambda$ as above, is a Douglis Nirenberg system with corresponding classes $(0)$ and corresponding orders
\[
\begin{pmatrix}
0 & 0 \\ -1 & 0
\end{pmatrix}.
\] 
Thus, $\Pi\bD\Lambda\in\OP(0,0)$ in the standard sense of Green operators. 
\begin{definition}
\label{def:order_reduced_symbol}
Given corresponding orders and classes for $\bD$ as in \eqref{eq:corresponding_orders}, and the associated order-reducing operators $\Pi$ and $\Lambda$ defined in \eqref{eq:order_reducing_opertors}, the principal symbol 
\[
\sigma(\Pi \bD \Lambda) \in \mathrm{S}(0,0)
\]
is called the \emph{order-reduced principal symbol} of $\bD$ associated with the basic tuples $(J_0, L_0;\, I_0, K_0)$.
\end{definition}

Now, with basic tuples $(J_0,L_0;I_0,K_0)$ as in \eqref{eq:strict_tuples0}, let $\scrL(J_0,L_0;I_0,K_0)$ stand for the space of all continuous linear maps between the Hilbert spaces 
\[
W^{J_0,L_0+1/2}_2(\bbE;\bbJ) \rightarrow W^{I_0,K_0+1/2}_2(\bbF;\bbG)
\]
equipped with the operator norm $\|\cdot\|_{J_0,L_0;I_0,K_0}$, so $\bD \in \scrL(J_0,L_0;I_0,K_0)$ (as a continuous extension).
Adapting \cite[Cor.~2, p.~174]{RS82} to the Douglis-Nirenberg setting, the correspondence $\frakA \mapsto \sigma(\Pi \frakA \Lambda)$ between a Douglis-Nirenberg system and the associated order-reduced principle symbol extends to a continuous linear map from a subspace of $\scrL(J_0,L_0;I_0,K_0)$ onto an appropriate Banach space, with a norm denoted by $\|\cdot\|_{0}$, such that the following holds:
\beq
\underset{\frakC}{\inf}{\|\bD+\frakC\|_{J_0,L_0;I_0,K_0}} = \|\sigma(\Pi \frakA \Lambda)\|_{0}
\label{eq:principle_norm}
\eeq 
where the infimum is taken with respect to all compact operators $\frakC \in \scrL(J_0,L_0;I_0,K_0)$. Thus, the associated order-reduced principal symbol for the system $\bD$ vanishes precisely when its continuous extension \eqref{eq:strict_mapping_property} is compact, indicating that the corresponding orders used to generate $(J_0, L_0; I_0, K_0)$ were not accurate and could actually be reduced. However, even if the collective mapping $\bD$ is not compact, some of the isolated continuous extensions of its matrix components $\bD^{jl}_{ik}$ in \eqref{eq:douglas_nirenberg_components} may still be compact. In such cases, the fact that the order-reduced principal symbol does not identically vanish is incidental.  

As a result, unlike standard Green operators, we must develop a framework to handle Douglis-Nirenberg systems independently of the properties of their order-reduced symbol. This is, again, achieved by extracting information directly from mapping properties they admit. 

To that end, we introduce the concept of \emph{lenient} mapping properties, which refer to any continuous extensions that $\bD$ may admit when \emph{not} restricted to compactly supported sections:

\begin{definition}
\label{def:lenient_mapping}
Let $\bD$ be a Douglis-Nirenberg system, let $(J_0, L_0; I_0, K_0)$ be basic tuples for $\bD$ as in \eqref{eq:strict_tuples0}, and let $1 < p < \infty$. Let $(S, T; S', T')$ be any tuples of real numbers satisfying:
\[
|S| = |J_0|, \qquad |T| = |L_0|, \qquad |S'| = |I_0|, \qquad |T'| = |K_0|.
\] 
Then $(S, T; S', T')$ are called $p$-\emph{lenient tuples} for $\bD$ if, for all $\Psi \in \Gamma(\bbE; \bbJ)$,
\[
\|\bD\Psi\|_{S', T', p} \lesssim \|\Psi\|_{S, T, p},
\]
meaning that $\bD: \Gamma(\bbE; \bbJ) \rightarrow \Gamma(\bbF; \bbG)$ extends into a continuous map
\beq
\bD: W^{S, T}_p (\bbE; \bbJ) \rightarrow W^{S', T'}_{p}(\bbF; \bbG).
\label{eq:douglas_nirenberg_mapping_2}
\eeq
A mapping property of the form \eqref{eq:douglas_nirenberg_mapping_2} is called a $p$-\emph{lenient mapping property} of $\bD$.

When there is no ambiguity, the $p$ is omitted, and we refer to $(S, T; S', T')$ simply as \emph{lenient tuples} and to \eqref{eq:douglas_nirenberg_mapping_2} as a \emph{lenient mapping property}.
\end{definition}

The mapping properties in \eqref{eq:strict_mapping_property} arising from standard tuples for $\bD$ serve as a specific example of a lenient mapping property. The advantage of the broader notion of lenient mapping properties is that they may hold even before some corresponding orders and classes for a system $\bD$ are specified. For instance, recalling that $W^{s,p} \hookrightarrow W^{s',p}$ for every $s \geq s'$, it becomes evident that by choosing $S, T$ sufficiently large, we have:
\begin{corollary}
\label{corr:weak_tuples}
Let $\bD$ be a Douglis Nirenberg system and let $(J,L;I,K)$ be standard tuples for $\bD$. Then every $S',T'$ tuples of real numbers with $|S'|=|I|$ and $|T'|=|K|$ and $\min{S},\min{T'}\geq 0$, there exists tuples of real number $S,T$ such that $(S,T;S',T')$ are lenient tuples for $\bD$.  
\end{corollary}

The classes of Green operators $\OP(m, r)$ are useful for handling lenient mapping properties. That is because, as previously noted, Douglis–Nirenberg systems are technically Green operators and belong to $\OP(m, r)$ for sufficiently large $m$ and $r$—specifically, values that exceed all corresponding orders and classes of the components in \eqref{eq:douglas_nirenberg_components}. In fact, by applying \propref{prop:L2_continuity} to each component individually, one obtains the following:

\begin{proposition}
\label{prop:G0_criteria}
Let $\bD$ be a Douglis–Nirenberg system. If, for some $m, s \in \mathbb{R}$ and tuples $S, T$ of real numbers with $\max(S, T) \leq s + 1/2$, the operator $\bD$ satisfies the lenient mapping property
\[
\bD : W_{2}^{S,\, T }(\bbE; \bbJ) \rightarrow W_{2}^{S - m,\, T - m}(\bbF; \bbG),
\]
then $\bD \in \OP(m, r)$ for any $r \in \mathbb{Z}$ such that $r \leq s$.
\end{proposition}

In particular, if $\bD$ has the lenient mapping property
\[
\bD : L^2(\bbE; \bbJ) \rightarrow L^{2}(\bbF; \bbG),
\]
then by taking $S=(0)$, $T=(0)-1/2$ and $s=0$, one finds that $\bD \in \OP(0, 0)$. Indeed, the class $\OP(0,0)$ enjoys a range of desirable properties:

\begin{corollary}
\label{corr:G0props}
The class $\OP(0,0)$ is closed under composition and adjunction. Moreover, for any tuples $(S, T;\, S', T')$ of real numbers with 
\[
\min(S),\, \min(T),\, \min(S'),\, \min(T') \geq 0,
\]
any $\bD \in \OP(0,0)$ satisfies the lenient mapping property:
\[
\bD : W^{S, T}_{p}(\bbE; \bbJ) \rightarrow W^{S', T'}_{p}(\bbF; \bbG).
\]
\end{corollary}

It is also convenient to consider direct sums of Douglis–Nirenberg systems $\bD$ with systems in $\OP(0,0)$, as it is relatively straightforward to determine the corresponding orders, classes, and standard tuples for the resulting system.

\begin{proposition}
\label{prop:G0tuples}
Let $\bD : \Gamma(\bbE; \bbJ) \rightarrow \Gamma(\bbF; \bbG)$ be a Douglis–Nirenberg system with corresponding classes $(r^{j})$ and corresponding orders as in \eqref{eq:corresponding_orders}. Let $\tbP \in \OP(0,0)$ act as 
\[
\tbP : \Gamma(\bbE; \bbJ) \rightarrow \Gamma(\bbE; \bbJ).
\]
If $(J, L;\, I, K)$ are standard tuples for $\bD$, then 
\[
(J, L;\, (J, I), (L, K))
\]
are standard tuples for the direct sum system $\bD \oplus \tbP$.
\end{proposition}
\subsection{Adjoints and Green's formulae} 
When it comes to adjoints, it follows from \propref{prop:L2_continuity}, applied to each component separately, that a Douglis–Nirenberg system $\bD$ possesses an adjoint if and only if all of its corresponding classes are zero. In this case, the matrix components of the adjoint system $\bD^* : \Gamma(\bbF; \bbG) \rightarrow \Gamma(\bbE; \bbJ)$ are given by:
\beq
(\bD^*)_{ik}^{jl} = (\bD_{jl}^{ik})^* = 
\begin{pmatrix}
(E_i^{j})^{*} & (T_{k}^{j})^{*} \\
(K_{i}^{l})^* & (Q_{k}^{l})^*
\end{pmatrix}.
\label{eq:adjoint_douglas}
\eeq

By applying Green's formula \eqref{eq:integration_by_parts_convention} and invoking \propref{prop:formal_adjoint} componentwise, we find that if $\bD$ has all corresponding classes equal to zero, then there exist Douglis–Nirenberg systems
\[
\bB_{\bD} : \Gamma(\bbE; \bbJ) \rightarrow \Gamma(0; \bbL), \qquad 
\bB^*_{\bD} : \Gamma(\bbF; \bbG) \rightarrow \Gamma(0; \bbL),
\]
where $0$ denotes the zero bundle, and whose matrix components take the form:
\[
(\bB_{\bD})^{jl}_{\ell} = 
\begin{pmatrix}
0 & 0 \\
(B_{E_{i}^{j}})_{\ell} & 0
\end{pmatrix}, \qquad
(\bB^*_{\bD})^{jl}_{\ell} = 
\begin{pmatrix}
0 & 0 \\
(B_{(E_{i}^{j})})^*_{\ell} & 0
\end{pmatrix},
\]
such that for all $\Upsilon \in \Gamma(\bbE; \bbJ)$ and $\Theta \in \Gamma(\bbF; \bbG)$, the following identity holds:
\beq
\bra \bD \Upsilon, \Theta \ket = \bra \Upsilon, \bD^* \Theta \ket + \bra \bB_{\bD} \Upsilon, \bB^*_{\bD} \Theta \ket,
\label{eq:integration_by_parts_convention_2}
\eeq
where $\bra \cdot, \cdot \ket$ denotes the direct sum $L^2$-inner product induced on $\Gamma(\bbE; \bbJ)$ and $\Gamma(0; \bbL)$.

\begin{definition}
\label{def:normal_system_of_trace_operators2}
A \emph{system of boundary operators} is a Douglis–Nirenberg system 
\[
\bB : \Gamma(\bbE; \bbJ) \rightarrow \Gamma(0; \bbL)
\]
whose matrix components are of the form:
\[
\bB^{jl}_{\ell} = 
\begin{pmatrix}
0 & 0 \\
T_{\ell}^{j} & Q_{\ell}^{l}
\end{pmatrix}.
\]

Adapting \defref{def:normal_system_of_trace_operators1}, the system is called a \emph{normal system} if the trace operator
\[
T = \bigoplus_{j, \ell} T_{\ell}^{j} : \Gamma(\bbE) \rightarrow \Gamma(\bbJ)
\]
takes the form
\beq 
T_{\ell}^{j} = \sum_{\kappa} S_{\kappa, \ell}^{j}\D_{\frakn}^{\kappa}+\tilde{T}_{\ell}^{j},
\label{eq:T_explicit_proof} 
\eeq 
where each $S_{\kappa, j}^{j}$ is a surjective pseudodifferential operator, and $\tilde{T}_{\ell}^{j}$ is a lower order and class Green operator. 
\end{definition}

Regarding the normality of the boundary terms in \eqref{eq:integration_by_parts_convention_2}, one can readily extend the results of \propref{prop:normal_sytem}, \propref{prop:direct_sum_trace_operator_elliptic_operator}, and the surrounding discussion to the Douglis–Nirenberg setting.

\begin{corollary}
\label{corr:normal_system_2}
The following statements hold:
\begin{enumerate}
\item If $\bB$ is a normal system of boundary operators, then $\bB$ is surjective, and $\ker \bB$ is dense in $L^p(\bbE; \bbJ)$ for every $1 < p < \infty$.

\item Let $\mathfrak{L}$ be a Douglis–Nirenberg system of the form
\[
\mathfrak{L}^{jl}_{kl} = 
\begin{pmatrix}
L_{i}^{j} & 0 \\
0 & W_{k}^{l}
\end{pmatrix},
\]
where each $L_{i}^{j}$ is the truncation of an elliptic operator with the transmission property. Then there exists a vector bundle $\bbL \rightarrow \partial M$, and normal systems of boundary operators
\[
\bB_{\mathfrak{L}} : \Gamma(\bbE; \bbJ) \rightarrow \Gamma(0; \bbL), \qquad 
\bB_{\mathfrak{L}^*} : \Gamma(\bbF; \bbG) \rightarrow \Gamma(0; \bbL),
\]
such that the integration-by-parts identity \eqref{eq:integration_by_parts_convention_2} holds for both $\mathfrak{L}$ and its adjoint $\mathfrak{L}^*$.

\item If $\bB$ is another normal system of boundary operators and $\mathfrak{L}$ is as above, then the system $\bB \mathfrak{L} \oplus \bB_{\mathfrak{L}}$ is also a normal system of boundary operators.
\end{enumerate}
\end{corollary}

\begin{proof}
By using the same techniques of \cite[Ch.~1.6]{Gru96}, we may assume that $\tilde{T}_{\ell}^{j}=0$ in \eqref{eq:T_explicit_proof}.

 Item (2) follows directly from the definitions, and item (3) is a consequence of the discussion around \propref{prop:normal_sytem} by composition. Hence it remains to prove item (1).

In the notation of \defref{def:normal_system_of_trace_operators2}, surjectivity is clear since one can restrict $\bB$ to $\Gamma(\bbE; 0)$, obtaining the normal system of trace operators $T$, which is surjective into $\Gamma(0; \bbL)$. 

To prove the required density, denote $Q = \bigoplus_{\ell, l} Q_{\ell}^{l}$, let $(\psi, \lambda) \in \Gamma(\bbE; \bbJ)$ and take an approximating sequence $(\overbar{\psi}_{n}) \subset \ker T$ such that $\overbar{\psi}_{n} \to \psi$ in $L^p$ (which exists since $T$ is normal).

Now consider a sequence of open collars $\partial M \subset \Omega_{n+1} \subset \Omega_n \subset M$, with $\mathrm{Vol}(\Omega_n) \to 0$. By \defref{def:normal_system_of_trace_operators2}, the operator $T$ remains surjective when restricted to $\Gamma(\bbE|_{\Omega_n})$, since the surjectivity of the boundary pseudodifferential operators $S_{\kappa,\ell}^{j}$ is unaffected by the restriction to collars. Thus, for each $n$, we can find sections $\tilde{\psi}_{n} \in \Gamma(\bbE|_{\Omega_n})$ satisfying
\[
T \tilde{\psi}_{n} + Q \lambda = 0,
\]
with $\tilde{\psi}_{n}$ uniformly bounded in $L^{p}(\bbE|_{\Omega_{n}}; 0)$ (as their norm is controlled by that of the boundary section $\lambda$).

By extending $\tilde{\psi}_{n}$ smoothly to all of $M$ using a bump function supported in $\Omega_n$, we obtain sections in $\Gamma(\bbE)$ such that $\tilde{\psi}_{n} \to 0$ in the $L^p$-norm, as $\mathrm{Vol}(\Omega_n) \to 0$. Importantly, this extension does not affect the condition $T \tilde{\psi}_{n} + Q \lambda = 0$, since $T$ is supported near the boundary.

Define $\lambda_n := \lambda$ and $\psi_n := \overbar{\psi}_{n} + \tilde{\psi}_{n}$. Then $(\psi_n, \lambda_n) \to (\psi, \lambda)$ in $L^p$ by construction, and
\[
\bB(\psi_n; \lambda_n) = T \psi_n + Q \lambda_n = T \tilde{\psi}_n + Q \lambda = 0,
\]
as required.
\end{proof}
At this point, a basic observation about Douglis–Nirenberg systems 
\[
\bD : \Gamma(\bbE; \bbJ) \rightarrow \Gamma(\bbF; \bbG)
\]
is that the components in \eqref{eq:douglas_nirenberg_components} with class zero can be isolated from the matrix in a well-defined manner, yielding the \emph{zero-class constituent} of $\bD$, denoted by $\tbD$.

Since, by definition, the components of $\tbD$ have class zero, this part of $\bD$ admits an adjoint within the calculus, as described in \eqref{eq:adjoint_douglas}. We adopt a slight abuse of notation and denote this adjoint also by $\bD^*$, referring to it as the \emph{adjoint of the zero-class constituent of $\bD$}. It holds that $\bD^{**} = \tbD$ if and only if all corresponding classes of $\bD$ are zero. 

By $L^2$-continuity of systems with class zero, the adjoint of the zero-class constituent satisfies an adaptation of the relation in \eqref{prop:formal_adjoint} to the Douglis–Nirenberg setting:
\beq
\bra \bD \Upsilon, \Theta \ket = \bra \Upsilon, \bD^* \Theta \ket, \qquad 
\Upsilon \in \Gamma_c(\bbE; \bbJ), \quad \Theta \in \Gamma_c(\bbF; \bbG).
\label{eq:integration_by_parts_zero_class_cons}
\eeq

For reference, we record these properties formally:
\begin{definition}[Zero-class constituent and its adjoint]
\label{def:zero_class_constituent}
Given a Douglis–Nirenberg system 
\[
\bD : \Gamma(\bbE; \bbJ) \rightarrow \Gamma(\bbF; \bbG),
\]
the \emph{zero-class constituent} of $\bD$, denoted $\tbD$, is the submatrix of $\bD$ consisting of all components of class zero.

The adjoint of this zero-class constituent, also denoted $\bD^*$, is the unique system satisfying
\[
\bra \bD \Upsilon, \Theta \ket = \bra \Upsilon, \bD^* \Theta \ket
\]
for all $\Upsilon \in \Gamma_c(\bbE; \bbJ)$ and $\Theta \in \Gamma_c(\bbF; \bbG)$.
\end{definition}

\subsection{Overdetermined ellipticity}
\label{sec:overdetermined_douglas_nirenberg} 
Technically, every Douglis--Nirenberg system $\bD$ possesses a principal symbol as a Green operator, denoted by $\sigma(\bD)$. As a Green operator, the order and class of $\bD$ are the two highest among all its different corresponding orders in \eqref{eq:corresponding_orders}; as such, this principal symbol does not necessarily capture the leading-order components of $\bD$ as it operates on different summands of the varying structure. Instead, the order-reduced symbol $\sigma(\Pi\bD\Lambda)$ in \eqref{eq:principle_norm} serves as the immediate generalization:
\begin{definition}
\label{def:overdetermined_elliptic_varying}
A system $\bD$ is said to be \emph{overdetermined elliptic} with respect to basic tuples $(J_0, L_0; I_0, K_0)$ as in \eqref{eq:strict_tuples0} if the associated order-reduced symbol $\sigma(\Pi \bD \Lambda)$ is injective.
\end{definition}
Unlike overdetermined ellipticity for Green operators (\defref{def:overdetermined_elliptic_standard}), where the only requirement is the injectivity of the principal symbol—corresponding to the sharp order of the operator, which is uniquely determined—Douglis–Nirenberg systems can be overdetermined elliptic with respect to different order-reducing operators. That is, they can be overdetermined elliptic relative to multiple distinct sets of basic tuples $(J_0, L_0; I_0, K_0)$. 

This generalizes the concept of \emph{weight}-dependent ellipticity in the classical theory of systems of varying order \cite{DN55}. In \secref{sec:douglas_nirenberg_lop}, we will discuss the practical implications of this dependence for verifying overdetermined ellipticity in application. Until then, when referring to a system as overdetermined elliptic, and when there is no ambiguity, we will omit explicit reference to the specific basic tuples upon which the overdetermined ellipticity is based.


An immediate observation is that overdetermined ellipticity, as defined above, remains unaffected by the addition of lower-order terms. Specifically, suppose we can write:
\beq
\bD = \bD_0 + \mathfrak{K},
\label{eq:DplusC}
\eeq
where $\mathfrak{K}$ satisfies the lenient mapping property:
\beq
\mathfrak{K} : W^{J,L}_{2}(\bbE; \bbJ) \to W^{I,K}_{2}(\bbF; \bbG),
\label{eq:comapct_mapping_proprety}
\eeq
and its continuous extension as such is a compact map.

In this context, $\bD$ and $\bD_0$ clearly share the same standard tuples. Moreover, by \eqref{eq:principle_norm} and the compactness of $\mathfrak{K}$, we have:
\[
\sigma(\Pi \bD \Lambda) = \sigma(\Pi \bD_0 \Lambda).
\]
Hence, the contribution of $\mathfrak{K}$ to the overdetermined ellipticity of the system is negligible.

We record this fact for later reference:

\begin{proposition}
\label{prop:lower_order_correction_overdetermined_ellipticity}
Suppose a Douglis–Nirenberg system $\bD$ can be written in the form \eqref{eq:DplusC}, and that $\bD_0$ is overdetermined elliptic. Then $\bD$ is also overdetermined elliptic.
\end{proposition}

The following is the main theorem concerning overdetermined elliptic systems, obtained directly from the estimate \eqref{eq:overdetermined_sharp}, the composition rules \eqref{eq:comp_douglas_nirenberg_1}–\eqref{eq:comp_douglas_nirenberg_2}, the statement and discussion surrounding \propref{prop:overdetermined_tools_Green}, and the same argument used in \cite[Cor.~5.5, p.~336]{Gru90}:

\begin{theorem}
\label{thm:overdetermined_varying_orders}
Let $\bD$ be an overdetermined elliptic Douglis–Nirenberg system as defined in \eqref{def:overdetermined_elliptic_varying}, based on basic tuples $(J_0, L_0;\, I_0, K_0)$, which in turn are determined by corresponding orders as in \eqref{eq:corresponding_orders} and classes $(r^j)$. Then, for any standard tuples $(J, L;\, I, K)$ of $\bD$ as in \eqref{eq:strict_tuples}, the following hold:

\begin{enumerate}
\item For all $\Psi \in W^{J,L}_{p}(\bbE; \bbJ)$, there exists an a priori estimate:
\beq
\|\Psi\|_{J,L,p} \lesssim \|\bD\Psi\|_{I,K,p} + \|\frakI\Psi\|_{0,0,p}.
\label{eq:overdetermined_a_priori}
\eeq
Here, $\frakI \in \OP(-\infty, -\infty)$ is the $L^2$-orthogonal projection onto the finite-dimensional space $\ker \bD$, which is independent of the particular continuous extension of $\bD$ under any lenient mapping property \eqref{eq:douglas_nirenberg_mapping_2}. 

\item The mapping in \eqref{eq:strict_mapping_property} is semi-Fredholm, and $\ker \bD$ consists entirely of smooth sections.

\item If $\bD$ is injective, then it admits a left inverse within the calculus, which is continuous in the reverse direction of \eqref{eq:strict_mapping_property} for every $s > 1/p - 1$, and has corresponding classes
\beq
\tilde{r}^i = \underset{j,l}{\min} (r^j - m_i^j,\; t^l - q_i^l)
\label{eq:left_inverse_corresponding_classes}
\eeq
and basic tuples $(I_0,K_0;J_0,L_0)$. 
\item Conversely, if $\bD$ admits a left inverse within the calculus—associated with basic tuples $(I_0, K_0;\, J_0, L_0)$—then $\bD$ is injective and overdetermined elliptic with respect to $(J_0, L_0;\, I_0, K_0)$.
\end{enumerate}
\end{theorem}

\begin{proof}
Since $\Pi\bD\Lambda$ is overdetermined elliptic as a Green operator, \propref{prop:overdetermined_tools_Green} provides the estimate:
\[
\|\Psi\|_{\tilde{s},p} \lesssim \|(\Pi\bD\Lambda)\Psi\|_{\tilde{s},p} + \|\tilde{\frakI}\Psi\|_{0,p},
\]
where $\tilde{\frakI}$ denotes the projection onto $\ker(\Pi\bD\Lambda)$. Replacing $\Psi$ with $\Lambda^{-1}\Psi$ and applying the continuity of $\Lambda^{-1}$, we obtain:
\[
\|\Psi\|_{J,L,p} \lesssim \|\Lambda^{-1}\Psi\|_{\tilde{s},p} \lesssim \|(\Pi\bD)\Psi\|_{\tilde{s},p} + \|\tilde{\frakI}\Lambda^{-1}\Psi\|_{0,p}.
\]
Using the continuity of $\Pi$ and the fact that $\tilde{\frakI} \Lambda^{-1} : L^p(\bbE; \bbJ) \to W^{s_0,\, s_0 + 1/p}_p(\bbE; \bbJ)$ for any $s_0 < \min(J, L)$ sufficiently small, we deduce:
\[
\|\Psi\|_{J,L,p} \lesssim \|\bD\Psi\|_{I,K,p} + \|\Psi\|_{s_0,\, s_0 + 1/p,\, p}.
\]
Now, letting $\frakI$ be the projection onto $\ker \bD$, a standard compactness argument (since $W^{J,L}_p \hookrightarrow W^{s_0,\, s_0 + 1/p}_p$ compactly) yields:
\[
\|\Psi\|_{J,L,p} \lesssim \|\bD\Psi\|_{I,K,p} + \|\frakI\Psi\|_{0,0,p},
\]
which proves \eqref{eq:overdetermined_a_priori}. The fact that $\bD$ is semi-Fredholm then follows from a standard argument e.g.,~\cite[p.~30]{EE18}. 

If $\bD$ is injective, then so is $\Pi\bD\Lambda$, since it is the composition of an injective operator with invertible maps on both sides. By \propref{prop:overdetermined_tools_Green}, $\Pi\bD\Lambda$ admits a left inverse $\tilde{\bD}$ of order $0$ and class $0$. Then $\Lambda \tilde{\bD} \Pi$ is a left inverse for $\bD$, continuous in the reverse direction of \eqref{eq:strict_mapping_property}. The corresponding classes in \eqref{eq:left_inverse_corresponding_classes} follow either from the composition rules in \thmref{thm:composition_Douglis_Nirenberg}, or by applying \propref{prop:L2_continuity} componentwise to each entry in the matrix representation of the composed operator.

The converse follows by symmetry of the argument.
\end{proof}
We now wish to introduce a machinery that allows one to identify overdetermined ellipticity prior to determining the basic tuples upon which it is based. To that end, the concept of \emph{sharp} tuples will be useful:
\begin{definition}[Sharp tuples]
\label{def:sharp_tuples}
Let $(J_0, L_0;\, I_0, K_0)$ be basic tuples for a Douglis--Nirenberg system $\bD:\Gamma(\bbE;\bbJ) \rightarrow \Gamma(\bbF;\bbG)$, and let $\Lambda$, $\Pi$ be the associated order-reducing operators as defined in \eqref{eq:order_reducing_opertors}, so that $\Pi\bD\Lambda\in\OP(0,0)$. 

We say that $(J_0, L_0;\, I_0, K_0)$ are \emph{basic sharp tuples} if: 
\[
\Pi\bD\Lambda \notin \OP(m, r), \qquad \forall m, r < 0.
\]

In this case, the associated standard tuples as defined in \eqref{eq:strict_tuples0}, namely $(J, L;\, I, K)$, are called \emph{sharp tuples} for $\bD$. The corresponding mapping properties,
\[
\bD:W^{J,L}_p(\bbE;\bbJ) \rightarrow W^{I,K}_p(\bbF;\bbG),
\]
as given in \eqref{eq:strict_mapping_property}, are referred to as a \emph{sharp} mapping properties.
\end{definition}
Sharp tuples, by definition, are very convenient to work with, as they reduce the composition rules in \thmref{thm:composition_Douglis_Nirenberg} to a matter of mapping properties:
\begin{corollary}
\label{corr:sharp_composition}
If $\bD$ and $\bQ$ are Douglis–Nirenberg systems with sharp tuples $(J, L; I, K)$ and $(I, K; U, S)$, respectively, then the composition $\bQ \bD$ has sharp tuples $(J, L; U, S)$.
\end{corollary}
The reason we introduce sharp tuples only at this stage, and not earlier, is due to the following claim, which draws its relevance from the notion of overdetermined ellipticity. For the statement, recall that the basic tuples of a direct sum with a system in $\OP(0,0)$ are easily determined; see \propref{prop:G0tuples}.
\begin{proposition}
\label{prop:sharp_tuples_OD}
Let $\bD:\Gamma(\bbE;\bbJ)\rightarrow\Gamma(\bbE;\bbJ)$ be a Douglis--Nirenberg system with basic sharp tuples $(J_0, L_0;\, I_0, K_0)$, and suppose there exists $\tbP:\Gamma(\bbE;\bbJ)\rightarrow\Gamma(\bbE;\bbJ)$ such that:
\[
\bD \oplus \tbP
\]
is overdetermined elliptic with respect to basic tuples $(\tilde{J}_0, \tilde{L}_0;\, (\tilde{I}_0, \tilde{J}_0), (\tilde{K}_0, \tilde{L}_0))$. Then it is also overdetermined elliptic with respect to $(J_0, L_0;\, (I_0, J_0), (K_0, L_0))$.
\end{proposition}

\begin{proof}
It is clear that, without loss of generality, we may take $\tbP = 0$. Denote by $\tilde{\Lambda}$ and $\tilde{\Pi}$ the order-reducing operators associated with $(\tilde{J}_0, \tilde{L}_0;\, \tilde{I}_0, \tilde{K}_0)$, and by $\Lambda$ and $\Pi$ those associated with $(J_0, L_0;\, I_0, K_0)$. 

By the construction of order-reducing operators in \eqref{eq:order_reducing_opertors}, we can write:
\[
\tilde{\Lambda} = \Lambda \tilde{\Lambda}_0, \qquad \tilde{\Pi} = \tilde{\Pi}_0 \Pi
\]
for suitable order-reducing operators $\tilde{\Lambda}_0$ and $\tilde{\Pi}_0$ of the same form. If this were not the case, then clearly $
\Pi\bD\Lambda \in \OP(m,r)$
for some $m, r < 0$. Thus, by the homomorphism property of the symbol, the injectivity of $\sigma(\tilde{\Lambda} \bD \tilde{\Pi})$ is equivalent to that of $\sigma(\Lambda \bD \Pi)$.
\end{proof}
\subsubsection{Overdetermined ellipticity and semi-Fredholmness} 
To conclude this section, we complete the equivalence of overdetermined ellipticity with semi-Fredholmness, as discussed prior to \propref{prop:overdetermined_tools_Green}. Specifically, we show that the validity of an a priori estimate of the form \eqref{eq:overdetermined_a_priori}, with respect to some basic tuples $(J_0, L_0;\, I_0, K_0)$, is sufficient to guarantee that the system is overdetermined elliptic with respect to these tuples.

For ease of reference, we state this result for the case $p = 2$:
\begin{proposition}
\label{prop:estimate_overdetermined_sufficent_douglas}
A Douglis-Nirenberg system $\bD$ is overdetermined elliptic with respect to basic tuples $(J_0, L_0; I_0, K_0)$ if and only if there exists an estimate
\[
\|\Psi\|_{J_0, L_0 + 1/2, 2} \lesssim \|\bD \Psi\|_{I_0, K_0 + 1/2, 2} + \|\Psi\|_{s_0, s_0 + 1/2, 2},
\]
for some $s_0 < \min(J_0, L_0 + 1/2)$.
\end{proposition}
 
\begin{proof}
The first direction is established by \thmref{thm:overdetermined_varying_orders}. For the reverse direction, let $\Pi$ and $\Lambda$ be the order-reducing operators associated with the basic tuples $(J_0, L_0; I_0, K_0)$ as defined in \eqref{eq:order_reducing_opertors}. Then $\Pi \bD \Lambda \in \OP(0,0)$, and the estimate takes the form:
\[
\|\Psi\|_{0,0,2} \lesssim \|\Pi \bD \Lambda \Psi\|_{0,0,2} + \|\frakI \Psi\|_{0,0,2},
\]
where $\frakI$ is a compact operator. Thus, $\Pi \bD \Lambda: L^{2}(\bbE; \bbJ) \to L^{2}(\bbF; \bbG)$ has a closed range and finite-dimensional kernel, which implies that $(\Pi \bD \Lambda)^* \Pi \bD \Lambda: L^{2}(\bbE; \bbJ) \to L^{2}(\bbE; \bbJ)$ is Fredholm by the closed range theorem (cf. e.g., \cite[App.~A]{Tay11a}).

By \cite[Thm.~7, p.~197]{RS82}, it follows that $(\Pi \bD \Lambda)^* \Pi \bD \Lambda$ is an elliptic Green operator of order zero, hence its principle symbol $\sigma(\Pi \bD \Lambda)^* \Pi \bD \Lambda)$ is a bijection. However, due to \thmref{thm:full_Green_compositon_rules}: 
\[
\sigma((\Pi \bD \Lambda)^*\Pi \bD \Lambda)=\sigma(\Pi \bD \Lambda)^*\circ\sigma(\Pi \bD \Lambda)
\]
We conclude that $\sigma(\Pi \bD \Lambda)$ has a left inverse within $\mathrm{S}(0,0)$, hence it is injective. Consequently, $\bD$ is overdetermined elliptic as in  \defref{def:overdetermined_elliptic_varying}.
\end{proof}
\subsection{The weighted symbol} 
\label{sec:douglas_nirenberg_lop}
Although the injectivity of the order-reduced symbol $\sigma(\Lambda \bD \Pi)$ provides the most immediate analytical generalization of the injectivity of $\sigma(\bD)$ for $\bD$ a Green operator, this formulation depends on the particular choice of order-reducing operators $\Lambda$ and $\Pi$. In practical applications, we seek a criterion for overdetermined ellipticity that is independent of the choice of particular $\Lambda$,$\Pi$. Such a criterion should depend solely on the existence of basic tuples $(J_0, L_0;\, I_0, K_0)$ with respect to which $\bD$ is overdetermined elliptic.

Here, we carry this out using an approach that, as mentioned earlier, can be viewed as generalizing the machinery of ``weights'' from the classical theory of systems of varying orders \cite{DN55}. In that theory, determining the ellipticity of a system involves introducing ``weights''—analogous in our framework to the basic tuple–dependent order-reducing operators $\Lambda$ and $\Pi$. Once appropriate ``weights'' are identified, the classical theory observes that the actual verification of ellipticity ultimately remains independent of the specific choice of weights. For further details on this in the classical theory, see the discussion in \cite{Kha23} and the referenced (Russian) paper there \cite{Vol63}.

Let $\bD$ be a Douglis–Nirenberg system with matrix components as in \eqref{eq:douglas_nirenberg_components}, and with specified basic tuples $(J_0, L_0;\, I_0, K_0)$. Let the fibers of the vector bundles be expressed as in \eqref{eq:vector_bundles_fibers}. We consider the components as isolated Green operators, acting between Sobolev spaces as inherited from the mapping property \eqref{eq:strict_mapping_property} yielded by $(J_0, L_0;\, I_0, K_0)$:
\beq
\begin{split}
&E_{i}^{j} = 
\begin{pmatrix}
E_{i}^{j} & 0 \\ 
0 & 0 
\end{pmatrix} : W^{r^j, 2} \Gamma(\bbE_{j}) \to W^{\underset{j',l'}{\min}(r^{j'} - m_{i}^{j'}, t^{l'} - q_{i}^{l'}), 2} \Gamma(\bbF_{i}), \\
&T_{k}^{j} = 
\begin{pmatrix}
0 & 0 \\ 
T_{k}^{j} & 0 
\end{pmatrix} : W^{r^j, 2} \Gamma(\bbE_{j}) \to W^{\underset{j',l'}{\min}(r^{j'} - \tau_{k}^{j'}, t^{l'} - \sigma_{k}^{l'}) + 1/2, 2} \Gamma(\bbG_{k}), \\
&K^{l}_{i} = 
\begin{pmatrix}
0 & K^{l}_{i} \\ 
0 & 0 
\end{pmatrix} : W^{t^{l} + 1/2, 2} \Gamma(\bbJ_{l}) \to W^{\underset{j',l'}{\min}(r^{j'} - m_{i}^{j'}, t^{l'} - q_{i}^{l'}), 2} \Gamma(\bbF_{i}), \\
&Q_{k}^{l} = 
\begin{pmatrix}
0 & 0 \\ 
0 & Q_{k}^{l}
\end{pmatrix} : W^{t^{l} + 1/2, 2} \Gamma(\bbJ_{l}) \to W^{\underset{j',l'}{\min}(r^{j'} - \tau_{k}^{j'}, t^{l'} - \sigma_{k}^{l'}) + 1/2, 2} \Gamma(\bbG_{k}).
\end{split}
\label{eq:matrix_com_strict}
\eeq

Although some of these mappings may be compact, we consider the components as elements of:
\[
\begin{aligned}
&\OP(r^{j} - \underset{j',l'}{\min}(r^{j'} - m_{i}^{j'}, t^{l'} - q^{l'}_{i}), r^{j}), \quad && \OP(t^l - \underset{j',l'}{\min}(r^{j'} - m_{i}^{j'}, t^{l'} - q^{l'}_{i}), 0), \\
&\OP(r^{j} - \underset{j',l'}{\min}(r^{j'} - \tau_{k}^{j'}, t^{l'} - \sigma^{l'}_{k}), r^{j}), \quad && \OP(t^l - \underset{j',l'}{\min}(r^{j'} - \tau_{k}^{j'}, t^{l'} - \sigma^{l'}_{k}), 0).
\end{aligned}
\]
We now consider the isolated interior and boundary symbols of the above Green operators:  
\beq
\begin{aligned}
&\sigma_{M}(E_{i}^{j}), \, \sigma_{\dM}(E_{i}^{j}), \quad && \sigma_{\dM}(K_{i}^{l}), \\  
&\sigma_{\dM}(T_{k}^{j}), \quad && \sigma_{\dM}(Q_{k}^{l}),  
\end{aligned}
\label{eq:weighted_symbols_1}
\eeq
within the following symbol spaces:  
\beq
\begin{aligned}
&\mathrm{S}\Big(r^{j} - \underset{j',l'}{\min}(r^{j'} - m_{i}^{j'}, t^{l'} - q^{l'}_{i}), r^{j} \Big), \quad  
&& \mathrm{S}\Big(t^l - \underset{j',l'}{\min}(r^{j'} - m_{i}^{j'}, t^{l'} - q^{l'}_{i}), 0\Big), \\  
&\mathrm{S}\Big(r^{j} - \underset{j',l'}{\min}(r^{j'} - \tau_{k}^{j'}, t^{l'} - \sigma^{l'}_{k}), r^{j} \Big), \quad  
&& \mathrm{S}\Big(t^l - \underset{j',l'}{\min}(r^{j'} - \tau_{k}^{j'}, t^{l'} - \sigma^{l'}_{k}), 0\Big),  
\end{aligned}
\label{eq:weighted_symbols_spaces}
\eeq
even though some of these symbols may vanish.  

With this established, the symbols in \eqref{eq:weighted_symbols_1} are used to define two mappings: the \emph{weighted interior symbol}, which is the bundle map $\sigma_{M}(\bD): T^*M \otimes \bbE \to \bbF$, given for $x \in M$ and $\xi \in T^*_{x}M$ in matrix form as:
\beq
\sigma_{M}(\bD)(x,\xi) := (\sigma_{M}(E_{i}^{j})(x, \xi)) : \bbE_{x} \to \bbF_{x},
\label{eq:interior_symbol_systems}
\eeq
operating analogously to \eqref{eq:contraction_Green} on $\psi = (\psi_{j}) \in \bbE_{x}$ by contraction:
\[
(\sigma_{M}(\bD)(x,\xi)(\psi_{j}))_{i} = \sigma(E_{i}^{j})(x, \xi) \psi_{j}.
\]

The second mapping is the \emph{weighted boundary symbol}, which for $x \in \partial M$ and $\xi' \in T^*_{x} \partial M$, generalizes the boundary symbol of a Green operator \eqref{eq:boundary_symbol_Green}. It is a map:
\[
\sigma_\dM(\bD)(x, \xi'): 
\begin{matrix} 
\mathscr{S}(\overline{\mathbb{R}}_+; \mathbb{C} \otimes \mathbb{E}_x) \\\oplus\\ \mathbb{C} \otimes \mathbb{J}_{x} 
\end{matrix} 
\longrightarrow 
\begin{matrix} 
\mathscr{S}(\overline{\mathbb{R}}_+; \mathbb{C} \otimes \mathbb{F}_x) \\\oplus\\ \mathbb{C} \otimes \mathbb{G}_x 
\end{matrix},
\]
operating by contraction as: 
\beq
(\sigma_{\partial M}(\bD)(x,\xi')((\BRK{s\mapsto\psi_{j}(s)}); (\lambda_{j})))_{ik} = 
\begin{pmatrix} 
\{s \mapsto \sigma_{\dM}(E_{i}^{j})(x, \xi') \psi_{j}(s) + \sigma_{\dM}(K_{i}^{l})(x, \xi')\lambda_{k}(s)\} \\ 
\sigma_{\dM}(T_{k}^{j})(x, \xi') \psi_{j}(0) + \sigma_{\dM}(Q_{k}^{l})(x, \xi') \lambda_{l} 
\end{pmatrix}.
\label{eq:snythetic_boundary_symbol}
\eeq
\begin{definition}
\label{def:weighted_symbol}
The \emph{weighted symbol} of a system $\bD$, associated with basic tuples $(J_0, L_0;\, I_0, K_0)$ as in \eqref{eq:strict_tuples0}, is defined as the formal direct sum of the weighted interior symbol and the weighted boundary symbol of $\bD$:
\[
\sigma(\bD) = \sigma_{M}(\bD) \oplus \sigma_{\partial M}(\bD).
\]
We say that $\sigma(\bD)$ is \emph{injective} if both $\sigma_{M}(\bD)$ and $\sigma_{\partial M}(\bD)$ are injective. 
\end{definition}
Since it is defined term by term, the weighted symbol clearly inherits the properties listed in \thmref{thm:full_Green_compositon_rules}:
\begin{proposition}
Let $\bD$ and $\mathfrak{Q}$ be systems associated with basic tuples $(J_0, L_0; I_0, K_0)$ and $(I_0, K_0; U_0, R_0)$, respectively. Then the following properties hold:
\begin{enumerate}
\item The weighted symbol of the composition $\mathfrak{Q}\bD$ with respect to $(J_0, L_0,; U_0, R_0)$ decomposes as:
\[
\sigma(\mathfrak{Q} \bD) = \sigma(\mathfrak{Q}) \circ \sigma(\bD) = (\sigma_M(\mathfrak{Q}) \circ \sigma_M(\bD)) \oplus (\sigma_\dM(\mathfrak{Q}) \circ \sigma_\dM(\bD)).
\]

\item If $\bD_0$ and $\mathfrak{K}$ are systems as in \eqref{eq:DplusC}, then the weighted symbol satisfies:
\[
\sigma(\bD_0+ \mathfrak{K}) = \sigma(\bD_0).
\]

\item If the corresponding classes of $\bD$ are all zero, and $\bD^*$ is its adjoint, then:
\[
\sigma(\bD^*) = \sigma(\bD)^*.
\]
\end{enumerate}
\end{proposition}
The following theorem shows that the weighted symbol is more than just a formal object:
\begin{theorem}
\label{thm:overdetermined_synetheic}
A system $\bD$ with basic tuples $(J_0, L_0;\, I_0, K_0)$ has an injective order-reduced symbol $\sigma(\Pi\bD\Lambda)$ if and only if the associated weighted symbol $\sigma(\bD)$ is injective. 
\end{theorem}

\begin{proof}
We explicitly demonstrate only the equivalence between the injectivity of the interior symbol $\sigma_{M}(\Pi \bD \Lambda)$ and the weighted interior symbol $\sigma_{M}(\bD)(x, \xi)$. The argument for the boundary symbol follows analogously, though it involves more careful bookkeeping due to the larger number of terms.

First, expanding the operations of $\Pi$, $\Lambda$, and $\bD$, we have:
\[
(\sigma_{M}(\Pi\bD\Lambda)(x,\xi))_{i} = 
\sum_{j} \sigma\left(
\calL_{\bbF_{i}}^{\underset{j',l'}{\min}(r^{j'} - m_{i}^{j'},\, t^{l'} - q_{i}^{l'})} 
\, E_{i}^{j} \, 
\calL_{\bbE_{j}}^{-r^{j}}
\right)(x,\xi),
\]
where, for clarity, we explicitly abandon the Einstein summation convention and avoid the use of deltas in the definitions of $\Pi$ and $\Lambda$ from \eqref{eq:order_reducing_opertors}.

Using the homomorphism property in \thmref{thm:full_Green_compositon_rules}, each summand can be rewritten as:
\[
\begin{split}
\sigma\left(
\calL_{\bbF_{i}}^{\underset{j',l'}{\min}(r^{j'} - m_{i}^{j'},\, t^{l'} - q_{i}^{l'})} 
\, E_{i}^{j} \, 
\calL_{\bbE_{j}}^{-r^{j}}
\right)(x,\xi) 
&= 
\sigma\left(\calL_{\bbF_{i}}^{\underset{j',l'}{\min}(r^{j'} - m_{i}^{j'},\, t^{l'} - q_{i}^{l'})}\right)(x,\xi)
\circ 
\sigma(E_{i}^{j})(x,\xi) \\
&\qquad \circ 
\sigma\left(\calL_{\bbE_{j}}^{-r^{j}}\right)(x,\xi),
\end{split}
\]
where equality holds in the class $\mathrm{S}(0,0)$. By linearity, this yields:
\[
(\sigma_{M}(\Pi\bD\Lambda)(x,\xi))_{i} = 
\sigma\left(\calL_{\bbF_{i}}^{\underset{j',l'}{\min}(r^{j'} - m_{i}^{j'},\, t^{l'} - q_{i}^{l'})}\right)(x,\xi)
\circ 
\left(\sum_{j} \sigma(E_{i}^{j})(x,\xi) \circ \sigma(\calL_{\bbE_{j}}^{-r^{j}})(x,\xi)\right).
\]

Since the order-reducing operators are elliptic and invertible, their symbols are isomorphisms. Now evaluate both sides on $\psi = (\psi_{j}) \in \bbE_{x}$, with 
\[
\psi_{j} := (\calL_{\bbE_{j}}^{-r^{j}}(x,\xi))^{-1} \tilde{\psi}_{j},
\]
and compose from the left with 
\[
\sigma\left(\calL_{\bbF_{i}}^{\underset{j',l'}{\min}(r^{j'} - m_{i}^{j'},\, t^{l'} - q_{i}^{l'})}\right)(x,\xi)^{-1}.
\]
We then find:
\[
\begin{split}
\sigma\left(\calL_{\bbF_{i}}^{\underset{j',l'}{\min}(r^{j'} - m_{i}^{j'},\, t^{l'} - q_{i}^{l'})}\right)^{-1} 
\circ 
\sigma_{M}(\Pi\bD\Lambda)(x,\xi)
((\calL_{\bbE_{j}}^{-r^{j}})^{-1} \tilde{\psi}_{j})
&= 
\sum_{j} \sigma(E_{i}^{j})(x,\xi) \tilde{\psi}_{j} \\
&= 
(\sigma_{M}(\bD)(\psi_{j}))_{i}.
\end{split}
\]

Since this holds for every $i$, and since the symbols of the order-reducing operators are isomorphisms, the equivalence between the injectivity of $\sigma_{M}(\Pi\bD\Lambda)(x,\xi)$ and $\sigma_{M}(\bD)(x,\xi)$ is established.
\end{proof}

In the same spirit, we next generalize the classical Lopatinski-Shapiro condition, which was formulated for Green operators of the form \eqref{eq:Douglis_Nirenberg_Adapted} in \propref{prop:rud_lop}. We extend this criterion to Douglis-Nirenberg systems of the form $\bD = \bD_{0} + \mathfrak{K}$ as in \eqref{eq:DplusC}, where we assume:
\[
(\bD_0)_{ik}^{jl} = \begin{pmatrix} E^{j}_{i} & 0 \\ T_{k}^{j} & Q_{k}^{l} \end{pmatrix},
\label{eq:Douglis_Nirenberg_Adapted2}
\]
and $E^{j}_{i}$, $T_{k}^{j}$, and $Q_{k}^{l}$ are all \emph{differential operators} belonging to their respective classes. 

We note that under these assumptions, the weighted interior symbol $\sigma_{M}(\bD)$ does not change from \defref{def:weighted_symbol}, and is nothing but the formal direct sum of the symbols of the differential operators $E_{i}^{j}$.  

For interpreting the condition on the injectivity of the boundary symbol, let $x \in \partial M$ and $\xi' \in T^*_{x} \partial M$, and generalize \eqref{eq:boundary_symbol_E1} to systems by defining: 
\beq
\sigma(E)(x, \xi' + \iota\partial_s\,dr) = (\sigma(E_{i}^{j})(x, \xi' + \iota\partial_s\,dr)) : C^{\infty}(\overline{\mathbb{R}}_+; \mathbb{C} \otimes \bbE_{x}) \to C^{\infty}(\overline{\mathbb{R}}_+; \mathbb{C} \otimes \bbF_{x}),
\label{eq:boundary_symbol_E2}
\eeq
operating on functions $s \mapsto \psi(s) = (\psi_{j}(s))$ as:
\[
(\sigma(E)(x, \xi')(\psi_{j}))_{i} = \sigma(E_{i}^{j})(x, \xi' + \iota\partial_s\,dr)\psi_{j}
\]
where each $\sigma(E_{i}^{j})(x, \xi' + \iota\partial_s\,dr)$ is as defined in \eqref{eq:boundary_symbol_E2}. Define similarly:  
\beq
\begin{split}
\sigma(T)(x,\xi) &= \sigma(T_{k}^{j})(x, \xi' + \iota\partial_s\,dr), \\  
\sigma(Q)(x,\xi) &= \sigma(Q_{k}^{l})(x,\xi'),  
\end{split}
\eeq
where each weighted symbol $\sigma(T_{k}^{j})(x, \xi' + \iota\partial_s\,dr)$ and $\sigma(Q_{k}^{l})(x, \xi')$ are defined as in \eqref{eq:Xi_map1}. This is well-defined since all operators are differential.  

Finally, we generalize the initial condition map \eqref{eq:Xi_map1} to the \emph{weighted initial condition map}:
\beq
\begin{split}
\Xi_{x, \xi'} \begin{pmatrix} 
\sigma(T_{k}^{j})(x, \xi' + \iota\partial_s\,dr)|_{s=0} & \sigma(Q_{k}^{l})(x, \xi') 
\end{pmatrix} : 
\begin{matrix} 
C^{\infty}(\overline{\mathbb{R}}_+; \mathbb{C} \otimes \bbE_{x}) \\ 
\oplus \\ 
\mathbb{C} \otimes \bbJ_{x} 
\end{matrix} \to \mathbb{C} \otimes \bbG_{x}.
\end{split}
\label{eq:Xi_map2}
\eeq
This operates as:
\[
\Xi_{x, \xi'}((\psi_{j}); (\lambda_{l}))_{k} = \sigma(T_{k}^{j})(x, \xi' + \iota\partial_s\,dr)\psi_{j}(0) + \sigma(Q_{k}^{l})(x, \xi')\lambda.
\]
By the very construction of the weighted boundary symbol \defref{def:weighted_symbol}, we find that \propref{prop:rud_lop} generalizes immediately into:  
\begin{theorem}
\label{thm:lopatnskii_shapiro}
Given $\bD$ as in \eqref{eq:Douglis_Nirenberg_Adapted}, with an injective interior weighted symbol, let $x \in \dM$ and $\xi' \in T_x^*\partial M \setminus \{0\}$. Let $\bbM_{x,\xi'}^+ \subset C^{\infty}(\overline{\bbR}_+; \bbC \otimes \bbE_x)$ denote the space of decaying solutions of the linear $\bbC \otimes \bbE_x$-valued ordinary differential equation:
\beq
\sigma(E)(x, \xi' + \iota\partial_s\,dr)\psi(s) = 0,
\label{eq:overdeterminedE_lop}
\eeq
where the operator on the left-hand side is as defined in \eqref{eq:boundary_symbol_E2}. Then, the weighted symbol $\sigma(\bD)$ is injective if and only if the restriction of the weighted initial condition map:
\[
\Xi_{x,\xi} : 
\begin{matrix}
\bbM_{x,\xi'}^+ \\
\oplus \\
\bbC \otimes \bbJ_{x}
\end{matrix}
\longrightarrow \bbC \otimes \bbG_{x},
\]
is injective for every $x \in \dM$ and $\xi' \in T_x^*\partial M \setminus \{0\}$.
\end{theorem}

\chapter{Elliptic Pre-Complexes}
\label{chp:elliptic_pre_complexes}
In this chapter we shall prove all results presented in \secref{sec:Hodge_intro}, and develop the necessary notions to do so. 

\section{Adapted Green systems}
\label{sec:Adapted}
\subsection{Setting and basic constructions}
The following definition serves as a fundamental building block of our theory. As such, we aim to formulate it in the highest generality, incorporating several ingredients that, in practical applications, often simplify, become self-implied, or prove trivial. We are aware that, viewed in isolation, some of these components may appear excessive; nonetheless, they are included here for completeness and to ensure the theory remains sufficiently broad for future applications.

For the statement, recall the definition of sharp tuples from \defref{def:sharp_tuples}. 

\begin{definition}[Adapted Green system, adapted adjoint]
\label{def:adapting_operator}
Let $\bD : \Gamma(\bbE; \bbJ) \to \Gamma(\bbF; \bbG)$ and $\bD^* : \Gamma(\bbF; \bbG) \to \Gamma(\bbE; \bbJ)$ be Douglis--Nirenberg systems.  
We call $\bD$ a \emph{one-sided adapted Green system} and $\bD^*$ its \emph{adapted adjoint} if the following are specified:
\begin{itemize}
\item Sharp tuples for $\bD$ of the form $(M, T;\, 0, 0)$;
\item Sharp tuples for $\bD^*$ of the form $(M', T';\, 0, 0)$;
\item Normal systems of boundary operators (cf.~\defref{def:normal_system_of_trace_operators2}):
\[
\bB : \Gamma(\bbE; \bbJ) \to \Gamma(0; \bbL), \quad 
\bB^* : \Gamma(\bbF; \bbG) \to \Gamma(0; \bbL),
\]
where $\bB$ has sharp tuples of the form $(M, T;\, 0, C)$, and $\bB^*$ has sharp tuples of the form $(M', T';\, 0, C')$;
\item A formula valid for all $\Upsilon \in \Gamma(\bbE; \bbJ)$ and $\Theta \in \Gamma(\bbF; \bbG)$:
\beq 
\langle \bD \Upsilon, \Theta \rangle
= \langle \Upsilon, \bD^* \Theta \rangle
+ \langle \bB \Upsilon, \bB^* \Theta \rangle.
\label{eq:integration_adapted_refined}
\eeq
\end{itemize}
We call $\bD$ an \emph{adapted Green system} if, in addition, $\bD^*$ has sharp tuples of the form $(0, 0;\, -M, -T)$.
\end{definition}

The original notion of an adapted Green operator of order $m$ (and class zero), introduced in \cite[Sec.~3.1]{KL23}, can be recovered from this definition by taking (in the notation there)
\beq 
\bD=\begin{pmatrix}\bA & 0 \\ 0 & 0 \end{pmatrix}, \qquad 
\bD^*=\begin{pmatrix}\bA^* & 0 \\ 0 & 0 \end{pmatrix}, \qquad 
\bB=\begin{pmatrix}0 & 0 \\ B_{A} & 0 \end{pmatrix}, \qquad 
\bB^*=\begin{pmatrix}0 & 0 \\ B^*_{A} & 0 \end{pmatrix}. 
\label{eq:original_adapted}  
\eeq 
where $B_{A},B^*_{A}$ are assumed to be normal in the sense of \defref{def:normal_system_of_trace_operators1}. 

The definition extends the notion of an adapted Green operator in several ways: 

First, as a full Douglis--Nirenberg system, $\bD$ may include components of varying orders and classes that are not confined to the upper-left corner.

Second, as reflected in \defref{def:normal_system_of_trace_operators2}, the boundary terms $\bB$ and $\bB^*$ may be more general and may include both trace and pseudodifferential components mapping sections over the boundary to sections over the boundary.

Third, the condition on the sharp tuples of $\bD$ is intended to generalize the original requirement of uniform order $m$ and class $0$ to the greatest extent possible. Indeed, for systems of the form \eqref{eq:original_adapted}, the required sharp tuples for $\bD$ in \eqref{def:adapting_operator} are trivially $(m, 0;\, 0, 0)$ and $(2m,0;\,m,0)$, as $\bbJ = 0$. The significance of this requirement will become apparent in the proof of the main results concerning elliptic pre-complexes, for which adapted Green systems serve as the basic building blocks, as exposed in \secref{sec:elliptic_pre_complexes_intro}.


Fourth, note that if $\bD$ has vanishing corresponding classes, it may admit an adjoint in the calculus that differs from its adapted adjoint, since the latter depends on a specified boundary system $\bB$, $\bB^*$, which may not coincide with the boundary terms in Green's formula. Thus, defining an adapted Green operator requires a fixed choice of $\bB$, $\bB^*$, and $\bD^*$.

%

We now introduce several constructions pertinent to adapted Green systems.  

When referring to specific sets of sharp or lenient tuples associated with adapted Green systems and their corresponding operators, we will use the following terminology:
\begin{definition}
\label{def:suitable}
Let $\bD:\Gamma(\bbE;\bbJ)\rightarrow\Gamma(\bbF;\bbG)$ be a system and let $\bB:\Gamma(\bbE;\bbJ)\rightarrow\Gamma(0;\bbL)$ be a system of boundary operators.  
Lenient (respectively, sharp) tuples $(S, T; S', T')$ for $\bD$ are called \emph{suitable} for $\bB$ if there exists a tuple $T''$ such that $(S, T;\, 0, T'')$ are lenient (respectively, sharp) tuples for $\bB$. 
\end{definition}
Lenient tuples for $\bD$ that are also suitable for $\bB$ give rise to lenient mapping properties \eqref{eq:douglas_nirenberg_mapping_1} and \eqref{eq:douglas_nirenberg_mapping_2} for both $\bD$ and $\bB$.   


For any lenient tuples $(S, T;\, S', T')$ for $\bD$, we consider the range of the linear map 
\[
\bD: W^{S, T}_{p}(\bbE; \bbJ) \to W^{S', T'}_{p}(\bbF; \bbG)
\]
as a subspace of $W^{S', T'}_{p}(\bbF; \bbG)$, denoted by
\beq
\scrR^{S', T'}_p(\bD) = \bD(W^{S, T}_{p}(\bbE; \bbJ)).
\label{eq:range_adapted_Green_system}
\eeq
Let $\scrR(\bD)$ denote the smooth version,
\[
\scrR(\bD) = \bD(\Gamma(\bbE; \bbJ)).
\]
Similarly, define the spaces
\beq
\begin{aligned}
&\scrR^{S', T'}_{p}(\bD; \bB)     
=\ \bD\big(W^{S, T}_{p}(\bbE; \bbJ) \cap \ker \bB\big), 
&&\quad 
\scrR(\bD; \bB)     
=\ \bD\big(\Gamma(\bbE; \bbJ) \cap \ker \bB\big).
\end{aligned}
\label{eq:RDDirichlet}
\eeq

These space are not always well defined, as there are multiple tuples $S, T$ for which $(S, T; S', T')$ are lenient tuples for $\bD$ that are suitable for $\bB$. However, by considering the closure of these subspaces in the $W^{S', T'}_p$-topology:
\beq
\begin{split}
&\overline{\scrR^{S', T'}_p(\bD)} = \overline{\bD(W^{S, T}_{p}(\bbE; \bbJ))} \subseteq W^{S', T'}_p(\bbF; \bbG), \\
&\overline{\scrR^{S', T'}_p(\bD; \bB)} = \overline{\bD(W^{S, T}_{p}(\bbE; \bbJ) \cap \ker \bB)} \subseteq W^{S', T'}_p(\bbF; \bbG).
\end{split}
\label{eq:range_green_closure}
\eeq 
we obtain well-defined notions regardless of the choice of $S, T$ for which $\bD: W^{S, T}_{p}(\bbE; \bbJ) \rightarrow W^{S', T'}_p(\bbF; \bbG)$, $\bB: W^{S, T}_{p}(\bbE; \bbJ) \rightarrow W^{0, T''}_p(0; \bbL)$ continuously. 
\begin{proposition}
\label{prop:RD}
Let $(S, T; S', T')$ be any lenient tuples for a one-sided adapted Green system $\bD$ that are also suitable for $\bB$ as defined above. Then the following holds:
\[
\begin{split}
&\overline{\bD(W^{S, T}_{p}(\bbE; \bbJ))} = \overline{\scrR(\bD)}, \\
&\overline{\bD(W^{S, T}_{p}(\bbE; \bbJ) \cap \ker \bB)} = \overline{\scrR(\bD; \bB)},
\end{split} 
\]
where the closure is taken with respect to the $W^{S', T'}_p$-topology.
\end{proposition}
The proof is technical and straightforward, and is deferred to the end of the section.

We conclude that the spaces defined in \eqref{eq:range_adapted_Green_system} are well-defined precisely when the range of the continuous map 
\[
\bD: W^{S, T}_{p}(\bbE; \bbJ) \to W^{S', T'}_{p}(\bbF; \bbG),
\]
or, respectively, of 
\[
\begin{split} 
&\bD: W^{S, T}_{p}(\bbE; \bbJ) \cap \ker \bB \to W^{S', T'}_{p}(\bbF; \bbG)
\end{split} 
\]
is closed.

\propref{prop:RD} also shows that:
\[
\begin{gathered} 
\overline{\scrR_p^{S, T}(\bD)} = \overline{\bD(\Gamma(\bbE; \bbJ))}, \quad \overline{\scrR_p^{S, T}(\bD; \bB)} = \overline{\bD(\Gamma(\bbE; \bbJ) \cap \ker \bB)}
\end{gathered} 
\]
where the closure is taken with respect to the $W_p^{S, T}$-topology.

For the next definition, recall that the closed range theorem asserts that for a bounded linear map $T: V \to W$ between Banach spaces, $\ker T' = (T(V))^\bot$, where $\bot$ denotes the Banach annihilator functor \cite[p. 575]{Tay11a}. Then, as provided by \corrref{corr:weak_tuples}, there exist $S, T$ such that $\bD: W^{S, T}_{p}(\bbE; \bbJ) \rightarrow L^{p}(\bbF; \bbG)$ continuously, which makes it possible to define the following subspaces of $L^{p}(\bbF; \bbG)$:
\beq
\scrN^{0, 0}_p(\bD^*, \bB^*) = \overline{\scrR^{0, 0}_q(\bD)}^\bot, \qquad \scrN^{0, 0}_p(\bD^*) = \overline{\scrR^{0, 0}_q(\bD; \frakB)}^\bot,
\label{eq:Lp_ann_0}
\eeq 
where $1/p + 1/q = 1$, thus allowing interpretation of the annihilator of $\overline{\scrR^{0, 0}_q(\bD)}$ as a subspace of $L^p$-sections by invoking the $L^q$–$L^p$ duality. As the annihilators of closed subspaces, both $\scrN^{0, 0}_{p}(\bD^*)$ and $\scrN^{0, 0}_{p}(\bD^*, \bB^*)$ are closed subspaces of $L^{p}(\bbF; \bbG)$. Due to \propref{prop:RD} and the fact that $(T(V))^{\bot\bot} = \overline{T(V)}$, we find that the closures in the above definitions are redundant, yielding:
\beq
\scrN^{0, 0}_p(\bD^*, \bB^*) = \scrR(\bD)^\bot, \qquad \scrN^{0, 0}_p(\bD^*) = \scrR(\bD; \frakB)^\bot,
\label{eq:Lp_ann}
\eeq 

The generalized Green's formula \eqref{eq:integration_adapted_refined} then allows an explicit description of the spaces $\scrN(\bD^*,\bB^*)$, $\scrN(\bD^*)$ and their Sobolev versions. 
\begin{proposition}
\label{prop:NDR}
Let $\bD$ be a one-sided adapted Green system. If $(S,T;S',T')$ are lenient tuples for $\bD^*$, and $\Theta \in W^{S,T}_p(\bbF;\bbG)$, then $\Theta \in \scrN^{0,0}_p(\bD^*)$ if and only if  
\beq
\bD^* \Theta = 0.  
\label{eq:AWQ_relation}
\eeq
Moreover, if $(S,T;S',T')$ are also suitable for $\bB^*$, then $\Theta \in \scrN^{0,0}_{p}(\bD^*,\bB^*)$ if and only if, in addition to \eqref{eq:AWQ_relation}, it holds that  
\beq
\bB^* \Theta = 0.  
\label{eq:Dirichlet_conditions_D}
\eeq
\end{proposition}
The proof emphasizes the demand that $\bB$ is a normal system of boundary operators, so we include it here: 
\begin{proof}
In view of \eqref{eq:Lp_ann}, it holds that $\scrN^{0,0}_{p}(\bD^*)=\scrR(\bD;\frakB)^{\bot}$. Thus, due to $L^{p}$-$L^{q}$ duality, the statement $\Theta\in\scrN^{0,0}_{p}(\bD^*)$ is equivalent to that for all $\bD\Upsilon\in \bD(\Gamma(\bbE;\bbJ)\cap\ker\frakB)$ 
\[
\bra \bD\Upsilon,\Theta\ket=0. 
\]
Comparing with \eqref{eq:integration_adapted_refined}, taking $\bB\Upsilon=0$ and since $\bD^*:W^{S,T}_{p}(\bbE;\bbJ)\rightarrow W^{S',T'}_{p}(\bbF;\bbG)$ continuously
\[
\bra \Upsilon,\bD^{*}\Theta\ket=0 
\]
The density of $\ker\bB$ in $L^{2}(\bbE;\bbJ)$ then provides that $\bD^{*}\Theta=0$.  

For the second statement, note that $\scrN^{0,0}_{p}(\bD^*,\frakB^*)\subseteq\scrN^{0,0}_{p}(\bD^*)$. Then, if $\Theta$ has sufficient regularity as in the statement, then $\bB^*$ is defined and so combining $\bD^*\Theta=0$ with \eqref{eq:integration_adapted_refined} yields
\[
\bra \bB\Upsilon,\bB^*\Theta\ket =0.
\]
Since $\bB$ is surjective, for an arbitrary $\tilde{\Upsilon}$ on the boundary it is possible to prescribe $\bB\Upsilon=\tilde{\Upsilon}$. Thus, $\bra \tilde{\Upsilon},\bB^*\Theta\ket=0$ for arbitrary $\tilde{\Upsilon}$, hence $\bB^*\Theta=0$. 

The other direction of the claim is clear by retracing the argument. 
\end{proof}

The following is the first analytical point where we refer to the specified sharp tuples from \defref{def:adapting_operator}. It essentially states that, under suitable assumptions, distributional mapping properties hold for the adapted adjoint and its associated boundary system $\bB^*$.

\begin{lemma}
\label{lem:D'mapping}
Let $\bD : \Gamma(\bbE; \bbJ) \rightarrow \Gamma(\bbF; \bbG)$ be an adapted Green system with sharp tuples $(M, T; 0, 0)$, as in \defref{def:adapting_operator}. Let $\Theta \in L^{p}(\bbF; \bbG)$, and let $(\Theta_n) \subset \Gamma(\bbF; \bbG)$ be an approximating $L^p$-sequence for $\Theta$. Let $(I, K; II, KK)$ be standard tuples for $\bD^*$. Suppose that there exists $\Xi \in W^{I, K}_p(\bbF; \bbG)$ such that
\[
\Theta - \Xi \in \scrN^{0, 0}_{p}(\bD^*).
\]
Then,
\beq
\lim_{n\rightarrow\infty}\bD^* \Theta_n =\bD^*\Xi \qquad \text{in} \; W^{-M,-T}_{p}(\bbE;\bbJ)
\label{eq:adated_uniform_boundness}
\eeq
If, in addition, there exists a tuple $K'$ such that $(M, T; 0, K')$ are standard tuples for $\bB^*$ and $\Theta - \Xi \in \scrN_{p}^{0, 0}(\bD^*, \bB^*)$, then:
\beq
\lim_{n\rightarrow\infty} \bB^* \Theta_n=\bB^*\Xi  \qquad \text{in} \; W^{0,-C}_{p}(0;\bbL)
\label{eq:adated_uniform_boundnessII}
\eeq
\end{lemma}
In the context of elliptic pre-complexes, these properties will later be refined further once combined with overdetermined ellipticity conditions on systems incorporating $\bD$. The proof is technical and is deferred to the end of the section.  
\subsection{Auxiliary decompositions} 
Given lenient tuples $(S, T; S, T')$ for $\bD^*$ that are also suitable for $\frakB^*$, \propref{prop:NDR} shows that the space $\scrN^{S, T}_{p}(\bD^*)$ coincides with the kernel of the continuous map $\bD^* : W^{S, T}_p \Gamma(\bbF; \bbG) \rightarrow W^{S', T'}_p \Gamma(\bbE; \bbJ)$, which implies that $\scrN^{S, T}_p(\bD^*)$ is a closed subspace in the corresponding Banach topology. The same holds for $\scrN^{S, T}_p(\bD^*, \bB^*)$. By similar reasoning, $\scrN(\bD^*, \bB^*)$ and $\scrN(\bD^*)$ are closed subspaces in the Fréchet topology, as they coincide with the kernel of a continuous map between Fréchet spaces. 

However, the subspaces $\scrR(\bD)$, $\scrR(\bD; \frakB)$, and their Sobolev versions are not necessarily closed without further assumptions on $\bD$. Recall that an algebraically direct decomposition of a Hilbert, Banach, or Fréchet space is topologically direct if and only if both subspaces in the decomposition are closed \cite[Ch.~2]{Bre11}. Unlike in Hilbert spaces, in Fréchet and Banach spaces a closed subspace may fail to induce a direct decomposition. 

In general, the following proposition provides the best one can expect:

\begin{proposition}
\label{prop:L2_aux_adapted_pair}
Let $\bD$ be a one-sided adapted Green system. There exist $L^{2}$-orthogonal, topologically-direct decompositions
\beq
\begin{split}
&L^{2}(\bbF; \bbG) = \overline{\scrR^{0, 0}_2(\bD)} \oplus \scrN^{0, 0}_2(\bD^*, \bB^*), \\
& L^{2}(\bbF; \bbG) = \overline{\scrR^{0, 0}_2(\bD; \frakB)} \oplus \scrN^{0, 0}_2(\bD^*),
\end{split}
\label{eq:L2_decom}
\eeq
where the overline denotes closure in the $L^2$-norm. 
\end{proposition}

\begin{proof}
Only the first statement is proven, as the second is completely analogous. By the isomorphism $L^2 \simeq (L^2)^*$, the Banach annihilator of a subspace coincides with its orthogonal complement. Thus, \eqref{eq:Lp_ann_0} gives
\[
\scrN^{0, 0}_2(\bD^*, \bB^*) = (\overline{\scrR^{0, 0}_2(\bD)})^\bot, \qquad \scrN^{0, 0}_2(\bD^*) = (\overline{\scrR^{0, 0}_2(\bD; \bB)})^\bot. 
\]
Since $\scrN^{0, 0}_2(\bD^*, \bB^*)$ is closed, and every closed subspace of a Hilbert space induces an orthogonal decomposition, \eqref{eq:L2_decom} holds.  
\end{proof}
A closed subspace yields a topologically direct decomposition if and only if it admits a continuous projection $\tbP$ onto it. When the range of a continuous map $\bD$ is closed and induces a topologically direct decomposition, and its kernel does so as well, a routine application of the open mapping theorem shows that the projection $\tbP$ onto the range of $\bD$ yields a continuous map $\bP$ satisfying:
\[
\tbP = \bD \bP.
\]
We aim to make these observations systematic within the framework of adapted Green systems. Considering the low-regularity decompositions \eqref{eq:L2_decom}, which are associated with any adapted Green system, we expect a projection in this setting, if it belongs to the calculus, to also belong to $\OP(0,0)$ due to its $L^{2} \to L^{2}$ mapping property (by applying \propref{prop:G0_criteria} for $S',T'=0$ and $m=0$). 

This motivates the following definitions, which rely on the notion of a \emph{balance}, introduced to facilitate the comparison of orders and classes between adapted Green systems:

\begin{definition}[Balance]
\label{def:balance}
Let $\bD: \Gamma(\bbE; \bbJ) \to \Gamma(\bbF; \bbG)$ be a one-sided adapted Green system. A \emph{balance} for $\bD$ is a system $\bP: \Gamma(\bbF; \bbG) \to \Gamma(\bbE; \bbJ)$ with lenient tuples $(0,0;M,T)$, i.e., satisfying the lenient mapping property
\[
\bP: W^{0,0}_{p}(\bbF; \bbG) \to W^{M,T}_{p}(\bbE; \bbJ).
\]
In addition, we say that $\bP$ is a \emph{balance with respect to $\bB$} if $\bP$ is a balance for $\bD$ and $\bB \bP = 0$.  
\end{definition}

By \propref{prop:G0_criteria}, the lenient mapping property of a balance implies that $\bD \bP \in \OP(0,0)$. This fact is what justifies the term, as by composition $\bP$ literally reduces both the corresponding orders and classes of $\bD$ to zero from the right. The reason a balance is not defined simply as any system in the calculus satisfying $\bD \bP \in \OP(0,0)$ is that such a condition could hold due to incidental ``cancellations" in the action of $\bD$ on $\bP$, which do not reflect genuine order and class balancing.  

\begin{definition}[Neumann auxiliary decomposition]
\label{def:aux_decomposition}
Let $\bD$ be an adapted Green system as in \defref{def:adapting_operator}. It is said that $\bD$ induces a  \emph{Neumann auxiliary decomposition} if the following holds: 
\begin{enumerate}[itemsep=0pt,label=(\alph*)]
\item There is a topologically-direct, $L^{2}$-orthogonal decomposition of Fréchet spaces:
\beq
\Gamma(\bbF;\bbG)=\scrR(\bD)\oplus \scrN(\bD^*,\bB^*).
\label{eq:aux_smooth}
\eeq
\item  The $L^{2}$-orthogonal projection onto $\scrR(\bD)$ in the above decomposition, denoted by $\tbP:\Gamma(\bbF;\bbG)\rightarrow \Gamma(\bbF;\bbG)$, is within the calculus and satisfies $\tbP\in\OP(0,0)$.  
\item   
There exists a balance $\bP:\Gamma(\bbF;\bbG)\rightarrow\Gamma(\bbE;\bbJ)$ for $\bD$ such that 
\[
\tbP=\bD\bP.
\]
\end{enumerate}
\end{definition}
The following is the \emph{Dirichlet} analogue, so named because the boundary condition shifts from the kernel of $\bD^*$ to the domain of $\bD$:
\begin{definition}[Dirichlet auxiliary decomposition]
\label{def:aux_decompositionD}
Let $\bD$ be an adapted Green system as in \defref{def:adapting_operator}. It is said that $\bD$ induces a  \emph{Dirichlet auxiliary decomposition} if the following holds: 
\begin{enumerate}[itemsep=0pt,label=(\alph*)]
\item There is a topologically-direct, $L^{2}$-orthogonal decomposition of Fréchet spaces:
\beq
\Gamma(\bbF;\bbG)=\scrR(\bD;\frakB)\oplus \scrN(\bD^*).
\label{eq:aux_smoothD}
\eeq
\item  The $L^{2}$-orthogonal projection onto $\scrR(\bD;\frakB)$ in the above decomposition, denoted by $\tbP:\Gamma(\bbF;\bbG)\rightarrow \Gamma(\bbF;\bbG)$, is within the calculus and satisfies $\tbP\in\OP(0,0)$.  
\item   
There exists a balance $\bP:\Gamma(\bbF;\bbG)\rightarrow\Gamma(\bbE;\bbJ)$ for $\bD$ with respect to $\bB$ such that 
\[
\tbP=\bD\bP.
\] 
\end{enumerate}
\end{definition}

As apparent from \defref{def:aux_decomposition} and \defref{def:aux_decompositionD}, auxiliary decompositions are formulated within the smooth Fréchet topology. However, since the projections associated with the direct decomposition \eqref{eq:aux_smooth} belong to the calculus, it follows—by a careful density/continuity argument—that all Sobolev extensions of $\Gamma(\bbF;\bbG)$ decompose accordingly in their respective Sobolev topologies.
\begin{lemma}
\label{lem:Wspaux}
If an adapted Green system $\bD$ induces a Neumann auxiliary decomposition, then for every $1 < p < \infty$ and $(S,T;S',T')$ lenient tuples for $\bD$, there exists a topologically direct decomposition of Banach spaces 
\beq
W^{S',T'}_p(\bbF;\bbG) = \overline{\scrR^{S',T'}_p(\bD)} \oplus \scrN^{S',T'}_p(\bD^*,\bB^*).
\label{eq:Wsp_aux_decomp}
\eeq
Moreover, $\bD(W^{S,T}_p(\bbF;\bbG))$ is closed whenever $(S',T';S,T)$ are lenient tuples for $\bP$, in which case one writes $\overline{\scrR^{S',T'}_p(\bD)} = \scrR^{S',T'}_p(\bD)$.
\end{lemma}

As the final clause of the lemma suggests, continuous extensions of $\tbP$ are not always given by the composition of the continuous extensions of $\bD$ and $\bP$. This is because, even if $(S, T; S', T')$ are lenient tuples for $\bD$, it does not necessarily follow that $(S', T'; S, T)$ are lenient tuples for $\bP$. However, by the definition of a balance, this implication does hold when $(S, T; S', T') = (J, L ; I, K)$, where $(J, L; I, K)$ are the standard tuples associated with $\bD$ and $\bP$. In this case, $\bP$ maps in the reverse direction of $\bD$ in the mapping property \eqref{eq:strict_mapping_property}.

In establishing the existence of an auxiliary decomposition, it is important to note the converse of the claim in the lemma: if \eqref{eq:Wsp_aux_decomp} holds for every standard tuple $(J, L; I, K)$ for some $1 < p < \infty$, then the smooth version \eqref{eq:aux_smooth} also holds.

The Dirichlet version is formulated and proven in the same manner:
\begin{lemma}
\label{lem:WspauxD}
If an adapted Green system $\bD$ induces a Dirichlet auxiliary decomposition, then for every $1 < p < \infty$ and $(S,T;S',T')$ lenient tuples for $\bD$ that are also suitable for $\frakB$, there exists a topologically direct decomposition of Banach spaces 
\beq
W^{S',T'}_p(\bbF;\bbG) = \overline{\scrR^{S',T'}_p(\bD;\frakB)} \oplus \scrN^{S',T'}_p(\bD^*).
\label{eq:Wsp_aux_decompD}
\eeq
Moreover, $\bD(W^{S,T}_p(\bbF;\bbG))$ is closed whenever $(S',T';S,T)$ are lenient tuples for $\bP$, in which case one writes $\overline{\scrR^{S',T'}_p(\bD)} = \scrR^{S',T'}_p(\bD)$.
\end{lemma}
The proof is technical and is presented at the end of the section.
\subsection{Disjoint unions}
\label{sec:disjoint_union}
In the context of auxiliary decompositions, it is prudent to address the situation of disjoint unions of adapted Green systems. Specifically, following \defref{def:direct_sums_unions_sums}, if $\bD^{j} : \Gamma(\bbE_{j}; \bbJ_j) \rightarrow \Gamma(\bbF_j; \bbG_j)$ are adapted Green systems, then their disjoint union $\bD = \bD^{1} \sqcup \bD^{2} : \Gamma(\bbE; \bbJ) \rightarrow \Gamma(\bbF; \bbG)$ also forms an adapted Green system, where $\bbE = \bbE_{1} \oplus \bbE_{2}$, etc., with the associated systems from \defref{def:adapting_operator} defined through corresponding disjoint unions.

Due to the nature of the disjoint union, we find that
\[
\begin{aligned}
\scrR(\bD) &= \scrR(\bD^{1}) \oplus \scrR(\bD^{2}), & \qquad \scrR(\bD; \bB) &= \scrR(\bD^{1}; \bB^{1}) \oplus \scrR(\bD^{2}; \bB^{2}),
\end{aligned}
\]
with similar relations holding for the $S, P$ versions and the adapted adjoints $(\bD^{1})^*,(\bD^{2})^*$. Furthermore, each of $\bD^1$ and $\bD^2$ produces an auxiliary decomposition, either Dirichlet or Neumann as defined in \defref{def:aux_decomposition}, if and only if e.g., in the Neumann case:
\beq
\Gamma(\bbF; \bbG) = \scrR(\bD) \oplus \scrN(\bD^*, \bB^*).
\label{eq:aux_Direct_sum_disjoint}
\eeq
This holds because the spaces in the decomposition remain separate, as the section spaces themselves are disjoint. 

In this case, if $\bP^{1}, \tbP^{1}$ and $\bP^{2}, \tbP^{2}$ are the mappings from the auxiliary decompositions of $\bD^1$ and $\bD^2$, respectively, then by setting $\bP = \bP^1 \sqcup \bP^{2}$ and $\tbP = \tbP^{1} \sqcup \tbP^{2}$, it follows by construction that $\tbP = \bD \bP$ is the projection onto $\scrR(\bD)$ and indeed lies in $\OP(0,0)$ as required.


\subsection{Technical proofs}
\begin{PROOF}{\propref{prop:RD}}
Only the first statement is proven here, as the proof of the second is entirely analogous. By construction of Sobolev spaces, $\Gamma(\bbE; \bbJ) \hookrightarrow W^{S, T}_{p}(\bbE; \bbJ)$ densely and continuously, so $\bD : \Gamma(\bbE; \bbJ) \rightarrow W^{S', T'}_p(\bbF; \bbG)$ continuously. It follows immediately that
\[
\overline{\bD(\Gamma(\bbE; \bbJ))} \subseteq \overline{\bD(W^{S, T}_{p}(\bbE; \bbJ))}.
\]
In the other direction, let $\Theta \in \overline{\bD(W^{S, T}_{p}(\bbE; \bbJ))}$, which means that there exists a sequence $\bD \Psi_{n} \in \bD(W^{S, T}_{p}(\bbE; \bbJ))$ such that $\bD \Psi_{n} \rightarrow \Theta$ in $W^{S', T'}_p$. Let $\Psi_{n, j} \in \Gamma(\bbE; \bbJ)$ be an approximating sequence for each $\Psi_{n}$ in the $W^{S, T}_p$-topology. By induction, for each $n \in \Nzero$, we can select $j_{n} \in \Nzero$ such that $j_{n} > j_{n-1}$ and
\[ 
\|\Psi_{n, j_{n}} - \Psi_{n}\|_{S, T, p} < 2^{-n}.
\]
Then by the continuity of $\bD : W^{S, T}_{p}(\bbE; \bbJ) \rightarrow W^{S', T'}_p(\bbF; \bbG)$, we have
\[
\lim_{n \to \infty} \bD \Psi_{n, j_{n}} = \lim_{n \to \infty} \bD \Psi_{n} = \Theta.
\] 
Since $\bD \Psi_{n, j_{n}} \in \bD(\Gamma(\bbE; \bbJ))$, the claim is proven.
\end{PROOF}

\begin{PROOF}{\lemref{lem:Wspaux}}
Since $\tbP \in \OP(0,0)$, it follows from \corrref{corr:G0props} that $\tbP$ has the lenient mapping property  
\[
\tbP : W^{S',T'}_{p}(\bbF;\bbG) \to W^{S',T'}_{p}(\bbF;\bbG).
\]
Recall that this continuous extension is defined as follows: given $\Theta \in W^{S',T'}_{p}(\bbF;\bbG)$ and any $W^{S',T'}_{p}$-approximating sequence $(\Theta_{n}) \subset \Gamma(\bbF;\bbG)$ for $\Theta$, $\tbP$ acts on $\Theta$ as  
\[
\tbP \Theta = \lim_{n \to \infty} \tbP \Theta_{n},
\]
where the limit is taken with respect to the $W^{S',T'}_{p}$-topology. Since the projection property $\tbP \tbP = \tbP$ is preserved under the limit, the $W^{S',T'}_{p}$-extension of $\tbP$ remains a projection. Thus, $\tbP$ has a closed range in this topology, denoted by $\scrR^{S',T'}_p(\tbP)$ for the sake of this proof.  

The complement of $\scrR^{S',T'}_p(\tbP)$ in $W^{S',T'}_{p}(\bbF;\bbG)$ is then the range of the projection $\id - \tbP$. We now show that this range is precisely $\scrN^{S',T'}_{p}(\bD^*,\bB^*)$. Let $\Phi \in W^{S',T'}_{p}(\bbF;\bbG)$ with $(\id - \tbP)\Phi = \Phi$ and let $(\Phi_{n}) \subset \Gamma(\bbF;\bbG)$ be a $W^{S',T'}_{p}$-approximating sequence for $\Phi$. Then, by continuity,  
\[
(\id - \tbP)\Phi_{n} \to \Phi \quad \text{in } W^{S',T'}_{p}.
\]
Since $\Phi_{n} \in \Gamma(\bbF;\bbG)$, it follows from the properties of $\tbP$ in the auxiliary decomposition that  
\[
(\bD^*\oplus\bB^*)(\id - \tbP)\Phi_{n} = 0.
\]
Thus, as $W^{S',T'}_{p}$-approximating sequences are also $L^{p}$-approximating sequences, for any $\Upsilon \in \Gamma(\bbE;\bbJ)$, it follows from the generalized Green's formula \eqref{eq:integration_adapted_refined} that  
\[
\bra \Phi, \bD \Upsilon \ket = \lim_{n \to \infty} \bra (\id - \tbP)\Phi_{n}, \bD \Upsilon \ket = 0.
\]
Therefore, from \eqref{eq:Lp_ann}, we conclude that  
\[
\Phi \in \scrN^{0,0}_{p}(\bD^*,\bB^*) \cap W^{S',T'}_{p}(\bbF;\bbG) = \scrN^{S',T'}_{p}(\bD^*,\bB^*),
\]
establishing the direct decomposition  
\[
W^{S',T'}_p(\bbF;\bbG) = \scrR^{S',T'}_p(\tbP) \oplus \scrN^{S',T'}_p(\bD^*,\bB^*).
\]
Thus, to establish \eqref{eq:Wsp_aux_decomp}, it remains to show that  
\[
\scrR^{S',T'}_p(\tbP) = \overline{\scrR^{S',T'}_p(\bD)}.
\]

For the containment $\overline{\scrR^{S',T'}_p(\bD)} \subseteq \scrR^{S',T'}_p(\tbP)$, let $\Theta \in \overline{\scrR^{S',T'}_p(\bD)}$. By \propref{prop:RD}, there exists an approximating sequence $(\Psi_{n}) \subset \Gamma(\bbE;\bbJ)$ such that  
\[
\bD \Psi_{n} \to \Theta \quad \text{in } W^{S',T'}_p. 
\]
Since $\bD \Psi_{n} \in \scrR(\bD)$ and $\tbP$ is the projection onto $\scrR(\bD)$ as per the definition of the auxiliary decomposition \defref{def:aux_decomposition}, we have  
\[
\tbP \bD \Psi_{n} = \bD \Psi_{n}.
\]
Since $\tbP$ is continuous in the $W^{S',T'}_p$-topology, we conclude that  
\[
\Theta \in \scrR^{S',T'}_p(\tbP),
\]
proving $\overline{\scrR^{S',T'}_p(\bD)} \subseteq \scrR^{S',T'}_p(\tbP)$.  

Conversely, for the containment $\scrR^{S',T'}_p(\tbP) \subseteq \overline{\scrR^{S',T'}_p(\bD)}$, let $\tbP \Theta \in \scrR^{S',T'}_p(\tbP)$ and let $(\Theta_{n}) \subset \Gamma(\bbF;\bbG)$ be a sequence converging to $\Theta$ in the $W^{S',T'}_{p}$-topology. By continuity,  
\[
\tbP \Theta_{n} \to \tbP \Theta.
\]
Since $\tbP \Theta_{n} = \bD \bP \Theta_{n}$, it follows that  
\[
\bD \bP \Theta_{n} \to \tbP \Theta \quad \text{in } W^{S',T'}_{p},
\]
which means that $\tbP \Theta \in \overline{\scrR^{S',T'}_p(\bD)}$. Thus,  
\[
\scrR^{S',T'}_p(\tbP) \subseteq \overline{\scrR^{S',T'}_p(\bD)}.
\]
This completes the proof of \eqref{eq:Wsp_aux_decomp}.  

Finally, we show that if $\bP : W^{S',T'}_{p}(\bbF;\bbG) \to W^{S,T}_{p}(\bbF;\bbG)$ is continuous, i.e., $(S',T';S,T)$ are lenient tuples for $\bP$, then $\scrR^{S',T'}_{p}(\bD)$ is closed.  

By \propref{prop:RD}, let $\bD \Psi_{n} \in \scrR(\bD)$ be a $W^{S',T'}_{p}$-Cauchy sequence with limit $\Theta \in \overline{\scrR^{S',T'}_p(\bD)}$. Since $\bP$ is continuous, $\bP \bD \Psi_{n}$ is $W^{S,T}_{p}$-Cauchy and converges to $\bP \Theta$, which in turn implies that $\bD \bP \bD\Psi_{n}$ is $W^{S,T}_p$-Cauchy and converges to $\bD \bP\bD \Theta$.  

By the properties of the auxiliary decomposition, since $\bD \Psi_{n} \in \scrR(\bD)$, we have $\bD \bP\bD \Psi_{n} = \bD \Psi_{n}$. By the uniqueness of limits, this implies $\bD \bP \Theta = \Theta$, meaning  
\[
\Theta \in \bD(W^{S,T}_{p}(\bbF;\bbG)).
\]
Thus, $\scrR^{S',T'}_p(\bD)$ is closed, completing the proof.
\end{PROOF}
\begin{PROOF}{\lemref{lem:D'mapping}}
The first limit follows directly from \propref{prop:NDR} and from the sharp mapping property of $\bD^*$ with respect to the sharp tuples $(0,0;\, -M, -T)$:
\[
\bD^* : L^{p}(\bbF;\bbG) \rightarrow W^{-M,-T}_{p}(\bbE;\bbJ).
\]

For the second limit, note that the fact that $(M, T;\, 0, 0)$ are sharp tuples for $\bD$ implies, by \propref{prop:RD}, the identity:
\[
\overline{\scrR^{0,0}_{q}(\bD)} = \overline{\bD\big(W^{M,T}_{p}(\bbE;\bbJ)\big)}.
\]

Using \eqref{eq:Lp_ann} and the $L^p$--$L^q$ duality, the assumption $\Theta - \Xi \in \scrN_{p}^{0,0}(\bD^*, \bB^*)$ implies that, for all $\Upsilon \in W^{J,L}_{q}(\bbE;\bbJ)$,
\[
0 = \bra \Theta - \Xi, \bD \Upsilon \ket = \lim_{n \to \infty} \bra \Theta_n - \Xi, \bD \Upsilon \ket.
\]

Since $\Upsilon$, $\Theta_n$, and $\Xi$ all have sufficient regularity, we may apply the Green's formula \eqref{eq:integration_adapted_refined} to obtain:
\beq
\lim_{n \to \infty} \left[ \bra \bD^* \Theta_n - \bD^* \Xi, \Upsilon \ket + \bra \bB^* \Theta_n - \bB^* \Xi, \bB \Upsilon \ket \right] = 0.
\label{eq:limit_in_the_proof}
\eeq

The first term vanishes since $\bD^* \Theta_n \to \bD^* \Xi$ in $W^{-M,-T}_{p}(\bbE;\bbJ)$, which is strictly stronger than convergence in $W^{-M,-T}_{p,0}(\bbE;\bbJ)$—the regularity required for dual pairing against $\Upsilon \in W^{M,T}_{q}(\bbE;\bbJ)$. Hence, we are left with:
\[
\lim_{n \to \infty} \bra \bB^* \Theta_n - \bB^* \Xi, \bB \Upsilon \ket = 0.
\]

Since $\bB$ is normal, $\bB \Upsilon \in W^{0,C}_{q}(0;\bbL)$ can be prescribed arbitrarily, so $\lim_{n \to \infty} \bB^* \Theta_n=  \bB^* \Xi$ in $(W^{0,C}_{q}(0;\bbL))^*$. Since $\bbL$ is a vector bundle over $\dM$, we have the duality:
\[
(W^{0,C}_{q}(0;\bbL))^* = W^{0,-C}_{p}(0;\bbL),
\]
which yields the second limit.
\end{PROOF}

\section{Elliptic pre-complexes}
\label{sec:Elliptic}
\subsection{Definitions and main theorems} 
\label{sec:def_elliptic_pre_complex}
Let now $(\bD_{\alpha})_{\alpha \in \Nzero}$ be a sequence of adapted Green systems, cast into the following diagram: 
\beq
\begin{xy}
(-30,0)*+{0}="Em1";
(0,0)*+{\Gamma(\bbF_0;\bbG_0)}="E0";
(30,0)*+{\Gamma(\bbF_1;\bbG_{1})}="E1";
(60,0)*+{\Gamma(\bbF_2;\bbG_{2})}="E2";
(90,0)*+{\Gamma(\bbF_3;\bbG_{3})}="E3";
(101,0)*+{\cdots}="E4";
(-30,-25)*+{0}="Gm1";
(0,-25)*+{\Gamma(0;\bbL_0)}="G0";
(30,-25)*+{\Gamma(0;\bbL_1)}="G1";
(60,-25)*+{\Gamma(0;\bbL_2)}="G2";
(90,-25)*+{\Gamma(0;\bbL_3)}="G3";
(100,-25)*+{\cdots}="G4";
{\ar@{->}@/^{1pc}/^{\bD_{0}}"E0";"E1"};
{\ar@{->}@/^{1pc}/^{\bD_{0}^*}"E1";"E0"};
{\ar@{->}@/^{1pc}/^{\bD_{1}}"E1";"E2"};
{\ar@{->}@/^{1pc}/^{\bD_{1}^*}"E2";"E1"};
{\ar@{->}@/^{1pc}/^{\bD_{2}}"E2";"E3"};
{\ar@{->}@/^{1pc}/^{\bD_{2}^*}"E3";"E2"};
{\ar@{->}@/^{1pc}/^{\bD_{-1}}"Em1";"E0"};
{\ar@{->}@/^{1pc}/^{\bD^*_{-1}}"E0";"Em1"};
{\ar@{->}@/_{0pc}/^{\bB_0}"E0";"G0"};
{\ar@{->}@/_{0pc}/^{\bB_1}"E1";"G1"};
{\ar@{->}@/_{0pc}/^{\bB_2}"E2";"G2"};
{\ar@{->}@/_{0pc}/^{\bB_{3}}"E3";"G3"};
{\ar@{->}@/^{0pc}/^{\bB_{-1}}"Em1";"Gm1"};
{\ar@{->}@/^{0.8pc}/^{\bB_{0}^*}"E1";"G0"};
{\ar@{->}@/^{0.8pc}/^{\bB_{1}^*}"E2";"G1"};
{\ar@{->}@/^{0.8pc}/^{\bB_{2}^*}"E3";"G2"};
{\ar@{->}@/^{0.8pc}/^{\bB^*_{-1}}"E0";"Gm1"};
\end{xy}
\label{eq:elliptic_pre_complex_diagram}
\eeq
The additional systems in the diagram are the ones associated with each adapted Green system $\bD_{\alpha}$ through the generalized Green’s formula \eqref{eq:integration_adapted_refined}:
\beq
\bra \bD_{\alpha}\Psi,\Theta\ket = \bra\Psi,\bD_{\alpha}^*\Theta\ket + \bra \bB_{\alpha}\Psi, \bB^*_{\alpha}\Theta\ket. 
\label{eq:integration_by_parts_elliptic_pre_complex}
\eeq  
In this setup, for $\alpha = -1$, we set $\bD_{-1} = 0$, $\bB^*_{-1} = 0$, etc. 

Collectively, we refer to the diagram \eqref{eq:elliptic_pre_complex_diagram_intro} as $(\bD_{\bullet})$, the bullet notation serves to refer to the entire diagram of mappings rather than a single level. 

For the definitions that follow, recall the notion of a \emph{balance}~(\defref{def:balance}):
\begin{definition}[Elliptic pre-complex --- Neumann conditions]
\label{def:NN_segment}
Let $\alpha_0\in\Nzero\cup\BRK{\infty}$. A diagram $(\bD_{\bullet})$ as above is called an $\alpha_0$-\emph{elliptic pre-complex} based on \emph{Neumann conditions} if the following holds for all $\alpha\leq \alpha_0$: 
\begin{enumerate}[label=(\roman*), itemsep=0.5em]
\item\emph{(Neumann overdetermined ellipticity)} The following systems are overdetermined elliptic:
\begin{enumerate}
\item $\bD_{\alpha} \oplus \bD^*_{\alpha-1} \oplus \bB^*_{\alpha-1}$;
\item $\bD^*_{\alpha} \bD_{\alpha} \oplus \frakB^*_{\alpha} \bD_{\alpha} \oplus \bD^*_{\alpha-1} \oplus \bB^*_{\alpha-1}$. 
\end{enumerate}
\item\emph{(Order-reduction property)} For every balance $\bP$ for $\bD_{\alpha}$,  
\[
\bD_{\alpha+1} \bD_{\alpha} \bP \in \OP(0,0) 
\textand
\sigma(\bD_{\alpha+1} - \bD_{\alpha+1} \bD_{\alpha} \bP) = \sigma(\bD_{\alpha+1}),
\]
where $\sigma$ denotes the order-reduced symbol with respect to the sharp tuples of $\bD_{\alpha+1}$.
\end{enumerate}
\end{definition}
\begin{definition}[Elliptic pre-complex — Dirichlet conditions]
\label{def:DD_segment}
Let $\alpha_0\in\Nzero\cup\BRK{\infty}$. A sequence $(\bD_{\bullet})$ as above is called an $\alpha_0$-\emph{elliptic pre-complex} based on \emph{Dirichlet conditions} if the following hold for all $\alpha\leq \alpha_0$:
\begin{enumerate}[label=(\roman*), itemsep=0.5em]

\item \emph{(Dirichlet overdetermined ellipticity)}  
The following systems are overdetermined elliptic:
\begin{enumerate}
\item $\bD_{\alpha} \oplus \bD^*_{\alpha-1} \oplus \bB_{\alpha}$;
\item $\bD^*_{\alpha} \bD_{\alpha} \oplus \bD^*_{\alpha-1} \oplus \bB_{\alpha}$.
\end{enumerate}

\item \emph{(Order-reduction property)}  The following relations hold:
\begin{enumerate}
\item $\ker \bB_{\alpha-1} \subseteq \ker \bB_{\alpha} \bD_{\alpha-1}$;
\item For every balance $\bP$ for $\bD_{\alpha}$ with respect to $\bB_{\alpha}$,
\[
\bD_{\alpha+1} \bD_{\alpha} \bP \in \OP(0,0) \textand 
\sigma(\bD_{\alpha+1} - \bD_{\alpha+1} \bD_{\alpha} \bP) = \sigma(\bD_{\alpha+1}),
\]
where $\sigma$ denotes the order-reduced symbol with respect to the sharp tuples of $\bD_{\alpha+1}$.
\end{enumerate}
\end{enumerate}
\end{definition}

Supplementary to the introductory discussion in \secref{sec:elliptic_pre_complexes_intro}, a few technical remarks on these definitions are now in order. 

First, we have not specified the basic tuples upon which the overdetermined ellipticities in (ii) are based, as would be technically required by \defref{def:overdetermined_elliptic_varying}. In this respect, our only assumption concerns the general form of the sharp tuples for $\bD_{\alpha}$ and $\bD_{\alpha}^*$, as inherent in \defref{def:adapting_operator}. These general forms will later allow us to determine the relevant basic tuples for overdetermined ellipticity via \propref{prop:sharp_tuples_OD}.
 
Second, note that in both cases---as discussed in \secref{sec:elliptic_pre_complexes_intro}, and in view of the discussion surrounding the definition of a balance in \defref{def:balance}---the \emph{order-reduction property} roughly states that the orders and classes of $\bD_{\alpha} \bD_{\alpha-1}$ are less than or equal to those of $\bD_{\alpha-1}$. Unlike the case of a standard element in $\OP(m,r)$, where these sets of orders and classes reduce to two single numbers, the notion of a balance is introduced here precisely to avoid the complexity of comparing systems with varying orders and classes. In many practical examples, as shown in \secref{sec:Examples}, verifying this condition reduces to an algebraic check based solely on the composition rules of the calculus.

Third, in most of the examples studied in this paper, the required overdetermined ellipticity condition in item (i.b) of \defref{def:NN_segment}--\defref{def:DD_segment} is redundant, as it follows from the condition in item (i.a), together with the order-reduction property and the symbolic calculus developed in \secref{sec:douglas_nirenberg_lop}. Nevertheless, these conditions are included explicitly in the definitions to avoid the need for such derivations at the analytical stage of the proofs and to ensure that the theory remains sufficiently abstract to apply to more general systems in future studies.

%

The following is the basic theorem concerning elliptic pre-complexes:
\begin{theorem}[Lifted complex]
\label{thm:corrected_complex}
Let $(\bD_{\bullet})$ be an $\alpha_0$-elliptic pre-complex. For $\alpha \leq \alpha_0$, there exists a sequence of continuous linear maps of Fréchet spaces
\[
\fbD_{\alpha+1} : \Gamma(\bbF_{\alpha+1}; \bbG_{\alpha+1}) \rightarrow \Gamma(\bbF_{\alpha+2}; \bbG_{\alpha+2}),
\]
each uniquely characterized by the following properties:
\begin{enumerate}[label=(\roman*), itemsep=0.5em]
\item $\mathrm{(\NN)}$ If the elliptic pre-complex is based on Neumann conditions (\defref{def:NN_segment}):
\begin{enumerate}[label=(\alph*), itemsep=0pt]
\item $\scrR(\fbD_{\alpha}) \subseteq \scrN(\fbD_{\alpha+1})$.
\item $\fbD_{\alpha+1} = \bD_{\alpha+1}$ on $\scrN(\fbD_{\alpha}^*, \frakB^*_{\alpha})$.
\end{enumerate}
\item $\mathrm{(\DD)}$ If the elliptic pre-complex is based on Dirichlet conditions (\defref{def:DD_segment}):
\begin{enumerate}[label=(\alph*), itemsep=0pt]
\item $\scrR(\fbD_{\alpha}; \bB_{\alpha}) \subseteq \scrN(\fbD_{\alpha+1})$.
\item $\fbD_{\alpha+1} = \bD_{\alpha+1}$ on $\scrN(\fbD^*_{\alpha})$.
\end{enumerate}
\end{enumerate}
The resulting operators $\fbD_{\alpha}:
\Gamma(\bbF_{\alpha};\bbG_{\alpha})\rightarrow\Gamma(\bbF_{\alpha+1};\bbG_{\alpha+1})$ are adapted Green systems. Collectively, the induced sequence is denoted by $(\fbD_{\bullet})$ and is called the \emph{lifted complex} induced by $(\bD_{\bullet})$.
\end{theorem}

As will be elaborated upon in the proof of \thmref{thm:corrected_complex}, the systems in $(\fbD_{\bullet})$ are built inductively, with each level constructed upon an auxiliary decomposition emerging from the preceding level:
\begin{proposition}[Auxiliary decompositions]
\label{prop:complex_aux_decomp}
In the setting of \thmref{thm:corrected_complex}, for every $\alpha\leq \alpha_0$, the adapted Green system $\fbD_{\alpha}:
\Gamma(\bbF_{\alpha};\bbG_{\alpha})\rightarrow\Gamma(\bbF_{\alpha+1};\bbG_{\alpha+1})$ induces an auxiliary decomposition determined by the conditions upon which the elliptic pre-complex is based:
\begin{enumerate}[label=(\roman*), itemsep=0.5em]
\item[$(\NN)$] Under Neumann conditions, $\fbD_{\alpha}$ induces a Neumann auxiliary decomposition, as in \defref{def:aux_decomposition}:
\beq
\Gamma(\bbF_{\alpha+1}; \bbG_{\alpha+1}) = \scrR(\fbD_{\alpha}) \oplus \scrN(\fbD_{\alpha}^*, \frakB^*_{\alpha}).
\label{eq:aux_smooth_complex}
\eeq
In this setting, the decomposition further \emph{refines} to:
\beq 
\Gamma(\bbF_{\alpha+1}; \bbG_{\alpha+1}) \cap \ker \bB^*_{\alpha}
= \left( \scrR(\fbD_{\alpha}) \cap \ker \bB^*_{\alpha} \right)
\oplus \scrN(\fbD_{\alpha}^*).
\label{eq:aux_smooth_complexNRefined} 
\eeq
\item[$(\DD)$] Under Dirichlet conditions, $\fbD_{\alpha}$ induces a Dirichlet auxiliary decomposition, as in \defref{def:aux_decompositionD}:
\beq
\Gamma(\bbF_{\alpha+1}; \bbG_{\alpha+1}) 
= \scrR(\fbD_{\alpha}; \bB_{\alpha}) \oplus \scrN(\fbD_{\alpha}^*).
\label{eq:aux_smooth_complexD}
\eeq
In this setting, the decomposition further \emph{refines} to:
\beq 
\Gamma(\bbF_{\alpha+1}; \bbG_{\alpha+1}) \cap \ker \bB_{\alpha+1}
= \scrR(\fbD_{\alpha}; \bB_{\alpha}) 
\oplus ( \scrN(\fbD_{\alpha}^*) \cap \ker \bB_{\alpha+1}).
\label{eq:aux_smooth_complexDRefined} 
\eeq

\end{enumerate}
Denote by $\tbP_{\alpha}$ and $\bP_{\alpha}$ the systems from \defref{def:aux_decomposition}--\defref{def:aux_decompositionD} associated with these decompositions; that is, the systems for which $\tbP_{\alpha} = \fbD_{\alpha} \bP_{\alpha} \in \OP(0,0)$ is the projection onto the corresponding ranges in the direct decompositions \eqref{eq:aux_smooth_complex}--\eqref{eq:aux_smooth_complexD}.
\end{proposition}

The following proposition shows that the lifted complex may indeed be regarded as a “correction” of the original elliptic pre-complex by zero-order terms, and provides an explicit formula for the correcting terms.

\begin{proposition}[Correction terms]
\label{prop:correction}
In the setting of \thmref{thm:corrected_complex}, each of the systems in the lifted complex can be expressed as
\[
\fbD_{\alpha} = \bD_{\alpha} + \bC_{\alpha},
\]
where $\bC_{\alpha} \in \OP(0,0)$ is a Green operator satisfying the following properties:
\begin{enumerate}[itemsep=0.5em]
\item $\sigma(\fbD_{\alpha}) = \sigma(\bD_{\alpha})$, where $\sigma$ denotes the weighted principal symbol (cf.~\defref{def:weighted_symbol}) associated with the the basic sharp tuples of $\fbD_{\alpha}$.
\item $\bC_{\alpha}$ is given explicitly by the formula
\beq
\bC_{\alpha} = -\,\bD_{\alpha} \bD_{\alpha-1} \bP_{\alpha-1} 
= -\,\bD_{\alpha} \tbP_{\alpha-1}.
\label{eq:recursive_correction}  
\eeq
\end{enumerate}
\end{proposition}
Since elements in $\OP(0,0)$ are $L^{2}\rightarrow L^2$ continuous, they yield adjoints that integrate by parts without boundary terms. The conclusion is that $\fbD_{\alpha}$ satisfies a Green's formula \eqref{eq:integration_adapted_refined} with a boundary term similar to that of $\bD_{\alpha}$:  
\beq
\bra \fbD_{\alpha} \Psi, \Theta \ket = \bra \Psi, \fbD_{\alpha}^{*} \Theta \ket + \bra \bB_{\alpha} \Psi, \bB^*_{\alpha} \Theta \ket,  
\label{eq:integration_by_parts_corrected}
\eeq  
where the adapted adjoint of $\fbD_{\alpha}$ takes the form  
\beq
\fbD_{\alpha}^{*} = \bD^*_{\alpha} + \bC_{\alpha}^*
\label{eq:corrected_adjoint}
\eeq
where $\bC_{\alpha}^*$ is the adjoint of $\bC_{\alpha}\in\OP(0,0)$ as a Green operator. 

\subsection{Some variants}
\label{sec:variants}
\subsubsection{Disjoint Unions}  

In the context of disjoint unions of adapted Green systems, as outlined in \secref{sec:disjoint_union}, and due to the uniqueness clause of the construction in \thmref{thm:corrected_complex}, we directly have the following:

\begin{proposition}
\label{prop:disjoint_union_corrected}
Let $(\bD_{\bullet})$ be an $\alpha_0$-elliptic pre-complex. Suppose that for some $\beta_0 \in \Nzero$ and all $\alpha \geq \beta_0$, one has $\bD_{\alpha} = \bD_{\alpha}^{1} \sqcup \bD_{\alpha}^{2}$, where each $\bD_{\alpha}^{i}$ is an adapted Green system. Then the lifted complex likewise decomposes as $\fbD_{\alpha} = \fbD_{\alpha}^{1} \sqcup \fbD_{\alpha}^{2}$ and $\bC_{\alpha} = \bC_{\alpha}^{1} \sqcup \bC_{\alpha}^{2}$, with all auxiliary structures separating accordingly as in \eqref{eq:aux_Direct_sum_disjoint}.
\end{proposition}


\subsubsection{Disrupted elliptic pre-complexes}
From an applicative point of view, we observe a special situation in which the results established for elliptic pre-complexes continue to hold, even if the conditions of \defref{def:NN_segment}--\defref{def:DD_segment} are not entirely fulfilled. As will become evident, some caution is required when dealing with specific aspects of the resulting Hodge theory. 
\begin{definition}[Finite elliptic pre-complexes]
\label{def:finite_elliptic_pre_complex1}
A diagram \eqref{eq:elliptic_complex_diagram} of systems $(\bD_{\bullet})$ is called \emph{finite} if there exists $N \in \Nzero$ such that $\bD_{\alpha} = 0$ for all $\alpha > N$.  

It is called a \emph{finite elliptic pre-complex} if it is finite and satisfies the conditions of \defref{def:NN_segment}--\defref{def:DD_segment} for all $\alpha \in \Nzero$ (i.e., $\alpha_0=N+1$).  
\end{definition} 
\begin{definition}[Disrupted elliptic pre-complexes]
\label{def:finite_elliptic_pre_complex}
A diagram $(\bD_{\bullet})$ as in \eqref{eq:elliptic_complex_diagram} is called a \emph{disrupted $\alpha_0$-elliptic pre-complex} if it is finite, and the requirements of \defref{def:NN_segment}--\defref{def:DD_segment} hold for all $\alpha \in \Nzero$ \emph{except} at $\alpha = \alpha_0 - 1$, where the overdetermined ellipticity conditions fail (while still holding at $\alpha = \alpha_0$, where $\bD_{\alpha_0} = 0$).   
\end{definition}
\begin{theorem}
\label{thm:disrubted_elliptic_pre_complex}
\thmref{thm:corrected_complex}, \propref{prop:correction}, and \propref{prop:complex_aux_decomp} remain valid as stated for disrupted elliptic pre-complexes. 
\end{theorem}

\subsection{Hodge theory for Neumann conditions}
\label{sec:apps}
Under Neumann conditions, the defining relations in \thmref{thm:corrected_complex} imply the existence of a cochain complex:
\beq
\begin{tikzcd}
\cdots \arrow[r, "\fbD_{\alpha-1}"] &  \Gamma(\bbF_{\alpha};\bbG_{\alpha}) \arrow[rr, "\fbD_{\alpha}"] &  & \Gamma(\bbF_{\alpha+1};\bbG_{\alpha+1}) \arrow[rr, "\fbD_{\alpha+1}"] &  & \Gamma(\bbF_{\alpha+1};\bbG_{\alpha+2})  \cdots
\end{tikzcd}
\label{eq:Neumann_complex}
\eeq
Consider the spaces $\scrR(\fbD_{\alpha}^*;\bB^*_{\alpha})$, $\scrN(\fbD_{\alpha})$, $\scrR(\fbD_{\alpha}^*)$, and $\scrN(\fbD_{\alpha}, \bB_{\alpha})$ associated with any one-sided adapted Green system. The following lemma is obtained directly by comparing the decompositions in \eqref{eq:L2_decom} (applied to $\fbD^*_{\alpha+1}$ and $\fbD_{\alpha}$) and the defining relations in \thmref{thm:corrected_complex}:
\begin{lemma}
\label{lem:useful_neumann}
In the setting of \thmref{thm:corrected_complex}, under Neumann conditions, the following holds for every $\alpha\leq \alpha_0$:
\begin{enumerate}[itemsep=0pt,label=(\alph*)]
\item The subspaces $\scrN(\fbD_{\alpha})$ and $\scrR(\fbD_{\alpha+1}^*;\bB^*_{\alpha+1})$ are $L^2$-orthogonal and intersect trivially.
\item The subspaces $\scrR(\fbD_{\alpha})$ and $\scrR(\fbD_{\alpha+1}^*;\bB_{\alpha+1}^*)$ are $L^2$-orthogonal and intersect trivially.
\end{enumerate}
\end{lemma}
From item (b) and \eqref{eq:aux_smooth_complex}, it follows that
\beq
\scrR(\fbD_{\alpha+1}^*;\bB^*_{\alpha+1}) \subseteq \scrN(\fbD^*_{\alpha},\bB^*_{\alpha}).
\label{eq:DstarDstar_N}
\eeq
Moreover, define
\[
\Gamma_{\N}(\bbF_{\alpha}; \bbG_{\alpha}) = \Gamma(\bbF_{\alpha}; \bbG_{\alpha}) \cap \ker \bB^*_{\alpha}.
\]
Then, in addition to the cochain complex \eqref{eq:Neumann_complex}, the result also yields the following chain complex:
\beq
\begin{tikzcd}
\cdots &  \Gamma_{\N}(\bbF_{\alpha};\bbG_{\alpha}) \arrow[l, "\fbD_{\alpha-1}^*"'] &  & \Gamma_{\N}(\bbF_{\alpha+1};\bbG_{\alpha+1}) \arrow[ll, "\fbD_{\alpha}^*"'] &  & \Gamma_{\N}(\bbF_{\alpha+2};\bbG_{\alpha+2}) \arrow[ll, "\fbD_{\alpha+1}^*"']  \cdots
\end{tikzcd}
\label{eq:Neumann_complex_co}
\eeq
With these established, we have that the Neumann auxiliary decomposition \eqref{eq:aux_smooth_complex} further refines into a Hodge decomposition whenever the overdetermined ellipticity conditions holds, yielding the generalized theorem effectively encompassing all the Neumann results presented and discussed in \secref{sec:Hodge_intro}:
\begin{theorem}[Neumann Hodge decomposition]
\label{thm:hodge_like_corrected_complex}
In the setting of \thmref{thm:corrected_complex}, under Neumann conditions, every $\alpha <\alpha_0$ yields an $L^2$-orthogonal, topologically direct decomposition:
\beq
\Gamma(\bbF_{\alpha+1};\bbG_{\alpha+1}) = \scrR(\fbD_{\alpha}) \oplus \scrR(\fbD_{\alpha+1}^*;\bB^*_{\alpha+1}) \oplus \module_{\N}^{\alpha+1},
\label{eq:Hodgelikesmooth}
\eeq
where the finite-dimensional subspace $\module_{\NN}^{\alpha+1}:=\ker(\fbD_{\alpha+1},\fbD^*_{\alpha},\bB^*_{\alpha})$ refines into: 
\beq
\module_{\N}^{\alpha+1} = \ker(\bD_{\alpha+1},\ttbP_{\alpha-1}\bD^*_{\alpha},\bB^*_{\alpha}),
\label{eq:cohomology_groups}
\eeq
where $\ttbP_{\alpha-1}=\id-\tbP_{\alpha-1}$ is the projection onto $\scrN(\fbD_{\alpha-1}^*,\bB_{\alpha-1}^*)$ in the auxiliary decomposition. 

In particular, compared with the auxiliary decomposition \eqref{eq:aux_smooth_complex}, we have:
\beq
\scrN(\fbD^*_{\alpha},\bB^*_{\alpha}) = \scrR(\fbD_{\alpha+1}^*;\bB^*_{\alpha+1}) \oplus \module_{\N}^{\alpha+1}.
\label{eq:aux_refinmenetSmooth}
\eeq
Moreover, the decomposition \eqref{eq:Hodgelikesmooth} refines further into:
\beq
\Gamma_{\NN}(\bbF_{\alpha+1}; \bbG_{\alpha+1})
= ( \scrR(\fbD_{\alpha}) \cap \ker \bB^*_{\alpha+1} )
\oplus  \scrR(\fbD_{\alpha+1}^*; \bB_{\alpha}^*)
\oplus \module_{\NN}^{\alpha+1}.
\label{eq:HodgelikesmoothnRefined}
\eeq
\end{theorem}

The proof of this theorem is provided in \secref{sec:Construction}, following the construction of the induced elliptic complex.

In the disrupted case (cf. \defref{def:finite_elliptic_pre_complex}), we have:
\begin{theorem}[Neumann Hodge decomposition — disrupted case]
\label{thm:hodge_like_corrected_disrupted_complex}
In the disrupted case, the result of \thmref{thm:hodge_like_corrected_complex} remains valid, except for the finite dimensionality of $\module^{\alpha_0}_{\NN}$ when $\alpha=\alpha_0-1$. 
\end{theorem}

In the context of index theory, we introduce the generalization of the Euler characteristic discussed in \secref{sec:Hodge_intro_NN}:
\begin{definition}
\label{def:Neumann_Euler_characteristic} 
If $(\bD_{\bullet})$ is a finite $\alpha_0$-elliptic pre-complex, disrupted or not, based on Neumann conditions, then its \emph{Neumann Euler characteristic} is:
\[
\mathscr{X}_{\NN}=\sum_{\alpha=0}^{\alpha_0}(-1)^{\alpha}\dim\module^{\alpha}_{\NN}. 
\]  
\end{definition}
In the disrupted case, note that it might very well be that $\mathscr{X}_{\NN}=\pm\infty$. 
%

In the context of the cochain complex \eqref{eq:Neumann_complex}, in either the standard or disrupted case, the refinement of the auxiliary decomposition into a Hodge decomposition identifies $\module_{\N}^{\alpha}$ as the cohomology groups (whether they are finite-dimensional or not for $\alpha = \alpha_0-1$ in the disrupted case):

\begin{theorem}[Neumann Cohomology Groups]
\label{thm:cohomology}
Let $\Psi \in \Gamma(\bbF_{\alpha+1}; \bbG_{\alpha+1})$. Then, 
\[
\begin{gathered}
\Psi \in \scrR(\fbD_{\alpha}) \\
\text{if and only if} \\
\Psi \in \scrN(\fbD_{\alpha+1}) \text{ and } \bra \Psi, \Upsilon \ket = 0 \qquad \text{for every } \Upsilon \in \module_{\N}^{\alpha+1}. 
\end{gathered}
\]
Equivalently,
\beq
\scrN(\fbD_{\alpha+1}) = \scrR(\fbD_{\alpha}) \oplus \module_{\N}^{\alpha+1}.
\label{eq:cohomology_spaces}
\eeq
\end{theorem}

Combining Theorems~\ref{thm:hodge_like_corrected_complex} and \ref{thm:cohomology}, we obtain the following compound decompositions:
\[
\Gamma(\bbF_{\alpha+1}; \bbG_{\alpha+1}) = 
\lefteqn{\overbrace{\phantom{\scrR(\fbD_{\alpha}) \oplus \module_{\N}^{\alpha+1}}}^{\scrN(\fbD_{\alpha+1})}} \scrR(\fbD_{\alpha}) \oplus \underbrace{\module_{\N}^{\alpha+1} \oplus \scrR(\fbD_{\alpha+1}^*; \bB^*_{\alpha+1})}_{\scrN(\fbD_{\alpha}^*, \frakB^*_{\alpha})}
\]
which not only identifies the homology groups of the chain complex in \eqref{eq:Neumann_complex_co}, but also provides that $\fbD^*_{\alpha+1}$ itself induces a Dirichlet auxiliary decomposition as in \defref{def:aux_decompositionD}. The proof of \thmref{thm:NNintro} then follows directly from these decompositions by invoking the relations $\fbD_{\alpha+1}\fbD_{\alpha}=0$ and $\fbD_{\alpha}=\bD_{\alpha}$ on $\scrN(\fbD_{\alpha}^*,\bB^*_{\alpha})$. 

One consequence of the fact the projections onto the summands in \eqref{eq:Hodgelikesmooth} belong to the calculus is that the Hodge decompositions extends to suitable Sobolev versions using density and approximation arguments, as demonstrated in \lemref{lem:Wspaux}. 
\begin{corollary}
For any lenient tuples $(S,T;S',T')$ for $\fbD_{\alpha+1}$ and $(S'',T'';S',T')$ lenient tuples for $\fbD_{\alpha+1}^*$, and $1<p<\infty$, there exists a topologically direct decomposition:
\beq
W^{S',T'}_{p}(\bbF_{\alpha+1}; \bbG_{\alpha+1}) = \overline{\scrR^{S', T'}_{p}(\fbD_{\alpha})} \oplus \overline{\scrR_{p}^{S',T'}(\fbD_{\alpha+1}^*; \bB^*_{\alpha+1})} \oplus \module_{\N}^{\alpha+1}.
\label{eq:WspHodge}
\eeq
Moreover, $\overline{\scrR^{S', T'}_{p}(\fbD_{\alpha})}=\scrR^{S', T'}_{p}(\fbD_{\alpha})$ is closed when $(S',T';S,T)$ are lenient tuples for the balance of $\fbD_{\alpha}$ in its auxiliary decomposition, and $\overline{\scrR_{p}^{S',T'}(\fbD_{\alpha+1}^*; \bB^*_{\alpha+1})}=\scrR_{p}^{S',T'}(\fbD_{\alpha+1}^*; \bB^*_{\alpha+1})$ is closed when $(S'',T'';S',T')$ are lenient tuples for the balance of $\fbD_{\alpha+1}^*$ in its auxiliary decomposition. 
\end{corollary}
The decomposition \eqref{eq:WspHodge} then yields analogous Sobolev versions of \thmref{thm:cohomology} by means of an approximation/continuity argument. 
\subsection{Hodge theory for Dirichlet conditions}
\label{sec:appsD}
As in the exposition of \secref{sec:Hodge_intro_DD}, we begin by setting

\beq
\Gamma_{\D}(\bbF_{\alpha}; \bbG_{\alpha}) = \Gamma(\bbF_{\alpha}; \bbG_{\alpha}) \cap \ker\bB_{\alpha}.
\label{eq:D_spaces}
\eeq
We show that as a byproduct of \thmref{thm:corrected_complex} we have:
\begin{proposition}
\label{prop:containement_M}
The following relation holds:
\[
\scrR(\fbD_{\alpha};\bB_{\alpha})\subseteq\scrN(\fbD_{\alpha+1},\bB_{\alpha+1}).
\]
\end{proposition} 
Thus, we obtain analogously to \eqref{eq:Neumann_complex} the cochain complex:
\beq
\begin{tikzcd}
\cdots \arrow[r, "\fbD_{\alpha-1}"] 
& \Gamma_{\D}(\bbF_{\alpha}; \bbG_{\alpha}) 
\arrow[rr, "\fbD_{\alpha}"] 
& & \Gamma_{\D}(\bbF_{\alpha+1}; \bbG_{\alpha+1}) 
\arrow[rr, "\fbD_{\alpha+1}"] 
& & \Gamma_{\D}(\bbF_{\alpha+2}; \bbG_{\alpha+2}) \cdots
\end{tikzcd}
\label{eq:Dirichlet_cochain_complex}
\eeq

As in \lemref{lem:useful_neumann}, the following is obtained from the defining conditions of the lifted complex in \thmref{thm:corrected_complex} and by comparing the decompositions in \eqref{eq:L2_decom}:   
\begin{lemma}
\label{lem:useful_dirichlet}
In the setting of \thmref{thm:corrected_complex}, under Dirichlet conditions,  the following holds for every $\alpha\leq\alpha_0$: 
\begin{enumerate}[itemsep=0pt,label=(\alph*)] 
\item The subspaces $\scrN(\fbD_{\alpha+1},\bB_{\alpha+1})$ and $\scrR(\fbD_{\alpha+1}^*)$ are $L^2$-orthogonal, hence intersect trivially.
\item The subspace $\scrR(\fbD_{\alpha};\bB_{\alpha})$ and $\scrR(\fbD_{\alpha+1}^*)$ are $L^2$-orthogonal, hence intersect trivially.
\end{enumerate}
\end{lemma}
By comparing item (b) with \eqref{eq:aux_smooth_complexD}, we find that
\beq
\scrR(\fbD_{\alpha+1}^*)\subseteq \scrN(\fbD^*_{\alpha}).
\label{eq:DstarDstar_D}
\eeq 
The Dirichlet case then yields an analogous chain complex to that in \eqref{eq:Neumann_complex_co}:
\beq
\begin{tikzcd}
\cdots &  \Gamma(\bbF_{\alpha};\bbG_{\alpha}) \arrow[l, "\fbD_{\alpha-1}"'] &  & \Gamma(\bbF_{\alpha+1};\bbG_{\alpha+1}) \arrow[ll, "\fbD_{\alpha}^*"'] &  & \Gamma(\bbF_{\alpha+2};\bbG_{\alpha+2}) \arrow[ll, "\fbD_{\alpha+1}^*"'] \cdots
\end{tikzcd}
\label{eq:Dirichlet_complex_co}
\eeq
As in the Neumann picture, the auxiliary decomposition refines into a Hodge decomposition accordingly whenever the overdetermined ellipticity conditions holds, yielding the generalized theorem effectively encompassing all the Dirichlet results presented and discussed in \secref{sec:Hodge_intro}::

\begin{theorem}[Dirichlet Hodge decomposition]
\label{thm:hodge_like_corrected_complexD}
In the setting of \thmref{thm:corrected_complex}, under Dirichlet conditions, every $\alpha< \alpha_0$ yields an $L^{2}$-orthogonal topologically direct decomposition 
\beq
\Gamma(\bbF_{\alpha+1};\bbG_{\alpha+1}) = \scrR(\fbD_{\alpha};\bB_{\alpha})\oplus\scrR(\fbD_{\alpha+1}^*)\oplus \module_{\D}^{\alpha+1}
\label{eq:HodgelikesmoothD}
\eeq
where the finite-dimensional subspace $\module_{\D}^{\alpha+1}:=\ker(\fbD_{\alpha+1},\fbD^*_{\alpha},\bB_{\alpha+1})$ refines into:  
\beq
\module_{\D}^{\alpha+1}=\ker(\bD_{\alpha+1},\ttbP_{\alpha-1}\bD^*_{\alpha},\bB_{\alpha+1}), 
\label{eq:cohomology_groupsD}
\eeq
where $\ttbP_{\alpha-1}=\id-\tbP_{\alpha-1}$ is the projection onto $\scrN(\fbD_{\alpha-1}^*)$ in the auxiliary decomposition. 

In particular, comparing with the auxiliary decomposition \eqref{eq:aux_smooth_complexD}: 
\beq
\scrN(\fbD^*_{\alpha})=\scrR(\fbD_{\alpha+1}^*)\oplus \module_{\D}^{\alpha+1}.
\label{eq:aux_refinemenetSmoothD}
\eeq
Moreover, the decomposition \eqref{eq:HodgelikesmoothD} refines further into:
\beq
\Gamma_{\DD}(\bbF_{\alpha+1};\bbG_{\alpha+1})= \scrR(\fbD_{\alpha};\bB_{\alpha})\oplus(\scrR(\fbD_{\alpha+1}^*)\cap\ker\bB_{\alpha+1})\oplus \module_{\D}^{\alpha+1}
\label{eq:HodgelikesmoothDRefined}
\eeq
\end{theorem}

As in the Neumann picture, the proof of this theorem relies on machinery developed in \secref{sec:Construction}, hence it is presented in that same section. 

Also as in the Neumann case, in the disrupted case (cf. \defref{def:finite_elliptic_pre_complex}) we have:
\begin{theorem}[Dirichlet Hodge decomposition — disrupted case]
\label{thm:hodge_like_corrected_disrupted_complexD}
In the disrupted case, the result of \thmref{thm:hodge_like_corrected_complexD} remains valid, except for the finite dimensionality of $\module^{\alpha_0}_{\DD}$ when $\alpha=\alpha_0-1$. 
\end{theorem}

In the context of index theory, we introduce the generalization of the Euler characteristic discussed in \secref{sec:Hodge_intro_NN}:
\begin{definition}
\label{def:Dirichlet_Euler_characteristic} 
If $(\bD_{\bullet})$ is a finite $\alpha_0$-elliptic pre-complex, disrupted or not, based on Dirichlet conditions, then its \emph{Dirichlet Euler characteristic} is:
\[
\mathscr{X}_{\DD}=\sum_{\alpha=0}^{\alpha_0}(-1)^{\alpha}\dim\module^{\alpha}_{\DD}. 
\]  
\end{definition}
In the disrupted case, note that it might very well be that $\mathscr{X}_{\DD}=\pm\infty$. 

The following is the Dirichlet counterpart to \thmref{thm:cohomology}, based on the refined decomposition \eqref{eq:HodgelikesmoothDRefined}: 
\begin{theorem}[Dirichlet Cohomology groups]
\label{thm:cohomologyD}
Let $\Psi\in \Gamma(\bbF_{\alpha};\bbG_{\alpha})$. Then, 
\[
\begin{gathered}
\Psi\in \scrR(\fbD_{\alpha};\bB_{\alpha}) \\
\text{if and only if} \\
\Psi\in \scrN(\fbD_{\alpha+1},\bB_{\alpha+1}) \textand \bra\Psi,\Upsilon\ket=0 \qquad \text{for every }\Upsilon\in\module_{\D}^{\alpha+1}. 
\end{gathered}
\]
Equivalently,
\beq
\scrN(\fbD_{\alpha+1},\bB_{\alpha+1})=\scrR(\fbD_{\alpha};\bB_{\alpha})\oplus\module_{\D}^{\alpha+1},
\label{eq:cohomology_spacesD}
\eeq
\end{theorem}
Like the Neumann case, the proof of \thmref{thm:DDintro} then follows directly from these decompositions by invoking the relations $(\fbD_{\alpha+1}\oplus\bB_{\alpha+1})\fbD_{\alpha}=0$ on $\ker\bB_{\alpha}$ along with $\fbD_{\alpha}=\bD_{\alpha}$ on $\scrN(\fbD_{\alpha}^*)$. 

\subsection{Comparison with previous studies}
\label{sec:comparison_theory} 
In what follows, we review the theories developed in \cite{RS82, KTT07, Wal15, SS19} and references therein.

In contrast with our presentation in \secref{sec:Hodge_intro}, the original definition of an elliptic complex in \cite{AS68} was purely algebraic, making no explicit reference to Green’s formulas or ellipticity conditions. There, an elliptic complex is a sequence of differential operators of the same order:
\beq
\begin{xy}
(-30,0)*+{0}="Em1";
(0,0)*+{\Gamma(\bbF_0)}="E0";
(30,0)*+{\Gamma(\bbF_1)}="E1";
(60,0)*+{\Gamma(\bbF_2)}="E2";
(90,0)*+{\Gamma(\bbF_3)}="E3";
(101,0)*+{\cdots}="E4";
{\ar@{->}@/^{0pc}/^{A_{0}}"E0";"E1"};
{\ar@{->}@/^{0pc}/^{A_{1}}"E1";"E2"};
{\ar@{->}@/^{0pc}/^{A_{2}}"E2";"E3"};
{\ar@{->}@/^{0pc}/^{0}"Em1";"E0"};
\end{xy}
\label{eq:elliptic_complex_discussion}
\eeq
subject to the conditions:
\begin{itemize}
\item $\bD_{\alpha+1} \bD_{\alpha} = 0$,
\item $\image \sigma(\bD_{\alpha}) = \ker \sigma(\bD_{\alpha+1})$,
\end{itemize}
namely, that \eqref{eq:elliptic_complex_discussion} forms a cochain complex and the sequence of principal symbols is exact. It can then be shown that these properties imply the ellipticity of the associated “Laplacian” $A^*_{\alpha}A_{\alpha} + A_{\alpha-1} A^*_{\alpha-1}$ without explicitly imposing it.

Extending this concept to sequences of Douglis–Nirenberg systems over a compact manifold with boundary is possible by means of order-reducing operators. Indeed, \cite{RS82, KTT07, Wal15, SS19} define elliptic complexes on manifolds with boundary as sequences of systems in the calculus:
\beq
\begin{xy}
(-30,0)*+{0}="Em1";
(0,0)*+{\Gamma(\bbF_0;\bbG_0)}="E0";
(30,0)*+{\Gamma(\bbF_1;\bbG_{1})}="E1";
(60,0)*+{\Gamma(\bbF_2;\bbG_{2})}="E2";
(90,0)*+{\Gamma(\bbF_3;\bbG_{3})}="E3";
(101,0)*+{\cdots}="E4";
{\ar@{->}@/^{0pc}/^{\bD_{0}}"E0";"E1"};
{\ar@{->}@/^{0pc}/^{\bD_{1}}"E1";"E2"};
{\ar@{->}@/^{0pc}/^{\bD_{2}}"E2";"E3"};
{\ar@{->}@/^{0pc}/^{0}"Em1";"E0"};
\end{xy}
\label{eq:quasicomplex_intro}
\eeq
and require only that \eqref{eq:quasicomplex_intro} satisfies an adaptation of the conditions on \eqref{eq:elliptic_complex_discussion}: namely, that it is a cochain complex with an exact sequence of order-reduced symbols (cf.~\defref{def:order_reduced_symbol}),
\[
\image \sigma(\Pi_{\alpha+1} \bD_{\alpha} \Pi^{-1}_{\alpha}) = \ker \sigma(\Pi_{\alpha+2} \bD_{\alpha+1} \Pi^{-1}_{\alpha+1})
\]
for appropriate order-reducing operators $\Pi_{\alpha}: W^{J_{\alpha};L_{\alpha}}_{2}(\bbF_{\alpha};\bbG_{\alpha}) \to L^{2}(\bbF_{\alpha};\bbG_{\alpha})$, such that $\Pi_{\alpha+1} \bD_{\alpha} \Pi^{-1}_{\alpha} \in \OP(0,0)$.

Under these assumptions, due to \corrref{corr:G0props}, the sequence \eqref{eq:quasicomplex_intro} extends continuously into a cochain complex between Hilbert spaces:
\beq
\begin{xy}
(-30,0)*+{0}="Em1";
(0,0)*+{L^{2}(\bbF_0;\bbG_0)}="E0";
(30,0)*+{L^{2}(\bbF_1;\bbG_{1})}="E1";
(60,0)*+{L^{2}(\bbF_2;\bbG_{2})}="E2";
(90,0)*+{L^{2}(\bbF_3;\bbG_{3})}="E3";
(105,0)*+{\cdots}="E4";
{\ar@{->}@/^{0pc}/^{\tilde{\bD}_{0}}"E0";"E1"};
{\ar@{->}@/^{0pc}/^{\tilde{\bD}_{1}}"E1";"E2"};
{\ar@{->}@/^{0pc}/^{\tilde{\bD}_{2}}"E2";"E3"};
{\ar@{->}@/^{0pc}/^{0}"Em1";"E0"};
\end{xy}
\label{eq:quasicomplex_introL2}
\eeq
where $\tilde{\bD}_{\alpha} = \Pi_{\alpha+1} \bD_{\alpha} \Pi^{-1}_{\alpha}$, and the exactness of the boundary symbols amounts to the ellipticity of the “Laplacian” $\tilde{\bD}_{\alpha}^* \tilde{\bD}_{\alpha} + \tilde{\bD}_{\alpha-1} \tilde{\bD}^*_{\alpha-1}$.

It then follows from this ellipticity that there exist $L^{2}$-orthogonal decompositions:
\beq
L^{2}(\bbF_{\alpha+1};\bbG_{\alpha+1}) = \image{\tilde{\bD}_{\alpha}} \oplus \image{\tilde{\bD}^*_{\alpha+1}} \oplus \ker(\tilde{\bD}_{\alpha+1}, \tilde{\bD}^*_{\alpha}),
\label{eq:L2_quasicomplex}
\eeq
so that, by defining an appropriate \emph{parametrix} $(\tbP_{\bullet})$ for the complex $(\bD_{\bullet})$, one obtains finite-dimensional modules $\module^{\alpha}$ consisting of smooth sections, and applying the isomorphisms $\Pi_{\alpha+1}^{-1}$ to \eqref{eq:L2_quasicomplex} yields topologically direct (though not necessarily $L^{2}$-orthogonal) decompositions:
\[
W^{J_{\alpha+1},L_{\alpha+1}}_{2}(\bbF_{\alpha+1};\bbG_{\alpha+1}) = \image{\bD_{\alpha}} \oplus \image{\tbP_{\alpha}} \oplus \module^{\alpha}.
\]
Therefore, this approach both generalizes the classical theory of elliptic complexes and reduces the setting to sequences of operators of order and class zero, so that a Green’s formula of the form \eqref{eq:Green_elliptic_intro}, as well as the explicit ellipticity conditions surveyed in \secref{sec:Hodge_intro} and in \defref{def:NN_segment}--\defref{def:DD_segment}, need not be imposed a priori.

The success of this approach naturally led to the concept of an \emph{elliptic quasicomplex} \cite{KTT07, Wal15, SS19}, which relaxes the requirement that \eqref{eq:quasicomplex_intro} forms a strict cochain complex by allowing $\bD_{\alpha+1} \bD_{\alpha}$ to be a compact operator within the calculus rather than identically zero—that is,
\[
\sigma\big(\Pi_{\alpha+2} \bD_{\alpha+1} \bD_{\alpha} \Pi_{\alpha}^{-1}\big) = 0 \quad \text{instead of} \quad \bD_{\alpha+1} \bD_{\alpha} = 0.
\]
The main result for an elliptic quasicomplex $(\bD_{\bullet})$ is that it can be “lifted” to a genuine elliptic complex $(\fbD_{\bullet})$ by adding lower-order terms, negligible at the symbolic level and possessing compact continuous extensions, by virtue of \eqref{eq:principle_norm}. For the lifted complex, the Hodge decomposition takes the form:
\beq
W^{J_{\alpha+1}, L_{\alpha+1}}_{2}(\bbF_{\alpha+1}; \bbG_{\alpha+1}) = \image{\fbD_{\alpha}} \oplus \image{\tbP_{\alpha}} \oplus \module^{\alpha}.
\label{eq:Hodge_theory_quasi}
\eeq

\subsubsection{Comparison with the present theory}
We now compare elliptic quasicomplexes with elliptic pre-complexes, in both structure and applicability.

\begin{itemize}
\item \textbf{Definitions.} Although the two notions share conceptual similarities, there is no general implication between them. An elliptic pre-complex does not, in general, yield an elliptic quasicomplex, since the latter requires the entire sequence of order-reduced symbols (together with boundary operators) to be exact, a condition absent from our theory. Conversely, the conditions defining \eqref{eq:quasicomplex_intro} as an elliptic quasicomplex do not guarantee that it embeds into a diagram such as \eqref{eq:elliptic_pre_complex_diagram} with the additional properties required of adapted Green systems and those listed in \defref{def:NN_segment}--\defref{def:DD_segment}; most distinctively, the generalized Green's formulae and the normality of the boundary operators. 

\item \textbf{Motivation and consequences.} In the cited works, the objective is primarily Fredholm and index-theoretic, and to that end it suffices that the “lifted” complex $(\fbD_{\bullet})$ merely differs from $(\bD_{\bullet})$ by compact terms. Thus, compactness of $\bD_{\alpha+1}\bD_{\alpha}$ together with exactness of the order-reduced symbols are essentially the only requirements. In contrast, our aim, as premised in \secref{sec:opening} and \secref{sec:Hodge_intro}, is to obtain explicit cohomological formulations for boundary-value problems. For \thmref{thm:NNintro}--\thmref{thm:DDintro} to hold, the symbiotic relationships summarized by the diagram \eqref{eq:elliptic_pre_complex_diagram} must be satisfied, and the lifted complex must differ from the original not merely by lower-order terms but specifically by operators of order and class zero (cf. the discussion in \secref{sec:main_results_intro}).

\item \textbf{Hodge decompositions.} The Hodge decompositions produced by elliptic pre-complexes, namely \eqref{eq:Hodgelikesmooth} and \eqref{eq:HodgelikesmoothD}, differ in several respects from those obtained for elliptic quasicomplexes \eqref{eq:Hodge_theory_quasi}:
\begin{itemize}
\item They hold in the Fréchet topology, with continuous extensions of the direct summands to arbitrary Sobolev regularity.
\item They are $L^{2}$-orthogonal (and remain so at the Sobolev and smooth levels), yielding a canonical description of the complement of $\scrR(\fbD_{\alpha})$ in terms of the range of the adapted adjoint and the cohomology groups.
\item The explicit expressions for the cohomology groups coincide with the kernels of the original overdetermined boundary-value problems directly related to the operators in the original elliptic pre-complex. In elliptic quasi-complexes, only the index is independent of the lifting terms. 
\item The decompositions do not depend on any auxiliary choice of order-reducing operators.
\end{itemize}
By contrast, the decompositions in \eqref{eq:Hodge_theory_quasi} depend on the chosen order-reducing operators, are not continuous across different Sobolev exponents, and the parametrix $\tbP_{\alpha}$ varies with both the Sobolev scale and the choice of order reduction. For example, in \eqref{eq:L2_quasicomplex}, even if $\bD_{\alpha}$ has zero class, one generally has
\[
\tilde{\bD}_{\alpha}^* \ne (\Pi^{-1}_{\alpha})^* \bD_{\alpha}^* \Pi^*_{\alpha+1}.
\]
As the index is invariant under compact perturbations, from a Fredholm and index-theoretic perspective this dependence is inconsequential; however, to establish \thmref{thm:cohomology}--\thmref{thm:cohomologyD}, it is essential that \thmref{thm:hodge_like_corrected_complex} and \thmref{thm:hodge_like_corrected_complexD} hold in full.
\end{itemize}


\section{Construction of the lifted complex}
\label{sec:Construction}

This section forms the main technical core of the work. Here we jointly prove \thmref{thm:corrected_complex}, \propref{prop:complex_aux_decomp}, and \propref{prop:correction}, by induction on $\alpha \leq \alpha_{0} \in \Nzero \cup \BRK{\infty}$.  The proof follows the approach of \cite[Sec.~4]{KL23}, with many technical distinctions arising from the much more general systems considered in the calculus. 

The main proof is organized into five stages:
\begin{enumerate}
\item Stage 1: the base of the induction and the setup of the induction step.
\item Stages 2--3: establishing necessary estimates, overdetermined ellipticity, and closed-range arguments for the lifted system.
\item Stage 4: producing the auxiliary decompositions.
\item Stage 5: defining $\fbD_{\alpha+1}$ and showing that it lies in the calculus and differs from $\bD_{\alpha+1}$ by a lower-order term.
\end{enumerate}
Afterwards, in \secref{sec:Hodge_proof}, we establish the Hodge decompositions for both Neumann and Dirichlet conditions (\thmref{thm:hodge_like_corrected_complex} and \thmref{thm:hodge_like_corrected_complexD}), along with the refined cohomological identities \eqref{eq:cohomology_groups} and \eqref{eq:cohomology_groupsD}.

Finally, in \secref{sec:disrupted_proof}, we treat the disrupted case (\defref{def:finite_elliptic_pre_complex}) and prove \thmref{thm:disrubted_elliptic_pre_complex}, \thmref{thm:hodge_like_corrected_disrupted_complex}, and \thmref{thm:hodge_like_corrected_disrupted_complexD}.

\subsection{Stage 1: Base and setup of induction step}
The proofs of \thmref{thm:corrected_complex}, \propref{prop:complex_aux_decomp}, and \propref{prop:correction} share the same analytical heart for both Neumann or Dirichlet conditions, with rather minor yet delicate adjustments. Therefore, to avoid redundancies on the one hand and keep arguments concise on the other, throughout this section we will often use an argumentative structure of the form:
\[
\begin{split}
&\NN:\qquad \text{statement 1},\\
&\DD:\qquad \text{statement 2}. 
\end{split}
\]
This notation indicates that, given a set of assumptions, statement 1 holds under Neumann conditions and statement 2 holds under Dirichlet conditions. 
\label{sec:stage1}
\subsubsection{Induction base} For the base of the induction, it convenient to set:
\[
\fbD_{-1}=0, \qquad \fbD_{0}=\bD_{0},
\]
and start at level $\alpha = -1$. At this initial level, the induction base requires the following conditions to hold:

\begin{enumerate}[itemsep=0pt,label=(\alph*)]
\item $\fbD_{-1}$ induces an auxiliary decomposition:
\[
\begin{aligned}
&\NN:\qquad &&\Gamma(\bbF_0;0) = \scrR(\fbD_{-1}) \oplus \scrN(\fbD^*_{-1},\bB^*_{-1}), \\
&\DD:\qquad &&\Gamma(\bbF_0;0) = \scrR(\fbD_{-1};\bB_{-1}) \oplus \scrN(\fbD_{-1}).
\end{aligned} 
\]
\item The following containment holds:
\[
\begin{aligned}
&\NN:\qquad &&\scrR(\fbD_{-1}) \subseteq \scrN(\fbD_0), \\
&\DD:\qquad &&\scrR(\fbD_{-1};\bB_{-1}) \subseteq \scrN(\fbD_0),
\end{aligned} 
\] 
\item The following relations hold:
\[
\begin{split}
&\NN:\qquad \fbD_0 = \bD_0 \text{ on } \scrN(\fbD_{-1}^*,\bB^*_{-1}), \\
&\DD:\qquad \fbD_0 = \bD_0 \text{ on } \scrN(\fbD_{-1}^*),
\end{split} 
\] 
\item $\fbD_0 = \bD_0 + \bC_0$ is an adapted Green system, with $\bC_0$ assuming the form specified in \propref{prop:correction}.
\end{enumerate}

Since in either case $\fbD_{-1}=0$, this set of requirements is satisfied trivially. For example, in (a) and (b) in the $\NN$ case, observe that $\scrR(\fbD_{-1}) = \{0\}$ while $\scrN(\fbD_{-1}^*,\bB_{-1}^*)=\Gamma(\bbF_{0};0)$ as $\fbD_{-1}^*=0$ and $\bB_{-1}^*=0$; condition (c) then implies that $\fbD_0 = \bD_0$ identically, which is indeed the case. Condition (d) is satisfied trivially, as the adapted Green system $\fbD_0 = \bD_0$ has $\bC_0 = 0$ by construction.

\subsubsection{Induction Hypothesis} Interpreting the conditions in \thmref{thm:corrected_complex}, \propref{prop:complex_aux_decomp}, and \propref{prop:correction}, the induction hypothesis is that $\fbD_{\alpha}$ and $\fbD_{\alpha-1}$ have been defined so that the following hold:

\begin{enumerate}[itemsep=0pt,label=(\alph*)]
\item Under Neumann conditions, $\fbD_{\alpha-1}$ induces a Neumann auxiliary decomposition, while under Dirichlet conditions, it induces a Dirichlet auxiliary decomposition. In both cases, this corresponds to the existence of a balance $\bP_{\alpha-1}$ for $\fbD_{\alpha-1}$ such that the system
\[
\tbP_{\alpha-1} = \fbD_{\alpha-1} \bP_{\alpha-1} \in \OP(0, 0)
\]
is the continuous projection onto the range in the corresponding $L^2$-orthogonal, topologically direct decomposition:
\beq
\begin{aligned}
&\NN:\qquad &&\Gamma(\bbF_{\alpha}; \bbG_{\alpha}) = \scrR(\fbD_{\alpha-1}) \oplus \scrN(\fbD^*_{\alpha-1}, \bB_{\alpha-1}^*), \\
&\DD:\qquad &&\Gamma(\bbF_{\alpha}; \bbG_{\alpha}) = \scrR(\fbD_{\alpha-1}; \bB_{\alpha-1}) \oplus \scrN(\fbD^*_{\alpha-1}).
\end{aligned}
\label{eq:aux_induction}
\eeq
We also assume the refined versions:
\beq
\begin{aligned}
&\NN:\qquad &&\Gamma(\bbF_{\alpha}; \bbG_{\alpha}) \cap \ker \bB_{\alpha-1}^*
=(\scrR(\fbD_{\alpha-1}) \cap \ker \bB^*_{\alpha-1})
\oplus \scrN(\fbD^*_{\alpha-1}, \bB_{\alpha-1}^*), \\
&\DD:\qquad &&\Gamma(\bbF_{\alpha}; \bbG_{\alpha}) \cap \ker \bB_{\alpha}
= \scrR(\fbD_{\alpha-1}; \bB_{\alpha-1})
\oplus(\scrN(\fbD_{\alpha-1}^*) \cap \ker \bB_{\alpha}).
\label{eq:aux_induction_Refined}
\end{aligned}
\eeq
\item The following containment holds:
\beq
\begin{aligned}
&\NN:\qquad &&\scrR(\fbD_{\alpha-1}) \subseteq \scrN(\fbD_\alpha), \\
&\DD:\qquad &&\scrR(\fbD_{\alpha-1}; \bB_{\alpha-1}) \subseteq \scrN(\fbD_\alpha),
\end{aligned} 
\label{eq:containments_induction} 
\eeq
with a further identity in the Dirichlet case, taking into account \propref{prop:containement_M}:
\beq
\DD:\qquad \scrR(\fbD_{\alpha-1}; \bB_{\alpha-1}) \subseteq \ker\bB_{\alpha}.
\label{eq:containments_inductionM} 
\eeq

\item The following relations hold:
\beq
\begin{aligned}
&\NN:\qquad &&\fbD_\alpha = \bD_\alpha \text{ on } \scrN(\fbD_{\alpha-1}^*,\bB^*_{\alpha-1}), \\
&\DD:\qquad &&\fbD_\alpha = \bD_\alpha \text{ on } \scrN(\fbD_{\alpha-1}^*),
\end{aligned} 
\label{eq:relations_induction}
\eeq
\item $\fbD_{\alpha} = \bD_{\alpha} + \bC_{\alpha}$ is an adapted Green system, with $\bC_{\alpha} \in \OP(0,0)$ in both cases, taking the form specified in \propref{prop:correction}:
\beq
\bC_{\alpha} = -\bD_{\alpha} \bD_{\alpha-1} \bP_{\alpha-1} = -\bD_{\alpha} \tbP_{\alpha-1}.
\label{eq:correction_induction}
\eeq
Moreover, in either case, $\sigma(\fbD_{\alpha} - \bD_{\alpha}) = 0$.
\end{enumerate}

\subsubsection{Induction step} Under the induction hypothesis, the remainder of the next four subsections is devoted to proving the following:
\begin{enumerate}[itemsep=0pt,label=(\alph*)]
\item The system $\fbD_{\alpha}$ induces an auxiliary decomposition, depending on the conditions upon which the elliptic pre-complex is based. Specifically, there exists a balance $\bP_{\alpha}$ for $\fbD_{\alpha}$ such that the system $\tbP_{\alpha} = \fbD_{\alpha} \bP_{\alpha} \in \OP(0, 0)$ is a continuous projection onto the range in one of the following $L^2$-orthogonal, topologically direct decompositions, which are shown to hold:
\beq
\begin{aligned}
&\NN:\qquad &&\Gamma(\bbF_{\alpha+1}; \bbG_{\alpha+1}) = \scrR(\fbD_{\alpha}) \oplus \scrN(\fbD^*_{\alpha}, \bB_{\alpha}^*), \\
&\DD:\qquad &&\Gamma(\bbF_{\alpha+1}; \bbG_{\alpha+1}) = \scrR(\fbD_{\alpha}; \bB_{\alpha}) \oplus \scrN(\fbD^*_{\alpha}).
\end{aligned}
\label{eq:aux_induction_step}
\eeq
We also establish the refined version in the Dirichlet case:
\beq
\begin{aligned} 
&\NN:\qquad &&\Gamma(\bbF_{\alpha+1}; \bbG_{\alpha+1}) \cap \ker \bB_{\alpha}^*
=(\scrR(\fbD_{\alpha}) \cap \ker \bB^*_{\alpha})
\oplus \scrN(\fbD^*_{\alpha}, \bB_{\alpha}^*), \\
&\DD:\qquad &&\Gamma(\bbF_{\alpha+1}; \bbG_{\alpha+1}) \cap \ker \bB_{\alpha+1}
= \scrR(\fbD_{\alpha}; \bB_{\alpha})
\oplus(\scrN(\fbD_{\alpha}^*) \cap \ker \bB_{\alpha+1}).
\label{eq:aux_induction_stepRefined}
\end{aligned} 
\eeq 
\item There exists a system $\fbD_{\alpha+1}:\Gamma(\bbF_{\alpha+1};\bbG_{\alpha+1})\rightarrow\Gamma(\bbF_{\alpha+2};\bbG_{\alpha+2})$ such that the following containments hold:
\beq
\begin{aligned}
&\NN:\qquad &&\scrR(\fbD_{\alpha}) \subseteq \scrN(\fbD_{\alpha+1}), \\
&\DD:\qquad &&\scrR(\fbD_{\alpha};\bB_{\alpha}) \subseteq \scrN(\fbD_{\alpha+1}),
\end{aligned} 
\label{eq:contain_induction_step}
\eeq
with the further identity in the Dirichlet case, taking into account \propref{prop:containement_M}:
\beq
\DD:\qquad \scrR(\fbD_{\alpha};\bB_{\alpha}) \subseteq \ker\bB_{\alpha+1}.
\label{eq:contain_induction_stepM} 
\eeq
\item The following relations hold:
\beq
\begin{aligned}
&\NN:\qquad &&\fbD_{\alpha+1} = \bD_{\alpha+1} \text{ on } \scrN(\fbD_{\alpha}^*,\bB^*_{\alpha}), \\
&\DD:\qquad &&\fbD_{\alpha+1} = \bD_{\alpha+1} \text{ on } \scrN(\fbD_{\alpha}^*),
\end{aligned} 
\label{eq:relation_induction_step}
\eeq
\item $\fbD_{\alpha+1} = \bD_{\alpha+1} + \bC_{\alpha+1}$ is an adapted Green system, with $\bC_{\alpha+1} \in \OP(0,0)$ in both cases, taking the form specified in \propref{prop:correction}:
\beq
\bC_{\alpha+1} = -\bD_{\alpha+1} \bD_{\alpha} \bP_{\alpha} = -\bD_{\alpha+1} \tbP_{\alpha}.
\label{eq:correction_induction_step}
\eeq
Moreover, in either case, $\sigma(\fbD_{\alpha+1} - \bD_{\alpha+1}) = 0$.

\end{enumerate}
\subsection{Stage 2: Additional elliptic estimates}
To establish the induction step, we begin by considering the following systems:
\beq
\begin{split}
&\NN:\qquad \fbD_{\alpha}\oplus\fbD_{\alpha-1}^* \oplus \bB^*_{\alpha-1}, \\
&\DD:\qquad \fbD_{\alpha}\oplus\fbD_{\alpha-1}^* \oplus \bB_{\alpha}.
\end{split}
\label{eq:corrected_D_N_adjoints}
\eeq
Due to the induction hypothesis, we have that the weighted symbol satisfies, in both cases:
\[
\sigma(\fbD_{\alpha} - \bD_{\alpha}) = 0 \quad \text{and} \quad \sigma(\fbD_{\alpha-1}^* - \bD_{\alpha-1}^*) = 0.
\]
Hence, by comparing with the overdetermined ellipticities (2.a) of either \defref{def:NN_segment} or \defref{def:DD_segment}, it follows from \propref{prop:lower_order_correction_overdetermined_ellipticity} that the systems in \eqref{eq:corrected_D_N_adjoints} are also overdetermined elliptic, as they differ from the original overdetermined elliptic systems by lower-order terms. 

Now, by the composition rules, the same holds for the corresponding lifted systems associated with those listed in (2.b):
\beq
\begin{aligned}
&\NN:\qquad &&\fbD^*_{\alpha} \fbD_{\alpha} \oplus \bB_{\alpha}^* \fbD_{\alpha} \oplus \fbD^*_{\alpha-1} \oplus \bB^*_{\alpha-1}, \\
&\DD:\qquad &&\fbD^*_{\alpha} \fbD_{\alpha} \oplus \fbD^*_{\alpha-1} \oplus \bB_{\alpha}.
\end{aligned}
\label{eq:corrected_D_N_adjointsII}
\eeq

We record these observations as a corollary:
\begin{corollary}
\label{corr:overdetermined_ellipticity_corrected}
Under the induction hypothesis, the systems listed in \eqref{eq:corrected_D_N_adjoints} and \eqref{eq:corrected_D_N_adjointsII} are overdetermined elliptic.
\end{corollary}

Through \thmref{thm:overdetermined_varying_orders}, the overdetermined ellipticities of \eqref{eq:corrected_D_N_adjoints} provide the finite-dimensionality of the corresponding kernels, each being a subspace of $\Gamma(\bbF_{\alpha}; \bbG_{\alpha})$:
\beq
\begin{aligned}
&\NN:\qquad &&\module^{\alpha}_{\N}:=\ker(\fbD_{\alpha},\fbD_{\alpha-1}^*,\bB^*_{\alpha-1}) = \ker(\bD_{\alpha},\fbD_{\alpha-1}^*,\bB^*_{\alpha-1}), \\
&\DD:\qquad &&\module^{\alpha}_{\D}:=\ker(\fbD_{\alpha},\fbD_{\alpha-1}^*,\bB_{\alpha}) = \ker(\bD_{\alpha},\fbD_{\alpha-1}^*,\bB_{\alpha}).
\end{aligned}
\label{eq:cohomology_in_the_proof}
\eeq
Here, the identity $\fbD_{\alpha} = \bD_{\alpha}$ on the kernel follows from the induction hypothesis \eqref{eq:relations_induction}, i.e., that $\fbD_\alpha = \bD_\alpha$ on either $\scrN(\fbD_{\alpha-1}^*)$ or $\scrN(\fbD_{\alpha-1}^*,\bB^*_{\alpha-1})$. 
\begin{proposition}
\label{prop:coinciding_kernels}
The kernels of the systems in \eqref{eq:corrected_D_N_adjoints} and \eqref{eq:corrected_D_N_adjointsII} coincide with $\module^{\alpha}_{\N}$ in the $\NN$ case and with $\module^{\alpha}_{\D}$ in the $\DD$ case.
\end{proposition}

\begin{proof}
The claim for the systems in \eqref{eq:corrected_D_N_adjoints} follows directly from comparing with \eqref{eq:cohomology_groups} and \eqref{eq:cohomology_groupsD}.

For the systems in \eqref{eq:corrected_D_N_adjointsII}, the inclusions:
\[
\begin{aligned}
&\NN:\qquad &&\module_{\N}^{\alpha} \subseteq \ker(\fbD^*_{\alpha} \fbD_{\alpha},\bB_{\alpha}^* \fbD_{\alpha},\fbD^*_{\alpha-1},\bB^*_{\alpha-1}), \\  
&\DD:\qquad &&\module_{\D}^{\alpha}\subseteq \ker(\fbD^*_{\alpha} \fbD_{\alpha},\fbD^*_{\alpha-1},\bB_{\alpha}),
\end{aligned}
\]
follow directly from comparing with \eqref{eq:cohomology_in_the_proof} above.  

To prove the reverse inclusions, consider $\Psi$ in one of the following spaces:
\[
\begin{aligned}
&\NN:\qquad &&\Psi \in \ker(\fbD^*_{\alpha} \fbD_{\alpha},\bB_{\alpha}^* \fbD_{\alpha},\fbD^*_{\alpha-1},\bB^*_{\alpha-1}),
\\&\DD:\qquad &&\Psi \in \ker(\fbD^*_{\alpha} \fbD_{\alpha},\fbD^*_{\alpha-1},\bB_{\alpha}).
\end{aligned}
\]

In particular, $\Psi$ satisfies:
\[
\begin{aligned}
&\NN:\qquad &&\Psi \in \scrN(\fbD_{\alpha-1}^*,\bB^*_{\alpha-1}) &&&\textand &&&&&\fbD_{\alpha} \Psi \in \scrN(\fbD_{\alpha}^*, \bB^*_{\alpha}), \\ 
&\DD:\qquad &&\Psi \in \scrN(\fbD_{\alpha-1}^*) \cap \ker\bB_{\alpha} &&&\textand &&&&&\fbD_{\alpha} \Psi \in \scrN(\fbD_{\alpha}^*).
\end{aligned}
\]

In all cases, by comparing with the $L^2$-orthogonal decompositions in \eqref{eq:L2_decom}, which hold for any adapted Green system, including $\fbD_{\alpha}$, we find that:
\[
\begin{aligned}
&\NN:\qquad &&\fbD_{\alpha} \Psi \in \scrN(\fbD_{\alpha}^*, \bB^*_{\alpha}) \cap \scrR(\fbD_{\alpha}) = \BRK{0}, \\ 
&\DD:\qquad &&\fbD_{\alpha} \Psi \in \scrN(\fbD_{\alpha}^*) \cap \scrR(\fbD_{\alpha}; \bB_{\alpha}) = \BRK{0}.
\end{aligned}
\]
This implies $\fbD_{\alpha} \Psi = 0$ in all cases. To summarize:
\[
\begin{aligned}
&\NN:\qquad &&\Psi \in \ker(\fbD_{\alpha},\fbD_{\alpha-1}^*,\bB^*_{\alpha-1}) = \module^{\alpha}_{\N}, \\ 
&\DD:\qquad &&\Psi \in \ker(\fbD_{\alpha},\fbD_{\alpha-1}^*,\bB_{\alpha}) = \module^{\alpha}_{\D},
\end{aligned}
\]
which completes the proof.
\end{proof}

Let $\frakI_{\alpha}\in\OP(-\infty, -\infty)$ denote the $L^2$-orthogonal projection onto $\module_{\R}^{\alpha}$, $\mathrm{R}\in\BRK{\NN,\DD}$, as described in \thmref{thm:overdetermined_varying_orders}. Direct summing $\frakI_{\alpha}$ with the systems in either \eqref{eq:corrected_D_N_adjoints} or \eqref{eq:corrected_D_N_adjointsII} yields injective systems due to \propref{prop:coinciding_kernels}. Therefore, by \thmref{thm:overdetermined_varying_orders}, these systems admit a left-inverse within the calculus of Green operators. 

\begin{proposition}
\label{prop:new_estimate1}
The following systems are also overdetermined elliptic and injective:
\beq
\begin{aligned}
&\NN:\qquad &&\fbD_{\alpha} \oplus \tbP_{\alpha-1} \oplus \frakI_{\alpha}, \\
&\DD:\qquad &&\fbD_{\alpha} \oplus \bB_{\alpha} \oplus \tbP_{\alpha-1} \oplus \frakI_{\alpha}.
\end{aligned}
\label{eq:order_reduced_elliptiices_proof}
\eeq
Here, the system $\tbP_{\alpha-1}\in\OP(0,0)$ is the one from the induction hypothesis, associated with the auxiliary decompositions \eqref{eq:aux_induction}. 
\end{proposition}
\begin{proof}
By \thmref{thm:overdetermined_varying_orders}, if a Douglis-Nirenberg system has a left-inverse within the calculus, then it is overdetermined elliptic and injective.

With this given, for the $\NN$ case, the Neumann auxiliary decomposition \eqref{eq:aux_induction} induced by $\fbD_{\alpha-1}$ and the properties of $\tbP_{\alpha-1}$ imply that $(\id - \tbP_{\alpha-1})$ is the projection onto $\scrN(\fbD_{\alpha-1}^*, \bB_{\alpha-1}^*)$. Therefore, for every $\Psi \in \Gamma(\bbF_{\alpha}; \bbG_{\alpha})$, we have:
\[
(\fbD_{\alpha-1}^* \oplus \bB_{\alpha-1}^*) \tbP_{\alpha-1} \Psi = (\fbD_{\alpha-1}^* \oplus \bB_{\alpha-1}^*) \Psi.
\]
This leads to the identity:
\[
\fbD_{\alpha} \oplus \fbD_{\alpha-1}^* \oplus \bB_{\alpha-1}^* \oplus \frakI_{\alpha} 
= (\id \oplus (\fbD_{\alpha-1}^* \oplus \bB_{\alpha-1}^*) \oplus \id)
(\fbD_{\alpha} \oplus \tbP_{\alpha-1} \oplus \frakI_{\alpha}).
\]
By the overdetermined ellipticity of the $\NN$ case, as established in \corrref{corr:overdetermined_ellipticity_corrected}, together with the identification of the system’s kernel in \propref{prop:coinciding_kernels}, the system on the left-hand side is injective and admits a left-inverse within the calculus. Consequently, the system $\fbD_{\alpha} \oplus \tbP_{\alpha-1} \oplus \frakI_{\alpha}$ also admits a left-inverse within the calculus and is therefore overdetermined elliptic.

In the $\DD$ case, the Dirichlet auxiliary decomposition induced by $\fbD_{\alpha-1}$, assumed in the induction step \eqref{eq:aux_induction}, combined with the properties of $\tbP_{\alpha-1}$, implies that $(\id - \tbP_{\alpha-1})$ is the projection onto $\scrN(\fbD_{\alpha-1}^*)$. Thus, for every $\Psi \in \Gamma(\bbF_{\alpha}; \bbG_{\alpha})$, we have:
\[
\fbD_{\alpha-1}^* \tbP_{\alpha-1} \Psi = \fbD_{\alpha-1}^* \Psi.
\]
We therefore obtain the identity:
\[
\fbD_{\alpha} \oplus \bB_{\alpha} \oplus \fbD_{\alpha-1}^* \oplus \frakI_{\alpha} 
= (\id \oplus \id \oplus \fbD_{\alpha-1}^* \oplus \id)
(\fbD_{\alpha} \oplus \bB_{\alpha} \oplus \tbP_{\alpha-1} \oplus \frakI_{\alpha}).
\]
By the overdetermined ellipticity for the $\DD$ case established in \corrref{corr:overdetermined_ellipticity_corrected}, together with the identification of the system’s kernel in \propref{prop:coinciding_kernels}, the system on the left-hand side is injective and admits a left-inverse within the calculus. This implies that the system $\fbD_{\alpha} \oplus \bB_{\alpha} \oplus \tbP_{\alpha-1} \oplus \frakI_{\alpha}$ is also overdetermined elliptic, allowing us to conclude the argument as in the $\NN$ case.
\end{proof}
Recall again the notion of sharp tuples from \defref{def:sharp_tuples}. Since $\fbD_{\alpha}-\bD_{\alpha}\in\OP(0,0)$, we find that $\fbD_{\alpha}$ and $\bD_{\alpha}$ share the same sharp tuples. Thus, from the definition of adapted Green system \defref{def:adapting_operator}, $\fbD_{\alpha}$ posses sharp tuples of the form $(M,T;0,0)$, which are also suitable for $\bB$, in the sense that there is a tuple $K'$ such that $(M,T;0,C)$ are sharp for $\bB$. More generally, if $(J,L;I,K)$ are any sharp tuples for $\bD_{\alpha}$ which are also suitable for $\bB_{\alpha}$, due to \propref{prop:sharp_tuples_OD} and the properties of overdetermined elliptic systems, we find: 


\begin{proposition}
Let $(J, L;\, I, K)$ be sharp tuples for $\fbD_{\alpha}$ also suitable for $\bB_{\alpha}$. Then the following estimate holds:
\beq
\|\Psi\|_{J, L, p} \lesssim \|\fbD_{\alpha} \Psi\|_{I, K, p} + \|\tbP_{\alpha-1} \Psi\|_{J, L, p} + \|\frakI_{\alpha} \Psi\|_{0, 0, p},
\label{eq:Ak_overdetermined_elliptic_estimate}
\eeq
for all $\Psi$ satisfying:
\beq
\begin{aligned}
&\NN:\qquad && \Psi \in W^{J, L}_p(\bbF_{\alpha}; \bbG_{\alpha}), \\
&\DD:\qquad && \Psi \in W^{J, L}_p(\bbF_{\alpha}; \bbG_{\alpha}) \cap \ker \bB_{\alpha}.
\end{aligned}
\label{eq:Ak_estiamte_psi_requirements}
\eeq
\end{proposition}

\begin{proof}
In the $\NN$ case, the estimate follows directly by adapting \eqref{eq:overdetermined_a_priori} to fit the overdetermined elliptic system in \propref{prop:new_estimate1}, once the sharp tuples have been identified and \propref{prop:sharp_tuples_OD} invoked. In the $\DD$ case, a similar adaption applies, with the additional observation that if $\Psi$ satisfies the specified conditions in \eqref{eq:Ak_estiamte_psi_requirements}, the summands in the a priori estimate involving the norm of $\bB_{\alpha} \Psi$ vanish.  
\end{proof}
The following proposition is proven similarly to \propref{prop:new_estimate1}, albeit using the second set of overdetermined ellipticities in \corrref{corr:overdetermined_ellipticity_corrected}:
\begin{proposition}
\label{prop:new_estimate2}
The following systems are overdetermined elliptic and injective:
\beq
\begin{aligned}
&\NN:\qquad &&\fbD_{\alpha}^* \fbD_{\alpha} \oplus \bB^*_{\alpha}\fbD_{\alpha}\oplus\tbP_{\alpha-1} \oplus \frakI_{\alpha}, \\
&\DD:\qquad &&\fbD_{\alpha}^* \fbD_{\alpha} \oplus \bB_{\alpha}\oplus \tbP_{\alpha-1} \oplus \frakI_{\alpha}.
\end{aligned}
\label{eq:order_reduced_elliptiices_proof2}
\eeq 
\end{proposition}
\begin{proof}
At this stage, we establish only the $\NN$ case, as it is clear how the $\DD$ case proceeds. As established in the proof of \propref{prop:new_estimate1}, we have 
\[
\fbD_{\alpha-1}^*\oplus\bB^*_{\alpha-1} =(\fbD_{\alpha-1}^*\oplus\bB^*_{\alpha-1})\tbP_{\alpha-1}
\]
allowing us to write
\[
\fbD_{\alpha}^* \fbD_{\alpha}\oplus\bB^*_{\alpha}\fbD_{\alpha} \oplus \fbD_{\alpha-1}^*\oplus\bB^*_{\alpha-1} \oplus \frakI_{\alpha} = (\id \oplus\id\oplus (\fbD^*_{\alpha-1}\oplus\bB_{\alpha-1}^*) \oplus \id)(\fbD_{\alpha}^* \fbD_{\alpha}\oplus\bB_{\alpha}^*\fbD_{\alpha} \oplus \tbP_{\alpha-1} \oplus \frakI_{\alpha}).
\]
Since the left-hand side has a left-inverse, it follows that the system
\[
\fbD_{\alpha}^* \fbD_{\alpha}\oplus\bB_{\alpha}^*\fbD_{\alpha} \oplus \tbP_{\alpha-1} \oplus \frakI_{\alpha}
\] 
also has a left-inverse, hence is overdetermined elliptic as well. 
\end{proof} 
By the listing of sharp tuples for the adapted Green system $\fbD_{\alpha}$ in \defref{def:adapting_operator}, $(M, T;\, 0, 0)$ are sharp tuples for $\fbD_{\alpha}$, and its adapted adjoint $\fbD^*_{\alpha}$ admits sharp tuples of the form $(0, 0;\, -M, -T)$. By the composition rules for sharp tuples (\corrref{corr:sharp_composition}), we therefore obtain the following:
\begin{corollary}
\label{corr:sharp_tuples_M_T} 
The systems:
\[
\begin{aligned}
&\NN: \quad && \fbD_{\alpha}^* \fbD_{\alpha} \oplus \bB^*_{\alpha} \fbD_{\alpha}, \\
&\DD: \quad && \fbD_{\alpha}^* \fbD_{\alpha} \oplus \bB_{\alpha}.
\end{aligned}
\]
admit sharp tuples of the form $(M, T;\, (-M, 0),\, (-T, C))$. Thus, for every $s \in \Nzero$, the tuples $(M+s, T+s;\, (-M+s, 0),\, (-T+s, C+s))$ are sharp for these systems, generated by composing the sharp tuples $(M+s, T+s; s, s)$ for $\fbD_{\alpha}$ and $(s, s; -M+s, -M+s)$ for $\fbD_{\alpha}^*$.
\end{corollary}
This corollary allows for a more general discussion of settings where $(J, L; I, K)$ are sharp tuples for $\fbD_{\alpha}$ with the property that simultaneously there exist $(I, K; II, KK)$ sharp tuples for $\fbD_{\alpha}^*$. Most of our analytical derivations will be valid for any such sharp tuples.
\begin{proposition}
Let $(J,L;I,K)$ be sharp tuples for $\fbD_{\alpha}$ such that there are $(I,K;II,KK)$ sharp tuples for $\fbD_{\alpha}^*$ also suitable for $\bB_{\alpha}^*$. Then the following estimates hold: 
\beq
\begin{aligned}
&\NN:\quad&&\|\Psi\|_{J, L, p} &&\lesssim \|\fbD_{\alpha}^*\fbD_{\alpha} \Psi\|_{II,KK, p}+\|\bB_{\alpha}^*\fbD_{\alpha} \Psi\|_{0, KK', p}\\& && &&\quad+ \|\tbP_{\alpha-1} \Psi\|_{J, L, p} + \|\frakI_{\alpha} \Psi\|_{0, 0, p},\\\\
&\DD:\quad&&\|\Psi\|_{J, L, p} &&\lesssim \|\fbD_{\alpha}^*\fbD_{\alpha}\Psi\|_{II, KK, p}+ \|\tbP_{\alpha-1} \Psi\|_{J, L, p} + \|\frakI_{\alpha} \Psi\|_{0, 0, p},\\
\label{eq:D^*D_estimate}
\end{aligned}
\eeq
valid for any $\Psi$ belonging to:
\[
\begin{aligned}
&\NN:\qquad &&\Psi \in W_p^{J, L}(\bbF_{\alpha}; \bbG_{\alpha}), \\
&\DD:\qquad &&\Psi \in W_p^{J, L}(\bbF_{\alpha}; \bbG_{\alpha}) \cap \ker\bB_{\alpha}.
\end{aligned}
\]
\end{proposition}
\begin{proof}
The estimate \eqref{eq:D^*D_estimate} follows from the overdetermined ellipticities in \propref{prop:new_estimate2}, once the sharp tuples for the corresponding systems have been identified by composition, direct summation, and the fact that $\tbP_{\alpha-1} \in \OP(0,0)$ (\propref{prop:G0tuples}).  
\end{proof}

\subsection{Stage 3: Closed range arguments and a priori estimates}
Due to \propref{prop:NDR}, we have that for every tuples of real numbers $S, T$ with $\min(S,T) \geq 0$:
\[
\begin{aligned}
&\NN:\qquad && \module_{\N}^{\alpha} \subseteq \scrN^{S,T}_{p}(\fbD_{\alpha-1}^*, \bB^*_{\alpha-1}), \\
&\DD:\qquad && \module_{\D}^{\alpha} \subseteq \scrN^{S,T}_{p}(\fbD_{\alpha-1}^*).
\end{aligned}
\]

Thus, writing $\id = (\id - \frakI_{\alpha}) + \frakI_{\alpha}$ when restricted to the spaces on the right-hand side, we obtain the topologically direct splittings:
\[
\begin{aligned}
&\NN:\qquad && \scrN^{S,T}_{p}(\fbD_{\alpha-1}^*, \bB^*_{\alpha-1}) = \scrN_{p,\bot}^{S,T}(\fbD_{\alpha-1}^*, \bB^*_{\alpha-1}) \oplus \module_{\N}^{\alpha}, \\
&\DD:\qquad && \scrN^{S,T}_{p}(\fbD_{\alpha-1}^*) = \scrN_{p,\bot}^{S,T}(\fbD_{\alpha-1}^*) \oplus \module_{\D}^{\alpha}.
\end{aligned}
\]

These decompositions are $L^2$-orthogonal for every $p \geq 2$ and any such $S, T$. Since they hold for every $S,T$, by the graded Fréchet structure, this further yields $L^2$-orthogonal topologically-direct decompositions of Fréchet spaces:
\[
\begin{aligned}
&\NN:\qquad && \scrN(\fbD_{\alpha-1}^*, \bB^*_{\alpha-1}) = \scrN_{\bot}(\fbD_{\alpha-1}^*, \bB^*_{\alpha-1}) \oplus \module_{\N}^{\alpha}, \\
&\DD:\qquad && \scrN(\fbD_{\alpha-1}^*) = \scrN_{\bot}(\fbD_{\alpha-1}^*) \oplus \module_{\D}^{\alpha}.
\end{aligned}
\]

Combined with the auxiliary decomposition \eqref{eq:aux_induction} induced by $\fbD_{\alpha}$, we therefore obtain the following topologically-direct, $L^2$-orthogonal decompositions:
\beq
\begin{aligned}
&\NN:\qquad && \Gamma(\bbF_{\alpha}; \bbG_{\alpha}) = \scrR(\fbD_{\alpha-1}) \oplus \scrN_{\bot}(\fbD_{\alpha-1}^*, \bB^*_{\alpha-1}) \oplus \module_{\N}^{\alpha}, \\
&\DD:\qquad && \Gamma(\bbF_{\alpha}; \bbG_{\alpha}) = \scrR(\fbD_{\alpha-1}; \bB_{\alpha-1}) \oplus \scrN_{\bot}(\fbD_{\alpha-1}^*) \oplus \module_{\D}^{\alpha}.
\end{aligned}
\label{eq:N0spdecomp}
\eeq
In both cases, the projection onto the middle summand is given by the map
\[
\tbP_{\alpha-1}^\bot = (\id - \frakI_{\alpha})(\id - \tbP_{\alpha-1}).
\]
By the composition rules of Green operators, $\tbP_{\alpha-1}^\bot \in \OP(0,0)$. Since all projections onto the closed subspaces in \eqref{eq:N0spdecomp} are in $\OP(0,0)$, it follows from a density/continuity argument similar to that of \lemref{lem:Wspaux} that:
\beq
\begin{aligned}
&\NN:\qquad && W^{S,T}_{p}(\bbF_{\alpha}; \bbG_{\alpha}) = \overline{\scrR^{S,T}_p(\fbD_{\alpha-1})} \oplus \scrN_{p,\bot}^{S,T}(\fbD_{\alpha-1}^*, \bB^*_{\alpha-1}) \oplus \module^{\alpha}_{\N}, \\
&\DD:\qquad && W^{S,T}_{p}(\bbF_{\alpha}; \bbG_{\alpha}) = \overline{\scrR^{S,T}_p(\fbD_{\alpha-1}; \bB_{\alpha-1})} \oplus \scrN_{p,\bot}^{S,T}(\fbD_{\alpha-1}^*) \oplus \module^{\alpha}_{\D}.
\end{aligned}
\label{eq:WspN0}
\eeq

Finally, for every $S, T$ as above we note that:
\beq 
\begin{aligned}
&\NN:\qquad && \scrN(\fbD_{\alpha-1}^*, \bB^*_{\alpha-1}) \hookrightarrow \scrN^{S,T}_p(\fbD_{\alpha-1}^*, \bB^*_{\alpha-1}), \\
&\DD:\qquad && \scrN(\fbD_{\alpha-1}^*) \hookrightarrow \scrN^{S,T}_p(\fbD_{\alpha-1}^*).
\end{aligned}
\label{eq:first_inclusions} 
\eeq 
\begin{lemma}
\label{lem:density_Nsp}
The continuous inclusions
\[
\begin{aligned}
&\NN:\qquad &&\scrN_{\bot}(\fbD_{\alpha-1}^*, \bB^*_{\alpha-1}) \hookrightarrow \scrN^{S,T}_{\bot, p}(\fbD_{\alpha-1}^*, \bB^*_{\alpha-1}), \\
&\DD:\qquad &&\scrN_{\bot}(\fbD_{\alpha-1}^*) \hookrightarrow \scrN^{S,T}_{\bot, p}(\fbD_{\alpha-1}^*)
\end{aligned}
\]
are dense.
\end{lemma}
\begin{proof}
The required inclusions are obtained by applying the projection $\id - \frakI_{\alpha}$ to both sides of \eqref{eq:first_inclusions}. For density, let $\Psi$ be an element of either $\scrN^{S,T}_{\bot, p}(\fbD_{\alpha-1}^*, \bB^*_{\alpha-1})$ in the $\NN$ case or $\scrN^{S,T}_{\bot, p}(\fbD_{\alpha-1}^*)$ in the $\DD$ case. Since $\Gamma(\bbF_{\alpha}; \bbG_{\alpha})$ is dense in $W^{S,T}_{p}(\bbF_{\alpha}; \bbG_{\alpha})$, there exists a sequence $\Psi_{n} \in \Gamma(\bbF_{\alpha}; \bbG_{\alpha})$ such that:
\[
\Psi_{n} \to \Psi, \qquad \text{in } W^{S,T}_{p}.
\]
Since $\tbP_{\alpha-1}^{\bot} \in \OP(0, 0)$ is $W^{S,T}_{p}$-continuous due to \corrref{corr:G0props}, the map $\tbP_{\alpha-1}^{\bot}: W^{S,T}_{p}(\bbF_{\alpha}; \bbG_{\alpha}) \to W^{S,T}_{p}(\bbF_{\alpha}; \bbG_{\alpha})$ is, by construction, the projection onto $\scrN^{S,T}_{\bot, p}(\fbD_{\alpha-1}^*, \bB^*_{\alpha-1})$ in the $\NN$ case or $\scrN^{S,T}_{\bot, p}(\fbD_{\alpha-1}^*)$ in the $\DD$ case. Therefore, $\tbP_{\alpha-1}^{\bot} \Psi = \Psi$, and by continuity:
\[
\begin{aligned}
&\NN:\qquad && \scrN_{\bot}(\fbD_{\alpha-1}^*, \bB^*_{\alpha-1}) \ni \tbP_{\alpha-1}^{\bot} \Psi_n \to \tbP_{\alpha-1}^{\bot} \Psi = \Psi, \\
&\DD:\qquad && \scrN_{\bot}(\fbD_{\alpha-1}^*) \ni \tbP_{\alpha-1}^{\bot} \Psi_n \to \tbP_{\alpha-1}^{\bot} \Psi = \Psi,
\end{aligned}
\]
which completes the proof.
\end{proof}
In the Dirichlet case, we have an additional refinement, based on \eqref{eq:aux_induction_Refined} in the induction hypothesis:
\begin{lemma}
\label{lem:N_Dirichlet_condition_induction}
In the $\DD$ case, for $(S, T)$ suitable for $\bB_{\alpha}$, if $\Psi \in W^{S,T}_{p}(\bbF_{\alpha}; \bbG_{\alpha}) \cap \ker\bB_{\alpha}$, then $\tbP_{\alpha-1}^{\bot} \Psi \in W^{S,T}_{p}(\bbF_{\alpha}; \bbG_{\alpha}) \cap \ker\bB_{\alpha}$. Hence, there is an additional dense inclusion:
\[
\scrN_{\bot}(\fbD_{\alpha-1}^*) \cap \ker\bB_{\alpha} \hookrightarrow \scrN^{S,T}_{\bot, p}(\fbD_{\alpha-1}^*) \cap \ker\bB_{\alpha}.
\]
\end{lemma}

\begin{proof}
In the $\DD$ case, since $\module_{\D}^{\alpha} \subseteq \ker \bB_{\alpha}$, the refined decomposition in \eqref{eq:aux_induction_Refined} reads that, when restricted to $\ker \bB_{\alpha}$, $\tbP^{\bot}_{\alpha-1}$ becomes the projection onto $(\scrN_{\bot}(\fbD_{\alpha-1}^*)\cap\ker\bB_{\alpha}) \oplus\module_{\D}^{\alpha}$. The Sobolev version then follows by continuity, and the density argument proceeds as in the proof of \lemref{lem:density_Nsp}.
\end{proof}

Recall the discussion surrounding \eqref{eq:range_adapted_Green_system}--\eqref{eq:range_green_closure}, and in particular the convention that we use the notation $\scrR^{I, K}_{p}(\fbD_{\alpha})$ and $\scrR^{I, K}_{p}(\fbD_{\alpha}; \bB_{\alpha})$ when the corresponding ranges of the adapted Green systems are closed.  
\begin{proposition}
\label{prop:Ak_closed_range}
Let $(J,L;I,K)$ be sharp tuples for $\fbD_{\alpha}$ also suitable for $\bB_{\alpha}$. Then the following subspaces are closed:
\beq
\begin{aligned}
&\NN:\quad && \scrR^{I, K}_{p}(\fbD_{\alpha}) \subseteq W^{I, K}_{p}(\bbF_{\alpha+1}; \bbG_{\alpha+1}), \\
&\DD:\quad && \scrR^{I, K}_{p}(\fbD_{\alpha}; \bB_{\alpha}) \subseteq W^{I, K}_{p}(\bbF_{\alpha+1}; \bbG_{\alpha+1}).
\end{aligned}
\label{eq:closed_ranges_proof}
\eeq
Moreover, the following estimate holds:
\beq
\|\Psi\|_{J, L, p} \lesssim \|\fbD_{\alpha} \Psi\|_{I, K, p},
\label{eq:estimateNsp}
\eeq
for any $\Psi$ that belongs to: 
\[
\begin{aligned}
&\NN:\quad && \Psi \in \scrN^{J, L}_{p, \bot}(\fbD_{\alpha-1}^*, \bB^*_{\alpha-1}), \\
&\DD:\quad && \Psi \in \scrN^{J, L}_{p, \bot}(\fbD_{\alpha-1}^*) \cap \ker\bB_{\alpha}.
\end{aligned}
\]
\end{proposition}

\begin{proof}
Consider the decompositions in \eqref{eq:WspN0}. In the $\DD$ case, combining \lemref{lem:N_Dirichlet_condition_induction} with \eqref{eq:WspN0}, the refined decomposition \eqref{eq:aux_induction_Refined} completes into the Sobolev version: 
\[
W^{J, L}_{p}(\bbF_{\alpha}; \bbG_{\alpha}) \cap \ker\bB_{\alpha} = \overline{\scrR^{J, L}_{p}(\fbD_{\alpha-1}; \bB_{\alpha-1})} \oplus (\scrN^{J, L}_{p, \bot}(\fbD_{\alpha-1}^*) \cap \ker\bB_{\alpha}) \oplus \module_{\D}^{\alpha}.
\]
Using the relations in the induction hypothesis \eqref{eq:containments_induction} and continuity, we obtain that $\fbD_{\alpha}$ vanishes identically on the first and third summands of this decomposition. Hence:
\beq
\begin{aligned}
&\NN:\quad && \fbD_{\alpha}(W^{J, L}_{p}(\bbF_{\alpha}; \bbG_{\alpha})) = \fbD_{\alpha}(\scrN^{J, L}_{p, \bot}(\fbD_{\alpha-1}^*, \bB^*_{\alpha-1})), \\
&\DD:\quad && \fbD_{\alpha}(W^{J, L}_{p}(\bbF_{\alpha}; \bbG_{\alpha})\cap\ker\bB_{\alpha}) = \fbD_{\alpha}(\scrN^{J, L}_{p, \bot}(\fbD_{\alpha-1}^*) \cap \ker\bB_{\alpha}).
\end{aligned}
\label{eq:ranges_in_the_proof}
\eeq

But now, for $\Psi$ in the corresponding subspaces: 
\[
\begin{aligned}
&\NN:\quad && \Psi \in \scrN^{J, L}_{p, \bot}(\fbD_{\alpha-1}^*, \bB^*_{\alpha-1}), \\
&\DD:\quad && \Psi \in \scrN^{J, L}_{p, \bot}(\fbD_{\alpha-1}^*) \cap \ker\bB_{\alpha},
\end{aligned}
\]
it holds by construction that $\tbP_{\alpha-1}\Psi=0$ and $\frakI_{\alpha}\Psi=0$ (and $\bB_{\alpha}\Psi=0$ in the $\DD$ case), hence the elliptic estimate \eqref{eq:Ak_overdetermined_elliptic_estimate} reduces to \eqref{eq:estimateNsp}. This implies that the ranges on the right hand side in \eqref{eq:ranges_in_the_proof} are closed subspaces (e.g., by \cite[Prop.~6.7, p.~583]{Tay11a}). Comparing with the subspaces on the left-hand side of \eqref{eq:ranges_in_the_proof}, we conclude that the ranges in the statement of the proposition \eqref{eq:closed_ranges_proof} are indeed closed, completing the proof.
\end{proof}

Since the proposition holds for any sharp tuples $(J,L;I,K)$, by applying it to the sharp tuples in \corrref{corr:sharp_tuples_M_T} for every $s \in \Nzero$ and using the Sobolev grading of the graded Fréchet space $\Gamma(\bbF_{\alpha+1};\bbG_{\alpha+1})$, we find:
\begin{corollary}
The following subspaces are closed in the Fréchet topology:
\[
\begin{aligned}
&\NN:\quad && \scrR(\fbD_{\alpha}) \subseteq \Gamma(\bbF_{\alpha+1}; \bbG_{\alpha+1}), \\
&\DD:\quad && \scrR(\fbD_{\alpha}; \bB_{\alpha}) \subseteq \Gamma(\bbF_{\alpha+1}; \bbG_{\alpha+1}).
\end{aligned}
\]
\end{corollary}
We now arrive at the point where the assumption on the sharp tuples of $\fbD_{\alpha}$ as an adapted Green system \defref{def:adapting_operator} becomes essential, and which is absent in the setting of non-varying ordered studied in \cite{KL23}. Since $\fbD_{\alpha}$ admits sharp tuples of the form $(M, T;\, 0, 0)$, combining this with \propref{prop:Ak_closed_range} yields that the following subspaces are closed:
\beq
\begin{aligned}
&\NN:\quad && \scrR^{0,0}_{p}(\fbD_{\alpha}) \subseteq L^{p}(\bbF_{\alpha+1}; \bbG_{\alpha+1}), \\
&\DD:\quad && \scrR^{0,0}_{p}(\fbD_{\alpha}; \bB_{\alpha}) \subseteq L^{p}(\bbF_{\alpha+1}; \bbG_{\alpha+1}).
\end{aligned}
\label{eq:closed_0_spaces}
\eeq
Without the closedness of these subspaces, the proposition below, which forms the analytical heart of the induction step, would fail\footnote{This proposition generalizes \cite[Prop.~4.7--4.8]{KL23}. The lemmas there remain valid as stated, but the analytical argument in their proof is flawed: distributional limits were incorrectly identified with weak $L^2$ limits. Here we remedy this by a more careful treatment. With substitutions from \eqref{eq:original_adapted}, our argument also applies in the narrower scope of \cite{KL23}.}.

\begin{proposition}
\label{prop:AprioriAstarAk} 
Let $(J, L; I, K)$ be sharp tuples for $\fbD_{\alpha}$ such that there exist $(I, K; II, KK)$ sharp tuples for $\fbD^*_{\alpha}$ and $(I, K; 0, KK')$ sharp tuples for $\bB_{\alpha}^*$. Let $\Theta$ belong to one of the spaces:
\beq
\begin{aligned}
&\NN:\quad && \Theta \in \scrR^{0,0}_{p}(\fbD_{\alpha}), \\
&\DD:\quad && \Theta \in \scrR^{0,0}_{p}(\fbD_{\alpha}; \bB_{\alpha}).
\end{aligned}
\label{eq:Theta_spaces_proof}
\eeq
Suppose there exists $\Xi \in W^{I, K}_p(\bbF_{\alpha+1}; \bbG_{\alpha+1})$ such that:
\[
\begin{aligned}
&\NN:\quad && \Theta - \Xi \in \scrN^{0,0}_{p}(\fbD_{\alpha}^*,\bB_{\alpha}^*), \\
&\DD:\quad && \Theta - \Xi \in \scrN^{0,0}_{p}(\fbD_{\alpha}^*).
\end{aligned}
\]
Then there exists $\Psi$ with:
\beq
\begin{aligned}
&\NN:\quad && \Psi \in \scrN^{J, L}_{p, \bot}(\fbD_{\alpha-1}^*, \bB^*_{\alpha-1}), \\
&\DD:\quad && \Psi \in \scrN^{J, L}_{p, \bot}(\fbD_{\alpha-1}^*) \cap \ker\bB_{\alpha},
\end{aligned}
\label{eq:potential_proof}
\eeq
such that $\Theta = \fbD_{\alpha} \Psi$ and the following estimates hold:
\beq
\begin{aligned}
&\NN:\quad && \|\Psi\|_{J, L, p} \lesssim \|\fbD_{\alpha}^* \fbD_{\alpha} \Psi\|_{II, KK, p} + \|\bB_{\alpha}^* \fbD_{\alpha} \Psi\|_{0, KK', p}, \\
&\DD:\quad && \|\Psi\|_{J, L, p} \lesssim \|\fbD_{\alpha}^* \fbD_{\alpha} \Psi\|_{II, KK, p}.
\end{aligned}
\label{eq:estimateN0p2}
\eeq
\end{proposition}
\begin{proof}
Since $(J, L; I, K)$ are sharp tuples for $\fbD_{\alpha}$, $(I, K; II, KK)$ are sharp tuples for $\fbD_{\alpha}^*$, and $(I, K; 0, KK')$ are sharp tuples for $\bB_{\alpha}^*$, it follows that
\[
(J, L; II, (KK, KK'))
\]
are sharp tuples for the system $\fbD_{\alpha}^* \fbD_{\alpha} \oplus \bB_{\alpha}^* \fbD_{\alpha}$.

By the assumption on $\Theta$ in \eqref{eq:Theta_spaces_proof} and the fact that the smooth version of \eqref{eq:ranges_in_the_proof} reads as
\[
\begin{aligned}
&\NN:\quad && \scrR(\fbD_{\alpha}) = \fbD_{\alpha}\big(\scrN_{\bot}(\fbD_{\alpha-1}^*, \bB^*_{\alpha-1})\big), \\
&\DD:\quad && \scrR(\fbD_{\alpha}; \bB_{\alpha}) = \fbD_{\alpha}\big(\scrN_{\bot}(\fbD_{\alpha-1}^*) \cap \ker\bB_{\alpha}\big),
\end{aligned}
\]
together with the density inclusions in \lemref{lem:density_Nsp}--\lemref{lem:N_Dirichlet_condition_induction}, we find that there exists a sequence $(\Psi_n)$ of smooth sections in:
\beq
\begin{aligned}
&\NN:\quad && (\Psi_n) \subset \scrN_{\bot}(\fbD_{\alpha-1}^*, \bB^*_{\alpha-1}), \\
&\DD:\quad && (\Psi_n) \subset \scrN_{\bot}(\fbD_{\alpha-1}^*) \cap \ker\bB_{\alpha},
\end{aligned}
\label{eq:approximating_sequence_prop_proof}
\eeq
such that $\fbD_{\alpha} \Psi_n \to \Theta$ in $L^{p}$.

Using \eqref{eq:Ak_overdetermined_elliptic_estimate} for the sharp tuples $(M, T;\, 0, 0)$ of $\fbD_{\alpha}$ and the closedness of the spaces in \eqref{eq:closed_0_spaces}, we may assume that $\Psi_n \to \Psi$ in $L^{p}$ and $\Theta = \fbD_{\alpha} \Psi$.

Let $\frakS_{\R}$, $\R \in \BRK{\NN,\DD}$, be the corresponding left inverses of the injective systems from \propref{prop:new_estimate2}. By the construction of the approximating sequence $(\Psi_{n})$ above, we have: 
\beq 
\begin{aligned} 
&\NN:\quad && \frakS_{\NN}(\fbD_{\alpha}^* \fbD_{\alpha} \Psi_n, \bB^*_{\alpha} \fbD_{\alpha} \Psi_{n}, 0, 0) = \Psi_{n}, \\
&\DD:\quad && \frakS_{\DD}(\fbD_{\alpha}^* \fbD_{\alpha} \Psi_n, 0, 0, 0) = \Psi_{n}.
\end{aligned} 
\label{eq:inverse_analytical_lemma}
\eeq 

Restricting these inverses to the relevant arguments gives the systems:
\[
\begin{aligned} 
&\NN:\quad && \frakS_{\NN}(\cdot,\cdot,0,0): \Gamma(\bbF_{\alpha};\bbG_{\alpha}) \oplus \Gamma(0;\bbL_{\alpha}) \to \Gamma(\bbF_{\alpha};\bbG_{\alpha}),\\
&\DD:\quad && \frakS_{\DD}(\cdot,0;0,0): \Gamma(\bbF_{\alpha};\bbG_{\alpha}) \to \Gamma(\bbF_{\alpha};\bbG_{\alpha}).
\end{aligned} 
\]

In the $\NN$ case, $\fbD_{\alpha}^*\fbD_{\alpha}$ has sharp tuples $(M,T;-M,-T)$, and $\bB^*_{\alpha}\fbD_{\alpha}$ takes values in $\bbL_{\alpha}$ over the boundary $\dM$, where there is no class restriction. By the properties of the left inverse from \thmref{thm:overdetermined_varying_orders}, for sufficiently small tuples $(S,T)$ with $\min(S,T) < 0$, we have the \emph{lenient} mapping property:
\[
\begin{aligned} 
&\NN:\quad && \frakS_{\NN}(\cdot,\cdot,0,0): W^{-M,-T}_{p}(\bbF_{\alpha};\bbG_{\alpha}) \oplus W_p^{0,-C}(0;\bbL_{\alpha}) \to W^{S,T}(\bbF_{\alpha};\bbG_{\alpha}),\\
&\DD:\quad && \frakS_{\DD}(\cdot,0;0,0): W^{-M,-T}_{p}(\bbF_{\alpha};\bbG_{\alpha}) \to W_p^{S,T}(\bbF_{\alpha};\bbG_{\alpha}).
\end{aligned} 
\]
We emphasize again that this lenient mapping property guarantees continuity between Sobolev norms but does not imply mapping into a space sufficient to invert the original problem (which would require a sharp mapping property). The mapping may even be compact.

Now, by the assumptions in the proposition, we see that the conditions of \lemref{lem:D'mapping} are satisfied for the adapted Green system $\fbD_{\alpha}$, with $\Theta_n = \fbD_{\alpha} \Psi_n$ and $\Theta, \Xi$ as stated. This yields:
\[
\lim_{n \to \infty} \fbD_{\alpha}^* \fbD_{\alpha} \Psi_n = \fbD_{\alpha}^* \Xi, \quad \text{in } W^{-M,-T}_{p},
\]
and, in the $\NN$ case,
\[
\NN:\quad \lim_{n \to \infty} \bB^*_{\alpha} \fbD_{\alpha} \Psi_{n}=\bB^*_{\alpha} \fbD_{\alpha} \Psi \quad \text{in } W^{0,-C}_{p}.
\]

By continuity of $\frakS_{\R}$, we obtain in $W^{S,T}_{p}$:
\[
\begin{aligned} 
&\NN:\quad && \lim_{n \to \infty} \frakS_{\NN}(\fbD_{\alpha}^* \fbD_{\alpha} \Psi_n,\, \bB^*_{\alpha} \fbD_{\alpha} \Psi_{n}, 0, 0) 
= \frakS_{\NN}(\fbD_{\alpha}^* \Xi,\, \bB^*_{\alpha} \Xi, 0, 0), \\
&\DD:\quad && \lim_{n \to \infty} \frakS_{\DD}(\fbD_{\alpha}^* \fbD_{\alpha} \Psi_n, 0, 0, 0) 
= \frakS_{\DD}(\fbD_{\alpha}^* \Xi, 0, 0, 0).
\end{aligned} 
\]

Comparing with \eqref{eq:inverse_analytical_lemma} and using uniqueness of limits (since $L^{p}$-convergence is stronger than the $W^{S,T}_{p}$-convergence) gives:
\[
\begin{aligned} 
&\NN:\quad && \frakS_{\NN}(\fbD_{\alpha}^* \Xi,\, \bB^*_{\alpha} \Xi, 0, 0) = \Psi, \\
&\DD:\quad && \frakS_{\DD}(\fbD_{\alpha}^* \Xi, 0, 0, 0) = \Psi.
\end{aligned} 
\]

Finally, since $\fbD_{\alpha}^* \Xi \in W^{II,KK}_{p}$ and $\bB_{\alpha}^* \Xi \in W^{0,KK'}_{p}$ by assumption, and given the specification of the relevant sharp tuples made earlier for the corresponding injective problems in \propref{prop:new_estimate2}, the properties of the left inverse imply that $\Psi \in W^{J,L}_p$ and, in particular, that it satisfies \eqref{eq:potential_proof}. The estimates in \eqref{eq:D^*D_estimate} then follow, yielding \eqref{eq:estimateN0p2} and thus completing the proof.
\end{proof}

\subsection{Stage 4: The auxiliary decomposition}
We now apply \propref{prop:L2_aux_adapted_pair} to the adapted Green system $\fbD_{\alpha}$. Since the ranges in \eqref{eq:closed_0_spaces} are closed, the particular case $p=2$ yields the $L^{2}$-orthogonal decompositions:
\beq
\begin{aligned}
&\NN:\quad && L^{2}(\bbF_{\alpha+1}; \bbG_{\alpha+1}) = \scrR^{0,0}_2(\fbD_{\alpha}) \oplus \scrN^{0,0}_2(\fbD_{\alpha}^*, \bB^*_{\alpha}), \\
&\DD:\quad && L^{2}(\bbF_{\alpha+1}; \bbG_{\alpha+1}) =\scrR^{0,0}_2(\fbD_{\alpha}; \bB_{\alpha})\oplus \scrN^{0,0}_2(\fbD_{\alpha}^*).
\end{aligned}
\label{eq:L2_splitting_proof}
\eeq 

\begin{proposition}
\label{prop:aux_first_step}
There exists an $L^{2}$-orthogonal, topologically-direct decomposition of Fréchet spaces:
\beq
\begin{aligned}
&\NN:\quad && \Gamma(\bbF_{\alpha+1}; \bbG_{\alpha+1}) = \scrR(\fbD_{\alpha}) \oplus \scrN(\fbD^*_{\alpha}, \bB_{\alpha}^*), \\
&\DD:\quad && \Gamma(\bbF_{\alpha+1}; \bbG_{\alpha+1}) = \scrR(\fbD_{\alpha}; \bB_{\alpha}) \oplus \scrN(\fbD^*_{\alpha}).
\end{aligned}
\label{eq:aux_decompositionAksmooth}
\eeq
Moreover, the continuous projection associated with these decompositions, onto either $\scrR(\fbD_{\alpha})$ or $\scrR(\fbD_{\alpha}; \bB_{\alpha})$, extends continuously to the $L^{2}$-orthogonal projection onto either $\scrR^{0,0}_{2}(\fbD_{\alpha})$ or $\scrR^{0,0}_{2}(\fbD_{\alpha};\bB^*_{\alpha})$, respectively, in the decomposition \eqref{eq:L2_splitting_proof}.
\end{proposition}

\begin{proof}
Let $(J, L; I, K)$ be sharp tuples for $\fbD_{\alpha}$ as in \propref{prop:AprioriAstarAk}. On the one hand, the following subspaces are closed due to \propref{prop:Ak_closed_range}:
\[
\begin{aligned}
&\NN:\quad && \scrR^{I, K}_2(\fbD_{\alpha}) \subseteq W^{I, K}_2(\bbF_{\alpha+1}; \bbG_{\alpha+1}), \\
&\DD:\quad && \scrR^{I, K}_2(\fbD_{\alpha}; \bB_{\alpha}) \subseteq W^{I, K}_2(\bbF_{\alpha+1}; \bbG_{\alpha+1}).
\end{aligned}
\]
On the other hand, the following subspaces are closed as they are kernels of continuous linear maps:
\[
\begin{aligned}
&\NN:\quad && \scrN^{I, K}_2(\fbD^*_{\alpha}, \bB^*_{\alpha})\subseteq W^{I, K}_2(\bbF_{\alpha+1}; \bbG_{\alpha+1}), \\
&\DD:\quad && \scrN^{I, K}_2(\fbD^*_{\alpha})\subseteq W^{I, K}_2(\bbF_{\alpha+1}; \bbG_{\alpha+1}).
\end{aligned}
\]
Hence, together with the containments:
\[
\begin{aligned}
&\NN:\quad && \scrR^{I, K}_2(\fbD_{\alpha}) \subseteq \scrR^{0, 0}_{2}(\fbD_{\alpha}), &&&&\scrN^{I, K}_2(\fbD_{\alpha}^*, \bB^*_{\alpha}) \subseteq \scrN^{0, 0}_2(\fbD_{\alpha}^*, \bB^*_{\alpha}), \\
&\DD:\quad && \scrR^{I, K}_2(\fbD_{\alpha}; \bB_{\alpha}) \subseteq \scrR^{0, 0}_{2}(\fbD_{\alpha}; \bB_{\alpha}), &&&&\scrN^{I, K}_2(\fbD_{\alpha}^*) \subseteq \scrN^{0, 0}_2(\fbD_{\alpha}^*),
\end{aligned}
\]
one finds that these subspaces are closed, intersect trivially, and are mutually $L^{2}$-orthogonal.

Thus, to prove:
\beq
\begin{aligned}
&\NN:\quad && W^{I, K}_{2}(\bbF_{\alpha+1}; \bbG_{\alpha+1}) = \scrR^{I, K}_2(\fbD_{\alpha}) \oplus \scrN^{I, K}_2(\fbD_{\alpha}^*, \bB^*_{\alpha}), \\
&\DD:\quad && W^{I, K}_{2}(\bbF_{\alpha+1}; \bbG_{\alpha+1}) = \scrR^{I, K}_2(\fbD_{\alpha}; \bB_{\alpha}) \oplus \scrN^{I, K}_2(\fbD_{\alpha}^*),
\end{aligned}
\label{eq:this_holds}
\eeq
it remains to show that the sum of spaces in each decomposition exhausts the whole of $W^{I, K}_{2}(\bbF_{\alpha+1}; \bbG_{\alpha+1})$. By the Sobolev grading of the graded Fréchet space $\Gamma(\bbF_{\alpha+1}; \bbG_{\alpha+1})$ and \corrref{corr:sharp_tuples_M_T}, if \eqref{eq:this_holds} holds for every sharp tuple $(J, L; I, K)$ for $\fbD_{\alpha}$, then \eqref{eq:aux_decompositionAksmooth} holds as well.

To prove this exhaustion, let $\Xi \in W^{I, K}_{2}(\bbF_{\alpha+1}; \bbG_{\alpha+1})$. Decompose it as an element in $L^{2}(\bbF_{\alpha+1}; \bbG_{\alpha+1})$ according to \eqref{eq:L2_splitting_proof}:
\[
\Xi = \Theta + \Phi,
\]
where:
\[
\begin{aligned}
&\NN:\qquad && \Theta \in \scrR^{0, 0}_2(\fbD_{\alpha}), \quad &&&&\Phi \in \scrN^{0, 0}_2(\fbD_{\alpha}^*, \bB^*_{\alpha}), \\
&\DD:\quad && \Theta \in \scrR^{0, 0}_2(\fbD_{\alpha}; \bB_{\alpha}), \qquad &&&&\Phi \in \scrN^{0, 0}_2(\fbD_{\alpha}^*).
\end{aligned}
\]
Since $\Xi \in W^{I, K}_2(\bbF_{\alpha+1}; \bbG_{\alpha+1})$, and $\Phi = \Xi - \Theta$ is in the corresponding kernel space, \propref{prop:AprioriAstarAk} applies, yielding $\Theta = \fbD_{\alpha} \Psi$ for some $\Psi \in W^{J, L}_2(\bbF_{\alpha}; \bbG_{\alpha})$ (with $\bB_{\alpha} \Psi = 0$ in the $\DD$ case). Therefore:
\[
\begin{aligned}
&\NN:\quad && \Phi \in \scrN^{0, 0}_2(\fbD_{\alpha}^*, \bB^*_{\alpha}) \cap W^{I, K}_2(\bbF_{\alpha+1}; \bbG_{\alpha+1}) = \scrN^{I, K}_2(\fbD_{\alpha}^*, \bB^*_{\alpha}), \\
&\DD:\quad && \Phi \in \scrN^{0, 0}_2(\fbD_{\alpha}^*) \cap W^{I, K}_2(\bbF_{\alpha+1}; \bbG_{\alpha+1}) = \scrN^{I, K}_2(\fbD_{\alpha}^*).
\end{aligned}
\]
This completes the proof. The $L^{2}$-continuity of the projections follows directly from this construction.
\end{proof}

\begin{theorem}
\label{thm:aux_induction_step}
$\fbD_{\alpha}$ induces a Neumann auxiliary decomposition in the $\NN$ case, and a Dirichlet auxiliary decomposition in the $\DD$ case. 
\end{theorem}

\begin{proof}
By surveying the requirements for the induction step \eqref{eq:aux_induction_step} and comparing them with the results of \propref{prop:aux_first_step}, it remains to establish the existence of a balance for $\fbD_{\alpha}$,
\[
\bP_{\alpha}: \Gamma(\bbF_{\alpha+1}; \bbG_{\alpha+1}) \to \Gamma(\bbF_{\alpha}; \bbG_{\alpha}),
\]
as specified in \defref{def:aux_decomposition}--\defref{def:aux_decompositionD}, such that
\[
\tbP_{\alpha} := \fbD_{\alpha} \bP_{\alpha} \in \OP(0,0)
\]
is the continuous projection onto the closed ranges in the decompositions \eqref{eq:aux_decompositionAksmooth}:
\[
\begin{aligned}
&\NN:\quad && \scrR(\fbD_{\alpha}) \subseteq \Gamma(\bbF_{\alpha+1};\bbG_{\alpha+1}), \\
&\DD:\quad && \scrR(\fbD_{\alpha}; \bB_{\alpha}) \subseteq \Gamma(\bbF_{\alpha+1};\bbG_{\alpha+1}).
\end{aligned}
\]

We begin by noting that the topologically direct decompositions \eqref{eq:aux_decompositionAksmooth} already provide projections onto the corresponding ranges,
\[
\tbP_{\alpha}: \Gamma(\bbF_{\alpha+1}; \bbG_{\alpha+1}) \to \Gamma(\bbF_{\alpha+1}; \bbG_{\alpha+1}).
\]
These projections are continuous in the Fréchet topology, though not yet known to belong to the calculus. By \propref{prop:aux_first_step}, however, they extend to the $L^{2}$-orthogonal projections
\beq 
\tbP_{\alpha}: L^{2}(\bbF_{\alpha+1}; \bbG_{\alpha+1}) \to L^{2}(\bbF_{\alpha+1}; \bbG_{\alpha+1}).
\label{eq:tbP_aux_proof} 
\eeq 
With this in mind, using the containment relations \eqref{eq:containments_induction} from the induction hypothesis and the decompositions \eqref{eq:N0spdecomp} from the previous level, we find that for every sharp tuple $(J,L;I,K)$ for $\fbD_{\alpha}$,
\[
\begin{aligned}
&\NN:\quad && \scrR_{p}^{I,K}(\fbD_{\alpha})
= \fbD_{\alpha}\big(\scrN^{J,L}_{p,\bot}(\fbD_{\alpha-1}^*, \bB^*_{\alpha-1})\big), \\
&\DD:\quad && \scrR_{p}^{I,K}(\fbD_{\alpha}; \bB_{\alpha})
= \fbD_{\alpha}\big(\scrN^{J,L}_{p,\bot}(\fbD_{\alpha-1}^*) \cap \ker \bB_{\alpha}\big).
\end{aligned}
\]

Together with the estimate \eqref{eq:estimateNsp}, applied at each Sobolev level, this shows that $\fbD_{\alpha}$ restricts to a bijection onto the above closed subspaces. By the open mapping theorem, we therefore obtain isomorphisms of Banach spaces with continuous inverses:
\[
\begin{aligned}
&\NN:\quad && (\fbD_{\alpha})^{-1}: \scrR_{p}^{I,K}(\fbD_{\alpha})
\to \scrN^{J,L}_{p,\bot}(\fbD_{\alpha-1}^*, \bB^*_{\alpha-1}), \\
&\DD:\quad && (\fbD_{\alpha})^{-1}: \scrR_{p}^{I,K}(\fbD_{\alpha}; \bB_{\alpha})
\to \scrN^{J,L}_{p,\bot}(\fbD_{\alpha-1}^*) \cap \ker\bB_{\alpha}.
\end{aligned}
\]
By the graded Fréchet structure and the validity of these statements for all sharp tuples $(J,L;I,K)$, we therefore obtain isomorphisms of Fréchet spaces:
\[
\begin{aligned}
&\NN:\quad && (\fbD_{\alpha})^{-1}: \scrR(\fbD_{\alpha})
\to \scrN_{\bot}(\fbD_{\alpha-1}^*, \bB^*_{\alpha-1}), \\
&\DD:\quad && (\fbD_{\alpha})^{-1}: \scrR(\fbD_{\alpha}; \bB_{\alpha})
\to \scrN_{\bot}(\fbD_{\alpha-1}^*) \cap \ker\bB_{\alpha}.
\end{aligned}
\]

This allows us to define a continuous linear map
\[
\bP_{\alpha}: \Gamma(\bbF_{\alpha+1};\bbG_{\alpha+1}) \to \Gamma(\bbF_{\alpha};\bbG_{\alpha})
\]
by
\beq 
\bP_{\alpha} := (\fbD_{\alpha})^{-1} \tbP_{\alpha}.
\label{eq:bP_def_proof}
\eeq 
This definition is well posed since, in both the $\NN$ and $\DD$ cases, the range of $\tbP_{\alpha}$ lies in the domain of $(\fbD_{\alpha})^{-1}$ (in the Dirichlet case by \lemref{lem:N_Dirichlet_condition_induction}). As a composition of continuous maps, $\bP_{\alpha}$ is continuous, and by construction $\tbP_{\alpha} = \fbD_{\alpha} \bP_{\alpha}$.

For the special case of sharp tuples $(M,T;0,0)$ for $\fbD_{\alpha}$ from \corrref{corr:sharp_tuples_M_T}, we have
\[
\begin{aligned}
&\NN:\quad && (\fbD_{\alpha})^{-1}: \scrR_{2}^{0,0}(\fbD_{\alpha})
\to \scrN^{M,T}_{2,\bot}(\fbD_{\alpha-1}^*, \bB^*_{\alpha-1}), \\
&\DD:\quad && (\fbD_{\alpha})^{-1}: \scrR_{2}^{0,0}(\fbD_{\alpha}; \bB_{\alpha})
\to \scrN^{M,T}_{2,\bot}(\fbD_{\alpha-1}^*) \cap \ker\bB_{\alpha}.
\end{aligned}
\]
In view of \eqref{eq:tbP_aux_proof}, it follows by composition that
\[
\bP_{\alpha}: L^{2}(\bbF_{\alpha+1}; \bbG_{\alpha+1}) \to W^{M,T}_{2}(\bbF_{\alpha}; \bbG_{\alpha})
\]
is continuous. Hence, to complete the argument, it remains to show that $\bP_{\alpha}$ and $\tbP_{\alpha}$ belong to the calculus. Indeed, once this is established, the identity $\tbP_{\alpha}=\fbD_{\alpha}\bP_{\alpha}$, together with the mapping property \eqref{eq:tbP_aux_proof}, implies—by \propref{prop:G0_criteria} applied to $\tbP_{\alpha}$ with $S=T=0$ and $m=0$—that $\tbP_{\alpha}\in\OP(0,0)$ and that $\bP_{\alpha}$ is a balance for $\fbD_{\alpha}$ (with respect to $\bB_{\alpha}$ in the $\DD$ case).

By the decompositions \eqref{eq:aux_decompositionAksmooth} and the fact that $\tbP_{\alpha}=\fbD_{\alpha}\bP_{\alpha}$ is the projection onto the corresponding ranges, we obtain
\[
\begin{aligned}
&\NN:\quad && \fbD_{\alpha}^* \fbD_{\alpha} \bP_{\alpha}
\oplus \bB_{\alpha}^* \fbD_{\alpha} \bP_{\alpha}
= \fbD_{\alpha}^* \oplus \bB_{\alpha}^*, \\
&\DD:\quad && \fbD_{\alpha}^* \fbD_{\alpha} \bP_{\alpha} = \fbD_{\alpha}^*.
\end{aligned}
\]

Moreover, since $\bP_{\alpha}$ takes values in
\[
\begin{aligned}
&\NN:\quad && \scrN_{\bot}(\fbD_{\alpha-1}^*, \bB_{\alpha-1}^*), \\
&\DD:\quad && \scrN_{\bot}(\fbD_{\alpha-1}^*) \cap \ker \bB_{\alpha},
\end{aligned}
\]
it follows in both cases that
\[
\tbP_{\alpha-1} \bP_{\alpha} = 0,
\qquad
\frakI_{\alpha} \bP_{\alpha} = 0,
\]
and additionally, in the $\DD$ case, $\bB_{\alpha} \bP_{\alpha}=0$. Summarizing,
\beq
\begin{aligned}
&\NN:\quad && (\fbD_{\alpha}^* \fbD_{\alpha}
\oplus \bB_{\alpha}^* \fbD_{\alpha}
\oplus \tbP_{\alpha-1}
\oplus \frakI_{\alpha}) \bP_{\alpha}
= \fbD_{\alpha}^* \oplus \bB_{\alpha}^* \oplus 0 \oplus 0, \\
&\DD:\quad && (\fbD_{\alpha}^* \fbD_{\alpha}
\oplus \bB_{\alpha}
\oplus \tbP_{\alpha-1}
\oplus \frakI_{\alpha}) \bP_{\alpha}
= \fbD_{\alpha}^* \oplus 0 \oplus 0 \oplus 0.
\end{aligned}
\label{eq:defining_relationGa}
\eeq

By \propref{prop:new_estimate2}, the systems
\[
\begin{aligned}
&\NN:\quad && \fbD_{\alpha}^* \fbD_{\alpha}
\oplus \bB_{\alpha}^* \fbD_{\alpha}
\oplus \tbP_{\alpha-1}
\oplus \frakI_{\alpha}, \\
&\DD:\quad && \fbD_{\alpha}^* \fbD_{\alpha}
\oplus \bB_{\alpha}
\oplus \tbP_{\alpha-1}
\oplus \frakI_{\alpha}
\end{aligned}
\]
are overdetermined elliptic and injective. Hence, the associated left inverses provided by \thmref{thm:overdetermined_varying_orders}, denoted in both cases by $\frakS_{\alpha}$, yield
\beq
\begin{aligned}
&\NN:\quad && \bP_{\alpha} = \frakS_{\alpha}(\fbD_{\alpha}^* \oplus \bB_{\alpha}^* \oplus 0 \oplus 0), \\
&\DD:\quad && \bP_{\alpha} = \frakS_{\alpha}(\fbD_{\alpha}^* \oplus 0 \oplus 0 \oplus 0).
\end{aligned}
\label{eq:G_def}
\eeq
This proves that $\bP_{\alpha}$ belongs to the calculus, as it is obtained by composition of operators in the calculus. Consequently, by composition, $\tbP_{\alpha} = \fbD_{\alpha} \bP_{\alpha}$ also belongs to the calculus. This completes the proof.
\end{proof}
The refined decompositions \eqref{eq:aux_induction_Refined} then clearly hold in the $\NN$ case, since $\scrN(\fbD_{\alpha}^*, \bB_{\alpha}^*) \subseteq \ker \bB_{\alpha}^*$. In the $\DD$ case, they will hold once we prove that $\bB_{\alpha+1} \fbD_{\alpha} = 0$ on $\ker \bB_{\alpha}$ in the next section.
\subsection{Stage 5: Completion of the induction step} 
By surveying what has been proven thus far, we see that the induction step is completed by proving the existence and uniqueness of $\fbD_{\alpha+1}$ satisfying the requirements in \eqref{eq:contain_induction_step}, \eqref{eq:relation_induction_step}, and \eqref{eq:correction_induction_step}, together with the additional condition \eqref{eq:contain_induction_stepM} in the Dirichlet case.  

In both cases, using the established auxiliary decompositions, define the continuous linear map of Fréchet spaces 
\[
\fbD_{\alpha+1} : \Gamma(\bbF_{\alpha+1}; \bbG_{\alpha+1}) \to \Gamma(\bbF_{\alpha+2}; \bbG_{\alpha+2})
\]
by setting
\beq 
\fbD_{\alpha+1} : \bD_{\alpha+1} - \bD_{\alpha+1} \tbP_{\alpha}
\label{eq:D_def_induction_step}
\eeq

\begin{proposition}
In the $\NN$ case, the definition \eqref{eq:D_def_induction_step} yields the relations \eqref{eq:contain_induction_step} and \eqref{eq:relation_induction_step}  required in the induction step, and it is the unique continuous map with these properties.
\end{proposition}

\begin{proof}
The relations \eqref{eq:contain_induction_step} and \eqref{eq:relation_induction_step} follow directly from the definition and the Neumann direct decomposition \eqref{eq:aux_decompositionAksmooth}, since $\tbP_{\alpha}$ is the projection onto $\scrR(\fbD_{\alpha})$, or equivalently, $\id - \tbP_{\alpha}$ is the projection onto $\scrN(\fbD_{\alpha}^*, \bB_{\alpha}^*)$.  
The uniqueness of $\fbD_{\alpha+1}$ as a system satisfying these two properties follows immediately from this decomposition, establishing the uniqueness requirement of \thmref{thm:corrected_complex}.
\end{proof}

In the $\DD$ case, we need to verify also: 
\begin{proposition}
\label{prop:extra_boundary_condition} 
In the Dirichlet case, one has
\[
\bB_{\alpha+1}\big(\scrR(\fbD_{\alpha}; \bB_{\alpha})\big) = 0.
\]
Hence, the definition \eqref{eq:D_def_induction_step} yields the relations \eqref{eq:contain_induction_step}, \eqref{eq:relation_induction_step} and \eqref{eq:contain_induction_stepM} required in the induction step, and it is the unique continuous map with these properties.
\end{proposition}
\begin{proof} 
By the induction hypothesis, the refined decomposition \eqref{eq:aux_induction_Refined} holds for $\alpha - 1$. Hence, since $\fbD_{\alpha} = 0$ on $\scrR(\fbD_{\alpha-1}; \bB_{\alpha-1})$, we obtain:
\[
\scrR(\fbD_{\alpha}; \bB_{\alpha}) = \fbD_{\alpha}(\scrN(\fbD_{\alpha-1}^*) \cap \ker \bB_{\alpha}).
\]
Moreover, by the induction hypothesis, $\fbD_{\alpha} = \bD_{\alpha}$ on $\scrN(\fbD_{\alpha-1}^*)$, and in particular on $\scrN(\fbD_{\alpha-1}^*) \cap \ker \bB_{\alpha}$. Thus,
\[
\scrR(\fbD_{\alpha}; \bB_{\alpha}) 
= \bD_{\alpha}(\scrN(\fbD_{\alpha-1}^*) \cap \ker\bB_{\alpha}).
\]
Since $\bB_{\alpha+1} \bD_{\alpha} = 0$ on $\ker \bB_{\alpha}$ by \defref{def:DD_segment}, the first claim follows.

For the second part of the claim, we note that by construction $\tbP_{\alpha} = \id$ on $\scrR(\fbD_{\alpha}; \bB_{\alpha})$. Hence, on this space, it is clear that $\fbD_{\alpha+1} = 0$ by how it is defined in \eqref{eq:D_def_induction_step}, so \eqref{eq:contain_induction_step} holds. Since $\tbP_{\alpha}$ vanishes on $\scrN(\fbD_{\alpha}^*)$, the relation in \eqref{eq:relation_induction_step} also holds, as implied by the auxiliary decomposition, and by construction it is the unique continuous map with these properties, since the decomposition on which it is defined is topologically direct. 
\end{proof}

It remains to verify the requirements in \eqref{eq:correction_induction_step} for both the Neumann and Dirichlet cases, and that in either case $\bC_{\alpha+1} \in \OP(0,0)$ and $\sigma(\fbD_{\alpha+1}-\bD_{\alpha+1})=0$. 
\begin{proposition} 
The relations in \eqref{eq:correction_induction_step} hold.
\end{proposition} 
\begin{proof} 
By the definition of $\fbD_{\alpha+1}$ in \eqref{eq:D_def_induction_step} and the representation $\tbP_{\alpha} = \fbD_{\alpha} \bP_{\alpha}$, we have
\[
\begin{aligned} 
\bC_{\alpha+1} = \fbD_{\alpha+1} - \bD_{\alpha+1} 
= -\bD_{\alpha+1} \fbD_{\alpha} \bP_{\alpha} 
= -\bD_{\alpha+1} \bD_{\alpha} \bP_{\alpha}.
\end{aligned} 
\]
In the $\NN$ case, by our construction of $\bP_{\alpha}$ in \thmref{thm:aux_induction_step}, the third equality follows from the fact that $\bP_{\alpha}$ takes values in $\scrN(\fbD_{\alpha-1}^*, \bB_{\alpha-1})$ (where $\fbD_{\alpha} = \bD_{\alpha}$ by the induction hypothesis). In the $\DD$ case, by the same construction, it follows from the fact that $\bP_{\alpha}$ takes values in $\scrN(\fbD_{\alpha-1}^*) \cap \ker \bB_{\alpha}$ (where $\fbD_{\alpha} = \bD_{\alpha}$ by the induction hypothesis and the order reduction properties in \defref{def:DD_segment}).

To prove that $\bC_{\alpha+1} \in \OP(0,0)$, observe that in the $\NN$ case, $\bP_{\alpha}$ is a balance for $\fbD_{\alpha}$, while in the $\DD$ case, it is a balance for $\fbD_{\alpha}$ with respect to $\bB_{\alpha}$. Together with the fact that $\fbD_{\alpha} - \bD_{\alpha} \in \OP(0,0)$, it follows that $\bP_{\alpha}$ also serves as a balance for $\bD_{\alpha}$ (respectively, with respect to $\bB_{\alpha}$). Thus, by the order-reduction property in either \defref{def:NN_segment} or \defref{def:DD_segment}, we conclude that $\bC_{\alpha+1} \in \OP(0,0)$, completing the proof of this part.  

Finally, $\sigma(\fbD_{\alpha+1} - \bD_{\alpha+1}) = 0$ follows from the definition of the lifting and from the order-reduction properties in \defref{def:NN_segment}--\defref{def:DD_segment}. 
\end{proof}

%
%
%
\subsection{The Hodge decompositions}
\label{sec:Hodge_proof}
At this stage, auxiliary decompositions for all levels of the elliptic pre-complex have been established, and the systems $\fbD_{\alpha}$ have been defined with the properties stated in the induction hypothesis in \secref{sec:stage1}. 

Using this collective structure, we now prove \thmref{thm:hodge_like_corrected_complex} and \thmref{thm:hodge_like_corrected_complexD}. 

We first make short work of the identities \eqref{eq:cohomology_groups} and \eqref{eq:cohomology_groupsD}. In view of the refinement in \eqref{eq:cohomology_in_the_proof}, these reduce to the following observation (which also applies in the disrupted case, hence is valid there as well): 
\begin{lemma}
We have
\beq
\begin{aligned}
&\NN:\quad && \fbD_{\alpha}^* = (\id - \tbP_{\alpha-1}) \bD^*_{\alpha}, \quad \text{on } \quad \ker \bB_{\alpha}^*, \\  
&\DD:\quad && \fbD_{\alpha}^* = (\id - \tbP_{\alpha-1}) \bD^*_{\alpha}. 
\end{aligned}
\label{eq:adjoint_expression} 
\eeq
\end{lemma}
\begin{proof}
We know that
\beq 
\fbD^*_{\alpha} = \bD_{\alpha}^* + \bC_{\alpha}^*,
\label{eq:expression_adjoint} 
\eeq 
where $\bC_{\alpha}^*$ is the adjoint of $\bC_{\alpha}$, and \propref{prop:correction} provides the formula
\[
\bC_{\alpha} = -\bD_{\alpha} \tbP_{\alpha-1}.
\]
Moreover, $\bC_{\alpha}$ and $\bC_{\alpha}^*$ belong to $\OP(0,0)$, so they are $L^{2}$-continuous and adjoint to each other, with no additional boundary term resulting from integration by parts.   

Given this, since $\bB_{\alpha} \tbP_{\alpha-1} \Psi = 0$ identically in the $\DD$ case, and assuming $\Theta \in \ker \bB_{\alpha}^*$ in the $\NN$ case, we can apply Green's formula \eqref{eq:integration_by_parts_corrected} iteratively and use the fact that $\tbP_{\alpha-1}$ is an $L^{2}$-orthogonal projection: 
\[ 
\bra\Psi, \bC_{\alpha}^* \Theta\ket = \bra \bC_{\alpha} \Psi, \Theta\ket = -\bra \bD_{\alpha} \tbP_{\alpha-1} \Psi, \Theta\ket = -\bra\Psi, \tbP_{\alpha-1} \bD_{\alpha}^* \Theta\ket. 
\]
Since this holds for arbitrary $\Psi$, and $\bC_{\alpha},\bC_{\alpha}^*$ are $L^{2}$-continuous, we conclude in both cases that, under the stated assumptions,
\[
\bC_{\alpha}^* = -\tbP_{\alpha-1} \bD_{\alpha}^*.
\]
Combining this with the expression for $\fbD_{\alpha}^*$ above in \eqref{eq:expression_adjoint} yields the required identity.
\end{proof}
  
By comparing the required decompositions with the refined auxiliary decompositions \eqref{eq:N0spdecomp} already established for all levels through the induction steps, we find that to prove \thmref{thm:hodge_like_corrected_complex} and \thmref{thm:hodge_like_corrected_complexD} it therefore remains to establish the following:
\begin{proposition}
\label{prop:Hodge_proof}
For all $\alpha< \alpha_0$, the following holds:
\[
\begin{aligned}
&\NN:\qquad &&\scrN_{\bot}(\fbD^*_{\alpha},\bB^*_{\alpha})=\scrR(\fbD^*_{\alpha+1};\bB^*_{\alpha+1}),
\\& \DD:\qquad &&\scrN_{\bot}(\fbD^*_{\alpha})=\scrR(\fbD^*_{\alpha+1}).
\end{aligned}
\]
\end{proposition}
We prove this in several stages.
\begin{proposition}
\label{prop:correction_dual_containment}
For all $\alpha< \alpha_0$, the following holds:
\[
\begin{aligned}
&\NN:\qquad &&\scrR(\fbD^*_{\alpha+1};\bB^*_{\alpha+1})\subseteq\scrN_{\bot}(\fbD^*_{\alpha},\bB^*_{\alpha}) ,
\\& \DD:\qquad &&\scrR(\fbD^*_{\alpha+1})\subseteq\scrN_{\bot}(\fbD^*_{\alpha}).
\end{aligned}
\]
\end{proposition}
\begin{proof}
This is obtained directly by dualizing the relation \eqref{eq:relation_induction_step} with respect to the $L^2$ inner product together with the relations already established in \lemref{lem:useful_neumann}, \lemref{lem:useful_dirichlet} and: 
\[
\begin{aligned}
&\NN:\qquad &&\scrR(\fbD^*_{\alpha+1};\bB^*_{\alpha+1})\,\bot\,\module^{\alpha+1}_{\N},
\\& \DD:\qquad &&\scrR(\fbD^*_{\alpha+1})\,\bot\,\module^{\alpha+1}_{\D}.
\end{aligned}
\] 
\end{proof}
For the next proposition, recall that by the assumptions on an adapted Green system, $\fbD^*_{\alpha+1}\fbD_{\alpha+1}$ has sharp tuples of the form $(2M,2T;0,0)$.
\begin{proposition}
\label{prop:correction_final_dual_new}
For all $\alpha< \alpha_0$, the following subspaces are closed:  
\[
\begin{aligned}
&\NN: \qquad && \scrR^{0,0}_{p}(\fbD^*_{\alpha+1}; \bB^*_{\alpha+1}), \\[0.5em]
&\DD: \qquad && \scrR^{0,0}_{p}(\fbD^*_{\alpha+1}).
\end{aligned}
\]

Moreover, let $(J, L; I, K)$ be sharp tuples for $\fbD_{\alpha+1}$ such that $(J, L; II, KK)$ are sharp tuples for the composition $\fbD^*_{\alpha+1} \fbD_{\alpha+1}$. Then the following subspaces are also closed, and the identities hold:
\[
\begin{aligned}
&\NN:\qquad && \scrR^{II,KK}_{p}(\fbD^*_{\alpha+1}; \bB^*_{\alpha+1}) 
= \scrR^{0,0}_{p}(\fbD^*_{\alpha+1}; \bB^*_{\alpha+1}) 
\cap W^{II,KK}_{p}(\bbF_{\alpha+1}; \bbG_{\alpha+1}), \\[0.5em]
&\DD:\qquad && \scrR^{II,KK}_{p}(\fbD^*_{\alpha+1}) 
= \scrR^{0,0}_{p}(\fbD^*_{\alpha+1}) 
\cap W^{II,KK}_{p}(\bbF_{\alpha+1}; \bbG_{\alpha+1}).
\end{aligned}
\]
\end{proposition}
\begin{proof}
Consider the subspaces of $W^{II,KK}_{p}(\bbF_{\alpha+1}; \bbG_{\alpha+1})$. 
\beq
\begin{aligned}
&\NN:\quad && \{\fbD_{\alpha+1}^*\Theta : \Theta \in W^{I,K}_p(\bbF_{\alpha+2};\bbG_{\alpha+2}) \cap \ker\bB_{\alpha+1}^*\}, \\
&\DD:\quad && \{\fbD_{\alpha+1}^*\Theta : \Theta \in W^{I,K}_p(\bbF_{\alpha+2};\bbG_{\alpha+2})\}.
\end{aligned}
\label{eq:Dstar_range_proof}
\eeq
For the first statement, we may assume that $(J, L; I, K) = (2M, 2L; 0, 0)$. Iterating the decompositions in \eqref{eq:aux_induction_Refined}--\eqref{eq:N0spdecomp} for $\alpha+1$ and $\alpha$ (this is why we require $\alpha<\alpha_0$, to ensure that these decompositions are available), and using the defining relations of $\fbD_{\alpha}$ together with their dualized counterparts for $\fbD_{\alpha}^*$, we find that these spaces are equal to
\[
\begin{aligned}
&\NN:\quad && \{\fbD_{\alpha+1}^*\fbD_{\alpha+1}\Psi : \Psi \in \scrN_{\bot,p}^{J,L}(\fbD^*_{\alpha},\bB_{\alpha}^*),\quad \bB^*_{\alpha+1}\fbD_{\alpha+1}\Psi=0\},\\
&\DD:\quad && \{\fbD_{\alpha+1}^*\fbD_{\alpha+1}\Psi : \Psi \in \scrN_{\bot,p}^{J,L}(\fbD^*_{\alpha})\cap\ker\bB_{\alpha+1}\}.
\end{aligned}
\]
In either case, we find that the estimates in \eqref{eq:D^*D_estimate} apply to the potential $\Psi$ in the defining relation for these spaces, yielding in all cases the estimate: 
\[
\|\Psi\|_{J,L,p} \lesssim \|\fbD_{\alpha+1}^*\fbD_{\alpha+1}\Psi\|_{II,KK,p}.
\]
By a standard argument (e.g., \cite[p.~583]{Tay11a}), this implies that the spaces above are closed. Retracing our steps, we conclude that the subspaces in \eqref{eq:Dstar_range_proof} are also closed, and therefore the subspaces in the claim are closed as well. The identities there then follow directly from the estimate above together with \propref{prop:AprioriAstarAk}.
\end{proof}
\begin{PROOF}{\propref{prop:Hodge_proof}}
Applying \propref{prop:L2_aux_adapted_pair} to the one-sided adapted Green system $\fbD_{\alpha}^*$, one obtains the orthogonal $L^{2}$-decomposition:
\[
\begin{aligned}
&\NN:\quad && L^2(\bbF_{\alpha+1}; \bbG_{\alpha+1}) =\scrR^{0,0}_{2}(\fbD_{\alpha+1}^*; \bB^*_{\alpha+1}) \oplus \scrN^{0,0}_{2}(\fbD_{\alpha+1}), \\
&\DD:\quad && L^2(\bbF_{\alpha+1}; \bbG_{\alpha+1}) = \scrR^{0,0}_{2}(\fbD_{\alpha+1}^*) \oplus \scrN^{0,0}_{2}(\fbD_{\alpha+1},\bB_{\alpha+1}).
\end{aligned}
\]
On the other hand, the $L^2$ version of the decompositions in \eqref{eq:N0spdecomp} reads
\[
\begin{aligned}
&\NN:\quad && L^2(\bbF_{\alpha+1}; \bbG_{\alpha+1}) = \scrR^{0,0}_{2}(\fbD_{\alpha})\oplus \scrN_{\bot,2}^{0,0}(\fbD_{\alpha}^*,\bB^*_{\alpha}) \oplus \module_{\N}^{\alpha+1}, \\
&\DD:\quad && L^2(\bbF_{\alpha+1}; \bbG_{\alpha+1}) = \scrR^{0,0}_{2}(\fbD_{\alpha}; \bB_{\alpha})\oplus \scrN_{\bot,2}^{0,0}(\fbD_{\alpha}^*) \oplus \module_{\D}^{\alpha+1}.
\end{aligned}
\]
Both decompositions are $L^2$-orthogonal. Using the continuity of the orthogonal projection $\frakI_{\alpha+1}: L^{2}(\bbF_{\alpha+1}; \bbG_{\alpha+1}) \to \module^{\alpha+1}_{\R}$, $\R\in\BRK{\NN,\DD}$, we observe:
\[
\begin{aligned}
&\NN:\quad&&\scrN^{0,0}_{2}(\fbD_{\alpha+1}) \cap \scrN_{2,\bot}^{0,0}(\fbD_{\alpha}^*, \bB^*_{\alpha}) = \module_{\N}^{\alpha+1} \cap \scrN_{\bot,2}^{0,0}(\fbD_{\alpha}^*, \bB^*_{\alpha}) = \{0\}, \\
&\DD:\quad&&\scrN^{0,0}_{2}(\fbD_{\alpha+1},\bB_{\alpha+1}) \cap \scrN_{2,\bot}^{0,0}(\fbD_{\alpha}^*) = \module_{\D}^{\alpha+1} \cap \scrN_{\bot,2}^{0,0}(\fbD_{\alpha}^*) = \{0\}.
\end{aligned}
\]
So comparing the decompositions, we conclude:
\[
\begin{aligned}
&\NN:\quad && \scrN_{2,\bot}^{0,0}(\fbD_{\alpha}^*; \bB^*_{\alpha}) \subseteq \scrR^{0,0}_{2}(\fbD_{\alpha+1}^*; \bB^*_{\alpha+1}), \\
&\DD:\quad && \scrN_{2,\bot}^{0,0}(\fbD_{\alpha}^*) \subseteq \scrR^{0,0}_{2}(\fbD_{\alpha+1}^*).
\end{aligned}
\]

By combining this with \propref{prop:correction_final_dual_new}, we obtain the required equality. Intersecting both sides with $W^{II, KK}_{p}(\bbF_{\alpha+1}; \bbG_{\alpha+1})$ and applying the second clause of \propref{prop:correction_final_dual_new}, we then have:
\[
\begin{aligned}
&\NN:\quad && \scrR_{p}^{II, KK}(\fbD_{\alpha+1}^*; \bB^*_{\alpha+1}) = \scrN_{p, \bot}^{II, KK}(\fbD_{\alpha}^*; \bB^*_{\alpha}), \\
&\DD:\quad && \scrR_{p}^{II, KK}(\fbD_{\alpha+1}^*) = \scrN_{p, \bot}^{II, KK}(\fbD_{\alpha}^*).
\end{aligned}
\]
Since this holds for every sharp tuple satisfying the assumptions of \propref{prop:correction_dual_containment}, the smooth version also holds due to the Sobolev Fréchet grading.
\end{PROOF}

\subsection{Disrupted case} 
\label{sec:disrupted_proof}
We now turn to the disrupted case, namely the lifting theorem \thmref{thm:disrubted_elliptic_pre_complex} and the Hodge decompositions \thmref{thm:hodge_like_corrected_disrupted_complex} and \thmref{thm:hodge_like_corrected_disrupted_complexD}. We show that, with an additional argument, these results follow from the analysis of the standard non-disrupted setting.

The first step is to observe that, by their very definition (\defref{def:finite_elliptic_pre_complex}), disrupted $\alpha_0$-elliptic pre-complexes are $\tilde{\alpha}_0$-elliptic pre-complexes with
\[
\tilde{\alpha}_0=\alpha_0-1.
\]
Thus, all relevant lifted systems up to $\fbD_{\tilde{\alpha}_0}$ are constructed as before and are characterized by the properties in \thmref{thm:corrected_complex} and \propref{prop:correction}. In particular, $\fbD_{\tilde{\alpha}_0-1}$ induces an auxiliary decomposition:
\beq
\begin{aligned}
\NN &:\qquad && \Gamma(\bbF_{\tilde{\alpha}_0}; \bbG_{\tilde{\alpha}_0}) 
= \scrR(\fbD_{\tilde{\alpha}_0-1}) \oplus \scrN(\fbD^*_{\tilde{\alpha}_0-1}, \bB^*_{\tilde{\alpha}_0-1}), \\[0.5em]
\DD &:\qquad && \Gamma(\bbF_{\tilde{\alpha}_0}; \bbG_{\tilde{\alpha}_0}) 
= \scrR(\fbD_{\tilde{\alpha}_0-1}; \bB_{\tilde{\alpha}_0-1}) \oplus \scrN(\fbD^*_{\tilde{\alpha}_0-1}).
\end{aligned} 
\label{eq:aux_disrupted_N-1}
\eeq 

Hence, to prove \thmref{thm:disrubted_elliptic_pre_complex}, \thmref{thm:hodge_like_corrected_disrupted_complex}, and \thmref{thm:hodge_like_corrected_disrupted_complexD}, it remains to show that $\fbD_{\tilde{\alpha}_0}$ induces an auxiliary decomposition (since $\fbD_{\tilde{\alpha}_0+1} = 0$), and that both this decomposition and the ones in \eqref{eq:aux_disrupted_N-1} refine to Hodge decompositions, albeit without the finite dimensionality of $\module^{\tilde{\alpha}_0}_{\R}$, $\R \in \BRK{\NN, \DD}$.

By the defining properties of disrupted $\alpha_0$-elliptic pre-complexes, and since $\bD_{\tilde{\alpha}_0+1} = 0$, a comparison with the overdetermined ellipticity conditions in \defref{def:NN_segment}--\defref{def:DD_segment} for $\alpha = \tilde{\alpha}_0 + 1$, together with the fact that $\sigma(\fbD_{\tilde{\alpha}_0}^* - \bD_{\tilde{\alpha}_0}^*) = 0$, shows that the systems
\beq
\begin{aligned}
\NN &:\quad && \fbD_{\tilde{\alpha}_0}^* \oplus \bB_{\tilde{\alpha}_0}^*, \\[0.5em]
\DD &:\quad && \fbD_{\tilde{\alpha}_0}^*
\end{aligned}
\label{eq:OD_disrupted1}
\eeq
are overdetermined elliptic. By composing these with the overdetermined ellipticity at $\alpha = \tilde{\alpha}_0$ and using the order-reduction property to disregard compact perturbations, the following systems are also overdetermined elliptic:
\beq
\begin{aligned}
\NN &:\quad && \fbD_{\tilde{\alpha}_0} \fbD_{\tilde{\alpha}_0}^* \oplus \bB_{\tilde{\alpha}_0}^*, \\[0.5em]
\DD &:\quad && \fbD_{\tilde{\alpha}_0} \fbD_{\tilde{\alpha}_0}^* \oplus \bB_{\tilde{\alpha}_0} \fbD_{\tilde{\alpha}_0}^*.
\end{aligned}
\label{eq:OD_disrupted2}
\eeq
Hence, the sequence
\[
\tilde{\bD}_{-1} = 0, \qquad \tilde{\bD}_{1} = \fbD_{\tilde{\alpha}_0}^*, \qquad \tilde{\bD}_{2} = 0, \qquad \tilde{\bD}_{3} = 0, \cdots
\]
is itself an $\tilde{\tilde{\alpha}}_0$-elliptic pre-complex, with $\tilde{\tilde{\alpha}}_0 = 1$. It is based on Dirichlet conditions when the original $(\bD_{\bullet})$ is based on Neumann conditions, and vice versa. Moreover, it holds that
\[
\tilde{\fbD}_{1} = \tilde{\bD}_{1} = \fbD_{\tilde{\alpha}_0}^*,
\]
by the unique characterization of the lifted system provided in \thmref{thm:corrected_complex}.

Thus, $\fbD_{\tilde{\alpha}_0}^*$ induces its own auxiliary decompositions (with the labeling referring to the conditions on which the original $(\bD_{\bullet})$ was based):  
\beq 
\begin{aligned}
\NN &:\qquad && \Gamma(\bbF_{\tilde{\alpha}_0}; \bbG_{\tilde{\alpha}_0}) 
= \scrR(\fbD^*_{\tilde{\alpha}_0}; \bB^*_{\tilde{\alpha}_0}) \oplus \scrN(\fbD_{\tilde{\alpha}_0}), \\[0.5em]
\DD &:\qquad && \Gamma(\bbF_{\tilde{\alpha}_0}; \bbG_{\tilde{\alpha}_0}) 
= \scrR(\fbD_{\tilde{\alpha}_0}^*) \oplus \scrN(\fbD_{\tilde{\alpha}_0}, \bB_{\tilde{\alpha}_0}).
\end{aligned}
\label{eq:aux_disrutped_adjoint_N}
\eeq 

We then have:
\begin{proposition} 
The decomposition \eqref{eq:aux_disrupted_N-1} refines into a Hodge decomposition as specified by \thmref{thm:hodge_like_corrected_disrupted_complex} and \thmref{thm:hodge_like_corrected_disrupted_complexD} (with $\tilde{\alpha}_0=\alpha-1$):
\beq 
\begin{aligned}
\NN &:\quad && \Gamma(\bbF_{\tilde{\alpha}_0}; \bbG_{\tilde{\alpha}_0}) 
= \scrR(\fbD_{\tilde{\alpha}_0-1}) \oplus \scrR(\fbD^*_{\tilde{\alpha}_0}; \bB^*_{\tilde{\alpha}_0})
\oplus \module^{\tilde{\alpha}_0}_{\NN} \\ 
\DD &:\quad &&\Gamma(\bbF_{\tilde{\alpha}_0}; \bbG_{\tilde{\alpha}_0}) 
= \scrR(\fbD_{\tilde{\alpha}_0-1}; \bB_{\tilde{\alpha}_0-1}) 
\oplus \scrR(\fbD_{\tilde{\alpha}_0}^*) \oplus  \module^{\tilde{\alpha}_0}_{\DD}.
\end{aligned}
\label{eq:Hodge_like_disrupted_1} 
\eeq
with $\module^{\tilde{\alpha}_0}_{\R}$, $\R\in\BRK{\NN,\DD}$ not necessarily finite dimensional.
\end{proposition}
\begin{proof} 
On the one hand, we have the auxiliary decompositions of both $\fbD_{\tilde{\alpha}_0}^*$ in \eqref{eq:aux_disrutped_adjoint_N} and $\fbD_{\tilde{\alpha}_0-1}$ in \eqref{eq:aux_disrupted_N-1}, which are $L^{2}$-orthogonal; on the other hand, we have the relations from the construction of the lifted complex from the main proof 
\beq 
\begin{aligned}
\NN &:\quad && \scrR(\fbD^*_{\tilde{\alpha}_0}; \bB^*_{\tilde{\alpha}_0}) \subseteq \scrN(\fbD_{\tilde{\alpha}_0-1}^*, \bB_{\tilde{\alpha}_0-1}^*), \qquad \scrR(\fbD_{\tilde{\alpha}_0-1}) \subseteq \scrN(\fbD_{\tilde{\alpha}_0}), \\[0.5em]
\DD &:\quad && \scrR(\fbD^*_{\tilde{\alpha}_0}) \subseteq \scrN(\fbD_{\tilde{\alpha}_0-1}^*), \qquad \scrR(\fbD_{\tilde{\alpha}_0-1}; \bB_{\tilde{\alpha}_0-1}) \subseteq \scrN(\fbD_{\tilde{\alpha}_0}, \bB_{\tilde{\alpha}_0}).
\end{aligned}
\label{eq:disrupted_relations_proof} 
\eeq 
Comparing these, we obtain the $L^{2}$-orthogonal decompositions
\[
\begin{aligned}
\NN &:\quad && \Gamma(\bbF_{\tilde{\alpha}_0}; \bbG_{\tilde{\alpha}_0}) 
=\scrR(\fbD_{\tilde{\alpha}_0-1}) \oplus \scrR(\fbD^*_{\tilde{\alpha}_0}; \bB^*_{\tilde{\alpha}_0}) \oplus 
(\scrN(\fbD_{\tilde{\alpha}_0}) \cap \scrN(\fbD_{\tilde{\alpha}_0-1}^*, \bB_{\tilde{\alpha}_0-1}^*)), \\[0.5em]
\DD &:\quad && \Gamma(\bbF_{\tilde{\alpha}_0}; \bbG_{\tilde{\alpha}_0}) 
= \scrR(\fbD_{\tilde{\alpha}_0-1}; \bB_{\tilde{\alpha}_0-1}) 
\oplus \scrR(\fbD_{\tilde{\alpha}_0}^*) \oplus (\scrN(\fbD_{\tilde{\alpha}_0}, \bB_{\tilde{\alpha}_0}) \cap \scrN(\fbD^*_{\tilde{\alpha}_0-1})).
\end{aligned}
\]
Moreover, since by the defining relation in \thmref{thm:corrected_complex} we have $\fbD_{\alpha} = \bD_{\alpha}$ on $\scrN(\fbD_{\alpha}^*)$ in the $\DD$ case and on $\scrN(\fbD_{\alpha}^*, \bB_{\alpha}^*)$ in the $\NN$ case, it follows that in the setting of both \thmref{thm:hodge_like_corrected_disrupted_complex} and \thmref{thm:hodge_like_corrected_disrupted_complexD} we have
\[
\begin{aligned}
\NN &:\qquad && \scrN(\fbD_{\tilde{\alpha}_0}) \cap \scrN(\fbD_{\tilde{\alpha}_0-1}^*, \bB_{\tilde{\alpha}_0-1}^*) = \module^{\tilde{\alpha}_0}_{\NN}, \\[0.5em]
\DD &:\qquad && \scrN(\fbD_{\tilde{\alpha}_0}, \bB_{\tilde{\alpha}_0}) \cap \scrN(\fbD^*_{\tilde{\alpha}_0-1}) = \module^{\tilde{\alpha}_0}_{\DD}.
\end{aligned}
\]
By comparing, these provide the desired Hodge decompositions. 
\end{proof} 
\begin{proposition}
$\fbD_{\tilde{\alpha}_0}$ induces an auxiliary decomposition, which refines to a Hodge decomposition as specified by \thmref{thm:hodge_like_corrected_disrupted_complex} and \thmref{thm:hodge_like_corrected_disrupted_complex} (by replacing $\tilde{\alpha}_0=\alpha_0-1$).
\end{proposition}
\begin{proof}
We find by iterating the relations in \eqref{eq:disrupted_relations_proof} upon the decompositions in \eqref{eq:Hodge_like_disrupted_1} that:
\[
\begin{aligned} 
\NN &:\qquad && \scrR(\fbD_{\tilde{\alpha}_0}) 
= \BRK{\fbD_{\tilde{\alpha}_0}\fbD_{\tilde{\alpha}_0}^*\Psi ~:~ \bB^*_{\tilde{\alpha}_0}\Psi=0}, \\[0.5em]
\DD &:\qquad && \scrR(\fbD_{\tilde{\alpha}_0};\bB_{\tilde{\alpha}_0}) 
= \BRK{\fbD_{\tilde{\alpha}_0}\fbD_{\tilde{\alpha}_0}^*\Psi ~:~ \bB_{\tilde{\alpha}_0}\fbD_{\tilde{\alpha}_0}^*\Psi=0}.
\end{aligned} 
\]

By the overdetermined ellipticities listed in \eqref{eq:OD_disrupted1}, the spaces on the right are closed due to the associated a~priori estimates, and so are their Sobolev versions. In particular, $\scrR^{0,0}_{2}(\fbD_{\tilde{\alpha}_0})$ and $\scrR^{0,0}_{2}(\fbD_{\tilde{\alpha}_0};\bB_{\tilde{\alpha}_0})$ are closed, and the $L^{2}$-orthogonal decompositions from  \propref{prop:L2_aux_adapted_pair} become:
\[
\begin{aligned}
\NN &:\qquad && L^{2}(\bbF_{\tilde{\alpha}_0+1};\bbG_{\tilde{\alpha}_0+1})
= \scrR^{0,0}_{2}(\fbD_{\tilde{\alpha}_0}) \oplus \scrN^{0,0}_{2}(\fbD_{\tilde{\alpha}_0}^*, \bB_{\tilde{\alpha}_0}^*), \\[0.5em]
\DD &:\qquad && L^{2}(\bbF_{\tilde{\alpha}_0+1};\bbG_{\tilde{\alpha}_0+1})
= \scrR^{0,0}_{2}(\fbD_{\tilde{\alpha}_0}; \bB_{\tilde{\alpha}_0}) \oplus \scrN^{0,0}_{2}(\fbD_{\tilde{\alpha}_0}^*).
\end{aligned}
\]

However, by the overdetermined ellipticities in \eqref{eq:OD_disrupted2}, the kernels on the right are finite dimensional and consist of smooth sections, and the projections onto them belong to the calculus, lying in $\OP(-\infty,-\infty)$. In fact, since $\fbD_{\tilde{\alpha}_0+1} = 0$, we directly find in the setting of \thmref{thm:hodge_like_corrected_disrupted_complex} and \thmref{thm:hodge_like_corrected_disrupted_complexD}:
\[
\begin{aligned}
\NN &:\qquad && \scrN^{0,0}_{2}(\fbD_{\tilde{\alpha}_0}^*, \bB_{\tilde{\alpha}_0}^*) 
= \scrN(\fbD_{\tilde{\alpha}_0}^*, \bB_{\tilde{\alpha}_0}^*) = \module^{\tilde{\alpha}_0+1}_{\NN}, \\[0.5em]
\DD &:\qquad && \scrN^{0,0}_{2}(\fbD_{\tilde{\alpha}_0}^*) 
= \scrN(\fbD_{\tilde{\alpha}_0}^*) = \module^{\tilde{\alpha}_0+1}_{\DD}.
\end{aligned}
\]

Hence, on the smooth level we have:
\[
\begin{aligned}
\NN &:\qquad && \Gamma(\bbF_{\tilde{\alpha}_0+1};\bbG_{\tilde{\alpha}_0+1})
= \scrR(\fbD_{\tilde{\alpha}_0}) \oplus \module^{\tilde{\alpha}_0+1}_{\NN}, \\[0.5em]
\DD &:\qquad && \Gamma(\bbF_{\tilde{\alpha}_0+1};\bbG_{\tilde{\alpha}_0+1})
= \scrR(\fbD_{\tilde{\alpha}_0};\bB_{\tilde{\alpha}_0}) \oplus \module^{\tilde{\alpha}_0+1}_{\DD}.
\end{aligned}
\]
This gives both the auxiliary decompositions and the Hodge decompositions induced by $\fbD_{\tilde{\alpha}_0}$, as required.
\end{proof}
\section{Tame smooth families}
\label{sec:Tame}
\subsection{Tame smooth families of systems} 
We go back for a moment to discussing general Douglis-Nirenberg systems, as was done in \secref{sec:Overdetermined}, stripped from the context of elliptic pre-complexes. Let $\scrU$ be a tame Fréchet manifold, serving as a moduli space, and let $\bbE_{\gamma}, \bbF_{\gamma}, \bbE, \bbF \rightarrow M$, and $\bbJ_{\gamma}, \bbG_{\gamma}, \bbJ, \bbG \rightarrow \partial M$ be vector bundles parameterized by the moduli space $\gamma \in \scrU$. Let $\scrV, \scrW \rightarrow \scrU$ be tame Fréchet vector bundles, with fibers:
\[
\scrV|_{\gamma}=\Gamma(\bbE_{\gamma};\bbJ_{\gamma}), \qquad \scrW|_{\gamma}=\Gamma(\bbF_{\gamma};\bbG_{\gamma})
\] 
and model spaces fixed $\Gamma(\bbE;\bbJ)$ and $\Gamma(\bbF;\bbG)$, respectively. 
\begin{definition}
A bundle map between tame Fréchet vector bundles as above 
\beq
\bD:\scrV\rightarrow\scrW
\label{eq:tame_bundle_maps}
\eeq
operating in the fashion of
\beq 
\begin{aligned}
(\gamma,\Psi)\mapsto(\gamma,\bD(\gamma)\Psi), \qquad \gamma\in\scrU,\,\,\,\Psi\in \Gamma(\bbE_{\gamma};\bbJ_{\gamma})
\end{aligned}
\label{eq:bundle_map_operation}
\eeq
is called a \emph{family of systems} if each of its \emph{fiber maps}
\[
\bD(\gamma):\Gamma(\bbE_{\gamma};\bbJ_{\gamma})\rightarrow \Gamma(\bbF_{\gamma};\bbG_{\gamma})
\] 
is a Douglis-Nirenberg system. $\frakA$ is called a \emph{tame smooth family of systems} if \eqref{eq:tame_bundle_maps} is tame and smooth as a map of tame Fréchet manifolds. 
\end{definition} 
Note that if $\scrV=\scrU\times\Gamma(\bbE;\bbJ)$ and $\scrW=\scrU\times\Gamma(\bbF;\bbG)$ are trivial bundles, and $\scrU\subseteq F$ is a an open subset of a tame Fréchet space $F$, then a family of systems effectively reduces into a mapping
\beq
\begin{split}
&\bD:(\scrU\subseteq F)\times \Gamma(\bbE;\bbJ)\rightarrow\Gamma(\bbF;\bbG),
\end{split}
\label{eq:linear_map} 
\eeq
operating as $(\gamma,\Psi)\mapsto \bD(\gamma)\Psi$, such that for all $\gamma\in\scrU$, $\bD(\gamma):\Gamma(\bbE;\bbJ)\rightarrow \Gamma(\bbF;\bbG)$ is a Douglis-Nirenberg system. 
Mappings of the form \eqref{eq:tame_bundle_maps} are thus a generalization of the notion of tame smooth families of linear maps as defined and discussed throughout \cite{Ham82}, and in particular generalizes tame smooth families of differential operators. 

The goal of this section is to establish that both adjunction and inversion of tame smooth families of systems preserve tame and smooth dependence on the parameter. The first result concerns the adjoints of the zero-class constituents of $\bD$, as defined in \eqref{eq:integration_by_parts_zero_class_cons}:
\begin{theorem}
\label{thm:adjoint_family}
Let $\bD:\scrV\rightarrow\scrW$ be a tame smooth family of systems. Let $g:\scrU\rightarrow\scrM_{M}$ be any tame smooth family of Riemannian metrics over $M$, and let $d\Volume:\scrU\rightarrow\Omega^{d}_{M}$ be any tame smooth family of volume forms. Let $\gamma\mapsto\bra\cdot,\cdot\ket_{\gamma}$ denote the family of induced $L^{2}$-inner products on the fibers $\scrV$ and $\scrW$, given by (boundary inner products omitted for conciseness):
\beq
\bra\cdot,\cdot\ket_{\gamma}=\int_{M}(\cdot,\cdot)_{g(\gamma)} \, d\Volume(\gamma).
\label{eq:smooth_familiy_inner_products}
\eeq
Then the \emph{family of adjoints} of the zero-class constituents of $\bD$ (\defref{def:zero_class_constituent}) with respect to \eqref{eq:smooth_familiy_inner_products}, i.e., the family of systems $\bD^*:\scrW\rightarrow\scrV$ defined for each $\gamma\in \scrU$ by the relation
\beq
\bra\bD(\gamma)\Psi,\Theta\ket_{\gamma}=\bra \Psi,\bD^*(\gamma)\Theta\ket_{\gamma}, \qquad \Psi\in\Gamma_{c}(\bbE_{\gamma};\bbJ_{\gamma}),\,\,\Theta\in\Gamma_{c}(\bbF_{\gamma};\bbG_{\gamma}), 
\label{eq:adjoint_relation_bundles}
\eeq
is also a tame smooth family.
\end{theorem}
 
The second result concerns left inverses:
\begin{theorem}
\label{thm:inverses_family}
Let $\bD:\scrV \rightarrow \scrW$ be a tame smooth family of systems such that each fiber map $\bD(\gamma)$ is an injective overdetermined-elliptic system. Then the family of systems $\frakS: \scrW \rightarrow \scrV$, whose fiber maps are defined by the relation
\[
\frakS(\gamma)\bD(\gamma)\Psi = \Psi, \qquad \Psi \in \Gamma(\bbE_{\gamma}; \bbJ_{\gamma}), \,\,\, \gamma \in \scrU,
\]
is also a tame smooth family.
\end{theorem} 

The proofs of these theorems are deferred to \secref{sec:techincal_tame}.

\subsection{Tame smooth families of elliptic pre-complexes} 
\label{sec:index_theory} 
We proceed to address the case where the systems in the diagram \eqref{eq:elliptic_pre_complex_diagram} are tamely and smoothly parameterized by the moduli space $\scrU$:
\beq
\begin{xy}
(-30,0)*+{0}="Em1";
(0,0)*+{\Gamma(\bbF_{0,\gamma};\bbG_{0,\gamma})}="E0";
(30,0)*+{\Gamma(\bbF_{1,\gamma};\bbG_{1,\gamma})}="E1";
(60,0)*+{\Gamma(\bbF_{2,\gamma};\bbG_{2,\gamma})}="E2";
(90,0)*+{\Gamma(\bbF_{3,\gamma};\bbG_{3,\gamma})}="E3";
(101,0)*+{\cdots}="E4";
(-30,-25)*+{0}="Gm1";
(0,-25)*+{\Gamma(0;\bbL_{0,\gamma})}="G0";
(30,-25)*+{\Gamma(0;\bbL_{1,\gamma})}="G1";
(60,-25)*+{\Gamma(0;\bbL_{2,\gamma})}="G2";
(90,-25)*+{\Gamma(0;\bbL_{3,\gamma})}="G3";
(100,-25)*+{\cdots}="G4";
{\ar@{->}@/^{1pc}/^{\bD_{0}(\gamma)}"E0";"E1"};
{\ar@{->}@/^{1pc}/^{\bD_{0}^*(\gamma)}"E1";"E0"};
{\ar@{->}@/^{1pc}/^{\bD_{1}(\gamma)}"E1";"E2"};
{\ar@{->}@/^{1pc}/^{\bD_{1}^*(\gamma)}"E2";"E1"};
{\ar@{->}@/^{1pc}/^{\bD_{2}(\gamma)}"E2";"E3"};
{\ar@{->}@/^{1pc}/^{\bD_{2}^*(\gamma)}"E3";"E2"};
{\ar@{->}@/^{1pc}/^{\bD_{-1}(\gamma)}"Em1";"E0"};
{\ar@{->}@/^{1pc}/^{\bD^*_{-1}(\gamma)}"E0";"Em1"};
{\ar@{->}@/_{0pc}/^{\bB_0(\gamma)}"E0";"G0"};
{\ar@{->}@/_{0pc}/^{\bB_1(\gamma)}"E1";"G1"};
{\ar@{->}@/_{0pc}/^{\bB_2(\gamma)}"E2";"G2"};
{\ar@{->}@/_{0pc}/^{\bB_{3}(\gamma)}"E3";"G3"};
{\ar@{->}@/^{0pc}/^{\bB_{-1}(\gamma)}"Em1";"Gm1"};
{\ar@{->}@/^{1pc}/^{\bB_{0}^*(\gamma)}"E1";"G0"};
{\ar@{->}@/^{1pc}/^{\bB_{1}^*(\gamma)}"E2";"G1"};
{\ar@{->}@/^{1pc}/^{\bB_{2}^*(\gamma)}"E3";"G2"};
{\ar@{->}@/^{0.8pc}/^{\bB^*_{-1}(\gamma)}"E0";"Gm1"};
\end{xy}
\label{eq:elliptic_pre_complex_diagram_tame}
\eeq
That is, for each $\alpha \in \bbN$, assume there exist tame Fréchet vector bundles $\scrV_{\alpha}, \scrW_{\alpha} \rightarrow \scrU$ with fibers 
\[
\scrV_{\alpha}|_{\gamma} = \Gamma(\bbF_{\alpha,\gamma};\bbG_{\alpha,\gamma}), \qquad \scrW_{\alpha}|_{\gamma} = \Gamma(0;\bbL_{\alpha,\gamma}),
\]
with model spaces $\Gamma(\bbF_{\alpha};\bbG_{\alpha})$ and $\Gamma(0;\bbL_{\alpha})$, respectively, 
and with tame smooth families of systems as defined in \secref{sec:Tame}:
\[
\begin{aligned}
&\bD_{\alpha}:\scrV_{\alpha} \rightarrow \scrV_{\alpha+1}, \qquad &&\bD_{\alpha}^*:\scrV_{\alpha+1} \rightarrow \scrV_{\alpha}, \\
&\bB_{\alpha}:\scrV_{\alpha} \rightarrow \scrW_{\alpha}, \qquad &&\bB^*_{\alpha}:\scrV_{\alpha+1} \rightarrow \scrW_{\alpha},
\end{aligned}
\] 
such that each system acts in the fashion of
\[
(\gamma, \Psi) \mapsto (\gamma, \bD_{\alpha}(\gamma) \Psi), \qquad \gamma \in \scrU,\,\, \Psi \in \Gamma(\bbF_{\alpha,\gamma}; \bbG_{\alpha,\gamma}). 
\]
For every $\gamma \in \scrU$, along each fiber, the systems $\bD_{\alpha}^*(\gamma)$ and $\bB^*_{\alpha}(\gamma)$ satisfy a generalized Green formula \eqref{eq:integration_by_parts_elliptic_pre_complex} with respect to the associated $L^{2}$-inner products $\bra \cdot, \cdot \ket_{\gamma}$, as in \thmref{thm:adjoint_family}: that is, for every $\Psi \in \Gamma(\bbF_{\alpha,\gamma}; \bbG_{\alpha,\gamma})$ and $\Theta \in \Gamma(\bbF_{\alpha+1,\gamma} \bbG_{\alpha+1,\gamma})$ it holds that
\[
\bra \bD_{\alpha}(\gamma) \Psi, \Theta \ket_{\gamma} 
= \bra \Psi, \bD^*_{\alpha}(\gamma) \Theta \ket_{\gamma} 
+ \bra \bB_{\alpha}(\gamma) \Psi, \bB^*_{\alpha}(\gamma) \Theta \ket_{\gamma}.
\]

\begin{definition}
\label{def:tame_family_of_elliptic_pre_complexes}
A family of diagrams \eqref{eq:elliptic_pre_complex_diagram_tame} is called a \emph{tame smooth family of elliptic pre-complexes} if, for every $\gamma \in \scrU$, the diagram $(\bD_{\bullet}(\gamma))$ is an elliptic pre-complex satisfying the properties described above.

The family is called \emph{disrupted} if, for every $\gamma \in \scrU$, the diagram $(\bD_{\bullet}(\gamma))$ is instead a disrupted elliptic pre-complex (cf. \defref{def:finite_elliptic_pre_complex})
\end{definition}
The lifted complexes provided by \eqref{thm:corrected_complex} for each elliptic pre-complex $(\bD_{\bullet}(\gamma))$ in the family collectively yield the following families of systems: 
\beq
\begin{aligned}
&\fbD_{\alpha} : \scrV_{\alpha} \rightarrow \scrV_{\alpha+1}, \qquad
&\fbD_{\alpha}^* : \scrV_{\alpha+1} \rightarrow \scrV_{\alpha}, \\ 
& \bP_{\alpha-1} : \scrV_{\alpha} \rightarrow \scrV_{\alpha-1}, \qquad
&\tbP_{\alpha-1} : \scrV_{\alpha} \rightarrow \scrV_{\alpha}.
\end{aligned}
\label{eq:rough_induced_bundle_maps}
\eeq
operating in the same fashion as above, where $\bP_{\alpha-1}(\gamma)$ and $\tbP_{\alpha-1}(\gamma)$ are the mappings associated with the auxiliary decomposition induced by $\fbD_{\alpha-1}(\gamma)$ in either \defref{def:aux_decomposition} and \defref{def:aux_decompositionD}. 

The conditions under which the tame and smooth dependence on the parameter $\gamma$ is preserved after passing to the lifted complex are now identified. Denote by
\[
\frakI_{\alpha} : \scrV_{\alpha} \rightarrow \scrV_{\alpha}
\]
the family of systems acting as
\[
(\gamma, \Psi) \mapsto (\gamma, \frakI_{\alpha}(\gamma)\, \Psi),
\]
where $\frakI_{\alpha}(\gamma): \Gamma(\bbF_{\alpha,\gamma}; \bbG_{\alpha,\gamma}) \rightarrow \Gamma(\bbF_{\alpha,\gamma}; \bbG_{\alpha,\gamma})$ is the $L^2$-orthogonal projection (with respect to the $\gamma$-dependent $L^2$-inner product) onto the finite-dimensional space
\[
\begin{aligned}
&\NN:\quad && \module_{\N}^{\alpha}(\gamma) = \ker\big(\bD_{\alpha}(\gamma),\bD_{\alpha-1}^*(\gamma),\bB_{\alpha-1}^*(\gamma)\big), \\
&\DD:\quad && \module_{\D}^{\alpha}(\gamma) = \ker\big(\bD_{\alpha}(\gamma),\bD_{\alpha-1}^*(\gamma),\bB_{\alpha}(\gamma)\big).
\end{aligned}
\]

\begin{theorem}
\label{thm:smooth_elliptic_pre_complex}
Suppose that there exists $\beta_0 \in \bbN$ such that for all $\alpha \leq \beta_0$, $\frakI_{\alpha}:\scrV_{\alpha} \rightarrow \scrV_{\alpha}$ is a tame smooth family of systems. Then for all $\alpha \leq \beta_0+1$, the families of systems in \eqref{eq:rough_induced_bundle_maps} are tame smooth families as well.  
\end{theorem}
The assumption on the mappings $\frakI_{\alpha}$ in the theorem is essential. As shown in the proof of \thmref{thm:corrected_complex}—particularly during the induction step at \eqref{eq:G_def}—for each $\gamma \in \scrU$, the mapping $\bP_{\alpha}(\gamma)$ is defined by the relations:
\beq
\begin{aligned}
&\NN:\quad && \bP_{\alpha}(\gamma) = \frakS_{\alpha}(\gamma)\big(\fbD_{\alpha}^*(\gamma) \oplus \bB^*_{\alpha}(\gamma) \oplus 0 \oplus 0\big), \\
&\DD:\quad && \bP_{\alpha}(\gamma) = \frakS_{\alpha}(\gamma)\big(\fbD_{\alpha}^*(\gamma) \oplus 0 \oplus 0 \oplus 0\big),
\end{aligned}
\label{eq:parametrized_Ggamma}
\eeq
where, in each case, $\frakS_{\alpha}(\gamma)$ is defined as a left inverse of the overdetermined elliptic systems:
\beq
\begin{aligned}
&\NN:\quad && \fbD_{\alpha}^*(\gamma)\fbD_{\alpha}(\gamma) \oplus \bB^*_{\alpha}(\gamma)\fbD_{\alpha}(\gamma) \oplus \tbP_{\alpha-1}(\gamma) \oplus \frakI_{\alpha}(\gamma), \\
&\DD:\quad && \fbD_{\alpha}^*(\gamma)\fbD_{\alpha}(\gamma) \oplus \bB_{\alpha}(\gamma) \oplus \tbP_{\alpha-1}(\gamma) \oplus \frakI_{\alpha}(\gamma).
\end{aligned}
\label{eq:parametrized_inverse}
\eeq

In view of this defining relation, the family of systems $\gamma \mapsto \frakS_{\alpha}(\gamma)$ may fail to retain smooth and tame dependence on $\gamma$ if the assumption on $\gamma\mapsto\frakI_{\alpha}(\gamma)$ is dropped. This is because projections onto kernels of even parameterized elliptic problems, including differential ones, may vary discontinuously with respect to the parameter. For instance, the family of projections onto the space of Killing fields, which arises as the kernel of an elliptic system smoothly and tamely parameterized by Riemannian metrics, is not continuous (cf.\ \cite{Ebi70}).

Since the proof is very brief and does not require any further technical detail, we provide it here:
\begin{PROOF}{\thmref{thm:smooth_elliptic_pre_complex}}
Without loss of generality, we assume that the family of elliptic pre-complexes is based on Neumann conditions, as the proof for Dirichlet conditions is analogous with minor adjustments. We proceed by induction on $\alpha \leq \beta_0+1$. 

The base case is trivially satisfied: indeed, by applying \thmref{thm:corrected_complex} for every $\gamma \in \scrU$ and noting that $\scrN(\fbD^*_{-1}(\gamma)) = 0$, we have:
\[
\begin{aligned}
&\fbD_{0} = \bD_{0}, \qquad &\fbD^*_{0} = \bD^*_{0}, \\[8pt]
&\bP_{-1} = 0, \qquad &\tbP_{-1} = 0.
\end{aligned}
\]

For the induction hypothesis, assume that for some $\alpha \leq \beta_0$, the families of systems in \eqref{eq:rough_induced_bundle_maps} are all tame and smooth. The induction step amounts to establish that the systems in \eqref{eq:rough_induced_bundle_maps} are tame and smooth for $\alpha$ replaced by $\alpha + 1$.

By the assumption on $\frakI_{\alpha}$, under the induction hypothesis, the family of systems in \eqref{eq:parametrized_inverse} is a tame smooth family of overdetermined elliptic injective systems. Thus, by \thmref{thm:inverses_family}, the corresponding family of left inverses $\gamma\mapsto\frakS(\gamma)$ is also a tame smooth family. By the relation \eqref{eq:parametrized_Ggamma}, this implies that $\gamma\mapsto\bP_{\alpha}(\gamma)$ is a tame smooth family since it is the composition of tame smooth families. Consequently, $\gamma\mapsto\tbP_{\alpha}(\gamma) = \fbD_{\alpha}(\gamma) \bP_{\alpha}(\gamma)$ is also a tame smooth family, being the composition of tame smooth maps.

Now, from the formula for $\bC_{\alpha+1}$ in \eqref{eq:recursive_correction}, we deduce that $\gamma\mapsto\bC_{\alpha+1}(\gamma)$ is a tame smooth family since it is the composition of tame smooth families. Hence, $\gamma\mapsto\fbD_{\alpha+1}(\gamma)$ is a tame smooth family due to the relation
\[
\fbD_{\alpha+1}(\gamma) = \bD_{\alpha+1}(\gamma) + \bC_{\alpha+1}(\gamma).
\]

By \thmref{thm:adjoint_family}, the adjoint families $\bC^*_{\alpha+1}(\gamma)$ and $\bD_{\alpha+1}^*(\gamma)$ are tame and smooth. Thus,
\[
\fbD_{\alpha+1}^*(\gamma) = \bD_{\alpha+1}^*(\gamma) + \bC_{\alpha+1}^*(\gamma)
\]
is also a tame smooth family, as it is the sum of tame smooth maps.

By induction, we conclude that the families of systems in \eqref{eq:rough_induced_bundle_maps} are tame and smooth for all $\alpha\leq\beta_0+1$.
\end{PROOF}

From the perspective of Fredholm and index theory, arguably the original motivation for studying elliptic complexes \cite{AB67, RS82, SS19, DR22}, it is worth noting that the \emph{Euler characteristic} of the lifted complex is independent of the parameter $\gamma \in \scrU$, as long as the original family of elliptic pre-complexes forms a tame smooth family.

\begin{theorem}
\label{thm:index}
Suppose that the family of elliptic pre-complexes $(\bD_{\bullet})$ is finite (cf.\ \defref{def:finite_elliptic_pre_complex1}) and continuously parameterized by $\scrU$. Then, the Euler characteristics defined in \defref{def:Neumann_Euler_characteristic} and \defref{def:Dirichlet_Euler_characteristic}:
\beq
\begin{aligned} 
&\NN:\qquad &&\mathscr{X}_{\N} = \sum_{\alpha=0}^{\alpha_0}(-1)^{\alpha}\dim\module^{\alpha}_{\N}(\gamma), \\[8pt]
&\DD:\qquad &&\mathscr{X}_{\D} = \sum_{\alpha=0}^{\alpha_0}(-1)^{\alpha}\dim\module^{\alpha}_{\D}(\gamma),
\end{aligned}
\label{eq:euler_charactristic} 
\eeq
remain constant for every $\gamma \in \scrU$ lying within the same connected component.
\end{theorem}


Throughout the next discussion and proofs of these theorems, we again freely invoke basic results on Fredholm and compact operators between Hilbert spaces (e.g., \cite[App.~A.6–A.7]{Tay11a}, \cite{EE18}, or \cite{Kat80}), as well as algebraic facts about cochain complexes (cf. \cite{RS82}).  

We note that since every vector bundle is locally trivializable, it suffices to assume that $\scrV_{\alpha} = \scrU \times \Gamma(\bbF_{\alpha}; \bbG_{\alpha})$ and $\scrW_{\alpha} = \scrU \times \Gamma(\bbF; \bbG)$ are trivial bundles. Under this assumption, the operation of the mappings in \eqref{eq:rough_induced_bundle_maps} reduces to the form of \eqref{eq:tame_bundle_maps}, namely:
\[
\begin{aligned}
&\fbD_{\alpha} : \scrU \times \Gamma(\bbF_{\alpha}; \bbG_{\alpha}) \rightarrow \Gamma(\bbF_{\alpha+1}; \bbG_{\alpha+1}), \qquad
&&\fbD_{\alpha-1}^* : \scrU \times \Gamma(\bbF_{\alpha}; \bbG_{\alpha}) \rightarrow \Gamma(\bbF_{\alpha-1}; \bbG_{\alpha-1}), \\[8pt]
&\frakG_{\alpha-1} : \scrU \times \Gamma(\bbF_{\alpha}; \bbG_{\alpha}) \rightarrow \Gamma(\bbF_{\alpha-1}; \bbG_{\alpha-1}), \qquad
&&\tbP_{\alpha-1} : \scrU \times \Gamma(\bbF_{\alpha}; \bbG_{\alpha}) \rightarrow \Gamma(\bbF_{\alpha}; \bbG_{\alpha}).
\end{aligned}
\]
Before proceeding, we discuss some considerations on how to approach the proof of \thmref{thm:index}. In particular, it is worth noting what we might have done, following the classical theory of Dirac operators \cite[Ch.~10]{Tay11b}. Under Neumann conditions, one would consider the tame smooth family of maps (due to the assumption in the statement of \thmref{thm:index}):
\[
\begin{split} 
&\fbD_{\mathrm{e}} + \fbD^*_{\mathrm{e}} : \scrU \times \Gamma(\bbF_{\mathrm{e}}; \bbG_{\mathrm{e}}) \rightarrow \Gamma(\bbF_{\mathrm{o}}; \bbG_{\mathrm{o}}), 
\end{split} 
\]
where
\[
\begin{split}
&\Gamma(\bbF_{\mathrm{e}}; \bbG_{\mathrm{e}}) = \bigoplus_{\alpha=0}^{\alpha_0+1} \Gamma(\bbF_{2\alpha}; \bbG_{2\alpha}), 
\end{split}
\]
and, for each $\gamma \in \scrU$ and $\alpha \in \Nzero$, operate as:
\[
\begin{split} 
&((\fbD_{\mathrm{e}}(\gamma) + \fbD^*_{\mathrm{e}}(\gamma))\Psi)_{2\alpha+1} = \fbD_{2\alpha}(\gamma)\Psi_{2\alpha} + \fbD^*_{2\alpha+2}(\gamma)\Psi_{2\alpha+2}. 
\end{split} 
\]
However, unlike in the classical theory, the operators in this family are not Fredholm since the Hodge-like decompositions in \eqref{eq:Hodgelikesmooth} applied for every $\gamma$ require that $\fbD^*_{\mathrm{e}}(\gamma)$ is supplemented by the boundary condition on $\bB^*_{\alpha}(\gamma)$, which varies with $\gamma$. 

To resolve this, we could instead consider:
\[
\begin{split} 
&\fbD_{\mathrm{e}} + \bP_{\mathrm{e}} : \scrU \times \Gamma(\bbF_{\mathrm{e}}; \bbG_{\mathrm{e}}) \rightarrow \Gamma(\bbF_{\mathrm{o}}; \bbG_{\mathrm{o}}), 
\end{split} 
\]
defined analogously to the above. By the construction of $\bP_{\alpha}$ in \eqref{eq:bP_def_proof}, it is surjective onto the complement of $\scrR(\fbD_{\alpha-1})$ (modulo the cohomology modules). Consequently, due to the Hodge-like decompositions applied for each $\gamma$, the systems in these families are indeed Fredholm. The kernel and cokernel of $\fbD_{\mathrm{e}}(\gamma) + \bP_{\mathrm{e}}(\gamma)$ are given by:
\beq 
\ker(\fbD_{\mathrm{e}}(\gamma) + \bP_{\mathrm{e}}(\gamma)) = \bigoplus_{\alpha=0}^{\left\lfloor\alpha_0/2 \right\rfloor} \module_{\N}^{2\alpha}(\gamma), \qquad
\mathrm{Coker}(\fbD_{\mathrm{e}}(\gamma) + \bP_{\mathrm{e}}(\gamma)) = \bigoplus_{\alpha=0}^{\left\lfloor\alpha_0/2 \right\rfloor} \module_{\N}^{2\alpha+1}(\gamma),
\label{eq:kernel_co_kernel_fredholm}
\eeq 
and both are finite-dimensional for every $\gamma \in \scrU$. The index of the Fredholm operator $\fbD_{\mathrm{e}} + \bP_{\mathrm{e}}$ is therefore $\mathscr{X}_{\N}$, as defined in \eqref{eq:euler_charactristic}. 

Unfortunately, the classical result stating that the index defines a continuous map into $\bbZ$ (e.g., \cite[App.~A.7]{Tay11a} or \cite[Ch.~3]{EE18}) does not hold in the Fréchet category \cite[p.~15]{RS82}. It is worth noting that there exist additional criteria under which a family of Fredholm operators between Fréchet spaces does yield a continuous index (see, e.g., \cite{DR22, Ger16}, and references therein), but such assumptions are too restrictive for the level of generality considered here.

A possible solution would be to adapt the operation of $\fbD_{\mathrm{e}} + \bP_{\mathrm{e}}$ to act between Banach spaces. However, in the most general setting, this cannot be done due to the fact that the systems involve operators of varying orders, causing $\fbD_{\mathrm{e}}$ and $\bP_{\mathrm{e}}$ to take values in different Sobolev spaces—where \eqref{eq:WspHodge}, rather than \eqref{eq:Hodgelikesmooth}, applies. Consequently, it is possible that the ranges of these mappings are not even closed when extended to Sobolev spaces, let alone that they satisfy the Fredholm property.

Therefore, the only practical approach is to order-reduce the entire complex $(\fbD_{\bullet})$, as is done in the works surveyed in \secref{sec:comparison_theory}. Our proof is essentially a careful adaptation of their arguments, except that in our setting the boundary systems are technically not part of the complex and must be appended separately to accommodate the different summands in our Hodge decompositions \thmref{thm:hodge_like_corrected_complex}--\thmref{thm:hodge_like_corrected_complexD}. Also different from their analysis, the normality of the boundary systems also plays an important role in identifying the cokernel of the resulting Fredholm maps:
\begin{PROOF}{\thmref{thm:index}}
Without loss of generality, we prove the theorem for Neumann conditions, where the relevant cochain complex is $(\fbD_{\bullet})$ as in \eqref{eq:Neumann_complex}. For Dirichlet conditions, the corresponding complex is $(\fbD^*_{\bullet})$ as in \eqref{eq:Dirichlet_complex_co}, and the proof adapts accordingly. We also note that it suffices to establish the result in a neighborhood of a fixed $\gamma_{0} \in \scrU$.

For each $\alpha \in \Nzero$, we produce sharp tuples $(J_{\alpha}, L_{\alpha}; J_{\alpha+1},L_{\alpha+1})$ for each adapted Green system $\fbD_{\alpha}(\gamma_0)$ as in \defref{def:sharp_tuples}--\defref{def:adapting_operator}, sufficiently large so that  
\[
\fbD_{\alpha}(\gamma_0) : W^{J_{\alpha}, L_{\alpha}}_{2}(\bbF_{\alpha}; \bbG_{\alpha}) \rightarrow W^{J_{\alpha+1}, L_{\alpha+1}}_{2}(\bbF_{\alpha+1}; \bbG_{\alpha+1}).
\]
By \corrref{corr:G0props}, since $\bD_{\alpha}(\gamma) - \fbD_{\alpha}(\gamma) \in \OP(0,0)$ and $\gamma \mapsto \bD_{\alpha}(\gamma)$ is tame and smooth, the tuples $(J_{\alpha}, L_{\alpha}; J_{\alpha+1}, L_{\alpha+1})$ remain sharp for $\fbD_{\alpha}(\gamma)$ near $\gamma_{0}$.

By continuity, the Sobolev extensions preserve the relation $\fbD_{\alpha+1}(\gamma)\,\fbD_{\alpha}(\gamma)=0$, and each $\fbD_{\alpha}(\gamma)$ has closed range (\propref{prop:Ak_closed_range}). Let 
\[
\Pi_{\alpha} : W^{J_{\alpha}, L_{\alpha}}_{2}(\bbF_{\alpha}; \bbG_{\alpha}) \rightarrow L^{2}(\bbF_{\alpha}; \bbG_{\alpha}), 
\qquad 
\Lambda_{\alpha} : W^{0, L'_{\alpha}}_{2}(0; \bbL_{\alpha}) \rightarrow L^{2}(0; \bbL_{\alpha})
\]
be order-reducing operators independent of $\gamma$. Define
\[
\begin{split}
\tilde{\bD}_{\alpha}(\gamma) &= \Pi_{\alpha+1}\bD_{\alpha}(\gamma)\Pi_{\alpha}^{-1}, \\
\tilde{\bB}^*_{\alpha}(\gamma) &= \Lambda_{\alpha}\bB_{\alpha}(\gamma)\Pi_{\alpha}^{-1}, \\
\tilde{\fbD}_{\alpha}(\gamma) &= \Pi_{\alpha+1}\fbD_{\alpha}(\gamma)\Pi_{\alpha}^{-1},
\end{split}
\]
with all maps acting between the corresponding $L^{2}$-spaces.  

This construction provides $\tilde{\fbD}_{\alpha}(\gamma)\tilde{\fbD}_{\alpha-1}(\gamma)=0$, which is to say that $(\tilde{\fbD}_{\bullet})$ is itself a cochain complex, with cohomology groups denoted $\tilde{\module}^{\alpha}_{\NN}(\gamma)$. Since the order-reducing operators are isomorphisms, the resulting cochain complex is isomorphic to the original one, and hence the cohomology groups themselves are isomorphic $\tilde{\module}^{\alpha}_{\NN}(\gamma) \simeq \module^{\alpha}_{\NN}(\gamma)$. The Sobolev extensions of the Hodge decompositions  (\thmref{thm:hodge_like_corrected_complex} and \thmref{thm:hodge_like_corrected_disrupted_complex} in the disrupted case) then yield after composing with the order reducing operators: 
\[
L^{2}(\bbF_{\alpha+1}; \bbG_{\alpha+1}) 
\simeq \image{\tilde{\fbD}_{\alpha}(\gamma)} 
\oplus \image{\tilde{\fbD}^*_{\alpha+1}(\gamma)|_{\ker\tilde{\bB}^*_{\alpha}(\gamma)}} 
\oplus \tilde{\module}^{\alpha}(\gamma).
\]
Hence, we define the boundary-augmented family of operators, analogous to the Dirac operators discussed above,
\[
(\tilde{\fbD}_{\mathrm{e}} + \tilde{\fbD}_{\mathrm{e}}^*) \oplus \tilde{\bB}_{\mathrm{e}}^* 
: \scrU \times L^{2}(\bbF_{\mathrm{e}}; \bbG_{\mathrm{e}}) 
\to L^{2}(\bbF_{\mathrm{o}}; \bbG_{\mathrm{o}})\times L^{2}(0; \bbL_{\mathrm{e}}).
\]
By the order-reduced Hodge decompositions, and since $\tilde{\bB}_{\mathrm{e}}^*$ is surjective 
(the original boundary operators are normal and the order-reducing operators are isomorphisms), 
this map is Fredholm (semi-Fredholm in the disrupted case, where $\module^{\alpha_0}_{\N}$ may fail to be finite dimensional), with kernel and cokernel finite dimensional and isomorphic to those in \eqref{eq:kernel_co_kernel_fredholm}, 
due to the isomorphism between of cohomology groups of $(\tilde{\fbD}_{\bullet})$ and $(\fbD_{\bullet})$. 
B
y comparison with \eqref{eq:euler_charactristic}, we conclude that the index of this Fredholm (semi-Fredholm) operator is
\[
\mathrm{Ind}\big((\tilde{\fbD}_{\mathrm{e}}(\gamma) + \tilde{\fbD}_{\mathrm{e}}^*(\gamma)) \oplus \tilde{\bB}_{\mathrm{e}}^*(\gamma)\big) 
= \mathscr{X}_{\N}.
\]

On the other hand, since $\sigma(\fbD(\gamma) - \bD(\gamma))=0$ (\propref{prop:correction}), the difference $\tilde{\bD}_{\alpha}(\gamma) - \tilde{\fbD}_{\alpha}(\gamma)$ is compact for each $\gamma$, and the same holds for their adjoints. Consequently, also
\[
(\tilde{\bD}_{\mathrm{e}} + \tilde{\bD}_{\mathrm{e}}^*) \oplus \tilde{\bB}_{\mathrm{e}}^* 
: \scrU \times L^{2}(\bbF_{\mathrm{e}}; \bbG_{\mathrm{e}}) 
\to L^{2}(\bbF_{\mathrm{o}}; \bbG_{\mathrm{o}})\times L^{2}(0; \bbL_{\mathrm{e}})
\]
is a Fredholm family (resp. semi-Fredholm in the disrupted case).  

By standard Fredholm theory, the index
\[
\gamma \mapsto\mathrm{Ind}\big((\tilde{\bD}_{\mathrm{e}}(\gamma) + \tilde{\bD}_{\mathrm{e}}^*(\gamma) ) \oplus \tilde{\bB}_{\mathrm{e}}^*(\gamma) \big)
\]
is locally constant in $\gamma$ and invariant under compact perturbations. Hence
\[
\mathrm{Ind}\big((\tilde{\bD}_{\mathrm{e}}(\gamma)  + \tilde{\bD}_{\mathrm{e}}^*(\gamma) ) \oplus \tilde{\bB}_{\mathrm{e}}^*(\gamma) \big)
= \mathrm{Ind}\big((\tilde{\fbD}_{\mathrm{e}}(\gamma)  + \tilde{\fbD}_{\mathrm{e}}^*(\gamma) ) \oplus \tilde{\bB}_{\mathrm{e}}^*(\gamma) \big)= \mathscr{X}_{\N}
\]
proving the claim 
\end{PROOF}
\subsection{Technical proofs} 
\label{sec:techincal_tame}
By \cite[Thm.~3.1.1]{Ham82} and the definition of a tame smooth map \cite[Sec.~2.1, p.~140]{Ham82}, tameness, like smoothness, is a local property. Consequently, since every vector bundle is locally trivializable, it suffices to prove \thmref{thm:adjoint_family} and \thmref{thm:inverses_family} in the case where $\scrV = \scrU \times \Gamma(\bbE; \bbJ)$ and $\scrW = \scrU \times \Gamma(\bbF; \bbG)$ are trivial bundles, reducing the operation of $\bD$ to \eqref{eq:tame_bundle_maps}.

Throughout the proofs, for each $s \in \mathbb{R}$ and $\gamma$, let $\|\cdot\|_{s, \gamma}$ denote the $W^{s, s+1/2}_{2}$-norm induced by the Riemannian metric $g(\gamma)$ and the volume form $d\Volume(\gamma)$ given in the statement of \thmref{thm:adjoint_family}. Since the manifold is compact, and both the Riemannian metrics and volume forms depend tamely and smoothly on $\gamma$, it follows by a standard argument that the induced Sobolev norms are smoothly and tamely equivalent. That is, given a fixed $\gamma_{0} \in \scrU$, for each $s \in \mathbb{R}$, there exist smooth functions $C_{s}, c_{s} : \scrU \to \mathbb{R}_{> 0}$ such that $C_{s}(\gamma), c_{s}(\gamma) \to 1$ as $\gamma \to \gamma_{0}$, and  
\beq
c_{s}(\gamma) \, \|\cdot\|_{s, \gamma} \leq \|\cdot\|_{s, \gamma_0} \leq C_{s}(\gamma) \, \|\cdot\|_{s, \gamma}. 
\label{eq:gamma_sobolev_norms}
\eeq  

This implies that, regardless of which $\gamma$-dependent Sobolev spaces induce the tame grading for $\Gamma(\bbE;\bbJ)$ and $\Gamma(\bbF;\bbG)$, the resulting gradings will be tamely equivalent \cite[Def.~1.1.3, p.~134]{Ham82}. Therefore, to prove the tame estimate \eqref{eq:tame_estimate} with respect to these norms, it suffices to establish it for a fixed $\gamma_0 \in \scrU$.  

Moreover, in the proofs of both theorems, it suffices to assume that $\bD(\gamma)$ has the same standard (or sharp) tuples for every $\gamma \in \scrU$. Indeed, by choosing a tame grading of $\Gamma(\bbE; \bbJ)$ and $\Gamma(\bbF; \bbG)$ based on products of Sobolev spaces induced by some fixed $\gamma_0$, it follows that $\bD$ defines a continuous map
\[
\bD: \scrU \to \scrL(J_0, L_0; I_0, K_0),
\]
given by $\gamma \mapsto \bD(\gamma)$.

%

With this established, we conclude that \thmref{thm:adjoint_family} and \thmref{thm:inverses_family} hold under the assumption that $\bD$ is given by \eqref{eq:tame_bundle_maps}, and that the families of adjoints $\bD^*$ and left inverses $\frakS$ reduce to mappings  
\[
(\scrU \subseteq F) \times \Gamma(\bbF; \bbG) \to \Gamma(\bbE; \bbJ),
\]  
with the same standard tuples for every $\gamma \in \scrU$.  

For the remainder of this section, when deriving estimates, we abandon the convention $\lesssim$ to emphasize that the constants in the inequalities are independent of $\gamma$ and $\Psi$, as required by the definition of tame smoothness outlined around \eqref{eq:tame_estimate}. 
In the proof of \thmref{thm:adjoint_family}, we need to linearize the family of inner products \eqref{eq:smooth_familiy_inner_products}. To this end, we use the identification $T_{\gamma} \scrU = F$ and differentiate under the integral along the curves of volume forms $d\Volume(\gamma + t\sigma)$ and fiber metrics $(\cdot, \cdot)_{g(\gamma + t\sigma)}$, where $\sigma \in F$. Specifically, for every $\Upsilon, \Xi \in \Gamma(\bbE; \bbJ)$, we obtain  
\beq
\dertZero \bra \Upsilon, \Xi \ket_{\gamma + t\sigma} = \int_{M} \dertZero (\Upsilon, \Xi)_{g(\gamma+t\sigma)} \, d\Volume(\gamma) + (\Upsilon, \Xi)_{g(\gamma)} \, \dertZero d\Volume(\gamma + t\sigma).
\label{eq:difffentiate_inner_product}
\eeq
Since $\gamma \mapsto d\Volume(\gamma)$ defines a tame smooth family of volume forms, and $d\Volume(\gamma + t\sigma)$ is a top form for every $t$, it follows that  
\[
\dertZero d\Volume(\gamma + t\sigma) = f(\gamma) \sigma \, d\Volume(\gamma),
\]  
for a tame smooth family of functionals $f : \scrU \times \Gamma(\bbE; \bbJ) \rightarrow \mathbb{R}$.

Moreover, since $g(\gamma)$ is a fiber metric, we can define a self-adjoint bundle homomorphism, denoted—by a slight abuse of notation from \eqref{eq:partial_derivative}—as  
\[
\D_{\sigma} g : (\scrU \subseteq F) \times \Gamma(\bbE; \bbJ) \rightarrow \Gamma(\bbE; \bbJ),
\]  
such that  
\[
\dertZero(\Upsilon, \Xi)_{g(\gamma+t\sigma)} + f(\gamma) \sigma \, (\Upsilon, \Xi)_{g(\gamma)} = (\D_{\sigma} g(\gamma) \Upsilon, \Xi)_{\gamma} = (\Upsilon, \D_{\sigma} g(\gamma) \Xi)_{\gamma}.
\]
Since $ \D_{\sigma}g(\gamma) $ is defined fiberwise, it is a tensorial, smooth, and tame operation, making it a tame smooth family of systems in $ \OP(0,0) $. Combining this in \eqref{eq:difffentiate_inner_product}, we thus have the formula: 
\beq
\dertZero \bra \Upsilon, \Xi \ket_{\gamma + t\sigma} = \bra \Upsilon, \D_{\sigma}g(\gamma) \Xi \ket_{\gamma} = \bra \D_{\sigma}g(\gamma) \Upsilon, \Xi \ket_{\gamma}, \qquad \Upsilon, \Xi \in \Gamma(\bbE; \bbJ).
\label{eq:fiber_metric_differentiated}
\eeq

\begin{lemma}
The family of adjoints $\bD^*$ (of the zero-class constituents of $\bD$) is tame.
\end{lemma}

\begin{proof}
For $s, s' \in \bbR$, let $\|\cdot\|_{\mathrm{op}(s,s',\gamma)}$ denote the corresponding operator norm on continuous linear maps $W^{s,s+1/2}_{2} \rightarrow W^{s',s'+1/2}_{2}$, where the domain and codomain are equipped with the norms in \eqref{eq:gamma_sobolev_norms}. Due to the equivalence in \eqref{eq:gamma_sobolev_norms}, these operator norms are also equivalent by definition, in the following manner:
\[
\frac{c_{s'}(\gamma)}{C_{s}(\gamma)} \, \|\cdot\|_{\mathrm{op}(s, s', \gamma)} \leq \|\cdot\|_{\mathrm{op}(s, s', \gamma_0)} \leq \frac{C_{s'}(\gamma)}{c_{s}(\gamma)} \, \|\cdot\|_{\mathrm{op}(s, s', \gamma)}.
\]

Given this setup, since $\bD^*$ depends only on the zero-class constituent of $\bD$, we can assume that the latter has zero corresponding classes. Let $m \in \bbZ$ be sufficiently large so that $\bD \in \OP(m, 0)$. Then for every $s\in\bbR$ sufficiently large, by the relation \eqref{eq:norms_operator_duality_adjoints} applied to each inner product separately, we have:
\[
\|\bD^*(\gamma)\|_{\mathrm{op}(s+m, s, \gamma)} = \|\bD(\gamma)\|_{\mathrm{op}(-s + 1/2, -s-m+1/2, \gamma)}.
\] 

Using the equivalences between the norms above, for every $\Theta \in \Gamma(\bbF; \bbG)$ and $\gamma \in \scrU$, we then have
\[
\begin{split}
\|\bD^*(\gamma)\Theta\|_{s, \gamma_0} &\leq C_{s}(\gamma) \, \|\bD^*(\gamma)\Theta\|_{s, \gamma} \\
&\leq C_{s}(\gamma) \, \|\bD^*(\gamma)\|_{\mathrm{op}(s+m, s, \gamma)} \|\Theta\|_{s+m, \gamma} \\
&\leq C_{s}(\gamma) \, \|\bD(\gamma)\|_{\mathrm{op}(-s + 1/2, -s-m+1/2, \gamma)} \|\Theta\|_{s+m, \gamma} \\
&\leq \frac{C_{s}(\gamma) C_{-s+1/2}(\gamma)}{c_{-s-m+1/2}(\gamma) c_{s+m}(\gamma)} \, \|\bD(\gamma)\|_{\mathrm{op}(-s + 1/2, -s-m+1/2, \gamma_0)} \|\Theta\|_{s+m, \gamma_0}.
\end{split}
\]

Since the map $(\gamma, \Psi) \mapsto \bD(\gamma) \Psi$ is tame when the range is equipped with the grading by Sobolev spaces, and the operator norm is continuous, we conclude in particular that  
\[
\gamma \mapsto \|\bD(\gamma)\|_{\mathrm{op}(-s + 1/2, -s-m+1/2,\gamma_0)}
\]  
is a continuous function into $\mathbb{R}_+$. Moreover, since the maps $\gamma \mapsto C_{s}(\gamma)$ and $\gamma \mapsto c_{s}(\gamma)$ are tame—being mappings from a graded Fréchet space into $\mathbb{R}$—it follows that for each $s \in \mathbb{R}$, there exists a norm $\|\cdot\|_{s}$ in the tame grading for $\scrU$ and a constant $C'_{s} > 0$ such that for $\gamma \in \scrU$ near $\gamma_0$,  
\[
\frac{C_{s}(\gamma) C_{-s+1/2}(\gamma)}{c_{-s-m+1/2}(\gamma) c_{s+m}(\gamma)} \, \|\bD(\gamma)\|_{\mathrm{op}(-s+1/2, -s-m+1/2, \gamma_0)} \leq C'_{s} (1 + \|\gamma\|_{s}).
\]

Combining these, we obtain:
\[
\|\bD^*(\gamma)\Theta\|_{s, \gamma_0} \leq C'_{s}(1 + \|\gamma\|_{s}) \|\Theta\|_{s+m, \gamma_0},
\]
showing that $\bD^*$ satisfies a tame estimate \eqref{eq:tame_estimate} in a neighborhood of $\gamma_{0}$. Since $\gamma_0\in\scrU$ was arbitrary, we are done.
\end{proof}

\begin{lemma}
The family of adjoints $\bD^*$ (of the zero-class constituent of $\bD$) is continuous.
\end{lemma}

\begin{proof}
To prove continuity, let $(\gamma_n, \Theta_n) \to (\gamma_0, \Theta_0)$ in $\scrU \times \Gamma(\bbF; \bbG)$ as $n \to \infty$. We aim to show that 
\[
\bD^*(\gamma_n) \Theta_n \to \bD^*(\gamma_0) \Theta_0 \quad \text{in } \Gamma(\bbE; \bbJ).
\]

Fix $\Psi \in W^{S,T}_{2,0}(\bbE; \bbJ)$ for any standard tuples $(S,T;S',T')$ for $\bD(\gamma_0)$. Since we assume we work with the zero-part constituents only, using the definition of the $L^2$-adjoint with respect to the varying coupling $\bra \cdot, \cdot \ket_{\gamma}$, we can write:
\beq
\begin{split}
\bra \Psi, \bD^*(\gamma_n) \Theta_n \ket_{\gamma_0} 
= \bra \bD(\gamma_n) \Psi, \Theta_n \ket_{\gamma_0} 
&+ \Big[\bra \bD(\gamma_n) \Psi, \Theta_n \ket_{\gamma_n} - \bra \bD(\gamma_n) \Psi, \Theta_n \ket_{\gamma_0} \Big] \\
&+ \Big[\bra \Psi, \bD^*(\gamma_n) \Theta_n \ket_{\gamma_0} - \bra \Psi, \bD^*(\gamma_n) \Theta_n \ket_{\gamma_n} \Big].
\end{split}
\label{eq:L2_adjoint_family_continuity}
\eeq

We now analyze the three terms on the right-hand side:

For the first term, by continuity of $\gamma\mapsto\bD(\gamma)$ and the convergence $\Theta_n \to \Theta_0$, we have
\[
\bra \bD(\gamma_n) \Psi, \Theta_n \ket_{\gamma_0} 
\to \bra \bD(\gamma_0) \Psi, \Theta_0 \ket_{\gamma_0} 
= \bra \Psi, \bD^*(\gamma_0) \Theta_0 \ket_{\gamma_0}.
\]

For the second term, since the family of inner products $\bra \cdot, \cdot \ket_{\gamma}$ is continuous in $\gamma$ and $\Theta_n \to \Theta_0$, we obtain
\[
\bra \bD(\gamma_n) \Psi, \Theta_n \ket_{\gamma_n} - \bra \bD(\gamma_n) \Psi, \Theta_n \ket_{\gamma_0} \to 0.
\]

For the third term, we claim that 
\[
\bra \Psi, \bD^*(\gamma_n) \Theta_n \ket_{\gamma_0} - \bra \Psi, \bD^*(\gamma_n) \Theta_n \ket_{\gamma_n} \to 0.
\]
To justify this, note that by the tame estimate from the previous lemma, we have uniform boundedness:
\[
\sup_n \| \bD^*(\gamma_n) \Theta_n \|_{0, \gamma_0} < \infty.
\]
The continuity of the inner products then implies that the operator norms of the functionals
\[
\varphi_n : L^2(\bbE; \bbJ) \to \bbR, \quad \varphi_n(u) = \bra \Psi, u \ket_{\gamma_n} - \bra \Psi, u \ket_{\gamma_0}
\]
tend to zero. Applying this to $u = \bD^*(\gamma_n) \Theta_n$ yields the desired convergence.

Combining all three terms in \eqref{eq:L2_adjoint_family_continuity}, we conclude:
\[
\bra \Psi, \bD^*(\gamma_n) \Theta_n \ket_{\gamma_0} \to \bra \Psi, \bD^*(\gamma_0) \Theta_0 \ket_{\gamma_0}.
\]

Since $\Psi \in W^{S,T}_{2,0}(\bbE; \bbJ)$ was arbitrary, we have:
\[
\bD^*(\gamma_n) \Theta_n \to \bD^*(\gamma_0) \Theta_0 \quad \text{in } W^{-S,-T}(\bbE; \bbJ),
\]
and since this is true for arbitrary tuples $S,T$, we have convergence in $\Gamma(\bbE;\bbJ)$ by the tame grading of Sobolev spaces. 
\end{proof}

\begin{PROOF}{\thmref{thm:adjoint_family}}
To establish that $\bD^*$ is smooth and that its derivatives are tame, it suffices to verify smoothness and tameness with respect to the first variable, since the map is tame linear in the second variable.

Let us identify $T_{\gamma} \scrU \simeq F$, and fix $\gamma \in \scrU$, $\sigma \in F$, and $t \neq 0$. For $\Psi \in W^{S,T}_{2,0}(\bbE; \bbJ)$ and $\Theta \in \Gamma(\bbF; \bbG)$ as in the previous proof, consider the difference quotient derived from the adjoint identity \eqref{eq:adjoint_relation_bundles}:
\beq 
\frac{\bra \Psi, \bD^*(\gamma + t\sigma)\Theta \ket_{\gamma + t\sigma} - \bra \Psi, \bD^*(\gamma)\Theta \ket_{\gamma}}{t}
= \frac{\bra \bD(\gamma + t\sigma)\Psi, \Theta \ket_{\gamma + t\sigma} - \bra \bD(\gamma)\Psi, \Theta \ket_{\gamma}}{t}.
\label{eq:difference_quotient} 
\eeq

As $t \to 0$, by the chain rule and the fact that all quantities are tame and smooth, the right-hand side converges to:
\[
\dertZero \bra \bD(\gamma)\Psi, \Theta \ket_{\gamma + t\sigma}
+ \bra \D_{\sigma}\bD(\gamma)\Psi, \Theta \ket_{\gamma},
\]
where $\D_{\sigma}\bD(\gamma)$ denotes the directional derivative of $\bD$ in the direction $\sigma$ (cf.~\eqref{eq:partial_derivative}).

To compute the left-hand side of \eqref{eq:difference_quotient}, expand the inner product $\bra \cdot, \cdot \ket_{\gamma + t\sigma}$ in powers around $t=0$: 
\[
\bra \cdot, \cdot \ket_{\gamma + t\sigma} 
= \bra \cdot, \cdot \ket_{\gamma} 
+ t\, \dersZero \bra \cdot, \cdot \ket_{\gamma + s\sigma} + o(t).
\]

Substituting this expansion into the inner product on the left-hand side, we obtain:
\[
\begin{split}
\frac{\bra \Psi, \bD^*(\gamma + t\sigma)\Theta \ket_{\gamma + t\sigma} - \bra \Psi, \bD^*(\gamma)\Theta \ket_{\gamma}}{t}
&= \frac{\bra \Psi, \left( \bD^*(\gamma + t\sigma) - \bD^*(\gamma) \right) \Theta \ket_{\gamma}}{t} \\
&\quad + \dersZero \bra \Psi, \bD^*(\gamma + t\sigma)\Theta \ket_{\gamma + s\sigma} 
\\&\quad+ \frac{o(t)}{t}.
\end{split}
\]

Inserting this identity into the expression from the previous step, and using the continuity of $\gamma\mapsto\bD^*(\gamma)$ (as established in the previous lemma) to conclude that 
\[
\bD^*(\gamma + t\sigma) \to \bD^*(\gamma) \quad \text{as } t \to 0,
\]
we arrive at the limit:
\[
\begin{split}
\frac{\bra \Psi, \left( \bD^*(\gamma + t\sigma) - \bD^*(\gamma) \right) \Theta \ket_{\gamma}}{t}
&\to \dertZero \bra \bD(\gamma) \Psi, \Theta \ket_{\gamma + t\sigma}
+ \bra \D_{\sigma} \bD(\gamma) \Psi, \Theta \ket_{\gamma} \\
&\quad - \dersZero \bra \Psi, \bD^*(\gamma) \Theta \ket_{\gamma + s\sigma}.
\end{split}
\]

Applying now the identity  \eqref{eq:fiber_metric_differentiated} and integrating by parts yields:
\[
\frac{\bra \Psi, \left( \bD^*(\gamma + t\sigma) - \bD^*(\gamma) \right) \Theta \ket_{\gamma}}{t}
\to \bra \Psi, \left(\bD^*(\gamma)\D_{\sigma}g(\gamma) 
+ (\D_{\sigma}\bD)^*(\gamma) - \D_{\sigma}g(\gamma)\bD^*(\gamma)\right) \Theta \ket_{\gamma}.
\]

Here, $(\D_{\sigma} \bD)^*(\gamma)$ denotes the family of adjoints of the zero-class constituent of $\D_{\sigma} \bD(\gamma)$, which is continuous and tame in $\gamma$ by our earlier results on general families of adjoints.

Since $\Psi \in W^{S,T}_{2,0}(\bbE; \bbJ)$ and $\Theta \in \Gamma(\bbF; \bbG)$ are arbitrary, we may again conclude, as in the previous proof, that the difference quotient for $\bD^*$ converges in the tame Fréchet topology. Therefore, the partial derivative $\D_{\sigma} \bD^*(\gamma)$ exists and is given by:
\[
\D_{\sigma} \bD^*(\gamma) = \bD^*(\gamma) \D_{\sigma}g(\gamma) 
+ (\D_{\sigma} \bD)^*(\gamma) - \D_{\sigma}g(\gamma) \bD^*(\gamma).
\]

The right-hand side consists of compositions of smooth and tame maps in $\gamma$, confirming that $\bD^*$ is of class $C^1$ with tame derivative. Applying the chain rule, induction, and the smoothness of $\bD$ completes the proof that $\bD^*$ is smooth and tame.
\end{PROOF}
\begin{PROOF}{\thmref{thm:inverses_family}} The proof essentially generalizes the approach in \cite[Sec.~3.3]{Ham82}, which focuses on families of invertible differential operators parameterized by their coefficients, to the setting of Douglis–Nirenberg systems parameterized by an arbitrary tame Fréchet manifold.   

Since it is assumed that $\bD(\gamma)$ has fixed sharp tuples for every $\gamma$, we can compose $\bD(\gamma)$ from the left and right with order-reducing operators independent of $\gamma$, based on these sharp tuples. This allows us to simplify our analysis by assuming that $\bD(\gamma) \in \OP(0,0)$ for every $\gamma \in \scrU$. Given that $\bD(\gamma)$ is overdetermined elliptic and injective, the system $\bD^*(\gamma)\bD(\gamma)$ is elliptic and bijective, hence it admits an actual inverse within the calculus. The left inverse of $\bD(\gamma)$ is then simply the composition of the inverse of $\bD^*(\gamma)\bD(\gamma)$ with $\bD^*(\gamma)$.  

Thus, since $\bD^*$ is already a tame smooth family by \thmref{thm:adjoint_family}, we can replace $\bD(\gamma)$ with $\bD^*(\gamma)\bD(\gamma)$ and assume that $\bD(\gamma)\in\OP(0,0)$ is a bijective elliptic system for every $\gamma \in \scrU$, with $\frakS(\gamma)$ as its inverse. Under this assumption, we show that  
\[
\frakS: \scrU \times \Gamma(\bbF; \bbG) \to \Gamma(\bbE; \bbJ)
\]  
is a tame continuous map. By \cite[Thm.~3.1.1]{Ham82}, this immediately implies smoothness.  

To demonstrate tameness, we use similar notation as in the previous section. The elliptic estimate \eqref{eq:overdetermined_a_priori} satisfied by $\bD(\gamma) \in \OP(0,0)$ with respect to some $\gamma_0 \in \scrU$ reads, for every $s > 0$, as  
\beq
\|\Psi\|_{s, \gamma_0} \leq M_{s}(\gamma) \|\bD(\gamma)\Psi\|_{s, \gamma_0},
\label{eq:elliptic_estimate_tame}
\eeq
where (e.g., \cite[Thm.~3.4]{EE18}),   
\[
\frac{1}{M_{s}(\gamma)} = \inf \big\{ \|\bD(\gamma)\Psi\|_{s, \gamma_0} \;:\; \|\Psi\|_{s, \gamma_0} = 1 \big\} > 0.
\]
Since the correspondence $(\gamma, \Psi) \mapsto \bD(\gamma)\Psi$ is tame and smooth, it follows that $\gamma \mapsto M_{s}(\gamma) \in \mathbb{R}$ is also tame and smooth. Consequently, there exists a grading $\|\cdot\|_{s}$ for $\scrU \subseteq F$ such that for $\gamma \in \scrU$ near $\gamma_0$,  
\[
M_{s}(\gamma) \leq 1 + C_{s}\|\gamma\|_{s},
\]
for constants $C_{s} > 0$ depending only on $s$. Inserting into \eqref{eq:elliptic_estimate_tame}, we find that in the vicinity of $\gamma_0$ we have  
\[
\|\Psi\|_{s, \gamma_0} \leq (1 + C_{s}\|\gamma\|_{s})\|\bD(\gamma)\Psi\|_{s, \gamma_0}.
\]
Replacing $\Psi$ with $\frakS(\gamma)\Theta$, we obtain  
\[
\|\frakS(\gamma)\Theta\|_{s, \gamma_0} \leq (1 + C_{s}\|\gamma\|_{s})\|\Theta\|_{s, \gamma_0},
\]
which reads that $\frakS$ satisfies a tame estimate as in \eqref{eq:tame_estimate}.  

To prove continuity, for every $s > 0$ and fixed $(\gamma_0, \Theta_0) \in \scrU \times \Gamma(\bbF; \bbG)$, we insert $\Psi = \frakS(\gamma)\Theta - \frakS(\gamma_0)\Theta_0$ into \eqref{eq:elliptic_estimate_tame}:  
\[
\|\frakS(\gamma)\Theta - \frakS(\gamma_0)\Theta_0\|_{s, \gamma_0} \leq M_{s}(\gamma)\|\Theta - \bD(\gamma)\frakS(\gamma_0)\Theta_0\|_{s, \gamma_0}.
\]
Due to the continuity of all quantities involved with respect to $\Theta$ and $\gamma$, we find that  
\[
\|\frakS(\gamma)\Theta - \frakS(\gamma_0)\Theta_0\|_{s, \gamma_0} \to 0 \quad \text{as} \quad (\gamma, \Theta) \to (\gamma_0, \Theta_0) \quad \text{in} \quad \scrU \times \Gamma(\bbF; \bbG).
\]
Since this holds for every $s > 0$, we conclude that  
\[
\frakS(\gamma)\Theta \to \frakS(\gamma_0)\Theta_0 \quad \text{in } \Gamma(\bbE; \bbJ) \text{ as }(\gamma,\Theta)\rightarrow(\gamma_0,\Theta_0)  
\]
establishing the continuity of $\frakS$ as required.
\end{PROOF}

\chapter{Examples: Detailed Analysis}
\label{chp:examples}

\section{Examples template}
\label{sec:Examples}

\subsection{The template} 
\label{sec:template} 
The abstract definitions of an adapted Green system (\defref{def:adapting_operator}), of elliptic pre-complexes (\defref{def:NN_segment}--\defref{def:DD_segment}), and of their variants (\secref{sec:variants}), are designed to isolate the essential abstract ingredients needed for \thmref{thm:corrected_complex} and the associated Hodge theory to take form. This approach allowed the theory to be presented in full generality, without committing to a specific form that the systems $\bD_{\alpha}$ in \eqref{eq:elliptic_pre_complex_diagram} might take.

Now that we turn to the analysis of concrete examples, namely those laid out in \secref{sec:main_results_examples_1} and \secref{sec:examples_intro}, we outline a concrete, relatively easy to use \emph{template} for obtaining elliptic pre-complexes, based on either Neumann or Dirichlet conditions. We emphasize in advance that this part is technical, as the goal is to capture a unifying machinery behind as many examples as possible, rather than to focus on elegance.

Consider operators falling into the template: 
\beq
\begin{aligned}
&\bD_{\alpha} = 
\begin{pmatrix} A_{D,\alpha} & 0 \\ T_{\alpha} & Q_{K,\alpha} \end{pmatrix} :
\mymat{\Gamma(\bbF_{\alpha})\\\oplus\\\Gamma(\bbG_{\alpha})} 
\longrightarrow 
\mymat{\Gamma(\bbF_{\alpha+1})\\\oplus\\\Gamma(\bbG_{\alpha+1})}. 
\end{aligned}
\label{eq:A_pallete}
\eeq
To avoid ambiguity, we specify how the systems act explicitly on $\psi\in\Gamma(\bbF_{\alpha})$ and $\lambda\in\Gamma(\bbG_{\alpha})$: 
\beq
\begin{aligned}
(\psi; \lambda) \mapsto (A_{D,\alpha}\psi; T_{\alpha}\psi + Q_{\alpha}\lambda). 
\end{aligned}
\eeq
To ensure that $(\bD_{\bullet})$ satisfies the necessary properties to be an elliptic pre-complex, we incorporate the following assumptions:
\begin{itemize}
\item The sequence $A_{D,\alpha}:\Gamma(\bbF_{\alpha})\to\Gamma(\bbF_{\alpha+1})$ consists of Green operators of order $m_{\alpha}>0$ and class zero, which can be further written as:
\beq 
A_{D,\alpha} = A_{\alpha} + D_{\alpha},
\label{eq:AD_pallete}
\eeq
where $D_{\alpha}:\Gamma(\bbE_{\alpha})\rightarrow \Gamma(\bbE_{\alpha+1})$ is a \emph{Green operator} of order and class zero. Moreover, $A_{\alpha}$ is in general the truncation of an operator with transmission property, and is equipped with normal systems of trace operators associated with order $m_{\alpha}$ (recall \defref{def:normal_system_of_trace_operators1}):
\[
B_{\alpha}:\Gamma(\bbF_{\alpha})\to \Gamma(\bbJ_{\alpha}), 
\qquad 
B^*_{\alpha}:\Gamma(\bbF_{\alpha+1})\to \Gamma(\bbJ_{\alpha}),
\]
such that the following Green's formula with respect to $A_{\alpha}$ and its adjoint $A_{\alpha}^*$ holds for every $\psi\in \Gamma(\bbF_{\alpha})$ and $\eta\in \Gamma(\bbF_{\alpha+1})$:
\beq
\bra A_{\alpha}\psi,\eta\ket = \bra\psi,A_{\alpha}^*\eta\ket + \bra B_{\alpha}\psi, B^*_{\alpha}\eta\ket,
\label{eq:integration_example}
\eeq
Additionally, we impose:
\[
A_{D,-1} = 0, \quad A_{D,-1}^* = 0, \quad B_{D,-1} = 0, \quad B^*_{D,-1} = 0.
\]
\item There are pseudodifferential operators of order 0 over the boundary:
\[
S_{\alpha}:\Gamma(\bbJ_{\alpha})\rightarrow\Gamma(\bbG_{\alpha})
\]
such that the trace operator $T_{\alpha} : \Gamma(\bbF_{\alpha}) \to \Gamma(\bbG_{\alpha})$ decomposes as: 
\beq
T_{\alpha} = S_{\alpha} B_{\alpha}.
\label{eq:trace_operator_pallete}
\eeq
The components of $T_{\alpha}$ are written as $T_{k,\alpha}:\Gamma(\bbF_{\alpha})\to\Gamma(\bbG_{k,\alpha})$, and in matrix form:
\[
T_{\alpha} = (T_{k,\alpha}),
\]
where the indexing is implied, and the order and class of each $T_{k,\alpha}$ are $\tau_{k,\alpha}$ and $r_{k,\alpha}$, which is by design lower than $m_{\alpha}$ (due to $B_{\alpha}$ being associated with order $m_{\alpha}$ and $S_{\alpha}$ being of class zero). 

\item The sequences $Q_{K,\alpha}:\Gamma(\bbG_{\alpha})\to\Gamma(\bbG_{\alpha+1})$ consist of pseudodifferential operators on the boundary, which can be written in the form of \eqref{eq:DplusC}:
\beq
Q_{K,\alpha} = Q_{\alpha} + K_{\alpha}.
\label{eq:QC_pallete} 
\eeq
We abuse notation and write the matrix components of $Q_{K,\alpha}$ as:
\[
Q_{K,\alpha} = (Q_{k,\alpha}^{l}),
\]
where the corresponding orders are denoted by $\sigma_{k,\alpha}^{l}$. As with $T_{\alpha}$, the indexing is implied. 
\end{itemize} 

Also let: 
\beq
\begin{aligned}
\bD_{\alpha}^* = \begin{pmatrix} A_{D,\alpha}^* & 0 \\ 0 & Q_{K,\alpha}^* \end{pmatrix}, \qquad \bB_{\alpha} = \begin{pmatrix} 0 & 0 \\ B_{\alpha}& 0 \end{pmatrix}, \qquad \bB^*_{\alpha} = \begin{pmatrix} 0 & 0 \\ B^*_{\alpha} & S_{\alpha}^* \end{pmatrix}
\end{aligned}
\label{eq:supplement_operators_examples} 
\eeq
operating as:
\[
\begin{aligned}
\bD_{\alpha}^*(\psi;\lambda)=(A^*_{D,\alpha}\psi;Q_{K,\alpha}^*\lambda) \qquad  \bB_{\alpha}\psi=(0;B_{\alpha}\psi) \qquad \bB^*_{\alpha}(\psi;\lambda)=(0;B^*_{\alpha}\psi+S^*_{\alpha}\lambda). 
\end{aligned}
\] 
Note that the boundary systems $\bB_{\alpha},\bB_{\alpha}^*$ are by construction normal systems of boundary operators (\defref{def:normal_system_of_trace_operators2}) due to the assumed normality of $B_{\alpha}$ and $B_{\alpha}^*$. 
\begin{proposition}
\label{prop:adapted_examples}
For every $\alpha \in \Nzero$, under the above assumptions, the system $\bD_{\alpha}$ defined in \eqref{eq:A_pallete} is an adapted Green system (cf.~\defref{def:adapting_operator}), with adapted adjoints and associated boundary systems given in \eqref{eq:supplement_operators_examples}. Consequently, the sequence $(\bD_{\bullet})$ fits into the diagram \eqref{eq:elliptic_complex_diagram}.
\end{proposition}

At this stage, the proof is technical verification and referred to \secref{sec:technical_template}. 

To fit the sequence of adapted Green systems $\bD_{\alpha}$ into an elliptic pre-complex based on either Dirichlet or Neumann conditions (\defref{def:NN_segment}--\defref{def:DD_segment}), we further assume the following \emph{algebraic order-reduction properties}:

\beq
\begin{aligned}
(i) & \quad && \ord(A_{D,\alpha+1} A_{D,\alpha}) \leq m_{\alpha}, \\  
(ii) & \quad && \ord(Q_{k',\alpha+1}^{k} Q_{k,\alpha}^{l}) \leq \underset{k}{\max}(\sigma_{k,\alpha}^{l}), &&&& \forall l,k', \\  
(iii) & \quad &&
\begin{cases}  
\NN:\qquad \ord(T_{k',\alpha+1} A_{D,\alpha} + Q^{k}_{k',\alpha+1} T_{k,\alpha}) = 0, & \forall k', \\  
\qquad\quad\,\,\,\,\mathrm{class}(T_{\alpha+1} A_{D,\alpha} + Q_{\alpha+1} T_{\alpha}) \leq m_{\alpha}, & \forall k', \\[5pt] 
\DD: \qquad B_{\alpha+1}A_{D,\alpha}= 0 \quad\text{   on   }\quad\ker{B_{\alpha}}.  
\end{cases}  
\end{aligned}
\label{eq:order_reduction_conditions_examples}
\eeq

where $\ord(\cdot)$ denotes the order and $\mathrm{class}(\cdot)$ denotes the class a trace operator.

We also assume the following systems are individually overdetermined elliptic: 
\beq
\begin{aligned}
& \NN: \quad && 
\begin{pmatrix} A_{\alpha} \oplus A_{\alpha-1}^* & 0 \\ B^*_{\alpha-1} & 0 \end{pmatrix}, \;
\qquad&&\begin{pmatrix} 0 & 0 \\ 0 & Q_{\alpha} \oplus Q_{\alpha-1}^* \end{pmatrix},
\\[10pt]
& \DD: \quad &&
\begin{pmatrix} A_{\alpha} \oplus A_{\alpha-1}^* & 0 \\ B_{\alpha} & 0 \end{pmatrix},
\qquad&& \begin{pmatrix} 0 & 0 \\ 0 & Q_{\alpha} \oplus Q_{\alpha-1}^* \end{pmatrix}.
\end{aligned}
\label{eq:overdetermined_ellipticies_examples_refined}
\eeq
%

\begin{theorem}
\label{thm:elliptic_pre_complex_exmaples}
Under the above assumptions, $(\bD_{\bullet})$ is an elliptic pre-complex, based on either Neumann or Dirichlet boundary conditions, depending on the corresponding overdetermined ellipticities assumed in \eqref{eq:overdetermined_ellipticies_examples_refined}.
\end{theorem}
At this point, having stated all the necessary ingredients, the proof of \thmref{thm:elliptic_pre_complex_exmaples} reduces to a technical verification. Arguably, the only non-elementary verification is to show that the overdetermined ellipticities required in \eqref{def:NN_segment}--\eqref{def:DD_segment} indeed reduce to those in \eqref{eq:overdetermined_ellipticies_examples_refined}. We defer this verification to \secref{sec:technical_template}.

\subsection{Template's main results} 
\label{sec:outline_template} 
We now translate the main results concerning elliptic pre-complexes, as listed in \secref{sec:apps}--\secref{sec:appsD}, to fit the above template. For the elliptic pre-complex obtained from \thmref{thm:elliptic_pre_complex_exmaples}, the lifted complex $(\fbD)_{\bullet}$ assumes the form:
\[
\begin{aligned}
\fbD_{\alpha} = \begin{pmatrix} \mathpzc{A}_{\alpha} & \mathpzc{K}_{\,\,\alpha} \\ \mathpzc{T}_{\alpha} & \mathpzc{Q}_{\,\,\alpha} \end{pmatrix}
\end{aligned}
\]
where in either case $\fbD_{0} = \bD_{0}$. The relation $\fbN_{\alpha} - \bD_{\alpha}\in\OP(0,0)$ translates into the following properties, valid for all $\alpha \in \Nzero$:
\begin{enumerate}
\item $\mathpzc{A}_{\alpha} - A_{D,\alpha}$ is a Green operator of order and class zero.
\item $\mathpzc{Q}_{\,\,\alpha} - Q_{K,\alpha}$ is a pseudodifferential operator on the boundary, of order zero.
\item $\mathpzc{K}_{\,\,\alpha}$ is a Potential operator of order zero.
\item $\mathpzc{C}_{\alpha} = \mathpzc{T}_{\alpha} - T_{\alpha}$ is a trace operator of order $-1$ and class zero.
\end{enumerate}

For the Neumann case, the defining properties for the lifted complex \eqref{thm:corrected_complex} become:
\[
\begin{aligned}
&\NN: \quad && 
\begin{pmatrix} 
\mathpzc{A}_{\alpha} \mathpzc{A}_{\alpha-1} + \mathpzc{K}_{\,\,\alpha} \mathpzc{T}_{\alpha-1} & 
\mathpzc{A}_{\alpha} \mathpzc{K}_{\,\,\alpha-1} + \mathpzc{K}_{\,\,\alpha} \mathpzc{Q}_{\,\,\alpha-1} \\[4pt]
\mathpzc{T}_{\alpha} \mathpzc{A}_{\alpha-1} + \mathpzc{Q}_{\,\,\alpha} \mathpzc{T}_{\alpha-1} & 
\mathpzc{T}_{\alpha} \mathpzc{K}_{\,\,\alpha-1} + \mathpzc{Q}_{\,\,\alpha} \mathpzc{Q}_{\,\,\alpha-1} 
\end{pmatrix}
= 0.
\end{aligned}
\]

The adapted adjoints are given by:
\[
\begin{aligned}
&\NN: \quad && 
\fbD^*_{\alpha} = 
\begin{pmatrix} 
\mathpzc{A}^*_{\alpha} & \mathpzc{C}^*_{\alpha} \\ 
\mathpzc{K}^*_{\,\,\alpha} & \mathpzc{Q}^*_{\,\,\alpha} 
\end{pmatrix},
\end{aligned}
\]
and in both cases, the lifted interior operators thus inherit the Green’s formulae. In particular, from \eqref{eq:integration_example} we obtain a lifted Green’s formula, since terms of zero order and class integrate by parts without producing an additional boundary term:
\beq
\bra \mathpzc{A}_{\alpha}\psi,\eta\ket
= \bra\psi,\mathpzc{A}^*_{\alpha}\eta\ket
+ \bra B_{\alpha}\psi,B_{\alpha}^*\eta\ket.
\label{eq:lifted_Green_template}
\eeq
Comparing with \eqref{eq:supplement_operators_examples}, the condition $(\psi; \lambda) \in \scrN(\fbD^*_{\alpha}, \bB^*_{\alpha})$ reduces to the constraints: 
\beq
\begin{aligned}
&\NN: \quad && 
\mathpzc{A}^*_{\alpha} \psi + \mathpzc{C}^*_{\alpha} \lambda = 0, \qquad 
\mathpzc{Q}^*_{\,\,\alpha} \lambda + \mathpzc{K}^*_{\,\,\alpha} \psi = 0, \qquad B^*_{\alpha} \psi +S_{\alpha}^* \lambda = 0 \qquad M^*_{\alpha}\lambda=0. 
\end{aligned}
\label{eq:N_star_examples_NN}
\eeq
and the relations in \eqref{eq:DstarDstar_N} and \eqref{eq:DstarDstar_D} translate into:
\[
\begin{aligned}
&\NN: \quad && 
\begin{pmatrix} 
\mathpzc{A}^*_{\alpha-1} \mathpzc{A}^*_{\alpha} + \mathpzc{C}^*_{\alpha-1} \mathpzc{K}^*_{\,\,\alpha} & 
\mathpzc{A}^*_{\alpha-1} \mathpzc{C}^*_{\alpha} + \mathpzc{C}^*_{\alpha-1} \mathpzc{Q}^*_{\,\,\alpha} \\[4pt]
\mathpzc{K}^*_{\,\,\alpha-1} \mathpzc{A}^*_{\alpha} + \mathpzc{Q}^*_{\,\,\alpha-1} \mathpzc{K}^*_{\,\,\alpha} & 
\mathpzc{K}^*_{\,\,\alpha-1} \mathpzc{C}^*_{\alpha} + \mathpzc{Q}^*_{\,\,\alpha-1} \mathpzc{Q}^*_{\,\,\alpha}
\end{pmatrix}
= 0, \quad \text{on } \ker \frakB^*_{\alpha}.
\end{aligned}
\]

The analogous conditions for the $\DD$ picture are significantly simpler:
\begin{proposition}
\label{prop:Dirichlet_examples} 
In the $\DD$ case, we have identically
\[
\mathpzc{C}_{\alpha} = 0, \qquad \mathpzc{K}_{\,\,\alpha} = 0.
\]
Consequently, the operator $\fbD_{\alpha}$ takes the form
\[
\fbD_{\alpha} = 
\begin{pmatrix} 
\mathpzc{A}_{\alpha} & 0 \\ 
T_{\alpha} & \mathpzc{Q}_{\,\,\alpha} 
\end{pmatrix},
\]
and its adapted adjoint is given by:
\[
\fbD^*_{\alpha} = 
\begin{pmatrix} 
\mathpzc{A}^*_{\alpha} & 0 \\ 
0 & \mathpzc{Q}^*_{\,\,\alpha} 
\end{pmatrix}.
\]

The condition $(\psi;\lambda) \in \scrN(\fbD_{\alpha}^*)$ reads
\beq 
\mathpzc{A}^*_{\alpha} \psi = 0, \qquad \mathpzc{Q}^*_{\,\,\alpha} \lambda = 0
\label{eq:N_star_examples_DD}
\eeq

and the defining properties of the lifted operators reduce to
\beq
\begin{aligned}
& \mathpzc{A}_{\alpha+1} \mathpzc{A}_{\alpha} \psi = 0, 
&& \qquad \text{for } \psi \in \ker B_{\alpha}, 
&& \qquad \mathpzc{A}_{\alpha+1} = A_{D,\alpha+1} \ \text{on } \ker \mathpzc{A}_{D,\alpha}^*, \\[6pt]
& \mathpzc{Q}_{\,\,\alpha+1} \mathpzc{Q}_{\,\,\alpha} \lambda = 0, 
&& \qquad \text{for all } \lambda \in \Gamma(\bbG_{\alpha-1}), 
&& \qquad \mathpzc{Q}_{\alpha+1} = Q_{K,\alpha+1} \ \text{on } \ker \mathpzc{Q}_{\,\,\alpha}^*.
\end{aligned}
\label{eq:Dirichelt_examples_defining}
\eeq

\end{proposition}
The proof is deferred to the technical proofs section. 

At the $\alpha = 0$ level, where no correction terms appear, the relations \eqref{eq:N_star_examples_NN} and \eqref{eq:N_star_examples_DD} reduces to:
\beq
\begin{aligned}
&\NN: \quad && A^*_{D,0}\psi= 0 \qquad B^*_{0}\psi+S^*_{0}\lambda=0 \qquad Q^*_{D,0}\lambda=0\\
&\DD: \quad && A^*_{D,0} \psi= 0 \qquad Q^*_{D,0}\lambda=0
\end{aligned}
\label{eq:N0_apps}
\eeq
The other cohomology groups can be derived by comparing the condition $\bD_{\alpha}(\psi;\lambda)=0$ with \eqref{eq:N_star_examples_NN}. 


Worth noting is the fact that, due to \propref{prop:Dirichlet_examples}, $\module^{\alpha+1}_{\DD}$ clearly splits as a disjoint sum:
\[
\module^{\alpha+1}_{\DD} = 
\ker(A_{D,\alpha+1},A^*_{D,\alpha},B_{\alpha+1})
\oplus 
\ker(Q_{K,\alpha+1},Q^*_{K,\alpha}),
\]
so we write:
\beq 
\begin{split} 
&\module^{\alpha+1}_{\DD}(\mathpzc{A}_{\bullet}) = \ker(A_{D,\alpha+1},\cttbP_{\alpha-1}A^*_{D,\alpha} ,B_{\alpha+1}), 
\\&\module^{\alpha+1}_{\DD}(\mathpzc{Q}_{\,\,\bullet}) = \ker(Q_{K,\alpha+1},\cttbP_{\alpha-1}Q^*_{K,\alpha}).
\end{split} 
\label{eq:cohomology_examplesDD}
\eeq 

By combining everything, the cohomological formulations \thmref{thm:NNintro}--\thmref{thm:DDintro} then become:

\begin{theorem}
\label{thm:compatibility_nn} 
Under Neumann conditions, for every $\alpha \in \Nzero$, the system
\beq
\begin{aligned}
&A_{D,\alpha} \psi = \omega, 
& \qquad & 
\mathpzc{A}^*_{\alpha} \psi + \mathpzc{C}^*_{\alpha} \lambda = 0,
& \qquad & \text{in } M,
\\[6pt]
&T_{\alpha} \psi + Q_{\alpha} \lambda = \rho, 
& \qquad &
\mathpzc{Q}^*_{\,\,\alpha} \lambda + \mathpzc{K}^*_{\,\,\alpha} \psi = 0, 
\qquad 
B^*_{\alpha} \psi + S^*_{\alpha} \lambda = 0,
& \qquad & \text{on } \partial M
\end{aligned}
\label{eq:neumann_system_with_gauge}
\eeq
admits a solution if and only if
\beq
\mathpzc{A}_{\alpha+1} \omega + \mathpzc{K}_{\,\,\alpha+1} \rho = 0, 
\qquad 
\mathpzc{T}_{\alpha+1} \omega + \mathpzc{Q}_{\,\,\alpha+1} \rho = 0, 
\qquad 
(\omega; \rho) \perp \module_{\N}^{\alpha+1}.
\label{eq:neumann_compatibility_conditions}
\eeq
The solution is unique modulo $\module^{\alpha}_{\N}$.
\end{theorem}
For the Dirichlet case, recognizing that $\scrN(\fbD_{\alpha}) = \ker \fbD_{\alpha}$ and that $(\omega; \rho) \in \scrN(\fbD_{\alpha}, \frakB_{\alpha})$ if and only if $B_{\alpha} \omega = 0$, and that for such $\omega$ we have $T_{\alpha} \omega = 0$ due to \eqref{eq:trace_operator_pallete}, \thmref{thm:DDintro} becomes:
\begin{theorem}
\label{thm:compatibility_dd} 
Under Dirichlet conditions, for every $\alpha \in \Nzero$, the system
\beq
\begin{aligned}
A_{D,\alpha} \psi &= \omega, 
& \qquad & 
\mathpzc{A}^*_{\alpha} \psi = 0,
& \qquad & \text{in } M,
\\[6pt]
Q_{K,\alpha} \lambda &= \rho, 
& \qquad & 
\mathpzc{Q}^*_{\,\,\alpha} \lambda = 0, 
\qquad 
B_{\alpha} \psi = 0,
& \qquad & \text{on } \partial M
\end{aligned}
\label{eq:dirichlet_system_with_gauge}
\eeq
admits a solution if and only if
\beq
\mathpzc{A}_{\alpha+1} \omega = 0, 
\qquad 
B_{\alpha+1} \omega = 0, 
\qquad 
\mathpzc{Q}_{\,\,\alpha+1} \rho = 0, 
\qquad 
\omega \perp \module_{\D}^{\alpha+1}(\mathpzc{A}_{\bullet}), 
\qquad 
\rho \perp \module_{\D}^{\alpha+1}(\mathpzc{Q}_{\,\,\bullet}).
\label{eq:dirichlet_compatibility_conditions}
\eeq
The solution is unique modulo $\module_{\D}^{\alpha}(\mathpzc{A}_{\bullet})$ for the $\omega$ component, and modulo $\module_{\D}^{\alpha}(\mathpzc{Q}_{\,\,\bullet})$ for the $\lambda$ component.
\end{theorem}

\subsection{The prototypical setting}
\label{sec:dirichelt_special_case} 
We note a special case, relevant in applications, where in an elliptic pre-complex fitting the template \eqref{eq:AD_pallete}, all components of $\bD_{\alpha}$ vanish except in the upper-left corner, namely:
\beq 
\bD_{\alpha}=\begin{pmatrix}A_{\alpha}&0 \\ 0 & 0\end{pmatrix}.
\label{eq:upper_left_A} 
\eeq 
This is precisely the simplified setting introduced in \secref{sec:elliptic_pre_complexes_intro} and \secref{sec:main_results_intro} as being prototypical; it can be represented by the same diagram used for elliptic complexes (with the bundle names changing from $\bbF_{\alpha}$ to $\bbE_{\alpha}$ and from $\bbG_{\alpha}$ to $\bbJ_{\alpha}$, as detailed in \eqref{eq:lifted_dirichet_specfic}), as exposed in \secref{sec:Hodge_intro}:
\beq
\begin{xy}
(-30,0)*+{0}="Em1";
(0,0)*+{\Gamma(\E_0)}="E0";
(30,0)*+{\Gamma(\E_1)}="E1";
(60,0)*+{\Gamma(\E_2)}="E2";
(90,0)*+{\Gamma(\E_3)}="E3";
(100,0)*+{\cdots}="E4";
(0,-20)*+{\Gamma(\bbJ_0)}="G0";
(30,-20)*+{\Gamma(\bbJ_1)}="G1";
(60,-20)*+{\Gamma(\bbJ_2)}="G2";
(90,-20)*+{\Gamma(\bbJ_3)}="G3";
(100,-20)*+{\cdots}="G4";
{\ar@{->}@/^{1pc}/^{A_0}"E0";"E1"};
{\ar@{->}@/^{1pc}/^{A_0^*}"E1";"E0"};
{\ar@{->}@/^{1pc}/^{A_1}"E1";"E2"};
{\ar@{->}@/^{1pc}/^{A_1^*}"E2";"E1"};
{\ar@{->}@/^{1pc}/^{A_2}"E2";"E3"};
{\ar@{->}@/^{1pc}/^{A_2^*}"E3";"E2"};
{\ar@{->}@/^{1pc}/^0"Em1";"E0"};
{\ar@{->}@/^{1pc}/^0"E0";"Em1"};
{\ar@{->}@/_{0pc}/^{B_0}"E0";"G0"};
{\ar@{->}@/_{0pc}/^{B_1}"E1";"G1"};
{\ar@{->}@/_{0pc}/^{B_2}"E2";"G2"};
{\ar@{->}@/^{1pc}/^{B^*_0}"E1";"G0"};
{\ar@{->}@/^{1pc}/^{B^*_1}"E2";"G1"};
{\ar@{->}@/^{1pc}/^{B^*_2}"E3";"G2"};
{\ar@{->}@/_{0pc}/^{B^*_3}"E3";"G3"};
\end{xy}
\label{eq:elliptic_complex_diagramTe}
\eeq
In the Neumann case, this situation is precisely the subject of \cite{KL23}, although the distinction between the Dirichlet and Neumann pictures was not employed there, as it had not yet been developed at the time. 

The list of properties for \eqref{eq:AD_pallete} translates as follows for a sequence of systems of the form \eqref{eq:upper_left_A}. The operators still interact via Green’s formula:
\[
\langle A_\alpha\psi, \eta \rangle
= \langle \psi, A_\alpha^*\eta \rangle
+ \langle B_\alpha\psi, B^*_\alpha\eta \rangle.
\]

The conditions for being an elliptic pre-complex then translate from \thmref{thm:elliptic_pre_complex_exmaples} into the following two notions, which combines into \defref{def:elliptic_pseudo_complex_intro} when $\alpha_0=\infty$:  
\begin{definition}
\label{def:elliptic_pre_complexTem}
Let $\alpha_0\in\Nzero\cup\BRK{-1}$. A diagram of operators \eqref{eq:elliptic_complex_diagramTe} is called a $\alpha_0$-\emph{elliptic pre-complex} (based on Dirichlet conditions) if, for every $\alpha\leq \alpha_0$, the following conditions are satisfied:
\begin{enumerate}[itemsep=0pt,label=(\alph*)]
\item $B_{\alpha}$ and $B_{\alpha}^*$ are normal systems of trace operators.
\item The sequence obeys the order reduction property:
\[
\ord(A_{\alpha+1} A_{\alpha}) \leq \ord(A_{\alpha}),
\qquad
\ker B_{\alpha-1}\subseteq \ker B_{\alpha}A_{\alpha-1}.
\]
\item $A_{\alpha} \oplus A_{\alpha-1}^* \oplus B_{\alpha}$ is overdetermined elliptic. 
\end{enumerate}
\end{definition}
\begin{definition}
\label{def:elliptic_pre_complexTemNeu}
Let $\alpha_0\in\Nzero\cup\BRK{-1}$. A diagram of operators \eqref{eq:elliptic_complex_diagramTe} is called a $\alpha_0$-\emph{elliptic pre-complex} (based on Neumann conditions) if, for every $\alpha\leq \alpha_0$, the following conditions are satisfied:
\begin{enumerate}[itemsep=0pt,label=(\alph*)]
\item $B_{\alpha}$ and $B_{\alpha}^*$ are normal systems of trace operators.
\item The sequence obeys the order reduction property:
\[
\ord(A_{\alpha+1} A_{\alpha}) \leq \ord(A_{\alpha}).
\]
\item $A_{\alpha} \oplus A_{\alpha-1}^* \oplus B^*_{\alpha-1}$ is overdetermined elliptic. 
\end{enumerate}
\end{definition}
The characterization of the lifted complex from \secref{sec:outline_template}, with $T_{\alpha} = 0$ and $Q_{\alpha} = 0$, then implies that:  
\beq
\fbD_{\alpha}=\begin{pmatrix}\mathpzc{A}_{\alpha} & 0 \\ 0 & 0 \end{pmatrix}, \qquad 
\bB_{\alpha}=\begin{pmatrix}0 & 0 \\ B_{\alpha} & 0 \end{pmatrix}.
\label{eq:lifted_dirichet_specfic} 
\eeq
i.e., we retain \thmref{thm:lifted_complexIntro}:
\begin{theorem}
\label{thm:lifted_complexTem}
For every $\alpha \leq \alpha_0$, there exists a continuous linear map of Fréchet spaces: 
\[
\mathpzc{A}_{\alpha+1} : \Gamma(\bbE_{\alpha}) \to \Gamma(\bbE_{\alpha+1}),
\]
uniquely characterized by the following properties:
\begin{enumerate}
\item In the Neumann case: 
\[
\begin{aligned}
& (1) \quad && \mathpzc{A}_{\alpha+1}\mathpzc{A}_{\alpha}= 0,
& \qquad &&& \text{identically}, \\[4pt]
& (2) \quad && \mathpzc{A}_{\alpha} = A_{\alpha},
& \qquad &&& \text{on } \ker (\mathpzc{A}_{\alpha-1}^*,B_{\alpha-1}^*).
\end{aligned}
\]
\item In the Dirichlet case: 
\[
\begin{aligned}
& (1) \quad && \mathpzc{A}_{\alpha+1}\mathpzc{A}_{\alpha}= 0,
& \qquad &&& \text{on } \ker B_{\alpha}, \\[4pt]
& (2) \quad && \mathpzc{A}_{\alpha} = A_{\alpha},
& \qquad &&& \text{on } \ker \mathpzc{A}_{\alpha-1}^*.
\end{aligned}
\]
\end{enumerate}
The unique operator $\mathpzc{A}_{\alpha}$ differs from $A_{\alpha}$ by a Green operator of order and class zero. We call $(\mathpzc{A}_{\bullet})$ the \emph{lifted complex} induced by $(A_{\bullet})$.
\end{theorem}

As in \eqref{eq:lifted_Green_template}, since $\mathpzc{A}_{\alpha} - A_{\alpha}$ is of zero order and class, the lifted operators $\mathpzc{A}_{\alpha}$ and their adjoints $\mathpzc{A}_{\alpha}^*$ inherit the Green’s formula without additional boundary terms:
\beq
\bra \mathpzc{A}_{\alpha}\psi,\eta\ket
= \bra\psi,\mathpzc{A}^*_{\alpha}\eta\ket
+ \bra B_{\alpha}\psi,B_{\alpha}^*\eta\ket.
\label{eq:Green_forumula_correction}
\eeq
As a byproduct, by translating the relations \eqref{eq:Dirichelt_examples_defining}, we find that in the Neumann case, for every $\alpha \in \Nzero \cup \{-1\}$:
\beq 
\begin{aligned}
(1) \quad & \mathpzc{A}_{\alpha}^* \, \mathpzc{A}_{\alpha+1}^* = 0 \;\;\text{on } \ker(B_{\alpha+1}^*), \\[4pt]
(2) \quad & \ker B^*_{\alpha+1} \subseteq \ker(B^*_{\alpha}\mathpzc{A}_{\alpha+1}),
\end{aligned}
\label{eq:lifted_adjoint_relationsNeum} 
\eeq 
whereas in the Dirichlet case: 
\beq 
\begin{aligned}
(1) \quad & \mathpzc{A}_{\alpha}^* \, \mathpzc{A}_{\alpha+1}^* = 0 \;\;\text{identically}, \\[4pt]
(2) \quad & \ker B_{\alpha-1} \subseteq \ker(B_{\alpha}\mathpzc{A}_{\alpha-1}).
\end{aligned}
\label{eq:lifted_adjoint_relations} 
\eeq
Hence, in this simplified setting, the explicit form of the Hodge decompositions obtained from \thmref{thm:hodge_like_corrected_complexD} by substituting \eqref{eq:lifted_dirichet_specfic} becomes the one in \secref{sec:Hodge_intro_NN} and \secref{sec:Hodge_intro_DD}: 

\begin{theorem}[Hodge decomposition -- Neumann]
\label{thm:Hodge_decompositionTemNeu}
In the Neumann case, for every $\alpha < \alpha_0$, there is a topologically direct, $L^{2}$-orthogonal decomposition of Fréchet spaces:
\beq
\Gamma(\bbE_{\alpha+1}) =
\lefteqn{\overbrace{\phantom{\image{(\mathpzc{A}_{\alpha})} \oplus \module_{\N}^{\alpha+1}}}^{\ker(\mathpzc{A}_{\alpha+1})}}
\image{(\mathpzc{A}_{\alpha})} \oplus 
\underbrace{\module_{\N}^{\alpha+1} \oplus \image{(\mathpzc{A}^*_{\alpha+1}|_{\ker B_{\alpha+1}^*})}}_{\ker (\mathpzc{A}_{\alpha}^*|_{\ker B_{\alpha}^*})}.
\label{eq:Hodge_introTemNeu}
\eeq
where the cohomology group $\module_{\N}^{\alpha+1} := \ker(\mathpzc{A}_{\alpha+1}, \mathpzc{A}_{\alpha}^*, B^*_{\alpha})$ is finite-dimensional and satisfies: 
\[
\module_{\N}^{\alpha+1} \simeq \ker(\mathpzc{A}_{\alpha+1}) \big/ \image(\mathpzc{A}_{\alpha}) \simeq \ker(\mathpzc{A}_{\alpha}^*|_{\ker{B_{\alpha}^*}}) \big/ \image(\mathpzc{A}_{\alpha+1}^*|_{\ker B_{\alpha+1}^*}),
\]
and moreover refines into an expression:
\beq 
\module_{\N}^{\alpha+1} = \ker(A_{\alpha+1}, \cttbP_{\alpha-1}A_\alpha^*, B^*_{\alpha}),
\label{eq:cohomology_expressionNeu}
\eeq 
where $\cttbP_{\alpha-1}$ is orthogonal projection into $\ker(\fbA^*_{\alpha-1},B_{\alpha-1}^*)$. The $L^{2}$-orthogonal projections onto the summands in \eqref{eq:Hodge_introTemNeu} are Green operators of order and class zero.
\end{theorem}

\begin{theorem}[Hodge decomposition -- Dirichlet]
\label{thm:Hodge_decompositionTem}
In the Dirichlet case, for every $\alpha < \alpha_0$, there is a topologically direct, $L^{2}$-orthogonal decomposition of Fréchet spaces:
\beq
\Gamma(\bbE_{\alpha+1}) =
\lefteqn{\overbrace{\phantom{\image{(\mathpzc{A}_{\alpha}|_{\ker B_{\alpha}})} \oplus \module_{\D}^{\alpha+1}}}^{\ker(\mathpzc{A}_{\alpha+1}|_{\ker B_{\alpha+1}})}}
\image{(\mathpzc{A}_{\alpha}|_{\ker B_{\alpha}})} \oplus 
\underbrace{\module_{\D}^{\alpha+1} \oplus \image{\mathpzc{A}^*_{\alpha+1}}}_{\ker (\mathpzc{A}_{\alpha}^*)}.
\label{eq:Hodge_introTem}
\eeq
where the cohomology group $\module_{\D}^{\alpha+1} := \ker(\mathpzc{A}_{\alpha+1}, \mathpzc{A}_{\alpha}^*, B_{\alpha+1})$ is finite-dimensional, satisfies
\[
\module_{\D}^{\alpha+1} \simeq \ker(\mathpzc{A}_{\alpha+1}|_{B_{\alpha+1}}) \big/ \image(\mathpzc{A}_{\alpha}|_{\ker{B_{\alpha}}}) \simeq \ker(\mathpzc{A}_{\alpha}^*) \big/ \image(\mathpzc{A}_{\alpha+1}^*),
\]
and moreover refines into an expression:
\beq 
\module_{\D}^{\alpha+1} = \ker(A_{\alpha+1}, \cttbP_{\alpha-1}A_\alpha^*, B_{\alpha+1}),
\label{eq:cohomology_expression}
\eeq 
where $\cttbP_{\alpha-1}$ is orthogonal projection into $\ker(\fbA^*_{\alpha-1},B_{\alpha-1}^*)$. The $L^{2}$-orthogonal projections onto the summands in \eqref{eq:Hodge_introTem} are Green operators of order and class zero.
\end{theorem}

The cohomological formulations in \thmref{thm:compatibility_nn}---\thmref{thm:compatibility_dd} then simplifies to: 
\begin{theorem}[Cohomological formulation -- Neumann]
\label{thm:cohomology_Neumann}
Given $\eta \in \Gamma(\bbE_{\alpha+1})$, the system
\[
\begin{aligned}
&A_\alpha \omega = \eta, \qquad \mathpzc{A}^*_{\alpha-1}\omega = 0 
\qquad &\text{in } M, \\[3pt]
&B^*_{\alpha-1}\omega = 0 
\qquad &\text{on } \dM
\end{aligned}
\]
admits a solution $\omega\in\Gamma(\bbE_{\alpha})$ if and only if: 
\[
\mathpzc{A}_{\alpha+1}\eta = 0, \qquad \eta \perp_{L^2} \module^{\alpha+1}_{\NN}.
\]
The solution is unique modulo $\module_{\NN}^\alpha$.
\end{theorem}
\begin{theorem}[Cohomological formulation -- Dirichlet]
\label{thm:cohomology_dirichlet}
Given $\eta \in \Gamma(\bbE_{\alpha+1})$, the system
\[
\begin{aligned}
&A_\alpha \omega = \eta, \qquad \mathpzc{A}^*_{\alpha-1}\omega = 0 
\qquad &\text{in } M, \\[3pt]
&B_{\alpha}\omega = 0 
\qquad &\text{on } \dM
\end{aligned}
\]
admits a solution $\omega\in\Gamma(\bbE_{\alpha})$ if and only if: 
\[
\mathpzc{A}_{\alpha+1}\eta = 0, \qquad B_{\alpha+1}\eta = 0, \qquad \eta \perp_{L^2} \module^{\alpha+1}_{\D}.
\]
The solution is unique modulo $\module_{\D}^\alpha$.
\end{theorem}

Finally, the notion of a disrupted complex \defref{def:finite_elliptic_pre_complex} in this simplified situation reduces to the following. We state this here only for Dirichlet conditions, as the statement for the Neumann case is completely analogous:
\begin{definition}[Disrupted elliptic pre-complex]
\label{def:disrupted_pre_complexTem}
A diagram $(A_{\bullet})$ as in \eqref{eq:elliptic_complex_diagram} is called a \emph{disrupted} $\alpha_0$-\emph{elliptic pre-complex} (based on Dirichlet conditions) if $A_{\alpha}=0$ for all $\alpha>\alpha_0\in\Nzero$, and all requirements of \defref{def:elliptic_pre_complexTem} are satisfied, except that overdetermined ellipticity fails at the penultimate level:
\[
A_{\alpha_0}\oplus A_{\alpha_0-1}^*\oplus B_{\alpha_0}\qquad \text{is not overdetermined elliptic.}
\]
\end{definition}

The disrupted case theorems \thmref{thm:disrubted_elliptic_pre_complex} then reads: 
\begin{proposition}
\label{prop:disrupted_validTem}
All results pertinent to elliptic pre-complexes hold in the disrupted case, except for the finite dimensionality of $\module^{\alpha_0}_{\DD}$ when $\alpha = \alpha_0-1$. 
\end{proposition}

\subsection{Technical proofs} 
\label{sec:technical_template}
%

First, we prove \propref{prop:adapted_examples}, namely that under the conditions outlined in \eqref{eq:A_pallete} and \eqref{eq:supplement_operators_examples}, the operator $\bD_{\alpha}$ is an adapted Green system, $\bD_{\alpha}^*$ is its adapted adjoint, together with the boundary systems $\bB_{\alpha}$ and $\bB^*_{\alpha}$ as defined in \defref{def:normal_system_of_trace_operators2}.  

For the fact that $\bD_{\alpha}$  is an adapted Green system, by the list of assumptions under \eqref{eq:A_pallete}, it remains to verify two things. The fist is that the sharp tuples of $\bD_{\alpha}$ are indeed of the required from, and the second is that the Green's formula \eqref{eq:integration_adapted_refined} holds for every $\Psi \in \Gamma(\bbF_{\alpha}; \bbG_{\alpha})$ and $\Theta \in \Gamma(\bbF_{\alpha+1}; \bbG_{\alpha+1})$:  
\beq
\bra \bD_{\alpha} \Psi, \Theta \ket = \bra \Psi, \bD^*_{\alpha} \Theta \ket + \bra \bB_{\alpha} \Psi, \bB^*_{\alpha} \Theta \ket.
\label{eq:integration_example2} 
\eeq
We leave the identification of the sharp tuples to the next proposition, as we will have to analyze them in any case when addressing the order reduction properties, and first prove that \eqref{eq:integration_example2} indeed holds:
\begin{PROOF}{\propref{prop:adapted_examples}}
Writing  
\[
\Psi = (\psi; \lambda), \quad \Theta = (\theta; \rho),
\]
where $\psi \in \Gamma(\bbF_{\alpha})$, $\theta \in \Gamma(\bbF_{\alpha+1})$, $\lambda \in \Gamma(\bbG_{\alpha})$, and $\rho \in \Gamma(\bbG_{\alpha+1})$, the operations of the systems in \eqref{eq:A_pallete} and \eqref{eq:supplement_operators_examples} read:  
\[
\begin{aligned}
&\bD_{\alpha}(\psi;\lambda) = (A_{D,\alpha}\psi; T_{\alpha}\psi + Q_{K,\alpha} \lambda), 
\qquad 
&&\bD^*_{\alpha}(\theta;\rho) = (A^*_{D,\alpha} \theta; Q^*_{K,\alpha} \rho),
\\[4pt] 
&\bB_{\alpha}(\psi;\lambda) = (0; B_{\alpha}\psi), 
\qquad 
&&\bB^*_{\alpha}(\theta;\rho) = (0; B_{\alpha}^* \theta + S^*_{\alpha} \rho).
\end{aligned}  
\]
Applying the specified Green’s formula for $A_{\alpha}$ in \eqref{eq:integration_example}, the fact that $D_{\alpha}$ is tensorial and hence integrates by parts into $D^*_{\alpha}$ without boundary terms, and using that $Q_{K,\alpha}^*$ is the adjoint of $Q_{K,\alpha}$, we find:  
\[
\begin{split}
\bra \bD_{\alpha} \Psi, \Theta \ket 
&= \bra A_{D,\alpha} \psi, \theta \ket 
+ \bra T_{\alpha} \psi, \rho \ket 
+ \bra Q_{K,\alpha} \lambda, \rho \ket
\\
&= \bra \psi, A_{D,\alpha}^* \theta \ket 
+ \bra B_{\alpha} \psi, B^*_{\alpha} \theta \ket 
+ \bra S_{\alpha} B_{\alpha} \psi, \rho \ket
+ \bra \lambda, Q_{K,\alpha}^* \rho \ket
\\
&= \bra \psi, A_{D,\alpha}^* \theta \ket 
+ \bra B_{\alpha} \psi, B^*_{\alpha} \theta + S^*_{\alpha} \rho \ket 
+ \bra \lambda, Q_{K,\alpha}^* \rho \ket
\\
&= \Brk{ \bra \psi, A_{D,\alpha}^* \theta \ket + \bra \lambda, Q_{K,\alpha}^* \rho \ket } 
+ \Brk{ \bra B_{\alpha} \psi, B^*_{\alpha} \theta + S^*_{\alpha} \rho \ket }
\\
&= \bra \Psi, \bD^*_{\alpha} \Theta \ket + \bra \bB_{\alpha} \Psi, \bB^*_{\alpha} \Theta \ket.
\end{split}
\]
In the second step, we expanded $T_{\alpha}$ as in \eqref{eq:trace_operator_pallete}, in the third step we integrated by parts $S_{\alpha}$, and in the final step we rearranged the terms so that they fit into \eqref{eq:integration_example2}.

\end{PROOF} 

We next prove \thmref{thm:elliptic_pre_complex_exmaples}, which also includes the identification of the sharp tuples associated with the systems $\bD_{\alpha}$. 

The first step is to translate the purely algebraic order-reduction properties listed in \eqref{eq:order_reduction_conditions_examples} into the corresponding ``abstract'' balance conditions required by \defref{def:NN_segment} and \defref{def:DD_segment}. 

To this end, we observe that each $\bD_{\alpha}$ satisfies the assumption $r_{k,\alpha} \leq m_{\alpha}$ (by construction of $T_{\alpha}$), with the associated orders given by \eqref{eq:corresponding_orders}:
\[
\begin{aligned}
\begin{pmatrix} 
m_{\alpha} & 0 \\[3pt] 
\tau_{k,\alpha} & \sigma_{k,\alpha}^{l} 
\end{pmatrix}, 
\qquad r_{k,\alpha} \leq m_{\alpha}.
\end{aligned}
\]
Since $m_{\alpha} - \tau_{k,\alpha} > 0$ by construction (the quantities $\tau_{k,\alpha}$ correspond to the orders of $T_{\alpha}$, generated by composing an operator of order zero with the system of trace operators of order $m_{\alpha}$, namely $B_{\alpha}$), we may choose in \eqref{eq:strict_mapping_property} for every $1<p<\infty$: 
\[
t^{l} = \underset{k}{\max} \, \sigma_{k,\alpha}^{l} - \tfrac{1}{p}.
\]
Hence, the operators in the lower right corner of \eqref{eq:AD_pallete} map into $L^{2}$-sections over the boundary, while those in the lower left corner map compactly into the same space. 

It follows that the system overall satisfies the sharp mapping property
\beq
\bD_{\alpha}: 
W^{m_{\alpha}, (t^{l}) + 1/p}_{p}(\bbF_{\alpha}; \bbG_{\alpha}) 
\;\rightarrow\;
L^{p}(\bbF_{\alpha+1}; \bbG_{\alpha+1}),
\label{eq:app_mapping_properties}
\eeq
which is precisely the sharp mapping property required in the definition of an adapted Green system, \defref{def:adapting_operator}.


Now, as required in \defref{def:NN_segment}--\defref{def:DD_segment}, take $\bP$ to be a balance for $\bD_{\alpha}$ in the $\NN$ case, and with respect to $\bB_{\alpha}$ in the $\DD$ case. In both settings the mapping property \eqref{eq:app_mapping_properties} is sharp, and so $\bP$ satisfies the reverse-direction mapping properties:
\beq
\begin{aligned}
\bP: L^{p}(\bbF_{\alpha+1}; \bbG_{\alpha+1}) \to W^{m_{\alpha}, (t^{l}) + 1/p}_{p}(\bbF_{\alpha}; \bbG_{\alpha}).
\end{aligned}
\label{eq:balance_mapping_properties}
\eeq

%

\begin{proposition}
\label{prop:order_reduction_template}
In the above setting, the algebraic order-reduction properties in \eqref{eq:order_reduction_conditions_examples} collectively translate into each of the required order-reduction properties in \defref{def:NN_segment}--\defref{def:DD_segment}.
\end{proposition}

\begin{proof}
In both cases, the composition $\bD_{\alpha+1} \bD_{\alpha}$ is given by
\[
\bD_{\alpha+1} \bD_{\alpha} = \begin{pmatrix} 
A_{\alpha+1} A_{\alpha} & 0 \\ 
T_{\alpha+1} A_{\alpha} + Q_{\alpha+1} T_{\alpha} & Q_{\alpha+1} Q_{\alpha} 
\end{pmatrix}.
\]
By construction, and \eqref{eq:order_reduction_conditions_examples}, this system has class at most $m_{\alpha}$. Combining this fact with the other items in \eqref{eq:order_reduction_conditions_examples}, we deduce that $\bD_{\alpha+1} \bD_{\alpha}$ satisfies the lenient mapping property
\[
\bD_{\alpha+1} \bD_{\alpha}: W^{m_{\alpha}, (t^{l}) + 1/2}_{2}(\bbF_{\alpha}; \bbG_{\alpha}) 
\to L^{2}(\bbF_{\alpha+2}; \bbG_{\alpha+2}).
\]
Composing this with the lenient mapping property of any balance $\bP$ for $\bD_{\alpha}$ as in \eqref{eq:balance_mapping_properties}, we obtain
\[
\bD_{\alpha+1} \bD_{\alpha} \bP: L^{2}(\bbF_{\alpha+1}; \bbG_{\alpha+1}) 
\to L^{2}(\bbF_{\alpha+2}; \bbG_{\alpha+2}).
\]
It therefore follows from \propref{prop:G0_criteria} that $\bD_{\alpha+1} \bD_{\alpha} \bP \in \OP(0,0)$.

\end{proof}

Since the order-reduction properties are satisfied,  the proof of \thmref{thm:elliptic_pre_complex_exmaples} will be complete once we establish:   

\begin{proposition}
\label{prop:overdetermined_ellipticity_examples}
In the above setting, the overdetermined ellipticity in \eqref{eq:overdetermined_ellipticies_examples_refined} collectively translate into each of the required conditions in \defref{def:NN_segment}--\defref{def:DD_segment}.  
\end{proposition}
\begin{proof}
Recall that the overdetermined ellipticity conditions in \defref{def:NN_segment}--\defref{def:DD_segment} consist of two sets:
\[
\begin{aligned}
&\NN: \quad && \bD_{\alpha} \oplus \bD_{\alpha-1}^* \oplus \bB^*_{\alpha-1}, \\ 
&\DD: \quad && \bD_{\alpha} \oplus \bD_{\alpha-1}^* \oplus \bB_{\alpha},
\end{aligned}
\]
and
\[
\begin{aligned}
&\NN: \quad && \bD_{\alpha}^* \bD_{\alpha} \oplus \bB^*_{\alpha} \bD_{\alpha} \oplus \bD_{\alpha-1}^* \oplus \bB^*_{\alpha-1}, \\ 
&\DD: \quad && \bD_{\alpha}^* \bD_{\alpha} \oplus \bD_{\alpha-1}^* \oplus \bB_{\alpha}.
\end{aligned}
\]
If the first set is overdetermined elliptic for all $\alpha \in \Nzero$, then the second set automatically satisfies the same property. Indeed, composing the overdetermined ellipticity at levels $\alpha$ and $\alpha+1$ yields operators whose leading terms preserve ellipticity, while the product $\bD_{\alpha+1}\bD_{\alpha}$ contributes only lower-order terms. Such corrections are negligible for overdetermined ellipticity by \propref{prop:lower_order_correction_overdetermined_ellipticity}. We leave the argument to the reader.


We now establish the overdetermined ellipticity of the first set. By expanding the definitions of $\bD_{\alpha}, \bD_{\alpha}^*, \bB_{\alpha}$, and $\bB^*_{\alpha}$ given after \eqref{eq:A_pallete}, the required systems become explicitly:
\beq
\begin{aligned}
&\NN: \quad && 
\begin{pmatrix} 
A_{\alpha} \oplus A_{\alpha-1}^* & 0 \\[3pt] 
T_{\alpha} \oplus B^*_{\alpha-1} & Q_{\alpha} \oplus S^{*}_{\alpha} \oplus Q_{\alpha-1}^*
\end{pmatrix}, 
\\[5pt]  
&\DD: \quad && 
\begin{pmatrix} 
A_{\alpha} \oplus A_{\alpha-1}^* & 0 \\[3pt] 
T_{\alpha} \oplus B_{\alpha} & Q_{\alpha} \oplus Q_{\alpha-1}^*
\end{pmatrix}.
\end{aligned}
\label{eq:required_overdetermined_ellipticies_examples} 
\eeq
In other words, the corresponding maps are
\[
\begin{aligned}
&\NN: \quad && (\psi;\lambda) \mapsto 
\big( A_{\alpha} \psi, A_{\alpha-1}^* \psi;\,
T_{\alpha}\psi + Q_{\alpha}\lambda,\,
B^*_{\alpha-1}\psi + S^*_{\alpha}\lambda,\,
Q^*_{\alpha-1}\lambda \big), \\[5pt]
&\DD: \quad && (\psi;\lambda) \mapsto 
\big( A_{\alpha}\psi, A_{\alpha-1}^*\psi;\,
T_{\alpha}\psi + Q_{\alpha}\lambda,\,
B_{\alpha}\psi,\,
Q^*_{\alpha-1}\lambda \big).
\end{aligned}
\]

We prove that, without loss of generality in the $\NN$ case, the above system is equivalent (for the purpose of overdetermined ellipticity) to
\beq
\begin{pmatrix} 
A_{\alpha} \oplus A_{\alpha-1}^* & 0 \\[3pt]
T_{\alpha} \oplus B^*_{\alpha-1} & Q_{\alpha} \oplus 0 \oplus Q_{\alpha-1}^*
\end{pmatrix}.
\label{eq:system_overdetermined_examples_proof}
\eeq
The difference between \eqref{eq:required_overdetermined_ellipticies_examples} and \eqref{eq:system_overdetermined_examples_proof} is the additional term
\[
\begin{pmatrix} 
0 & 0 \\[3pt]
T_{\alpha} & S^*_{\alpha}
\end{pmatrix}.
\]
With respect to the sharp tuples defining the mapping property in \eqref{eq:app_mapping_properties}, this term induces a compact map: the operator $T_{\alpha}$ has order strictly lower than $m_{\alpha}$ by construction, while $S^*_{\alpha}$ is a pseudodifferential operator of order zero on the boundary. Hence, this contribution can be regarded as a lower-order term (in the varying order sense):
\[
\begin{split}
\begin{pmatrix} 
A_{\alpha} \oplus A_{\alpha-1}^* & 0 \\ 
T_{\alpha} \oplus B^*_{\alpha-1} & Q_{\alpha} \oplus S^{*}_{\alpha} \oplus Q_{\alpha-1}^*
\end{pmatrix}
&=
\begin{pmatrix} 
A_{\alpha} \oplus A_{\alpha-1}^* & 0 \\ 
0 \oplus B^*_{\alpha-1} & Q_{\alpha} \oplus 0 \oplus Q_{\alpha-1}^*
\end{pmatrix}
+
\begin{pmatrix} 
0 & 0 \\ 
T_{\alpha} & S^{*}_{\alpha}
\end{pmatrix}.
\end{split}
\]
Therefore, by \propref{prop:lower_order_correction_overdetermined_ellipticity}, this lower-order term does not affect the overdetermined ellipticity of the system. 

We therefore observe that effectively, the operations the interior and boundary sections decouple, and we are left with the overdetermined ellipticities in \eqref{eq:overdetermined_ellipticies_examples_refined}, which, for reference, read
\beq
\begin{pmatrix} 
A_{\alpha} \oplus A_{\alpha-1}^* & 0 \\[3pt] 
B^*_{\alpha-1} & 0 
\end{pmatrix}, 
\qquad 
\begin{pmatrix} 
0 & 0 \\[3pt] 
0 & Q_{\alpha} \oplus Q^*_{\alpha-1}
\end{pmatrix},
\label{eq:system_overdetermined_examples_proofII}
\eeq
as required. The argument for the $\DD$ case is analogous.

\end{proof}
\begin{PROOF}{\propref{prop:Dirichlet_examples}}
We prove the claim by induction on $\alpha \in \Nzero$. For $\alpha = -1$, it holds trivially since $\bD_{0} = \fbD_{0}$. 

Assume now that the statement is valid for some $\alpha \in \Nzero$.  
By \thmref{thm:corrected_complex}, the lifted operators are uniquely characterized by 
\[
\fbD_{\alpha+1}\fbD_{\alpha} = 0 
\quad \text{on } \ker \frakB_{\alpha}, 
\qquad 
\fbD_{\alpha+1} = \bD_{\alpha+1} 
\quad \text{on } \scrN(\fbD_{\alpha}^*).
\]

We claim that
\[
\widetilde{\fbD}_{\alpha+1} =
\begin{pmatrix} 
\mathpzc{A}_{\alpha+1} & 0 \\ 
T_{\alpha+1} & \mathpzc{Q}_{\,\,\alpha+1} 
\end{pmatrix}
\]
also satisfies these properties. Hence, by uniqueness, $\widetilde{\fbD}_{\alpha+1} = \fbD_{\alpha+1}$, and the claim follows.

For any $\Psi \in \ker \bB_{\alpha}$, we have
\[
\fbD_{\alpha+1} \fbD_{\alpha} \Psi = 0.
\]
Writing $\Psi = (\psi; \lambda)$, the condition $\Psi \in \ker \bB_{\alpha}$ implies that $\lambda$ is arbitrary while $\psi$ satisfies $B_{\alpha}\psi = 0$; in particular, $T_{\alpha}\psi = 0$ by \eqref{eq:trace_operator_pallete}.  

Expanding the relation $\fbD_{\alpha+1} \fbD_{\alpha} \Psi = 0$ and using the induction hypothesis gives
\[
\begin{split} 
\fbD_{\alpha+1} \fbD_{\alpha} \Psi 
&= 
\begin{pmatrix} 
\mathpzc{A}_{\alpha+1} & \mathpzc{K}_{\alpha+1} \\
\mathpzc{T}_{\alpha+1} & \mathpzc{Q}_{\,\,\alpha+1} 
\end{pmatrix}
\begin{pmatrix} 
\mathpzc{A}_{\alpha} & 0 \\
T_{\alpha} & \mathpzc{Q}_{\,\,\alpha} 
\end{pmatrix}
(\psi; \lambda) \\[4pt]
&=
\left( 
\mathpzc{A}_{\alpha+1}\mathpzc{A}_{\alpha}\psi 
+ \mathpzc{K}_{\alpha+1}\mathpzc{Q}_{\,\,\alpha}\lambda;\,
\mathpzc{T}_{\alpha+1}\mathpzc{A}_{\alpha}\psi 
+ \mathpzc{Q}_{\,\,\alpha+1}\mathpzc{Q}_{\,\,\alpha}\lambda 
\right) = 0.
\end{split}
\]

Taking $\lambda = 0$ and $\psi \in \ker B_{\alpha}$ gives $\mathpzc{A}_{\alpha+1}\mathpzc{A}_{\alpha}\psi = 0$.  
Similarly, taking $\psi = 0$ and $\lambda$ arbitrary yields $\mathpzc{Q}_{\,\,\alpha+1}\mathpzc{Q}_{\,\,\alpha}\lambda = 0$.  
Moreover, for such $\psi$, the induction hypothesis and the condition $\bB_{\alpha+1} \fbD_{\alpha} = 0$ on $\ker \bB_{\alpha}$ imply $T_{\alpha+1}\mathpzc{A}_{\alpha}\psi = 0$.  
Therefore, for $(\psi; \lambda) \in \ker \bB_{\alpha}$,
\[
\widetilde{\fbD}_{\alpha+1}\fbD_{\alpha}\Psi =
\begin{pmatrix} 
\mathpzc{A}_{\alpha+1} & 0 \\ 
T_{\alpha+1} & \mathpzc{Q}_{\,\,\alpha+1} 
\end{pmatrix}
\begin{pmatrix} 
\mathpzc{A}_{\alpha} & 0 \\ 
T_{\alpha} & \mathpzc{Q}_{\,\,\alpha} 
\end{pmatrix}
(\psi; \lambda)
=
\left( 
\mathpzc{A}_{\alpha+1}\mathpzc{A}_{\alpha}\psi;\,
\mathpzc{Q}_{\,\,\alpha+1}\mathpzc{Q}_{\,\,\alpha}\lambda 
\right) = 0.
\]

By the induction hypothesis, since $\mathpzc{C}_{\alpha}=0$ and $\mathpzc{K}_{\alpha}=0$, the adapted adjoint of $\fbD_{\alpha}$ is
\[
\fbD_{\alpha}^* = 
\begin{pmatrix}
\mathpzc{A}_{\alpha}^* & 0 \\ 
0 & \mathpzc{Q}_{\,\,\alpha}^*
\end{pmatrix}.
\]
Thus, the condition $\fbD_{\alpha+1} = \bD_{\alpha+1}$ on $\scrN(\fbD_{\alpha}^*)$ becomes
\[
\mathpzc{A}_{\alpha+1} = A_{D,\alpha+1} 
\quad \text{on } \ker \mathpzc{A}_{\alpha}^*, 
\qquad 
\mathpzc{Q}_{\,\,\alpha+1} = Q_{K,\alpha+1} 
\quad \text{on } \ker \mathpzc{Q}_{\,\,\alpha}^*.
\]
Hence $\widetilde{\fbD}_{\alpha+1} = \bD_{\alpha+1}$ also on this space.  

\end{PROOF}

\section{Exterior covariant derivatives}
\label{sec:exterior_covariant_derivatives}
As promised in \secref{sec:examples_intro}, we now elaborate on several elliptic pre-complexes consisting of exterior covariant derivatives that fall within the framework of \eqref{eq:A_pallete}. In the Neumann case, we consider a broader family of examples generalizing those discussed in \secref{sec:exterior_intro}.
\subsection{Dirichlet picture}
\label{sec:DD_exterior_intro}
An elliptic pre-complex consisting of exterior covariant derivatives and based on Dirichlet conditions is obtained by fitting the following systems into the scheme of \eqref{eq:A_pallete} for the $\D$-case:
\beq
\bD_\alpha=\begin{pmatrix}\dU & 0 \\ 0 & 0\end{pmatrix}:\mymat{\Omega^{\alpha}(M;\bbU)\\\oplus\\0}\longrightarrow \mymat{\Omega^{\alpha+1}(M;\bbU)\\\oplus\\0}.
\label{eq:D_exterior_systems} 
\eeq
Specifically, for the operators in \eqref{eq:AD_pallete}, we recognize:
\[
A_{D,\alpha} = \dU.
\]
More explicitly, we have:
\[
A_{\alpha} = \dU, \qquad 
D_{\alpha} = 0, \qquad 
B_{\alpha} = \PtD, \qquad 
B^*_{\alpha} = \PnD, \qquad 
A_{\alpha}^* = \delU.
\]
Here, we set $D_{\alpha} = 0$, although, in principle, it can be any tensorial operation without affecting the validity of the theory. For instance, $D_{\alpha}$ could be taken as the tensorial operation arising from the connection difference $\nabla - \nabla^0$, where $\nabla^0$ is a reference connection.  

The required properties from \eqref{eq:A_pallete} hold immediately due to \eqref{eq:Green's_formula_exterior}, while the order-reduction properties in \eqref{eq:order_reduction_conditions_examples} follow directly from the relations in \eqref{eq:order_reduction_exterior}.

As for the required overdetermined ellipticities in \eqref{eq:overdetermined_ellipticies_examples_refined}, after computing the adapted adjoints and boundary systems in \eqref{eq:supplement_operators_examples}, these become:
\beq
\begin{pmatrix} 
\dU \oplus \delU & 0 \\ 
\PtD & 0 
\end{pmatrix}
\label{eq:DD_systems_overdetermined}
\eeq

where we note that $\delU=0$ on $\Omega^{0}(M;\bbU)$. 
\begin{proposition}
\label{prop:DD_exterior} 
The systems in \eqref{eq:DD_systems_overdetermined} are overdetermined elliptic. 
\end{proposition}
Although straightforward, since this proposition demonstrates the simplest example of how the machinery introduced in \secref{sec:principle_symbol_green} is used to verify overdetermined ellipticity, we include it here:
\begin{proof}
Note that the system \eqref{eq:DD_systems_overdetermined} is in fact a Green operator belonging to $\OP(1,1)$, since $\dU \oplus \delU$ is of order $1$, and $\PtD$ is of order $0$ (one less than $1$) and class $1$. To calculate its symbols as described in \secref{sec:principle_symbol_green}, we follow \cite[Ch.~2.10]{Tay11a} and remark that the required interior symbols are:
\[
\sigma(\dU)(x, \xi) = \iota \xi \wedge, \qquad \sigma(\delU)(x, \xi) = -\iota \ixi.
\]
Thus, the interior symbol of $\dU \oplus \delU$ is:
\[
\sigma(\dU \oplus \delU)(x, \xi) = (\iota \xi \wedge) \oplus (-\iota \ixi),
\]
which is injective by standard multilinear algebra. 

Next, we verify the Lopatinski-Shapiro condition in \propref{prop:rud_lop}. Let $x \in \dM$ and $\xi'\in T_x^*\dM \setminus \{0\}$. The system of overdetermined Es at this stage becomes, for a function $s \mapsto \psi(s)$ taking values in $\mathbb{C} \otimes \Lambda^\alpha T^*_x M \oplus \bbU_x$:
\[
\begin{aligned}
& \iota \xi'\wedge \psi - dr \wedge \dot{\psi} = 0, \\
& \iota \ixi \psi - i_{\partial_r} \dot{\psi} = 0.
\end{aligned}
\]
Following \cite[Sec.~1.6]{Sch95b}, applied to $\bbU$-valued forms instead of scalar valued forms, we note how can we assume that $|\xi| = 1$ and decompose:
\[
\psi = \psi_0 + \xi'\wedge \psi_1 + dr \wedge \psi_2 + \xi'\wedge dr \wedge \psi_3,
\]
where $\ixi \psi_j = 0$ and $i_{\partial_r} \psi_j = 0$ for $j \in \{0, 1, 2, 3\}$. Using the relations $\xi'\wedge \xi'\wedge = 0$, $dr \wedge dr \wedge = 0$, and their adjunct counterparts $\idr\idr = 0$, $\ixi\ixi = 0$, the equations decouple as:
\[
\psi_0 \equiv 0, \qquad -\iota \psi_2 = \dot{\psi}_1, \qquad \iota \psi_1 = \dot{\psi}_2, \qquad \psi_3 \equiv 0.
\]
The solutions in $\bbM_{x,\xi}^+$ are thus of the form $\psi = \xi'\wedge \psi_1 + dr \wedge \psi_2$, where:
\[
\psi_1(s) = -e^{-s} \omega_0, \qquad \psi_2(s) = \iota e^{-s} \omega_0,
\]
with $\omega_0$ being an ``integration constant" satisfying $\ixi \omega_0 = 0$ and $i_{\partial_r} \omega_0 = 0$.

Calculating the map $\Xi_{x,\xi}$ \eqref{eq:Xi_map1} for the system in question \eqref{eq:DD_systems_overdetermined} yields that:
\[
\Xi_{x,\xi} (\{s \mapsto \psi(s)\}) = \PtD \psi(0).
\]
Thus, if $\psi \in \bbM_{x,\xi}^+$ as above, the condition $\Xi_{x,\xi} (\{s \mapsto \psi(s)\}) = 0$ reduces to:
\[
\PtD \omega_0 = 0,
\]
since $\PtD(dr \wedge) = 0$ and $\PtD(\xi'\wedge) = \xi'\wedge \PtD$. As $\idr \omega_0 = 0$, $\omega_0$ has only tangential components, so $\PtD \omega_0 = 0$ implies $\omega_0 \equiv 0$, and hence $\psi \equiv 0$. Thus, $\Xi_{x,\xi}$ is injective when restricted to $\bbM_{x,\xi}^+$, as required.
\end{proof}

Applying the results in \secref{sec:outline_template}, we obtain that the lifted complex consists of a sequence of operators $\mathpzc{d}_{\nabla}:\Omega^\alpha(M;\bbU) \rightarrow \Omega^{\alpha+1}(M;\bbU)$, differing from $\dU$ by terms of order and class zero, and satisfying:
\[
\mathpzc{d}_{\nabla}\mathpzc{d}_{\nabla}\omega=0 \textand \PtD\mathpzc{d}_{\nabla}\omega=0 \qquad \text{for} \qquad \omega\in\Omega^\alpha(M;\bbU)\cap\ker\PtD, 
\]
with adjoints $\pzcdel_{\nabla}:\Omega^{\alpha+1}(M;\bbU)\rightarrow \Omega^\alpha(M;\bbU)$ satisfying $\pzcdel_{\nabla}\pzcdel_{\nabla}=0$.

Specifically, for $\alpha = 0$, since $\dU = \nabla$ on zero-forms and $\PtD = |_{\partial M}$ denotes restriction to the boundary, we have
\[
\module_{\D}^{0}(\fbN_{\bullet}) 
= \ker(\dU, \PtD) 
= \ker(\nabla, |_{\partial M}) 
= \{0\},
\]
which is the space of all $\nabla$-parallel fields vanishing on the boundary (and therefore vanishing identically by invariance under parallel transport).

The triviality of the zeroth cohomology, independent of the choice of connection $\nabla$, translates in the setting of \thmref{thm:smooth_elliptic_pre_complex} as follows:  
when the systems \eqref{eq:D_exterior_systems} are regarded as an elliptic pre-complex depending tamely and smoothly on $\nabla$, the application of \thmref{thm:smooth_elliptic_pre_complex} at $\alpha = 0$ implies that the lifted operator at the next level,
\[
\mathpzc{d}_{\nabla}:\Omega^{1}(M;\bbU) \to \Omega^{2}(M;\bbU),
\]
also depends tamely and smoothly on the connection $\nabla$.

Finally, \thmref{thm:compatibility_dd} then takes the cohomological form described in \secref{sec:DD_exterior_intro}.

\subsection{Neumann picture}
\label{sec:NN_exterior_intro}
For a class of elliptic pre-complexes consisting of exterior covariant derivatives and based on Neumann conditions—covering also the example in \secref{sec:NN_exterior_intro}—let $\beta \in \Nzero$ and consider
\beq 
\begin{aligned}
\bD_\alpha &= 
\begin{pmatrix} \dU \oplus \delU & 0 \\ \PtD & 0 \end{pmatrix} 
\qquad &&& \text{if } \alpha = 0, \\[1em]
\bD_\alpha &= 
\begin{pmatrix} \dU & 0 \\ \PtD & -d_{\jmath^*\nabla} \end{pmatrix} \sqcup 
\begin{pmatrix} \delU & 0 \\ 0 & 0 \end{pmatrix} 
\qquad &&& \text{if } \alpha > 0.
\end{aligned}
\label{eq:exterior_disjoint}
\eeq
Here, the symbol $\sqcup$ denotes the disjoint union of systems introduced in \defref{def:direct_sums_unions_sums}. Specifically:
\begin{itemize}
\item For $\alpha = 0$, the operator $\bD_{0}$ acts as
\[
\begin{pmatrix} \dU \oplus \delU & 0 \\ \PtD & 0 \end{pmatrix}:
\mymat{\Omega^{\beta}(M;\bbU)\\\oplus\\0}
\longrightarrow
\mymat{\Omega^{\beta+1}(M;\bbU) \oplus \Omega^{\beta-1}(M;\bbU)\\\oplus\\\Omega^{\beta}(\dM;\jmath^*\bbU)}.
\]

\item For $\alpha \ge 1$, the right-hand system in the disjoint union acts as
\[
\begin{pmatrix} \delU & 0 \\ 0 & 0 \end{pmatrix}:
\mymat{\Omega^{\beta-\alpha}(M;\bbU)\\\oplus\\0}
\longrightarrow
\mymat{\Omega^{\beta-\alpha-1}(M;\bbU)\\\oplus\\0},
\]
while the left-hand system acts as
\[
\begin{pmatrix} \dU & 0 \\ \PtD & -d_{\jmath^*\nabla} \end{pmatrix}:
\mymat{\Omega^{\beta+\alpha}(M;\bbU)\\\oplus\\\Omega^{\beta+\alpha-1}(\dM;\jmath^*\bbU)}
\longrightarrow
\mymat{\Omega^{\beta+\alpha+1}(M;\bbU)\\\oplus\\\Omega^{\beta+\alpha}(\dM;\jmath^*\bbU)}.
\]
\end{itemize}

From this point onward, the system naturally decomposes into two disjoint subsystems that operate independently, as described in \secref{sec:disjoint_union} and \propref{prop:disjoint_union_corrected}.

We verify that these indeed fall into the template  \eqref{eq:A_pallete} for the $\NN$-case:
\begin{itemize}
\item For the operators in \eqref{eq:AD_pallete}, we again recognize as in the $\DD$-case: 
\[
A_{D,\alpha} = \dU.
\]
Again, more explicitly, 
\[
\begin{aligned} 
&A_{0}=\dU\oplus\delU \quad &&D_{\alpha}=0 \quad &&A^*_{0}=\delU\oplus^*\dU \quad &&B_{0}=\PtD\oplus\PnD \quad  &&B_{0}^*=\PnD \oplus \PtD
\\& A_{\alpha}=\dU\sqcup\delU \quad &&D_{\alpha}=0 \quad &&A^*_{\alpha}=\delU\sqcup\dU \quad &&B_{\alpha}=\PtD\sqcup\PnD \quad &&B^*_{\alpha}=\PnD\sqcup\PtD, \quad \alpha\geq 1.
\end{aligned} 
\]
\item For the second item, in \eqref{eq:trace_operator_pallete} we also let $M_{\alpha}=0$ and 
\[
S_{\alpha} = \operatorname{Pr}_{\beta+\alpha} : \Omega^{\beta+\alpha}(\dM;\jmath^*\bbU) \oplus \Omega^{\beta-\alpha-1}(\dM;\jmath^*\bbU) \to \Omega^{\beta+\alpha}(\dM;\jmath^*\bbU)\oplus \Omega^{\beta-\alpha-1}(\dM;\jmath^*\bbU)
\]
be the orthogonal projection into $\Omega^{\beta+\alpha}(\dM;\jmath^*\bbU)$, i.e.,
\[
\operatorname{Pr}_{\beta+\alpha}(\rho,\lambda)=(\rho,0)
\]
which is a tensorial operation.  Consequently, $\PtD = \operatorname{Pr}_{\beta+\alpha}(\PtD \oplus \PnD)$ becomes a trace operator fitting into the same template as \eqref{eq:trace_operator_pallete}. 

\item Finally, for the boundary operators, we recognize: 
\[
\begin{aligned} 
&Q_{0}=0 \qquad &&K_{0}=0
\\&Q_{\alpha}=-d_{\jmath^*\nabla} \qquad &&K_{\alpha}=0, \qquad \alpha\geq 1.
\end{aligned} 
\] 
\end{itemize}
The required order-reduction properties in \eqref{eq:order_reduction_conditions_examples} are seen to be satisfied by the systems in \eqref{eq:exterior_disjoint}, by dualizing the corresponding properties in \eqref{eq:order_reduction_exterior}. Explicitly, since $\dU \dU = \mathrm{R}_{\nabla}$ is a zero-order operation, it follows that $\delU \delU = (\mathrm{R}_{\nabla})^*$ is also a zero-order operation, so the combined operator  
\[
(\dU \sqcup \delU)(\dU \oplus \delU) = \dU \dU \oplus \delU \delU
\]  
remains zero-order as well.

Before proceeding with the verification of the required overdetermined ellipticity, for illustration, we consider several specific cases of the rather general sequence \eqref{eq:exterior_disjoint}. Notably, for the case $\beta=0$, since $\Omega^{\beta-\alpha}(M;\bbU)=0$ for all $\alpha > 0$, the elliptic pre-complex simplifies into:
\beq
\begin{aligned}
\bD_\alpha &= 
\begin{pmatrix} 
\nabla & 0 \\ 
|_{\partial M} & 0 
\end{pmatrix}:
\mymat{\Omega^{0}(M;\bbU) \\ \oplus \\ 0}
\longrightarrow
\mymat{\Omega^{1}(M;\bbU) \\ \oplus \\ \Omega^{0}(\dM;\jmath^*\bbU)}
&& \text{if } \alpha = 0, \\[1em]
\bD_\alpha &= 
\begin{pmatrix} 
\dU & 0 \\ 
\PtD & -d_{\jmath^*\nabla} 
\end{pmatrix}:
\mymat{\Omega^{\alpha}(M;\bbU) \\ \oplus \\ \Omega^{\alpha-1}(\dM;\jmath^*\bbU)}
\longrightarrow
\mymat{\Omega^{\alpha+1}(M;\bbU) \\ \oplus \\ \Omega^{\alpha}(\dM;\jmath^*\bbU)}
&& \text{if } \alpha > 0
\end{aligned}
\label{eq:elliptic_pre_complex_exterior_NN_1}
\eeq
removing the disjoint union structure entirely. This is the example we introduced in \secref{sec:NN_exterior_intro}.

Similarly, for $\beta=\dim{M}$, since $\Omega^{\beta+\alpha}(M;\bbU)=0$ for every $\alpha > 0$ the system simplifies to: 
\[
\begin{aligned}
\bD_\alpha &= 
\begin{pmatrix} 
\delU & 0 \\ 
0 & 0 
\end{pmatrix}.
\end{aligned}
\]
By a duality argument, another elliptic pre-complex of this type is given by: 
\[
\begin{aligned}
\bD_\alpha &= 
\begin{pmatrix} 
\dU & 0 \\ 
0 & 0 
\end{pmatrix}
\end{aligned}
\]
which represents the same sequence of systems constituting \eqref{eq:D_exterior_systems}. However, instead of Dirichlet conditions, the required Neumann conditions in \eqref{eq:overdetermined_ellipticies_examples_refined} correspond to the overdetermined ellipticity of:
\[
\begin{pmatrix} 
\dU \oplus \delU & 0 \\ 
\PnD & 0
\end{pmatrix}.
\]
This demonstrates that the same sequence of operators can support elliptic pre-complexes based on either Dirichlet or Neumann conditions.

For a general $\beta \in \Nzero$, translating the systems in \eqref{eq:overdetermined_ellipticies_examples_refined} to this setting yields the corresponding overdetermined ellipticities for the $\NN$ case. After computing the associated adapted adjoints and boundary systems as in \eqref{eq:supplement_operators_examples}, these take the form
\beq
\begin{split}
&\begin{pmatrix} \dU \oplus \delU & 0 \\ \PtD & 0 \end{pmatrix}\sqcup 
\begin{pmatrix} \dU \oplus \delU & 0 \\ \PnD & 0 \end{pmatrix}, 
\qquad 
\begin{pmatrix}0 & 0 \\ 0 & d_{\jmath^*\nabla}\oplus\delta_{\jmath^*\nabla}\end{pmatrix}.
\end{split}
\label{eq:overdetermined_ellipticity_NN_exterior}
\eeq
Because of the disjoint union structure, the overdetermined ellipticity on the left decomposes into the overdetermined ellipticities of the two subsystems
\[
\begin{pmatrix} \dU \oplus \delU & 0 \\ \PtD & 0 \end{pmatrix}
\quad \text{and} \quad
\begin{pmatrix} \dU \oplus \delU & 0 \\ \PnD & 0 \end{pmatrix}.
\]

\begin{proposition}
\label{prop:ND_exteiror} 
The systems in \eqref{eq:overdetermined_ellipticity_NN_exterior} are overdetermined elliptic. 
\end{proposition}
Most of the details are already addressed in \propref{prop:DD_exterior}, allowing us to make short work of the proof:
\begin{proof}
Due to \propref{prop:DD_exterior} and the last comment, it remains to verify that the following systems are overdetermined elliptic:
\[
\begin{pmatrix} \dU \oplus \delU & 0 \\ \PnD & 0 \end{pmatrix}, \qquad \begin{pmatrix}0 & 0 \\ 0 & d_{\jmath^*\nabla}\oplus\delta_{\jmath^*\nabla}\end{pmatrix}.
\]
The system on the left is an element of $\OP(1,1)$, while the system on the right belongs to $\OP(1,-\infty)$. Thus, as in \propref{prop:DD_exterior}, we can employ the machinery developed in \secref{sec:principle_symbol_green}. 

For the system on the left in \eqref{eq:overdetermined_ellipticity_NN_exterior}, the only difference compared to \propref{prop:DD_exterior} is that $\PtD$ is replaced by $\PnD$. Consequently, the associated interior symbols and systems of overdetermined Es from \propref{prop:rud_lop} are identical to those for the system in \propref{prop:DD_exterior}. The sole distinction lies in the initial condition map $\Xi_{x,\xi'}$, which is given here by:
\[
\Xi_{x,\xi'} (\{s \mapsto \psi(s)\}) = \PnD\psi(0).
\]

Now, given any $\psi \in \bbM_{x,\xi'}$ satisfying  
\[
\Xi_{x,\xi'} (\{s \mapsto \psi(s)\}) = 0,
\]
we fit in the general form of $\psi \in \bbM_{x,\xi'}$ obtained in the proof of \propref{prop:DD_exterior}. We find that due to the relations $\idr \omega_{0} = 0$ and  
\[
\PnD (\xi' \wedge \omega_{0}) = -\xi' \wedge \PnD \omega_{0} = 0, \qquad \PnD (dr \wedge \omega_{0}) = \omega_{0},
\]
that $\PnD \psi(0) = 0$ implies $\omega_{0} = 0$. Hence, $\psi \equiv 0$, establishing that $\Xi_{x,\xi'}$ is injective, as required.

For the overdetermined ellipticity of the system on the right in \eqref{eq:overdetermined_ellipticity_NN_exterior}, which operates solely on boundary sections, \propref{prop:rud_lop} reduces the problem to verifying the injectivity of:
\[
\sigma(d_{\jmath^*\nabla})(x,\xi') \oplus \sigma(\delta_{\jmath^*\nabla})(x,\xi') = (\iota \xi' \wedge) \oplus (\iota \ixi),
\]
which, again, follows from standard results in multilinear algebra.
\end{proof}
Thus, in view of \thmref{thm:elliptic_pre_complex_exmaples}, the systems in \eqref{eq:exterior_disjoint} constitute an elliptic pre-complex based on Neumann conditions for any initialization point $\beta \in \Nzero$.  

Following the outline in \secref{sec:outline_template}, and due to the disjoint union structure, by \propref{prop:disjoint_union_corrected}, the lifted complex splits as well after $\alpha>0$, taking in general the form:
\[
\begin{aligned}
\fbD_\alpha &= \bD_{0} 
&& \qquad\text{if } \alpha = 0, \\[1em]
\fbD_\alpha &= 
\begin{pmatrix} 
\mathpzc{d}_{\nabla} & -\mathpzc{k}_{\nabla} \\ 
\PtD-\mathpzc{c}_{\nabla} & - \mathpzc{d}_{\jmath^*\nabla} 
\end{pmatrix} 
\sqcup
\begin{pmatrix}
\pzcdel_{\nabla} & 0 \\ 
0 & 0
\end{pmatrix}
&& \qquad\text{if } \alpha > 0.
\end{aligned}
\]

\thmref{thm:compatibility_nn} then reads in general: 
\begin{theorem}
Let $\omega \in \Omega^{\beta+\alpha+1}_{M; \bbU}$, $\eta \in \Omega^{\beta-\alpha-1}_{M; \bbU}$, and $\rho \in \Omega^{\beta+\alpha}_{\dM; \jmath^* \bbU}$. 
Then the boundary value problem
\[
\begin{aligned}
& \dU \psi = \omega, \qquad 
\delU \zeta = \eta, \qquad
\pzcdel_{\nabla}\psi = \mathpzc{c}^*_{\nabla} \lambda, 
\qquad \mathpzc{d}_{\nabla}\zeta = 0,
&& \text{in } M, \\[6pt]
& \mathpzc{d}_{\jmath^*\nabla}\lambda = -\mathpzc{k}^*_{\,\,\nabla}\psi, 
\qquad \PtD \psi - d_{\jmath^*\nabla}\lambda = \rho, 
\qquad \PnD \psi + \mathrm{Pr}_{\alpha+\beta}\lambda = 0,
&& \text{on } \dM.
\end{aligned}
\]

admits a solution $(\psi, \zeta; \lambda)$ if and only if
\[
\mathpzc{d}_{\nabla} \omega = \mathpzc{k}_{\nabla} \rho, 
\qquad \pzcdel_{\nabla}\eta = 0, 
\qquad \PtD \omega - \mathpzc{c}_{\nabla}\omega = \mathpzc{d}_{\jmath^*\nabla}\rho, 
\qquad (\omega, \eta; \rho) \perp \module_{\N}^{\alpha+1}.
\]
The solution $(\psi, \zeta; \lambda)$ is unique modulo an element of $\module_{\N}^{\alpha}$.
\end{theorem}

Note that for the cases $\beta=0$, $\beta=\dim{M}$ and $\alpha=0$, some of these conditions are satisfied trivially. In particular, for $\beta=\dim{M}$ and its dualized version, there are no boundary operators to be lifted. Hence, the statement simplifies significantly for all $\alpha \geq 0$ in this case. In particular, the case $\beta=0$ is nothing but \thmref{thm:exterior_chomology_intro}. 

%
%

\section{Killing and Hessian equations}
\label{sec:KillingHessianBody} 
Here, we complete the technical details and results outlined in \secref{sec:KillingHessianIntro} by formulating the systems in \eqref{eq:systems_curvature_intro} as an elliptic pre-complex in the prototypical sense \secref{sec:dirichelt_special_case}. We emphasize that all technical proofs are taken directly from \cite[Sec.~5]{KL23}, and presented here for completeness and re-framing so it fits the notation and terminology we use here. 

Let $1 \le m \le d$; we recast the Bianchi sequence \eqref{eq:calabi_complex} into the simplified template, diagram \eqref{eq:elliptic_complex_diagramTe}, which we break into two lines:
\[
\begin{xy}
(-50,0)*+{0}="Em1";
(-30,0)*+{\Ckm{0}{m}}="E0";
(20,0)*+{\Ckm{1}{m}}="E1";
(70,0)*+{\cdots}="E2";
(100,0)*+{\Ckm{m}{m}}="E3";
(-50,-20)*+{}="Gm1";
(-30,-20)*+{\plCkm{0}{m} \oplus \plCkm{0}{m-1}}="G0";
(20,-20)*+{\plCkm{1}{m} \oplus \plCkm{1}{m-1}}="G1";
(70,-20)*+{\cdots}="G2";
(100,-20)*+{(\plCkm{m}{m})^2}="G3";
{\ar@{->}@/^{1pc}/^{\dBianchi}"E0";"E1"};
{\ar@{->}@/^{1pc}/^{\delBianchi}"E1";"E0"};
{\ar@{->}@/^{1pc}/^{\dBianchi}"E1";"E2"};
{\ar@{->}@/^{1pc}/^{\delBianchi}"E2";"E1"};
{\ar@{->}@/^{1pc}/^{\dBianchi}"E2";"E3"};
{\ar@{->}@/^{1pc}/^{\delBianchi}"E3";"E2"};
{\ar@{->}@/^{1pc}/^0"Em1";"E0"};
{\ar@{->}@/^{1pc}/^0"E0";"Em1"};
{\ar@{->}@/_{0pc}/^{B_g}"E0";"G0"};
{\ar@{->}@/_{0pc}/^{B_g}"E1";"G1"};
{\ar@{->}@/_{0pc}/^{B_g}"E2";"G2"};
{\ar@{->}@/^{1pc}/^{B_g^*}"E1";"G0"};
{\ar@{->}@/^{1pc}/^{B_g^*}"E2";"G1"};
{\ar@{->}@/^{1pc}/^{B_g^*}"E3";"G2"};
{\ar@{->}@/_{0pc}/^{B_H}"E3";"G3"};
\end{xy}
\]
\[
\begin{xy}
(-50,0)*+{\Ckm{m}{m}}="Em1";
(0,0)*+{\Ckm{m+1}{m+1}}="E0";
(50,0)*+{\cdots}="E1";
(90,0)*+{\Ckm{d}{m+1}}="E2";
(110,0)*+{0}="E3";
(-50,-20)*+{(\plCkm{m}{m})^2}="Gm1";
(0,-20)*+{\plCkm{m+1}{m+1} \oplus \plCkm{m+1}{m}}="G0";
(50,-20)*+{\cdots}="G1";
(90,-20)*+{\plCkm{d}{m+1} \oplus \plCkm{d}{m}}="G2";
(110,-20)*+{}="G3";
{\ar@{->}@/^{1pc}/^{\dBianchi}"E0";"E1"};
{\ar@{->}@/^{1pc}/^{\delBianchi}"E1";"E0"};
{\ar@{->}@/^{1pc}/^{\dBianchi}"E1";"E2"};
{\ar@{->}@/^{1pc}/^{\delBianchi}"E2";"E1"};
{\ar@{->}@/^{1pc}/^0"E2";"E3"};
{\ar@{->}@/^{1pc}/^0"E3";"E2"};
{\ar@{->}@/^{1pc}/^{\Hg}"Em1";"E0"};
{\ar@{->}@/^{1pc}/^{\Hg^*}"E0";"Em1"};
{\ar@{->}@/_{0pc}/^{B_H}"Em1";"Gm1"};
{\ar@{->}@/_{0pc}/^{B_g}"E0";"G0"};
{\ar@{->}@/_{0pc}/^{B_g}"E1";"G1"};
{\ar@{->}@/_{0pc}/^{B_g}"E2";"G2"};
{\ar@{->}@/^{1pc}/^{B_g^*}"E1";"G0"};
{\ar@{->}@/^{1pc}/^{B_g^*}"E2";"G1"};
{\ar@{->}@/^{1pc}/^{B_H^*}"E0";"Gm1"};
\end{xy}
\]
The boundary operators and Green's formulae, as well as their normality, are detailed in the Appendix, \secref{sec:first_order}.

Once this is verified to be a Neumann elliptic pre-complex as in \defref{def:elliptic_pre_complexTemNeu}, \thmref{thm:hessian} and \thmref{thm:Killing} are then direct translations of \eqref{thm:cohomology_Neumann} for the cases $m=0$ and $m=1$, respectively, as outlined in \secref{sec:KillingHessianIntro}.

We shall use extensively the identities in \secref{sec:Appendix}. 

The required order-reduction properties in \defref{def:elliptic_pre_complexTemNeu} evidently translate into the following: 

\begin{proposition}
The following hold: 
\begin{enumerate}[itemsep=0pt, label=(\alph*)]
\item For $k \le m-2$, $\dBianchi\dBianchi : \Ckm{k}{m} \to \Ckm{k+2}{m}$ is of order 0. 
\item $\Hg\dBianchi : \Ckm{m-1}{m} \to \Ckm{m+1}{m+1}$ is of order 1. 
\item $\dBianchi \Hg : \Ckm{m}{m} \to \Ckm{m+2}{m+1}$ is of order 1. 
\item For $k \ge m+1$, $\dBianchi\dBianchi : \Ckm{k}{m+1} \to \Ckm{k+2}{m+1}$ is of order 0. 
\end{enumerate}
\end{proposition}

\begin{proof}
Items (a) and (d) are proved in \propref{prop:DstarDstar}. For Item (b) take for example, for $\psi\in \Ckm{m-1}{m}$, use the explicit expression for $\dBianchi$ in \secref{sec:first_order} to obtain: 
\[
\begin{split}
\Hg\dBianchi\psi &= \tfrac12(\dg \dgV + \dgV\dg) \brk{\dg - \tfrac12 \G\dgV} \psi \\
&= \tfrac12(\dg \dgV + \dgV\dg) \dg  \psi - \tfrac14 \G(\dg \dgV + \dgV\dg) \dgV \psi \\ 
\end{split}
\]
where in the passage to the second line we used the commutation of $\G$ with $\Hg$. Since $\dg\dg$ and $\dgV\dgV$ are tensorial, and so is the commutator of $\dg$ and $\dgV$, it follows that $\Hg\dBianchi$ is a first-order operator.
Item (c) follows similarly.
\end{proof}

The overdetermined ellipticities required in \defref{def:elliptic_pre_complexTemNeu} are cases $(a)$, $(b)$, and $(d)$ of the following more general claim: 

\begin{proposition}
\label{prop:overdeteremined_ellipticityBianchi}
The following systems are overdetermined elliptic: 
\begin{enumerate}[itemsep=0pt,label=(\alph*)]
\item $\delBianchi\oplus\dBianchi\oplus B_g^*$ with domain $\Ckm{k}{m}$ for $k < m$.
\item $\delBianchi\oplus H\oplus B_g^*$ with domain $\Ckm{m}{m}$.
\item $H^*\oplus\dBianchi\oplus B_H^*$ with domain $\Ckm{m+1}{m+1}$.
\item $\delBianchi\oplus\dBianchi\oplus B_g^*$ with domain $\Ckm{k}{m}$ for $k > m$.
\end{enumerate}
\end{proposition}

\begin{proof}
Noting the following Hodge-dualities (up to multiplicative constants),
\[
\begin{aligned}
& \starG\starGV(\Hg^*\oplus\dBianchi\oplus B_H^*)\starG\starGV = \Hg\oplus\delBianchi\oplus B_H \\
& \starG\starGV(\delBianchi\oplus\dBianchi\oplus B_g^*)\starG\starGV = \dBianchi\oplus\delBianchi\oplus B_g,
\end{aligned}
\]
we may equivalently prove that 
\begin{enumerate}[itemsep=0pt,label=(\alph*)]
\item $\delBianchi\oplus\dBianchi\oplus B_g^*$ and $\delBianchi\oplus\dBianchi\oplus B_g$ with domain $\Ckm{k}{m}$, $k<m$,  are overdetermined elliptic.
\item $\delBianchi\oplus \Hg\oplus B_g^*$ and $\delBianchi\oplus \Hg\oplus B_H$ with domain $\Ckm{m}{m}$ are overdetermined elliptic.
\end{enumerate}
Statement (a) is proved in \propref{prop:ellipticity_bianchi} below.
As for Statement (b), we note that $\delBianchi = \deltag$ and $B_g^* = \PnnD\oplus\PntD$, when restricted to $\Ckm{m}{m}$, and that $\Ckm{m}{m}$ is a subspace of the symmetric $(m,m)$-forms, hence (b) can be replaced by showing that  
\[
\deltag\oplus \Hg\oplus\PnnD\oplus\PntD \textand \deltag\oplus \Hg\oplus B_H
\]
restricted to symmetric $\Wkm{m}{m}$ forms are overdetermined elliptic. This is proved in \propref{prop:ellipticity_H}.
\end{proof}

\begin{proposition}
\label{prop:ellipticity_bianchi}
The systems 
\[
\delBianchi\oplus\dBianchi\oplus B_g^*
\Textand
\delBianchi\oplus\dBianchi\oplus B_g,
\] 
restricted to $\Ckm{k}{m}$, $k< m$, are overdetermined elliptic.
\end{proposition}

\begin{proof}
Since all the operators, including the boundary operators are differential operators,  we need to verify that the Loptanskii-Shaprio criteria in \thmref{thm:lopatnskii_shapiro} are satisfied: we need to show that for every $x\in M$ and $\xi\in T_x^*M\setminus\{0\}$, 
\[
\sigma_{\delBianchi}(x,\xi) \oplus \sigma_{\dBianchi}(x,\xi) : \Gkm{k}{m} \to \Gkm{k-1}{m} \oplus  \Gkm{k+1}{m}
\]
is injective, and so are the maps
\[
\begin{aligned}
\Xi^1_{x,\xi'} &= \sigma_{\PnnG}(x,\xi'+ \imath\,\partial_s dr)\oplus \sigma_{\PntG}(x,\xi'+ \imath\,\partial_s dr)|_{s=0} \\
\Xi^2_{x,\xi'} &= \sigma_{\PttG}(x,\xi'+ \imath\,\partial_s dr)\oplus \sigma_{\PtnG}(x,\xi'+ \imath\,\partial_s dr)|_{s=0}
\end{aligned}
\]
for $x\in\dM$ and $\xi'\in T_x^*\dM\setminus\{0\}$, when restricted to the space $\bbM_{x,\xi'}^+$ of decaying solutions to the ordinary differential system
\beq
\brk{\sigma_{\delBianchi}(x,\xi' + \imath\,\partial_s dr) \oplus \sigma_{\dBianchi}(x\xi' ,+ \imath\,\partial_s dr)}\psi(s) = 0.
\label{eq:overdeterminedE_ellipticity_bianchi}
\eeq

Let $x\in M$ and let $\xi\in T^*_xM$; we denote $\xi_V = \xi^T = \GV\xi$; without loss of generality, we may assume that $|\xi|_\g=1$.
For $\psi\in\Gkm{k}{m}|_x$, $k<m$,
\[
\begin{aligned}
-\imath \sigma_{\dBianchi}(x,\xi) \psi &= \xi\wedge\psi - \frac{1}{\alpha(m,k)} \G(\xi_V\wedge \psi)
\\
\imath \sigma_{\delBianchi}(x,\xi) \psi &= \ixi\psi.
\end{aligned}
\]
Suppose that $(\sigma_{\delBianchi}(x,\xi) \oplus \sigma_{\dBianchi}(x,\xi)) \psi=0$, then
\[
\ixi\brk{\xi\wedge\psi - \frac{1}{\alpha(m,k)} \G(\xi_V\wedge \psi)} = 0.
\]
Using the fact that $\sigma_{\delBianchi}(x,\xi) \psi=0$, $\ixi (\xi\wedge) + \xi\wedge\ixi = \id$ and the anti-commutation relation between $\ixi$ and $\G$, we obtain
\[
\psi + \frac{1}{\alpha(m,k)} \ixiV(\xi_V\wedge \psi) = 0.
\]
Taking an inner-product with $\psi$,
\[
|\psi|_\g^2 + \frac{1}{\alpha(m,k)} |\xi_V\wedge\psi|_\g^2 = 0, 
\]
which implies that $\psi=0$, i.e., $\sigma_{\delBianchi}(x,\xi) \oplus \sigma_{\dBianchi}(x,\xi)$ is injective.

We proceed to establish the injectivity of the boundary symbols.
Let $x\in\dM$ and $\xi \in T_x^*\dM$. 
The ordinary differential equation \eqref{eq:overdeterminedE_ellipticity_bianchi} takes the form of a system,
\begin{gather}
\xi\wedge \psi - \frac{1}{\alpha(m,k)} \G(\xi_V\wedge \psi) + \imath \brk{dr\wedge \dot{\psi} - \frac{1}{\alpha(m,k)} \G(dr_V\wedge \dot{\psi})} = 0
\label{eq:overdeterminedE1} \\ 
\ixi\psi + \imath i_{\dr} \dot{\psi} =0.
\label{eq:overdeterminedE2} 
\end{gather}
To solve it, we decompose $\psi$ orthogonally (as an element in $\Lkm{k}{m}|_x$)
\[
\begin{split}
\psi &= \brk{\psi_{00} + \xi\wedge \psi_{01} + dr\wedge \psi_{02} + \xi\wedge dr \wedge \psi_{03}} \\
&+ \xi_V\wedge \brk{\psi_{10} + \xi\wedge \psi_{11} + dr\wedge \psi_{12} + \xi\wedge dr \wedge \psi_{13}} \\
&+ dr^T\wedge \brk{\psi_{20} + \xi\wedge \psi_{21} + dr\wedge \psi_{22} + \xi\wedge dr \wedge \psi_{23}} \\
&+ \xi_V\wedge dr^T\wedge \brk{\psi_{30} + \xi\wedge \psi_{31} + dr\wedge \psi_{32} + \xi\wedge dr \wedge \psi_{33}},
\end{split}
\]
where 
\[
\ixi\psi_{ij}=0
\qquad  
\ixiV\psi_{ij}=0
\qquad  
i_{\dr}\psi_{ij}=0
\Textand
i_{\dr}^V\psi_{ij}=0
\]
for every $i,j=0,1,2,3$.
Substituting into \eqref{eq:overdeterminedE1} and \eqref{eq:overdeterminedE2}, equating like terms we obtain the following equations:
\beq
\psi_{00}  =  0
\label{eq:00}
\eeq
\beq
\Cases{
\psi_{01} + \imath \dot\psi_{02} = 0 \\
\psi_{02} - \imath \dot\psi_{01} = 0 
}
\qquad\qquad
\Cases{
\psi_{10} + \tfrac{1}{\alpha(m,k+1)}(-\G\psi_{01} + \psi_{10}) = 0 \\
\psi_{20} + \tfrac{\imath}{\alpha(m,k+1)}(-\G\dot\psi_{01} + \dot\psi_{10}) = 0
}
\label{eq:01}
\eeq
\beq
\Cases{
\psi_{11} + \imath \dot\psi_{12}  = 0 \\
\psi_{21} + \imath \dot\psi_{22}  = 0
}
\qquad\qquad
\Cases{
\psi_{12} - \imath\dot\psi_{11} + \frac{1}{\alpha(m,k+1)}(\psi_{12} - \psi_{21})  = 0\\
\psi_{22} - \imath\dot\psi_{21} + \frac{\imath}{\alpha(m,k+1)}(\dot\psi_{12} - \dot\psi_{21}) = 0
}
\label{eq:11}
\eeq
\beq
\Cases{
\psi_{31} + \imath \dot\psi_{32}  = 0 \\
\psi_{32} - \imath\dot\psi_{31} = 0
}
\label{eq:31}
\eeq
and
\beq
\psi_{30} + \tfrac{1}{\alpha(m,k+1)} (\G\psi_{21} + \psi_{30}) - \tfrac{\imath}{\alpha(m,k+1)} \G\dot\psi_{11} = 0,
\label{eq:30}
\eeq
and
\[
\psi_{i3}=0
\qquad\qquad
\text{for $i=0,1,2,3$}. 
\]

The initial conditions, can be orthogonally decomposed similarly (noting that terms including $dr$ and $dr^T$ are annihilated by the pullback to boundary).

Let $\psi\in \bbM_{x,\xi'}^+$. The condition that $\Xi^1_{x,\xi'}\psi=0$ results in
\[
\psi_{02}(0) = 0
\qquad
\psi_{12}(0) = 0
\qquad
\psi_{22}(0) = 0
\Textand
\psi_{32}(0) = 0.
\]
Substituting $\psi_{01}(0)=0$ into \eqref{eq:01}, along with the condition the solution is decaying at infinity, yields
\[
\psi_{01}(s) = 0
\qquad
\psi_{02}(s) = 0
\qquad
\psi_{10}(s) = 0
\Textand
\psi_{20}(s) = 0.
\]
Substituting $\psi_{12}(0) = 0$ and 
$\psi_{22}(0) = 0$ into \eqref{eq:11} yields
\[
\psi_{11}(s) = 0
\qquad
\psi_{12}(s) = 0
\qquad
\psi_{21}(s) = 0
\Textand
\psi_{22}(s) = 0.
\]
Substituting $\psi_{32}(0) = 0$ into \eqref{eq:31} yields
\[
\psi_{31}(s) = 0
\Textand
\psi_{32}(s) = 0.
\]
Finally, \eqref{eq:30} yield that
\[
\psi_{30}(s) = 0, 
\]
thus $\psi(s) = 0$, proving the injectivity of $\Xi^1_{x,\xi'}$ when restricted to $\bbM^+_{x,\xi'}$.

The condition that $\Xi^2_{x,\xi'}\psi=0$ results in
\[
\psi_{01}(0) = 0
\qquad
\psi_{11}(0) = 0 
\qquad
\psi_{21}(0) + \tfrac{1}{\alpha(m,k+1)}(\psi_{21}(0)  - \psi_{12}(0)) = 0 
\textand
\psi_{31}(0) = 0.
\]
Substituting $\psi_{01}(0) = 0$ into \eqref{eq:01} yields,
\[
\psi_{01}(s) = 0
\qquad
\psi_{02}(s) = 0
\qquad
\psi_{10}(s) = 0
\Textand
\psi_{20}(s) = 0.
\]
Substituting $\psi_{11}(0) = 0$ and 
$\psi_{21}(0) + \frac{1}{\alpha(m,k+1)}(\psi_{21}(0)  - \psi_{12}(0)) = 0$ into \eqref{eq:11}
yields
\[
\psi_{11}(s) = 0
\qquad
\psi_{12}(s) = 0
\qquad
\psi_{21}(s) = 0
\Textand
\psi_{22}(s) = 0.
\]
Substituting $\psi_{31}(0) = 0$ into \eqref{eq:31} yields
\[
\psi_{31}(s) = 0
\Textand
\psi_{32}(s) = 0.
\]
Finally, \eqref{eq:30} yield that
\[
\psi_{30}(s) = 0, 
\]
thus $\psi(s) = 0$, proving the injectivity of $\Xi^2_{x,\xi'}$ when restricted to $\bbM^+_{x,\xi'}$.
\end{proof}

\begin{proposition}
\label{prop:ellipticity_H}
The systems 
\[
\deltag\oplus H\oplus \PnnD\oplus\PntD
\Textand
\deltag\oplus H\oplus\PttD\oplus\frakT,
\] 
restricted to the symmetric elements of $\Wkm{m}{m}$ are overdetermined elliptic.
\end{proposition}

\begin{proof}
Since all the operators, including the boundary operators are differential operators,  we need to verify that the Lopatnskii-Shaprio condition in \thmref{thm:lopatnskii_shapiro} is satisfied:  we need to show that for every $x\in M$ and $\xi\in T_x^*M\setminus\{0\}$, 
\[
\sigma_{\deltag}(x,\xi) \oplus \sigma_{H}(x,\xi) : \Gkm{m}{m} \to \Gkm{m-1}{m} \oplus  \Gkm{m+1}{m+1}
\]
is injective, and so are the maps
\[
\begin{aligned}
\Xi^1_{x,\xi'} &= \sigma_{\PnnD}(x,\xi'+ \imath\,\partial_s dr)\oplus \sigma_{\PntD}(x,\xi'+ \imath\,\partial_s dr)|_{s=0} \\
\Xi^2_{x,\xi'} &= \sigma_{\PttD}(x,\xi'+ \imath\,\partial_s dr)\oplus \sigma_{\frakT}(x,\xi'+ \imath\,\partial_s dr)|_{s=0}
\end{aligned}
\]
for $x\in\dM$ and $\xi'\in T_x^*\dM\setminus\{0\}$, when restricted to the space $\bbM_{x,\xi'}^+$ of decaying solutions to the ordinary differential system
\beq
\brk{\sigma_{\deltag}(x,\xi'+ \imath\,\partial_s dr) \oplus \sigma_{H}(x,\xi'+ \imath\,\partial_s dr)}\psi(s) = 0.
\label{eq:overdeterminedE_H}
\eeq

Let $x\in M$ and $\xi\in T_x^*M\setminus\{0\}$; once again, we assume  without loss of generality that $|\xi|_\g^2=1$.  
The symbols of $H$ and $\deltag$ are
\[
\sigma_{\deltag}(x,\xi)\psi = \ixi \psi
\Textand
\sigma_H(x,\xi)\psi = \xi\wedge\xi_V\wedge \psi .
\]
Suppose that $\psi\in\Gkm{m}{m}|_x$ is symmetric and satisfies
\[
\brk{\sigma_{\deltag}(x,\xi)\oplus\sigma_H(x,\xi)}\psi = 0. 
\]
Then,
\[
\ixiV \ixi(\xi\wedge\xi_V\wedge \psi) = \psi = 0,
\]
proving the injectivity of  $\sigma_{\deltag}(x,\xi)\oplus\sigma_H(x,\xi)$  when restricted to symmetric $(m,m)$-covectors. 

We proceed with the boundary symbols. Let $x\in\dM$ and $\xi'\in T_x^*\dM\setminus\{0\}$. The ordinary differential system \eqref{eq:overdeterminedE_H} takes the explicit form
\beq
\begin{aligned}
& \xi\wedge\xi_V\wedge \psi + \imath dr\wedge\xi_V\wedge \dot{\psi} + \imath \xi\wedge dr^T\wedge \dot{\psi} - dr\wedge dr^T \wedge \ddot{\psi} = 0 \\
& \ixi \psi  + \imath i_{\dr} \dot{\psi} = 0.
\end{aligned}
\label{eq:eqsH}
\eeq
As in the proof of \propref{prop:ellipticity_bianchi}, we decompose $\psi$ orthogonally. Substituting into \eqref{eq:overdeterminedE_H} and equaling like terms, we obtain
\beq
\psi_{00} = 0,
\label{EQ:00}
\eeq
\beq
\Cases{
\psi_{02} - \imath \dot\psi_{01}  = 0 \\
\psi_{01} + \imath \dot\psi_{02}  = 0
}
\label{EQ:01}
\eeq
and
\beq
\Cases{\psi_{11} + \imath \dot\psi_{12} = 0\\
\psi_{12} + \imath \dot\psi_{22} = 0  \\
\psi_{22} - 2\imath\dot\psi_{12}   - \ddot\psi_{11} = 0
}
\label{EQ:11}
\eeq
and
\[
\psi_{i3}=0
\Textand
\psi_{3i}=0
\qquad
\text{for $i=0,1,2,3$}.
\]

Let $\psi\in \bbM_{x,\xi'}^+$ satisfy  $\Xi^1_{x,\xi'}\psi=0$. This results in 
\[
\psi_{02}(0)  = 0 
\qquad
\psi_{12}(0)  = 0 
\Textand
\psi_{22}(0)  = 0.
\]
Substituting $\psi_{02}(0)  = 0$  into \eqref{EQ:01} yields
\[
\psi_{01}(s)  = 0 
\Textand
\psi_{02}(s)  = 0.
\]
Substituting $\psi_{12}(0)  = 0$ and $\psi_{22}(0)  = 0$ into \eqref{EQ:11} yields
\[
\psi_{11}(s)  = 0 
\qquad
\psi_{12}(s)  = 0 
\Textand
\psi_{22}(s)  = 0.
\]
This proves the injectivity of $\Xi^1_{x,\xi'}\psi=0$ when restricted to $\bbM_{x,\xi'}^+$.

Let $\psi\in \bbM_{x,\xi'}^+$ satisfy  $\Xi^2_{x,\xi'}\psi=0$. 
The condition $\PttD\psi(0)=0$ yields
\[
\psi_{01}(0) = 0
\Textand
\psi_{11}(0) = 0.
\]
Substituting the first into \eqref{EQ:01} yields.
\[
\psi_{01}(s)  = 0 
\Textand
\psi_{02}(s)  = 0.
\]
We then note that,
\[
-\imath \sigma_{\frakT}(x,\xi + \imath\,\partial_s dr) \psi|_{s=0} = \imath \PttD\dot{\psi}(0) -  \xi\wedge \PntD\psi(0)  -  \xi_V\wedge \PtnD\psi(0) = 0,
\]
yielding the additional initial condition
\[
\imath\dot\psi_{11}(0) - 2\psi_{12}(0) = 0,
\]
which substituted together with $\psi_{11}(0)=0$ into \eqref{EQ:11}  yields
\[
\psi_{11}(s)  = 0 
\qquad
\psi_{12}(s)  = 0 
\Textand
\psi_{22}(s)  = 0.
\]
This proves the injectivity of $\Xi^2_{x,\xi'}\psi=0$ when restricted to $\bbM_{x,\xi'}^+$.
\end{proof}

\section{Riemann curvature equations}
\label{sec:prescribed_curvature}
Here, we complete the technical details and results outlined in \secref{sec:prescribed_curvature_intro} by formulating the systems in \eqref{eq:systems_curvature_intro} as an elliptic pre-complex within the template \eqref{eq:A_pallete}, based on either Neumann or Dirichlet conditions.

In the Dirichlet case, the sequence is given as in \eqref{eq:systems_curvature_intro}, in the prototypical setup of \secref{sec:dirichelt_special_case}:  
\beq
\begin{split}
&\bA_{0} := \Def : \frakX_{M} \rightarrow \Ckm{1}{1}, \\
&\bA_{1} := \D_{\Gamma}\Rm_{g} : \Ckm{1}{1} \rightarrow \Ckm{2}{2}, \\
&\bA_{2} := \dBianchi : \Ckm{2}{2} \rightarrow \Ckm{3}{2}, \\
&\bA_{3} := \dots
\end{split}
\label{eq:curvature_complex_DD}
\eeq
Whereas for the Neumann case it is given in the full template \eqref{eq:A_pallete} by: 
\beq
\begin{split}
&\bD_{0}=\begin{pmatrix}
\Def & 0 \\ |_{\partial M} & 0
\end{pmatrix}:\mymat{\frakX_{M}\\\oplus\\0}\longrightarrow \mymat{\Ckm{1}{1}\\\oplus\\\frakX_{M}|_{\dM}}
\\& \bD_{1}=\begin{pmatrix}
\D_{\Gamma}\Rm_{g} & 0 \\ \PttG \oplus \D\Ah_{g} & -\Defd\oplus-\Defh
\end{pmatrix}:\mymat{\Ckm{1}{1}\\\oplus\\\frakX_{M}|_{\dM}}\longrightarrow\mymat{\Ckm{2}{2}\\\oplus\\\scrC_{\dM}^{1,1}\oplus\scrC^{1,1}_{\dM}}
\\& \bD_{2}=\begin{pmatrix}
\dBianchi & 0 \\ \PtD & -\D\Gh_{\gD,\Ah_{g}} \oplus -\D\MCh_{\gD,\Ah_{g}}
\end{pmatrix}:\mymat{\Ckm{2}{2}\\\oplus\\\scrC_{\dM}^{1,1}\oplus\scrC^{1,1}_{\dM}}\longrightarrow \mymat{\Ckm{3}{2}\\\oplus\\\scrC_{\dM}^{2,2}\oplus\scrC^{2,1}_{\dM}}
\\& \bD_{3}=\begin{pmatrix}
\dBianchi & 0 \\ \PtD & -d_{\gD} \oplus -d_{\gD}
\end{pmatrix}:\mymat{\Ckm{3}{2}\\\oplus\\\scrC_{\dM}^{2,2}\oplus\scrC^{2,1}_{\dM}}\longrightarrow \mymat{\Ckm{4}{2}\\\oplus\\\scrC_{\dM}^{3,2}\oplus\scrC^{3,1}_{\dM}}
\\&\bD_{4}=\begin{pmatrix}
\dBianchi & 0 \\ \PtD & -d_{\gD} \oplus -d_{\gD}
\end{pmatrix}:\mymat{\Ckm{4}{2}\\\oplus\\\scrC_{\dM}^{3,2}\oplus\scrC^{3,1}_{\dM}}\longrightarrow \mymat{\Ckm{5}{2}\\\oplus\\\scrC_{\dM}^{4,2}\oplus\scrC^{4,1}_{\dM}}
\\&\bD_{5}=\cdots
\\& \cdots
\\& \bD_{\dim{M}}=0. 
\end{split}
\label{eq:curvature_complex}
\eeq
As in the Killing and Hessian equations \secref{sec:KillingHessianBody}, the appendix (\secref{sec:Appendix}) contains the required material to parse most of the operations here in context of Bianchi forms and the differential operators associated with them. 

There are however a few linearized systems appearing in \eqref{eq:curvature_complex} that we have not yet introduced in \secref{sec:prescribed_curvature_intro}.  These are the systems $\D\Gh_{\gD,\Ah_{g}}$ and $\D\MCh_{\gD,\Ah_{g}}$, which denote the linearizations of the corresponding nonlinear maps:
\[
\begin{split}
&(h, \mathrm{K}) \mapsto \Gh_{h, \mathrm{K}}: \scrM_{\dM} \times \scrC^{1,1}_{\dM} \to \scrC^{2,2}_{\dM}, 
\\
&(h, \mathrm{K}) \mapsto \MCh_{h, \mathrm{K}}: \scrM_{\dM} \times \scrC^{1,1}_{\dM} \to \scrC^{2,1}_{\dM},
\end{split}
\]
defined by
\beq
\Gh_{h, \mathrm{K}} = \Rm_{h} + \frac{1}{2}\mathrm{K} \wedge \mathrm{K}, \qquad
\MCh_{h, \mathrm{K}} = d_{h} \mathrm{K}.
\label{eq:GMC_intro}
\eeq

These maps respect gauge invariance with respect to $\varphi:\dM \to \dM$:  
\beq
\Gh_{\varphi^*h, \varphi^*\mathrm{K}} = \varphi^* \Gh_{h, \mathrm{K}}, \qquad 
\MCh_{\varphi^*h, \varphi^*\mathrm{K}} = \varphi^* \MCh_{h, \mathrm{K}}.
\label{eq:gauge_GMC_intro}
\eeq

Thus, reformulating the constraints \eqref{eq:contraints_Rm_intro} in terms of the data $(T; \rho, \tau)$, we obtain:
\beq
\begin{gathered}
d_g T = 0, \\ 
\PtD{T}-(\varphi^* \Gh_{\rho, \tau},\varphi^* \MCh_{\rho, \tau}) = 0. 
\end{gathered}
\label{eq:data_Rm_required}
\eeq

\subsection{Variation formulae} 
Here we derive several variation formulae that are needed to establish how the systems in \eqref{eq:curvature_complex_DD}---\eqref{eq:curvature_complex} fall into the template in \eqref{eq:A_pallete}.

We first recall the well-known variation formula referenced in \secref{sec:prescribed_curvature_intro}, which can be written as (cf. \cite[p.~560]{Tay11b})
\beq  
\D\Rm_g = \frac{1}{2}(\Hg - D_g).
\label{eq:Rm_variation}  
\eeq
where, for completeness, we recall also that:
\[
D_{g}\sigma=\frac{1}{2}(\trace_{g}(\Rm_{g}\wedge\sigma)-\trace_{g}\Rm_{g}\wedge\sigma-\Rm_{g}\wedge\trace_{g}\sigma).
\]
Next, we obtain a variation formula for the second fundamental form. To that end, it is known that near the boundary, $\Ah_{g} = \nabg \nu_{g}$, where $\nabg$ is the Levi-Civita covariant derivative induced by $g$, and $\nu_{g}=dr \in \Omega^{1}_{M}|_{\dM}$ is the induced normal $1$-form to the boundary. This $1$-form is related to the inward-pointing unit normal vector to the boundary, $n_{g} \in \frakX_{M}|_{\dM}$.

Consider also the operator $\calS_{g}:\Ckm{1}{1}\rightarrow\Ckm{1}{1}$:
\[
\calS_{g}\sigma(X;Y)=\sigma(\nabg_{X}n_{g};Y),
\]
and define the operator $S_{g} : \scrC^{1,1}_{\dM} \oplus \scrC^{1,1}_{\dM} \to \scrC^{1,1}_{\dM} \oplus \scrC^{1,1}_{\dM}$ as:
\beq
S_{g}(\rho, \tau) = (\rho, \frac{1}{2}(\tau + \frac{1}{2}\calS_g \rho + \frac{1}{2}(\calS_g \rho)^{T}))
\label{eq:Sg_prescribed}
\eeq
and note that by construction $S_{g}^*=S_{g}$. 
\begin{proposition}
\label{prop:variation_normal}
The following variation formulae hold:
\beq
\begin{split}
&(\dertZero n_{g+t\sigma})^{\flat_{g}} = -\frac{1}{2} \PnnG \sigma \, \nu_{g} - \PntG \sigma, \\
& \dertZero \nu_{g+t\sigma}=\frac{1}{2} \PnnG\sigma \, \nu_{g},\\ 
&(\PttG\sigma,\D\Ah_{g}\sigma) = S_{g}(\PttG \sigma, \frakT_{g} \sigma).
\end{split}
\label{eq:variation_normal_second_form}
\eeq
\end{proposition}

The proof is technical and is provided in the end of the chapter \secref{sec:techincal_curvature}. It is worth noting that similar variation formulae can also be found in, e.g., \cite{And08,AH22}. However, for completeness, we present the computations here within the framework of Bianchi forms.

For the other components in \eqref{eq:systems_curvature_intro}, recall the well-known deformation, or Killing, operator $\Def:\frakX_{M}\rightarrow \Ckm{1}{1}$ \cite[p.~155, Ch.~5.12]{Tay11a}:
\[
\Def{X}=\frac{1}{2}\calL_{X}g=\dBianchiV X^{\flat}.
\]
Also, let $\Defd : \frakX_{M}|_{\dM} \rightarrow \scrC_{\dM}^{1,1}$ be defined by
\[
\Defd{Y} = \frac{1}{2} \calL_{Y^{\parallel}} \gD = d_{\gD}^{V} \PttG Y^{\flat},
\]
where we identify $\frakX_{M}|_{\dM} \simeq \frakX_{\dM} \oplus \Gamma(N\dM)$ via the decomposition $X|_{\dM} = Y^{\parallel} + Y^{\bot} n_{g}$, with $N\dM$ denoting the normal bundle of $\dM$. Note that $\Defd$ is not quite the Killing operator on $(\dM, \gD)$, as it acts on the full restrictions of vector fields from the interior, which may in general include non-tangential components.
 
Next, we consider a special case of the variation formula in \propref{prop:variation_normal}, restricted to variations of $\Ah_{g}$ arising from Lie derivatives of the Riemannian metric:
\begin{corollary}
\label{corr:deformation_h}
Let $X \in \frakX_{M}$ and set $X|_{\dM}=Y\in\frakX_{M}|_{\dM}$. Let $Y^{\parallel} \in \frakX_{\dM}$ denote the tangential component on the boundary, and let $Y^{\bot} \in C^{\infty}_{\dM}$ denote the normal part with respect to a Riemannian metric $g \in \scrM_{M}$. Then, 
\beq
\dertZero \Ah_{g+t\calL_{X}g} = \calL_{Y^{\parallel}} \Ah_{g} - \mathrm{Hess}_{\gD} Y^{\bot} + Y^{\bot} \calS_{g} \Ah_{g}.
\label{eq:h_deformation}
\eeq
In particular, the operation $X \mapsto \dertZero \Ah_{g+t\calL_{X}g}$ depends solely on $X|_{\dM}=Y$.
\end{corollary}

With this given, we now turn our attention to the maps:
\[
\begin{split}
&\Defh : \frakX_{M}|_{\dM} \rightarrow \scrC^{1,1}_{\dM}, \\
&\D\Gh_{\gD,\Ah_{g}} : \scrC^{1,1}_{\dM} \oplus \scrC^{1,1}_{\dM} \rightarrow \scrC^{2,1}_{\dM}, \\
&\D\MCh_{\gD,\Ah_{g}} : \scrC^{1,1}_{\dM} \oplus \scrC^{1,1}_{\dM} \rightarrow \scrC^{2,1}_{\dM},
\end{split}
\]
defined by
\beq
\begin{split}
&\Defh{Y} = \dertZero \Ah_{g + t \Def{X}}, \qquad X|_{\dM} = Y, \\
&\D\Gh_{\gD,\Ah_{g}}(\rho, \tau) = \frac{1}{2}(H_{\gD} \rho - D_{\gD} \rho) + \tau \wedge \Ah_{g}, \\
&\D\MCh_{\gD,\Ah_{g}}(\rho, \tau) = d_{\gD} \tau + \dertZero d_{\gD + t\rho} \Ah_{g},
\end{split}
\label{eq:boundary_section_prescribed}
\eeq
where $\Defh$ is well defined due to \eqref{eq:h_deformation}. Moreover, directly from the variation formula for $\Rm_{\gD}$, the expressions for $\D\Gh_{\gD,\Ah_{g}}$ and $\D\MCh_{\gD,\Ah_{g}}$ coincide with the linearizations of the nonlinear maps in \eqref{eq:GMC_intro}.

Before continuing, we take a moment to use these variation formulae to produce a non-homogeneous version for \eqref{eq:linearized_prescribed_curvatureIntroII}.  Specifically, we consider the system:
\beq
\begin{gathered}
\Rm_g = \mathrm{T}, \\ 
\gD = \varphi^*h, \qquad 
\Ah_{g} = \varphi^*\mathrm{K}.
\end{gathered}
\label{eq:prescribed_curvatureNon}
\eeq
where $\varphi:\dM \rightarrow \dM$ is a variable boundary diffeomorphism and $g \in \scrM_{M}$. linearizing geometric constraints $\Gamma:\Ckm{1}{1}\rightarrow\Ckm{2}{2}$ and setting $Y = \dertZero \varphi_{t}$, we obtain the system for $(\sigma, Y) \in \Ckm{1}{1} \oplus \frakX_{M}|_{\dM}$:
\beq
\begin{gathered}
\D_{\Gamma}\Rm_g\sigma = T, \\ 
\PttD \sigma - \Defd Y = \rho, \qquad 
\D\Ah_{g} \sigma - \Defh Y = \tau.
\end{gathered}
\label{eq:prescribed_curvatureLinarizedNon}
\eeq
\subsection{Verification of the template} 
We now have all the necessary ingredients to formulate \eqref{eq:curvature_complex} within the template of \eqref{eq:A_pallete}, just as we did for the simpler examples of exterior covariant derivatives in \secref{sec:DD_exterior_intro}--\secref{sec:NN_exterior_intro}. 
\subsubsection{Neumann picture} 
For conciseness, we focus here exclusively on the first two segments of \eqref{eq:curvature_complex}. The analysis for the remaining segments proceeds along similar lines to the case of exterior covariant derivatives and poses no additional complications. 

We also note that when $\dim M = 3$, only the first two segments are relevant in any case, as $\Ckm{k}{m} = \BRK{0}$ for $k > \dim M$.
\begin{itemize}
\item For the operators in \eqref{eq:AD_pallete}, we recognize:
\[
A_{D,0} = \Def, \qquad 
A_{D,1} = \D_{\Gamma}\Rm_{g}, \qquad 
A_{D,2} = \dBianchi,
\]
where, more specifically:
\beq
\begin{aligned} 
&A_{0} = \Def, \quad &&A_{0}^* = \delBianchi, \quad &&D_{0} = 0, \quad &&B_{0} = |_{\dM}, \quad &&B_{0}^* = (\PnD\cdot)^{\sharp}, \\[0.5em]
&A_{1} = \Hg, \quad &&A_{1}^* = \Hg^*, \quad &&D_{1} = - D_{g}+\Gamma, \quad 
&&B_{1} = \PttG \oplus \frakT_{g}, \quad &&B_{1}^* =  \frakT^*_{g} \oplus -\PnnG, \\[0.5em]
&A_{2} = \dBianchi, \quad &&A_{2}^* = \delBianchi, \quad &&D_{2} = 0, \quad 
&&B_{2} = \PtD, \quad &&B_{2}^* = \PnD_{g}.
\end{aligned}
\label{eq:recongized_A_curvature}
\eeq
Note that for $\alpha = 1$, we obtain the required form of \eqref{eq:AD_pallete} due to \eqref{eq:Rm_variation}.

The fact that these systems satisfy the required Green's formulae in \eqref{eq:integration_example}, as well as the normality conditions, is due to basic facts about Bianchi forms found in \secref{sec:Appendix}. Specifically,  
\beq
\begin{aligned}
&\bra \Def{X},\sigma\ket = \bra X|_{\dM},\delBianchi\sigma\ket + \bra Y,(\PnD\sigma)^{\sharp_{g}}\ket, \\[0.5em]
&\bra \Hg\psi,\eta\ket = \bra\psi,\Hg^*\eta\ket + \bra \PttG\psi,\frakT_{g}^*\eta\ket - \bra \frakT_{g}\psi,\PnnG\eta\ket, \\[0.5em]
&\bra d_{g}\psi,\eta\ket = \bra \psi,\delta_{g}\eta\ket + \bra \PtD\psi,\PnD_{g}\eta\ket. 
\end{aligned}
\label{eq:Hg_Green}
\eeq
\item As for the trace operator \eqref{eq:trace_operator_pallete}, we set $M_{\alpha}=0$ uniformly. For the operators $S_\alpha$ required in \eqref{eq:trace_operator_pallete}, set: 
\[
S_{0} = \id, \qquad S_{1} = S_{g}, \qquad S_{2} = \id.
\]
So in \eqref{eq:curvature_complex} we recognize due to \eqref{eq:variation_normal_second_form}: 
\[
T_{0} = |_{\dM}, \qquad T_{1} = \PttG\oplus\D\Ah_{g}=S_{g}(\PttG \oplus \frakT_g), \qquad T_{2} = \PtD. 
\]
\item Finally, for the boundary operators, we recognize:
\[
Q_{K,0} = 0, \qquad 
Q_{K,1} = -\Defd \oplus -\Defh, \qquad 
Q_{K,2} = -\D\Gh_{\gD,\Ah_{g}} \oplus -\D\MCh_{\gD,\Ah_{g}}.
\]
More specifically, using the identifications:
\beq
\frakX_{M}|_{\dM} \simeq \frakX_{\dM} \oplus \Gamma(N\dM), \qquad 
\Gamma(N\dM) \simeq C^{\infty}_{\dM} \simeq \scrC^{0,0}_{\dM},
\label{eq:vector_field_decompostion_curvature}
\eeq
we find using the expressions for $\D\Gh_{\gD,\Ah_{g}}$, $\D\MCh_{\gD,\Ah_{g}}$, $\Defd$, and $\Defh$ from \eqref{eq:boundary_section_prescribed}, along with the variation formula in \corrref{corr:deformation_h}, that the above systems take the form:
\beq
\begin{aligned}
Q_{K,1}&=-\Defd \oplus -\Defh = \underbrace{(-\Defd \sqcup \mathrm{Hess}_{\gD})}_{:=Q_{1}} + K_{1}, \\[0.5em]
Q_{K,2}&=-\D\Gh_{\gD,\Ah_{g}} \oplus -\D\MCh_{\gD,\Ah_{g}} = \underbrace{(- H_{\gD} \sqcup -2d_{\gD})}_{:=Q_{2}} + K_{2},
\end{aligned}
\label{eq:systems_curvature_overdetermined}
\eeq
where $K_{1}$ and $K_{2}$ are given by:
\[
\begin{aligned}
K_{1}(Y^{\parallel},Y^{\bot}) &= (0, \calL_{Y^{\parallel}} \Ah_{g} + Y^{\bot} \calS_{g} \Ah_{g}), \\[0.5em]
K_{2}(\rho, \tau) &= (\frac{1}{2}D_{\gD} \rho - \tau \wedge \Ah_{g}, \dertZero d_{\gD + t \rho} \Ah_{g} ).
\end{aligned}
\]
\end{itemize}

We take a moment to verify that the systems in the last item indeed satisfy the required properties in \eqref{eq:QC_pallete}:

\begin{proposition}
\label{prop:boundary_compact_petrubation_prescribed}
The decomposition $Q_{K,\alpha} = Q_{\alpha} + K_{\alpha}$ falls into the form \eqref{eq:DplusC}.
\end{proposition}

\begin{proof}
Relabeling the exponents $t_{l}$ in \eqref{eq:strict_tuples0} as $t, s, t', s' \in \bbR$, the following mappings are identified as non-compact (corresponding to $\bD_{0}$ in \eqref{eq:DplusC}):
\beq
\begin{aligned}
\Defd &: W^{s+2,2} \frakX_{\dM} \rightarrow W^{s+1,2} \scrC^{1,1}_{\dM}, \\[0.5em]
\mathrm{Hess}_{\gD} &: W^{s'+2,2} \scrC^{0,0}_{\dM} \rightarrow W^{s',2} \scrC^{1,1}_{\dM}, \\[0.5em]
H_{\gD} &: W^{t+2,2} \scrC^{1,1}_{\dM} \rightarrow W^{t,2} \scrC^{2,2}_{\dM}, \\[0.5em]
d_{\gD} &: W^{t'+2,2} \scrC^{1,1}_{\dM} \rightarrow W^{t'+1,2} \scrC^{1,2}_{\dM}.
\end{aligned}
\label{eq:overdetermined_relevant_curvature}
\eeq

In contrast, the map that constitute $K_{1}$ and $K_{2}$ (corresponding collectively to $\mathfrak{K}$ in \eqref{eq:DplusC}) operate as:
\[
\begin{aligned}
(Y^{\parallel},Y^{\bot}) &\mapsto \calL_{Y^{\parallel}} \Ah_{g} + Y^{\bot} \calS_{g} \Ah_{g} : 
W^{s+2,2} \frakX_{\dM} \oplus W^{s'+2,2} \scrC^{0,0}_{\dM} \rightarrow W^{\min(s+1, s'+2),2} \scrC^{1,1}_{\dM}, \\[0.5em]
\tau &\mapsto \tau \wedge \Ah_{g} : 
W^{t'+2,2} \scrC^{1,1}_{\dM} \rightarrow W^{t'+2,2} \scrC^{2,2}_{\dM}, \\[0.5em]
\rho &\mapsto \dertZero d_{\gD + t \rho} \Ah_{g} : 
W^{t+2,2} \scrC^{1,1}_{\dM} \rightarrow W^{t+1,2} \scrC^{1,1}_{\dM}.
\end{aligned}
\]
To align these mappings with the negligible component $\mathfrak{K}$ in \eqref{eq:DplusC}, we observe that the compactness of the mapping above, when compared to the non-compact components in \eqref{eq:overdetermined_relevant_curvature}, depends on the choices of $t, s, t', s'$. 

Specifically, in our setting, we require the following inclusions to be compact:
\[
\begin{split}
&W^{\min(s+1, s'+2),2} \scrC^{1,1}_{\dM} \hookrightarrow W^{s',2} \scrC^{1,1}_{\dM}, \\[8pt]
&W^{t'+2,2} \scrC^{2,2}_{\dM} \hookrightarrow W^{t,2} \scrC^{1,1}_{\dM}, \\[8pt]
&W^{t+1,2} \scrC^{1,1}_{\dM} \hookrightarrow W^{t'+1,2} \scrC^{1,1}_{\dM}.
\end{split}
\]
By choosing $s+1 > s'$ and $t = t'+1$, the required inclusions hold, ensuring that $K_{1}$ and $K_{2}$ define compact perturbations to the mappings in \eqref{eq:overdetermined_relevant_curvature}.
\end{proof} 
We proceed with the verification of the order-reduction properties required in \eqref{eq:order_reduction_conditions_examples}:
\begin{proposition}
\label{prop:reduction_curvature}
When \eqref{eq:curvature_complex} is cast into the template \eqref{eq:A_pallete}, it satisfies the algebraic order-reduction properties required in \eqref{eq:order_reduction_conditions_examples}.
\end{proposition}

As promised in \secref{sec:prescribed_curvature_intro}, and specifically in \eqref{eq:order_reduction_curvature}, the proof relies on the geometric variation formulae established above, along with variation formulae for the differential Bianchi identity and the Gauss–Mainardi–Codazzi equations. Since the proof involves more than mere calculations and, as outlined in \eqref{eq:order_reduction_curvature}, emphasizes the key idea that linearizing geometric constraints yields the order-reduction property, we include it here:
\begin{proof}
As explained earlier, we focus on the segments $\alpha \leq 2$ in the diagram. Again, establishing the order reduction properties for the higher segments in \eqref{eq:curvature_complex} follow the same lines as the example of exterior covariant derivatives. Also, for these properties, it suffices to assume without loss of generality that $\Gamma=0$ since $\Gamma$ is tensorial. 

The order-reduction properties for the $\alpha = 0$ segment are trivially satisfied since the preceding segment consists of zero maps.

\subsubsection{The $\alpha = 1$ segment.} 
Here, $m_{1} = 1$ (the order of $\Def$). The order-reduction properties, as required correspondingly by \eqref{eq:order_reduction_conditions_examples}, amount to the following:

\begin{enumerate}
\item[\emph{(i)}] The linearized version of the invariance of the mapping $g \mapsto \Rm_{g}$ under pullback by diffeomorphisms elaborated upon in \eqref{eq:gauge_group_Rm} is evidently: 
\[
\D\Rm_{g} \Def{X} = \frac{1}{2}\calL_{X} \Rm_{g}, \qquad \text{(order 1)},
\]
Note that this identity holds for every $X \in \frakX_{M}$, not only for vector fields $X$ tangent to the boundary (i.e., those generating one-parameter groups of diffeomorphisms). This can be justified either through an approximation argument in $L^{2}$, where $X$ is approximated by compactly supported vector fields, or by performing an explicit linear-level computation.

\item[\emph{(ii)}] Since $\bD_{0}$ in \eqref{eq:curvature_complex} does not involve operations between boundary sections, there is nothing to verify.

\item[\emph{(iii)}] Similarly, by linearizing the gauge invariance \eqref{eq:boundary_symmetries_intro}, along with the variation formulae in \propref{prop:variation_normal} and \corrref{corr:deformation_h}, we find by letting $X|_{\dM}=Y$:
\[
(\PttG \oplus \D\Ah_{g}) \Def{X} - (\Defd{Y}, \Defh{Y}) = (Y^{\bot} \Ah_{g}, 0), \qquad \text{(order 0, class 1)}.
\]
\end{enumerate}

\subsubsection{The $\alpha = 2$ segment.} 
Here, $m_{2} = 2$ (the order of $\Hg$).

\begin{enumerate}
\item[\emph{(i)}] By linearizing the differential Bianchi identity as in \eqref{eq:order_reduction_curvature}:
\[
d_{g+t\sigma} \Rm_{g+t\sigma} = 0,
\]
we deduce:
\[
\ord(\dBianchi \D\Rm_{g}) \leq 1.
\]

\item[\emph{(ii)}] For the boundary sections, note that both $\Defd$ and $\Defh$ are first-order differential operators. Thus, the maximum order $\max_{k}(q_{k,1}^l)$ in \eqref{eq:order_reduction_conditions_examples} is 1.  

As with $\D\Rm_{g}\Def{X}$ above, by linearizing the invariance relation \eqref{eq:gauge_GMC_intro} and using the explicit expression \eqref{eq:GMC_intro}, one finds:
\[
\D\Gh_{\gD,\Ah_{g}}(\Defd Y, \Defh Y) = \calL_{Y^{\parallel}} \Rm_{\gD} + \Defh Y \wedge \Ah_{g}, \qquad \text{(order 1)}.
\]
To prove this identity, one can alternatively use the expansion of $\D\Gh_{\gD,\Ah_{g}}$ in \eqref{eq:boundary_section_prescribed}. A similar approach applies to $\D\MCh_{\gD,\Ah_{g}}$. For the sake of variety, let us do this explicitly using the variation formula for $\Defh$ in \corrref{corr:deformation_h}:

\[
\MCh_{\gD,\Ah_{g}}(\Defd Y, \Defh Y) = \dertZero d_{\gD+t\Defd Y} \Ah_{g} + d_{\gD} (\calL_{Y^{\parallel}} \Ah_{g} - \mathrm{Hess}_{\gD} Y^{\bot} + Y^{\bot} \calS_{g} \Ah_{g}).
\]
Using the naturality of the connection we find that the first expression on the right simplifies into: 
\[
\dertZero d_{\gD+t\Defd Y} \Ah_{g} + d_{\gD} \calL_{Y^{\parallel}} \Ah_{g} = \calL_{Y^{\parallel}} d_{\gD} \Ah_{g},
\]
which is a first-order differential operator in $Y$. Furthermore:
\[
\begin{aligned}
&d_{\gD} \mathrm{Hess}_{\gD} Y^{\bot} = \mathrm{R}_{\gD}dY^{\bot}, \\
&d_{\gD}(Y^{\bot} \calS_{g} \Ah_{g}) = dY^{\bot} \wedge \calS_{g} \Ah_{g} + Y^{\bot} d_{\gD} \calS_{g} \Ah_{g},
\end{aligned}
\]
are also first-order differential operations in $Y$. All in all:
\[
\ord(\D\Gh_{\gD,\Ah_{g}}\circ(\Defd \oplus \Defh )) \leq 1, \qquad \ord(\D\MCh_{\gD,\Ah_{g}}\circ(\Defd \oplus \Defh )) \leq 1
\]

\item[\emph{(iii)}] Finally, the remaining conditions follow by linearizing the Gauss-Mainardi-Codazzi equations for a variation $g + t\sigma$, which reads as in \eqref{eq:GMC_intro}: 
\[
\begin{aligned}
&\PtD\Rm_{g+t\sigma}-(\Rm_{\gD+t\PttG\sigma} + \frac{1}{2}\Ah_{g+t\sigma} \wedge \Ah_{g+t\sigma}, d_{\gD+t\PttG\sigma} \Ah_{g+t\sigma}) = 0.
\end{aligned}
\]
Using \propref{prop:variation_normal}, the definition of $S_{g}$ in \eqref{eq:Sg_prescribed}, and the variation formulae for $n_{g}$, $\Rm_{g}$, and $d_{g}$, we find, as outlined in \eqref{eq:order_reduction_curvature},
\[
\begin{aligned}
&\ord(\PtD \D\Rm_{g} - (\D\Gh_{\gD,\Ah_{g}} \circ (\PttG \oplus \D\Ah) \oplus \D\MCh_{\gD,\Ah_{g}} \circ (\PttG \oplus \D\Ah))) = 0, \\
&\mathrm{class}(\PtD \D\Rm_{g} - (\D\Gh_{\gD,\Ah_{g}} \circ (\PttG \oplus \D\Ah) \oplus \D\MCh_{\gD,\Ah_{g}} \circ (\PttG \oplus \D\Ah))) = 1.
\end{aligned}
\]

\end{enumerate}
\end{proof}
With the required Green's formulae and order-reduction properties for \eqref{eq:curvature_complex} established, it remains to establish the Neumann overdetermined ellipticities in \eqref{eq:overdetermined_ellipticies_examples_refined}. These become by order of appearance, by direct comparison with the recognized systems in \eqref{eq:recongized_A_curvature} and \eqref{eq:systems_curvature_overdetermined}:
\begin{proposition} 
\label{prop:curvature_overdetermined}
The following systems are overdetermined elliptic:
\beq
\begin{split}
&\begin{pmatrix} \Def & 0 \\ 0 & 0 \end{pmatrix}, \\[6pt]
&\begin{pmatrix} \Hg \oplus \delBianchi & 0 \\ (\PnD_{g})^{\sharp_{g}} & 0 \end{pmatrix} 
\qquad \text{and} \qquad 
\begin{pmatrix} 0 & 0 \\ 0 & \Defd \sqcup \mathrm{Hess}_{\gD} \end{pmatrix}, \\[6pt]
&\begin{pmatrix} \dBianchi \oplus \Hg^* & 0 \\ \frakT_{g}^* \oplus -\PnnG & 0 \end{pmatrix} 
\qquad \text{and} \qquad 
\begin{pmatrix} 0 & 0 \\ 0 & (H_{\gD} \sqcup d_{\gD}) \oplus (\delta_{\gD} \sqcup \mathrm{Hess}_{\gD}^*) \end{pmatrix}.
\end{split}
\label{eq:overdetermined_ellipticity_curvature}
\eeq
\end{proposition} 
\begin{proof}
To analyze these systems, we first apply the musical isomorphism and identify $\frakX_{M}$ with $\Ckm{0}{1}$. Under this identification, the operators $\Def$ and $\Defd$ correspond respectively to the Bianchi derivatives $\dBianchi$ and $d_{\gD}$, while the Hessian and its adjoint correspond to $H_{\gD}$ and $H_{\gD}^*$ when acting on $\Ckm{0}{0}$ and $\Ckm{1}{1}$, respectively.

By unraveling the disjoint union structure, we find that the overdetermined ellipticity of the systems in \eqref{eq:overdetermined_ellipticity_curvature} is equivalent (up to signs and scalar factors) to that of the following model systems:
\beq
\begin{aligned}
&\begin{pmatrix} \dBianchi & 0 \\ 0 & 0 \end{pmatrix}, \\[6pt]
&\begin{pmatrix} \Hg \oplus \delBianchi & 0 \\ \PnD_{g} & 0 \end{pmatrix}, 
\qquad 
\begin{pmatrix} 0 & 0 \\ 0 & d_{\gD} \end{pmatrix}, 
\qquad 
\begin{pmatrix} 0 & 0 \\ 0 & H_{\gD} \end{pmatrix}, \\[6pt]
&\begin{pmatrix} \dBianchi \oplus \Hg^* & 0 \\ \frakT_{g}^* \oplus \PnnG & 0 \end{pmatrix}, 
\qquad 
\begin{pmatrix} 0 & 0 \\ 0 & H_{\gD} \oplus \delta_{\gD} \end{pmatrix}, 
\qquad 
\begin{pmatrix} 0 & 0 \\ 0 & d_{\gD} \oplus H_{\gD}^* \end{pmatrix},
\end{aligned}
\eeq
with domains and codomains understood from context.

The verification that these systems satisfy the Lopatinski–Shapiro condition (cf. \thmref{thm:lopatnskii_shapiro}) is exactly the content of \propref{prop:overdeteremined_ellipticityBianchi}. 


\end{proof}
Each level thus gives rise to its own set of lifted operators and cohomological formulations for the associated boundary value problems, as described in \secref{sec:outline_template}. The most relevant of these theorems to the linearized prescribed curvature problem \eqref{eq:linearized_prescribed_curvatureIntroII} is that for $\alpha = 1$. In this case, the characterization of $\scrN(\bD^*_{\alpha},\bB^*_{\alpha})$ in \eqref{eq:N0_apps} reads as 
\[
\scrN(\bD^*_{0},\bB^*_{0}) = \BRK{(\sigma,(\PnD_{g}\sigma)^{\sharp_{g}}) \in \Ckm{1}{1} \oplus \frakX_{M}|_{\dM} ~:~ \sigma \in \ker\delBianchi}.  
\]
\thmref{thm:compatibility_nn} for $\alpha = 1$ then can be interpreted as the analog of \thmref{thm:curvature_prescription_intro} for the non-homogeneous version of \eqref{eq:prescribed_curvatureLinarizedNon}, and assumes the following form (with the lifted operators implied):  
\begin{theorem}
\label{thm:curvature_prescription_introNN}
Given $T \in \Ckm{2}{2}$ and $(\rho, \tau) \in \plCkm{1}{1} \oplus \plCkm{1}{1}$, 
the boundary-value problem \eqref{eq:linearized_prescribed_curvatureIntroII} admits a solution $(\sigma, Y) \in \Ckm{1}{1} \oplus \frakX_{M}|_{\dM}$ satisfying
\[
\begin{aligned}
& \delBianchi \sigma = 0,
&& \text{in } M, \\[6pt]
& (\PnD_{g}\sigma)^{\sharp_{g}} = Y,
&& \text{on } \dM,
\end{aligned}
\]
if and only if
\[
\mathpzc{d}_{g} T = \mathpzc{K}_{\,\,g}(\rho, \tau), 
\qquad 
\mathpzc{P}^{\frakt\frakt}T = \mathpzc{G}_{\gD,\Ah_{g}}(\rho, \tau),
\qquad 
\mathpzc{P}^{\frakt\frakn}T = \mathpzc{MC}_{\gD,\Ah_{g}}(\rho, \tau),
\qquad 
(T; \rho, \tau) \perp_{L^{2}} \module_{\N}^{2}.
\]
The solution $(\sigma, Y)$ is unique modulo an element of $\module_{\N}^{1}$.
\end{theorem}

\subsubsection{The Dirichlet picture}
For the Dirichlet pre-complex \eqref{eq:curvature_complex_DD}, the construction is even easier, since it fits right into the prototypical case \eqref{sec:dirichelt_special_case}, and at this stage, it can be readily understood from the outline in \secref{sec:prescribed_curvature_intro}. In particular, the algebraic order-reduction properties required in \defref{def:elliptic_pre_complexTem} for the first two segments of \eqref{eq:curvature_complex_DD} follow directly from the calculations in \eqref{eq:order_reduction_gauge} and \eqref{eq:order_reduction_curvature}. As before, the validity of the order-reduction properties for the remaining segments closely resembles the case of exterior covariant derivatives and introduces no additional analytical difficulties.

The only remaining point to address is the required overdetermined ellipticities in \defref{def:elliptic_pre_complexTem}, which, by order of appearance, can be verified through direct comparison with the recognized systems in \eqref{eq:recongized_A_curvature} and \eqref{eq:systems_curvature_overdetermined}.
\begin{proposition} 
\label{prop:curvature_overdeterminedDD}
The following systems are overdetermined elliptic:
\beq
\Def\oplus(\cdot)|_{\dM},\qquad \Hg\oplus\delBianchi\oplus\PttD\oplus\frakT_{g}, \qquad \dBianchi\oplus\Hg^*\oplus \PtD. 
\label{eq:overdetermined_ellipticity_curvatureDD}
\eeq
\end{proposition} 
The verification that these systems satisfy the Lopatinski–Shapiro condition (cf. \thmref{thm:lopatnskii_shapiro}) is, again, exactly the content of \propref{prop:overdeteremined_ellipticityBianchi}.


\subsection{Technical proofs}
\label{sec:techincal_curvature}
\begin{PROOF}{\propref{prop:variation_normal}}
The variation formula for $n_{g}$ and $\nu_{g}$ simply follow by linearizing the relations:
\[
g(n_g;n_{g})=1 \qquad \nu_{g}(X)=g(n_{g};X). 
\]
%

For the variation formula for $\Ah_{g}$, since $\Ah_{g+t\sigma}$ is an element in $\scrC^{1,1}_{\partial M}$ for any $\sigma$, it suffices  evaluating for tangent $X,Y\in\frakX_{M}|_{\partial M}$: 
\[
\begin{split}
\dertZero\Ah_{g+t\sigma}(X;Y)&=\dertZero(g+t\sigma)(\nabla_{X}^{g+t\sigma}n_{g+t\sigma};Y)
\\&=\sigma(\nabg_{X}n_{g};Y)+g(\dertZero\nabla_{X}^{g+t\sigma}n_{g};Y)+g(\nabla_{X}^{g}\dertZero n_{g+t\sigma};Y).
\end{split}
\] 
by definition of $\calS_{g}$, and since $X,Y$ are tangent:
\[
\sigma(\nabg_{X}n_{g};Y)=(\calS_{g}\PttD\sigma)(X;Y).
\]
As for the second term, using the well known variation formula for the Levi-civita connection \cite[p.~559]{Tay11b}:
\[
\begin{split}
g(\dertZero\nabla_{X}^{g+t\sigma}n_{g},Y)&=\frac{1}{2}(\nabg_{X}\sigma)(n_{g};Y)+(\frac{1}{2}(\nabg_{n_{g}}\sigma)(X;Y)-\frac{1}{2}(\nabg_{Y}\sigma)(X;n_{g}))
\\&=\frac{1}{2}X(\sigma(n_{g},Y))-\frac{1}{2}\sigma(\nabg_{X}n_{g};Y)-\frac{1}{2}\sigma(n_{g};\nabg_{X}Y)+\frac{1}{2}\dgV\sigma(X;n_g,Y)
\\&=\frac{1}{2}\nabla^{\jmath^*g}_{X}\PntG\sigma(Y)-\frac{1}{2}(\calS_{g}\PttD\sigma)(X;Y)+\frac{1}{2}\PnnG\sigma\,\Ah_{g}(X;Y)+\frac{1}{2}\PtnD\dgV\sigma(X;Y)
\\& =\frac{1}{2}d_{\nabla^{\jmath^*g}}\PntG\sigma(X;Y)-\frac{1}{2}(\calS_{g}\PttD\sigma)(X;Y)+\frac{1}{2}\PnnG\sigma\,\Ah_{g}(X;Y)+\frac{1}{2}\PtnD\dgV\sigma(X;Y) 
\end{split}
\]
where the definition of exterior covariant derivation as operating on double forms was used (note that $\PntG\sigma\in\Omega^{0,1}_{\partial M}$), along with the fundamental relation: 
\[
\nabg_{X}Y =\nabla^{\jmath^*g}_{X}Y-\Ah_{g}(X;Y)\,n_{g}.
\] 
Finally, expanding the third term in the main calculation:
\[
g(\nabla_{X}^{g}\dertZero n_{g+t\sigma};Y)=-(\nabg_{X}\PntG\sigma)(Y)-\frac{1}{2}\nabg_{X}(\PnnG\sigma\nu_{g})(Y)=-(\nabla^{\jmath^*g}_{X}\PntG\sigma)(Y)-\frac{1}{2}\PnnG\sigma\,\Ah_{g}(X;Y)
\]
where we used the fact that $\PntG\sigma$ has no normal components and the Leibniz rule together with $\nu_{g}(Y)=0$ since $Y$ is tangent. 

By combining these calculations we find that the $\PnnG\sigma\,\Ah_{g}$ terms cancel and we are left with: 
\[
\begin{split}
2\,\dertZero\Ah_{g+t\sigma}(X;Y)=(\calS_{g}\PttD\sigma)(X;Y)-d_{\nabla^{\jmath^*g}}\PntG\sigma(X;Y)+\PtnD\dgV\sigma(X;Y).
\end{split} 
\]
Since $\dertZero\Ah_{g+t\sigma}(X;Y)$ is symmetric, by symmetrization and comparing with the definition of $\frakT_{g}$ and $S_{g}$, we obtain the second identity in \eqref{eq:variation_normal_second_form}. 
\end{PROOF}
\begin{PROOF}{\corrref{corr:deformation_h}}
By setting $X|_{\partial M}=Y$, and using linearity:
\[
\dertZero \Ah_{g+t\calL_{X}g} = \dertZero \Ah_{g+t\calL_{Y^{\parallel} }g} + \dertZero \Ah_{g+t\calL_{Y^{\bot} n_{g}}g}.
\]
Since $Y^{\parallel} $ is tangent to $\partial M$, any extension of it generates a global flow $\varphi_{t}:M\rightarrow M$ restricting as $\varphi_{t}|_{\partial M}: \partial M \to \partial M$. As $\varphi^*_{t} g = g + t \calL_{Y^{\parallel} } g + o(t)$ \cite[p.~44]{Pet16}, it follows due to the naturality of the connection and $\Ah_{g}=\nabg\nu_{g}$ that: 
\[
\dertZero \Ah_{g+t\calL_{Y^{\parallel} }g} = \dertZero \Ah_{\varphi^*_{t}g} = \dertZero \varphi^*_{t} \Ah_{g} = \calL_{Y^{\parallel} }  \Ah_{g}
\]
Hence, in establishing \eqref{eq:h_deformation}, it remains to calculate 
\[
\dertZero \Ah_{g+t\calL_{Y^{\bot} n_{g}}g}. 
\]

Using the Leibniz rule for Lie derivatives \cite[p.~45]{Pet16} and the fact that $g(n_{g}, \cdot) = \nu_{g}$, we have:
\[
\begin{split}
\calL_{Y^{\bot} n_{g}} g &= dY^{\bot}  \otimes \nu_{g} + \nu_{g} \otimes dY^{\bot}  + Y^{\bot}  \calL_{n_{g}} g \\
&= dY^{\bot}  \otimes \nu_{g} + \nu_{g} \otimes dY^{\bot}  + 2 Y^{\bot}  \Ah_{g}.
\end{split}
\] 
We rewrite this in the language of wedge products and double forms:
\[
\calL_{Y^{\bot} n_{g}} g = dY^{\bot}  \wedge \nu_{g}^{T} + \nu_{g} \wedge (dY^{\bot} )^{T} + 2 Y^{\bot}  \Ah_{g}.
\]

This formulation enables us to use the formula established in the proof of \propref{prop:variation_normal}:
\[
2 \dertZero \Ah_{g+t\sigma} = \calS_{g} \PttD \sigma - d_{\nabla^{\jmath^*g}} \PntG \sigma + \PntG \dg \sigma.
\]
Then, by inserting in the identities:
\[
\PttD \nu_{g} = 0, \quad \PntG \nu_{g} = 1, \quad ddY^{\bot}  = 0, \quad \PttD \Ah_{g} = \Ah_{g}, \quad \PntG \Ah_{g} = 0, \quad \dg \nu_{g}^{T} = \Ah_{g},
\]
and invoking the commutation relations for wedge products of double forms and boundary projections, we deduce:
\[
\begin{aligned}
&\calS_{g} \PttD \calL_{Y^{\bot} n_{g}} g = Y^{\bot}  \calS_{g} \Ah_{g}, \\
&d^{\nabla^{\jmath^*g}} \PntG \calL_{Y^{\bot} n_{g}} g= \mathrm{Hess}_{\jmath^*g} Y^{\bot} , \\
&\PntG \dg \calL_{Y^{\bot} n_{g}} g = -\partial_{n_{g}} Y^{\bot}  \Ah_{g} - \PttD \dg (dY^{\bot} )^{T} + 2 \partial_{n_{g}} Y^{\bot}  \Ah_{g} + 2 Y^{\bot}  \PntG \dg \Ah_{g}.
\end{aligned}
\]

By expanding further we may find: 
\[
\PntG \dg \Ah_{g} = 0, \quad \PttD \dg (dY^{\bot} ) = \mathrm{Hess}_{\jmath^*g} Y^{\bot}  + \partial_{n_{g}} Y^{\bot}  \Ah_{g},
\]
Hence
\[
\dertZero \Ah_{g+t\calL_{Y^{\bot} n_{g}}g} = Y^{\bot}  \calS_{g} \Ah_{g} - \mathrm{Hess}_{\jmath^*g} Y^{\bot} .
\]
so combining everything yields \eqref{eq:h_deformation}.
\end{PROOF}
\appendix
\chapter{Survey on Bianchi forms}
\label{sec:Appendix}
This section is adapted from \cite[Sec.~5]{KL23}, with some notation adjustments to better suit the framework of this paper. Proofs available in \cite{KL23} are omitted. 
\section{The bundle of Bianchi covectors}
Let $(M,\g)$ be a $\dim{M}=d$ dimensional Riemannian manifold with smooth boundary. We denote by 
\[
\Lkm{k}{m} = \Lambda^k T^*M \otimes \Lambda^m T^*M
\]
the vector bundle of \emph{$(k,m)$-covectors} (i.e., $k$-covectors taking values in the bundle of $m$-covectors), and by 
\[
\LkmAll = \bigoplus_{k,m} \Lkm{k}{m}
\]
the graded vector bundle of \emph{double-covectors}. The bundle $\LkmAll$ is a graded algebra, endowed with a  graded wedge-product, 
\[
\wedge : \Lkm{k}{m} \times \Lkm{\ell}{n} \to \Lkm{k+\ell}{m+n},
\]
and a graded involution,
\[
(\cdot)^T : \Lkm{k}{m} \to \Lkm{m}{k},
\]
obtained by switching the form and vector parts.
A $(k,k)$-covector $\psi$ satisfying $\psi^T = \psi$ is called \emph{symmetric}. The vector bundle $\LkmAll$ is equipped with a graded \emph{Hodge-dual} isomorphism,
\[
\starG : \Lkm{k}{m} \to \Lkm{d-k}{m},
\]
defined by its action on the form part. To every operation on the form part corresponds an operation on the vector part, via   involution; in this case,
\[
\starGV : \Lkm{k}{m} \to \Lkm{k}{d-m},
\]
is defined by $\starGV\psi = (\starG\psi^T)^T$. Additional graded bundle maps are the \emph{interior products}
\[
i_X  : \Lkm{k}{m} \to  \Lkm{k-1}{m}
\Textand
i^V_X : \Lkm{k}{m} \to  \Lkm{k}{m-1},
\]
where $X$ is a tangent vector, $i_X$ is defined as usual via its action on the form part, and $i^V_X\psi = (i_X \psi^T)^T$, and the \emph{metric trace},
\[
\trace_\g : \Lkm{k}{m} \to  \Lkm{k-1}{m-1}
\qquad\text{defined by}\qquad
\trace_\g\psi = \sum_{i=1}^d i_{E_i} i_{E_i}^V\psi,
\]
where $\{E_i\}_{i=1}^d$ is an orthonormal basis. 

The \emph{Bianchi sum} $\G:\Lkm{k}{m}\to \Lkm{k+1}{m-1}$ is a smooth bundle map given by \cite{Kul72,Gra70},
\[
\G = \sum_{i=1}^d \vartheta^i\wedge i_{E_i}^V,
\]
where $\{\vartheta^i\}_{i=1}^d$ is the basis of covectors dual to $\{E_i\}_{i=1}^d$.
For $\psi\in \Lkm{k}{m}$ and $\eta\in \LkmAll$, the Bianchi sum satisfies the product rule 
\[
\G(\psi\wedge\eta) = \G\psi\wedge\eta + (-1)^{k+m}\psi\wedge \G\eta.
\]
The operator $\GV:\Lkm{k}{m}\to \Lkm{k-1}{m+1}$ is the smooth bundle map $\GV\psi=(\G\psi^T)^T$. The operators $\G$ and $\GV$ are mutually dual with respect to the fiber metric,
\[
(\G\psi,\eta)_\g = (\psi,\GV\eta)_\g.
\] 
The following algebraic commutation and anti-commutation relations are readily verifiable from the definitions:
\[
\begin{aligned}
&\Brk{\G,\GV}|_{\Lkm{k}{m}} = (k-m)\id \\ 
&[\G,\g\wedge] = 0 
&\quad
&[\GV,\g\wedge] = 0 \\
&[\G,\trace_\g] = 0 
&\quad
&[\GV,\trace_\g] = 0 \\
&\{\G,i_X\} = i_X^V 
&\quad 
&\{\G,i_X^V\} = 0 \\
&\{\GV,i_X^V\} = i_X 
&\quad
&\{\GV,i_X\} = 0,
\end{aligned}
\]
where $[A,B] = AB - BA$ and $\{A,B\} = AB+BA$. The tensorial operators $\G$, $\GV$, $\g\wedge$ and $\trace_\g$ are related via the Hodge duals $\starG$ and $\starG^V$.
The following orthogonal decompositions are established in \cite{Cal61},
\[
\LkmAll = \ker\G\oplus\image\GV = \ker\GV\oplus\image\G,
\]
with $\ker\G=\{0\}$ when $\G$ is restricted to $\Lkm{k}{m}$ for $k<m$ and $\ker\GV=\{0\}$  when $\GV$ is restricted to $\Lkm{k}{m}$ for $k>m$. 
That is, $\G$ is injective and $\GV$ is surjective on $\Lkm{k}{m}$ for $k<m$ and $\GV$ is injective and $\G$ is surjective on $\Lkm{k}{m}$ for $k>m$.

\begin{definition}
\label{def:Bianchi_forms}
We define the vector bundles of \emph{Bianchi $(k,m)$-covectors}, 
\[
\Gkm{k}{m} = \Cases{
\Lkm{k}{m}\cap\ker\GV & k\le m \\
\Lkm{k}{m}\cap\ker\G & k\ge m ,
}
\]
along with the graded \emph{bundle of Bianchi coverctors}, 
\[
\GkmAll = \bigoplus_{k,m=0}^d \Gkm{k}{m}.
\]
\end{definition}

For $k=m$, the kernels of $\G$ and $\GV$ coincide, and consist of symmetric double-covectors \cite[Prop.~2.2]{Gra70}. In particular, $\Gkm{1}{1}$ coincides with the bundle of symmetric $(1,1)$-covectors and $\Gkm{2}{2}$ is the bundle of $(2,2)$-covectors satisfying the \emph{algebraic Bianchi identities} (also known as \emph{algebraic curvature tensors}). 

We denote by $\calPG:\Lkm{k}{m}\to\Gkm{k}{m}$ the orthogonal projection of a double-covector on $\Gkm{k}{m}$; it has an explicit representation which will not be needed. Since $\GV\psi = (\G\psi^T)^T$, it follows that $\calPG$ commutes with the involution, i.e., $(\calPG\psi)^T = \calPG\psi^T$.

Let $\xi\in\Lkm10$. 
The operators $\ixi$ and $\psi\mapsto \xi\wedge\psi$, which are dual with respect to the fiber metric $(\cdot,\cdot)_\g$, can be restricted to Bianchi forms. Since the first commutes with $\GV$ and the second commutes with $\G$,
\[
\begin{aligned}
& \ixi : \Gkm{k}{m} \to\Gkm{k-1}{m} & \qquad\qquad & k\le m \\
& \xi\wedge : \Gkm{k}{m} \to\Gkm{k+1}{m} & \qquad\qquad & k\ge m.
\end{aligned}
\]
The Bianchi symmetry is however not preserved for arbitrary $k,m$. 
We introduce the \emph{Bianchi wedge-product} and the corresponding \emph{Bianchi interior product}:
\[
\calPG (\xi\wedge):\Gkm{k}{m}\to \Gkm{k+1}{m}
\Textand
\calPG \ixi:\Gkm{k}{m}\to \Gkm{k-1}{m}.
\]

For values of $k,m$ for which a projection is needed, we obtain the following explicit formulas:

\begin{proposition}
\label{prop:W_explicit}
Let $\psi\in\Gkm{k}{m}$. Then,
\[
\begin{aligned}
&\calPG(\xi\wedge\psi) =  \xi\wedge \psi - \frac{1}{\alpha(m,k)}\G(\xi_V\wedge\psi)
&\qquad\qquad & k<m  \\
&\calPG\ixi\psi = \ixi\psi - \frac{1}{\alpha(k,m)}\GV\ixiV\psi,
&\qquad\qquad & k>m ,
\end{aligned}
\]
where $\alpha(k,m) = k-m+1$.
\end{proposition}

\section{First-order differential operators}
\label{sec:first_order} 
We denote by $\Wkm{k}{m} = \Gamma(\Lkm{k}{m})$ the space of \emph{$(k,m)$-forms}, endowed with the inner-product
\beq
\bra\psi,\eta\ket= \int_M (\psi,\eta)_\g \,d\Volume_{g}.
\label{eq:L2_double_forms}
\eeq
All bundle maps defined on $\Lkm{k}{m}$ extend into tensorial operations on $\Wkm{k}{m}$. We denote by $\Ckm{k}{m} = \Gamma(\Gkm{k}{m})$ the space of  \emph{Bianchi $(k,m)$-forms}, and by 
\[
\CkmAll = \bigoplus_{k,m} \Ckm{k}{m}
\]
the graded space of \emph{Bianchi forms}.

We denote by
\[
\dg : \Wkm{k}{m} \to \Wkm{k+1}{m}
\Textand
\dgV : \Wkm{k}{m} \to \Wkm{k}{m+1}
\]
the exterior covariant derivative (defined in the same way as for any bundle-valued form) and 
its vectorial counterpart, $\dgV\psi = (\dg\psi^T)^T$. We denote by 
\[
\deltag : \Wkm{k+1}{m} \to \Wkm{k}{m}
\Textand
\deltagV : \Wkm{k}{m+1} \to \Wkm{k}{m}
\]
the respective formal $L^2$-adjoint of $\dg$ and $\dgV$, where 
$\deltagV\psi = (\deltag\psi^T)^T$.

These first-order operators satisfy the following commutation and anti-commutation relations with the tensorial operators:
\[
\begin{aligned}
&\{\dg,\g\wedge\} = 0 
&\quad
&\{\dgV,\g\wedge\} = 0 
&\quad
&\{\deltag,\g\wedge\} = -\dgV 
&\quad
&\{\deltagV,\g\wedge\} = - \dg \\
&\{\dg,\trace_\g\} = -\deltagV
&\quad
&\{\dgV,\trace_\g\} = -\deltag 
&\quad
&\{\deltag,\trace_\g\} = 0
&\quad
&\{\deltagV,\trace_\g\} = 0 \\ 
&\{\dg,\G\} = 0
&\quad
&\{\dgV,\G\} = \dg 
&\quad
&\{\deltag,\G\} = \deltagV
&\quad
&\{\deltagV,\G\} = 0 \\
&\{\dgV,\GV\} = 0
&\quad
&\{\dg,\GV\} = \dgV 
&\quad
&\{\deltagV,\GV\} = \deltag
&\quad
&\{\deltag,\GV\} = 0 \\
&[\dg,\starG^V] = 0 
&\quad
&[\dgV,\starG] = 0.
\end{aligned}
\]

The operators $\dg$ and $\deltag$ can be restricted to Bianchi forms.  Due to the commutation relations $\{\G,\dg\}=0$ and $\{\GV,\deltag\}=0$, 
\[
\begin{gathered}
\dg:\Ckm{k}{m}\to\Ckm{k+1}{m} \qquad\text{for $k\ge m$} \\
\deltag:\Ckm{k}{m}\to\Ckm{k-1}{m} \qquad\text{for $k\le m$}.
\end{gathered}
\]
The Bianchi symmetry is however not preserved by $\dg$ and $\deltag$ for every $(k,m)$-form. This can be rectified by projecting their image onto the Bianchi bundle.

\begin{definition}
The \emph{Bianchi derivative}, $\dBianchi:\Ckm{k}{m}\to\Ckm{k+1}{m}$, the \emph{Bianchi coderivative}, $\delBianchi:\Ckm{k+1}{m}\to\Ckm{k}{m}$, and their transposed counterparts, $\dBianchiV:\Ckm{k}{m}\to\Ckm{k}{m+1}$ and $\delBianchiV:\Ckm{k}{m+1}\to\Ckm{k}{m}$ are given by
\beq
\dBianchi\psi =  \calPG \dg\psi
\Textand
\delBianchi\psi =
\calPG \deltag\psi,
\label{eq:Bianchi_codifferential} 
\eeq
along with $\dBianchiV\psi =(\dBianchi\psi^T)^T$ and $\delBianchiV\psi = (\delBianchi\psi^T)^T$.
\end{definition}

The Bianchi derivative $\dBianchi$ and the Bianchi coderivative $\delBianchi$ (and likewise $\dBianchiV$ and $\delBianchiV$) are mutually adjoint with respect to the $L^2$-inner-product \eqref{eq:L2_double_forms}.

The following is proved in a similar way as \propref{prop:W_explicit}:

\begin{proposition}
\label{prop:Bianchi_derivative_explicit}
For $\psi\in\Ckm{k}{m}$,
\[
\begin{aligned}
& \dBianchi\psi =  
\dg\psi - \frac{1}{\alpha(m,k)}\G \dgV\psi
&\qquad\qquad
& k<m 
\\
&\delBianchi\psi = 
\deltag\psi - \frac{1}{\alpha(k,m)}\GV\deltagV\psi 
&\qquad\qquad
& k>m. 
\end{aligned}
\]
\end{proposition}

The fact that $\dg\dg$ is a tensorial operator yields the following:

\begin{proposition}
\label{prop:DstarDstar}
The maps $\dBianchi\dBianchi:\Ckm{k}{m}\to\Ckm{k+2}{m}$ and  $\delBianchi\delBianchi:\Ckm{k+2}{m}\to\Ckm{k}{m}$ are tensorial for every $k,m$, except when $k=m-1$.  
\end{proposition}

Let $\jmath:\dM\to M$ denote as before the inclusion map of the boundary.
We
introduce mixed projections of tangential and normal boundary components,
\[
\begin{aligned}
&\PttD:\Wkm{k}{m}\to \plWkm{k}{m}
&\qquad 
&\PntD:\Wkm{k}{m}\to \plWkm{k-1}{m} \\ 
&\PtnG:\Wkm{k}{m}\to \plWkm{k}{m-1} 
&\qquad 
&\PnnG:\Wkm{k}{m}\to \plWkm{k-1}{m-1}.
\end{aligned}
\]
The first superscript in $\frakt\frakt,\frakt\frakn,\frakn\frakt,\frakn\frakn$ refers to the projection of the form part, whereas the second superscript refers to the projection of the vector part. 
Specifically,
\[
\begin{gathered}
\PttD\psi = \jmath^*\psi
\qquad
\PntD\psi = \jmath^*\idr \psi 
\qquad
\PtnG\psi = \jmath^*\idr^T \psi
\textand
\PnnG\psi = \jmath^*\idr^T \idr \psi, 
\end{gathered}
\]
where $\partial_r$ is the unit vector field normal to the level-sets of the distance from the boundary, which is defined in a collar neighborhood of $\dM$, and $\jmath^*$ pulls back to the boundary both the form and vector parts. For $\psi\in\Wkm{k}{m}$ and $\eta\in\Wkm{k+1}{m}$,
\beq
\bra \dg\psi,\eta\ket = \bra\psi,\deltag\eta\ket + \bra (\PttD\oplus\PtnG)\psi, (\PntD\oplus\PnnG)\eta\ket.
\label{eq:dg_by_parts}
\eeq
The definition of the Bianchi sum implies that the pullback $\jmath^*$ commutes with both $\G$ and $\GV$. Furthermore, $\idr$ anti-commutes with $\GV$ and  $\idr^V$ anti-commutes with $\G$. A direct calculation gives the following commutation and anti-commutation relations,
\[
\begin{aligned}
&[\PttD,\G] = 0 
&\quad
&\{\PtnG,\G\} = 0 
&\quad
&\{\PntD,\G\} = \PtnG 
&\quad
&[\PnnG,\G] = 0 \\
&[\PttD,\GV] = 0 
&\quad
&\{\PtnG,\GV\} = \PntD 
&\quad
&\{\PntD,\GV\} = 0 
&\quad
&[\PnnG,\GV] =0 .
\end{aligned}
\] 
As a result,
\[
\begin{aligned}
& \PttD : \Ckm{k}{m} \to \plCkm{k}{m} &\qquad & \text{for every $k,m$} \\
& \PnnG :  \Ckm{k}{m} \to  \plCkm{k-1}{m-1} &\qquad & \text{for every $k,m$} \\
& \PtnG : \Ckm{k}{m}\to \plCkm{k}{m-1} &\qquad & \text{for $k\ge m$} \\
& \PntD : \Ckm{k}{m}\to \plCkm{k-1}{m} &\qquad & \text{for $k\le m$}.
\end{aligned}
\] 
For $k<m$, $\PtnG:\Ckm{k}{m}\to \plWkm{k}{m-1}$ does not yield a Bianchi form, since $\idr^V$ does not commute with $\GV$. The same is true for $\PntD:\Ckm{k}{m}\to \plWkm{k-1}{m}$ when $k>m$. In the same spirit as in formula \eqref{eq:Bianchi_codifferential} for the Bianchi derivatives, we define:

\begin{definition}
The \emph{Bianchi boundary operators} 
\[
\begin{aligned}
& \PttG : \Ckm{k}{m} \to \plCkm{k}{m} 
&\qquad\qquad
& \PnnG :  \Ckm{k}{m} \to  \plCkm{k-1}{m-1} \\
& \PtnG : \Ckm{k}{m}\to \plCkm{k}{m-1} 
&\qquad\qquad
& \PntG : \Ckm{k}{m}\to \plCkm{k-1}{m}
\end{aligned}
\] 
are given by, when there is no ambiguity: 
\[
\begin{gathered}
\PttG = \PttD
\qquad
\PntG = \calPG\PntD
\qquad
\PtnG = \calPG\PtnD
\textand
\PnnG = \PnnD,
\end{gathered}
\]
where $\calPG:\plLkm{k}{m}\to \plLkm{k}{m}$ denotes here the projection on Bianchi boundary forms. 
\end{definition}

Similarly to \propref{prop:Bianchi_derivative_explicit}, we have:

\begin{proposition}
For $\psi\in\Ckm{k}{m}$,
\[
\begin{aligned}
& \PtnG\psi = 
\PtnD\psi - \frac{1}{\alpha(m,k)} \G\PntD \psi
&\qquad\qquad
& k< m
\\
& \PntG\psi 
= \PntD\psi - \frac{1}{\alpha(k,m)} \GV\PtnD \psi
&\qquad\qquad
& k> m.
\end{aligned}
\]
\end{proposition}

\begin{proposition}
For all $\eta\in W^{1,p}\CkmAll$ and $\sigma\in W^{1,q}\CkmAll$ (the precise class determined by the context), with $1/p+1/q=1$, 
\beq
\bra\dBianchi\eta,\sigma\ket = \bra\eta,\delBianchi\sigma\ket +
\bra B_g \eta,B_g^* \sigma\ket,
\label{eq:integration_by_parts_DV}
\eeq
where 
\[
B_g = \PttG\oplus\PtnG
\Textand
B_g^* = \PntG\oplus\PnnG.
\]
\end{proposition}

\section{Second-order differential operators}
\label{sec:second_orderBianchi} 
As in \secref{sec:KillingHessianIntro}, we introduce the \emph{covariant curl-curl} operator, $H_{g}:\Wkm{k}{m}\to \Wkm{k+1}{m+1}$,
and its $L^2$-dual, $\Hg^*: \Wkm{k+1}{m+1}\to\Wkm{k}{m}$,
\[
H_{g}=\tfrac12(\dg\dgV+\dgV\dg) 
\Textand 
H_{g}^*=\tfrac12(\deltag\deltagV+\deltagV\deltag).
\]
These second-order operators satisfy integration by part formulas involving both tensorial and first-order boundary operators. 
We also defined the first-order boundary operators, 
\[
\frakT_{g} : \Wkm{k}{m}\to\plWkm{k}{m} 
\Textand
\frakT_{g}^* : \Wkm{k}{m}\to\plWkm{k-1}{m-1},
\]
given by 
\[
\begin{aligned}
\frakT_{g}\psi &= \tfrac12 \brk{\PntD \dg\psi-\dg \PntD\psi}+\tfrac12 \brk{\PtnG\dgV\psi- \dgV \PtnG\psi} \\
\frakT_{g}^*\psi &=  -\tfrac12\brk{ \PtnG\deltag\psi+\deltag \PtnG\psi} - \tfrac12\brk{\PntD\deltagV\psi+\deltagV\PntD\psi},
\end{aligned}
\]
such that 
\[
\bra H_g\psi,\eta\ket=\bra \psi,H_g^*\eta\ket +\bra B_H\psi,B_H^*\eta\ket ,
\]
where 
\[
B_H : \Wkm{k}{m}\to (\plWkm{k}{m})^2 
\Textand
B_H^* : \Wkm{k}{m}\to(\plWkm{k-1}{m-1})^2
\]h
are given by
\[
B_H=\PttD\oplus\frakT_{g}
\Textand
B_H^*=\frakT_{g}^*\oplus-\PnnG.
\] 

A direct verification shows that the operators $H_g$ and $H_g^*$ both commute with the Bianchi sums $\GV$, $\G$, which implies that for every $k,m$,
\[
H_g:\Ckm{k}{m}\rightarrow \Ckm{k+1}{m+1} 
\Textand
H_g^*:\Ckm{k+1}{m+1}\rightarrow \Ckm{k}{m}. 
\]
A similar calculation shows that the boundary operators also preserve the Bianchi structure:
\[
B_H : \Ckm{k}{m}\to (\plCkm{k}{m})^2 
\Textand
B_H^* : \Ckm{k}{m}\to(\plCkm{k-1}{m-1})^2.
\]

The fact that $B_H$ and $B_H^*$ are normal systems of trace operators of order $2$ follows from a direct expansion of the above expressions. One observes that the leading normal derivatives of $\frakT_{g}$ are $\PttD\nabg_{\dr}$, while those of $\frakT_{g}^*$ are $\PnnD_{g}\nabg_{\dr}$; thus, the systems are normal. The detailed computation is left to the reader.


\backmatter

\bibliographystyle{amsrefs}
\newcommand{\etalchar}[1]{$^{#1}$}
\providecommand{\bysame}{\leavevmode\hbox to3em{\hrulefill}\thinspace}
\providecommand{\MR}{\relax\ifhmode\unskip\space\fi MR }
\providecommand{\MRhref}[2]{%
\href{http://www.ams.org/mathscinet-getitem?mr=#1}{#2}
}
\providecommand{\href}[2]{#2}

\end{document}